\documentclass[10pt,reqno]{amsart} 

\usepackage[utf8]{inputenc} 

\usepackage[margin=1in]{geometry} 
\usepackage{graphicx} 
\usepackage{booktabs} 
\usepackage{array} 
\usepackage{paralist} 
\usepackage{verbatim} 
\usepackage{subfig} 
\usepackage{mathrsfs}
\usepackage{amssymb}
\usepackage{amsthm}
\usepackage{amsmath,amsfonts,amssymb,esint}
\usepackage{graphics}
\usepackage{enumerate}
\usepackage{mathtools}
\usepackage{xfrac}

\usepackage{fancyhdr} 
\numberwithin{equation}{section}

\def\Xint#1{\mathchoice
{\XXint\displaystyle\textstyle{#1}}%
{\XXint\textstyle\scriptstyle{#1}}%
{\XXint\scriptstyle\scriptscriptstyle{#1}}%
{\XXint\scriptscriptstyle\scriptscriptstyle{#1}}%
\!\int}
\def\XXint#1#2#3{{\setbox0=\hbox{$#1{#2#3}{\int}$ }
\vcenter{\hbox{$#2#3$ }}\kern-.6\wd0}}

\def\dashint{\Xint-}

\newtheorem{theorem}{Theorem}[section]
\newtheorem{corollary}[theorem]{Corollary}
\newtheorem{proposition}[theorem]{Proposition}
\newtheorem{lemma}[theorem]{Lemma}
\newtheorem{definition}[theorem]{Definition}
\theoremstyle{definition}
\newtheorem{remark}[theorem]{Remark}

\newcommand{\norm}[1]{\left\|#1\right\|}

\newcommand*{\supp}{\ensuremath{\mathrm{supp\,}}}
\newcommand*{\dist}{\ensuremath{\mathrm{dist\,}}}
\newcommand*{\Id}{\ensuremath{\mathrm{Id}}}
\renewcommand*{\div}{\ensuremath{\mathrm{div\,}}}

\newcommand*{\N}{\ensuremath{\mathbb{N}}}
\newcommand*{\T}{\ensuremath{\mathbb{T}}}
\newcommand*{\Tthreexi}{\ensuremath{\mathbb{T}_\xi^3}}
\newcommand*{\Z}{\ensuremath{\mathbb{Z}}}

\newcommand*{\R}{\ensuremath{\mathbb{R}}}

\newcommand{\eps}{\varepsilon}

\newcommand{\RR}{\mathring R}
\newcommand{\HH}{\mathring H}
\newcommand{\nn}{{\tilde{n}}}
\newcommand{\dpot}{{\mathsf d}}

\renewcommand*{\tilde}{\widetilde}
\renewcommand*{\hat}{\widehat}
\newcommand*{\curl}{\ensuremath{\mathrm{curl\,}}}

\newcommand{\cstar}{ \mathsf{c_0} }
\newcommand{\cstarn}{ \mathsf{c}_{\textnormal{n}} }
\newcommand{\cstarnprime}{ \mathsf{c}_{\textnormal{n}'} }
\newcommand{\cstarnn}{ \mathsf{c}_{\tilde{\textnormal{n}}} }

\newcommand{\cstarzero}{ \mathsf{c}_{0} }

\newcommand{\shaq}{{ -\mathsf{C_R}} }
\newcommand{\badshaq}{{ \mathsf{C_u} }}

\newcommand{\shaqqplusone}{ \Gamma_{q+1}^{-\mathsf{C_R}} }

\newcommand{\les}{\lesssim}
\newcommand{\imax}{{i_{\rm max}}}
\newcommand{\jmax}{{j_{\rm max}}}
\newcommand{\nmax}{{n_{\rm max}}}

\newcommand{\twopi}{2\pi}

\newcommand{\Ncut}{\mathsf{N}_{\rm cut}}

\newcommand{\NcutSmall}{\mathsf{N}_{\rm cut,t}}
\newcommand{\NcutLarge}{\mathsf{N}_{\rm cut,x}}
\newcommand{\NindSmall}{\mathsf{N}_{\textnormal{ind,t}}}
\newcommand{\NindLarge}{\mathsf{N}_{\textnormal{ind,v}}}
\newcommand{\NindRt}{\mathsf{N}_{\textnormal{ind,t}}}

\newcommand{\Nindvt}{\mathsf{N}_{\textnormal{ind,t}}}
\newcommand{\Nindt}{\mathsf{N}_{\textnormal{ind,t}}}

\newcommand{\Nindv}{\mathsf{N}_{\textnormal{ind,v}}}
\newcommand{\Nfin}{\mathsf{N}_{\rm fin}}
\newcommand{\Nfn}{\mathsf{N}_{\rm fin, n}}
\newcommand{\Nfnn}{\mathsf{N}_{\rm{fin,}\nn}}

\newcommand{\NN}[1]{\mathsf{N}_{#1}}
\newcommand{\CLebesgue}{{\mathsf{C}_{b}}}

\newcommand{\Ndec}{{\mathsf{N}_{\rm dec}}}
\newcommand{\WW}{\ensuremath{\mathbb{W}}}

\newcommand{\UU}{\ensuremath{\mathbb{U}}}
\newcommand{\Proj}{\ensuremath{\mathbb{P}}}
\newcommand{\MM}[1]{\ensuremath{\mathcal{M}}\left(#1\right)}

\newcommand{\rqnperptilde}{r_{q+1,\nn}}

\newcommand{\vlq}{v_{\ell_q}}
\newcommand{\vlqprime}{v_{\ell_{q'}}}
\newcommand{\vlqminus}{v_{\ell_{q-1}}}

\newcommand{\Pqx}{\mathcal{P}_{q,x}}

\newcommand{\Pqt}{\mathcal{P}_{q,t}}
\newcommand{\Pqxt}{\mathcal{P}_{q,x,t}}
\newcommand{\divH}{\mathcal{H}}
\newcommand{\divR}{\mathcal{R}^*}
\newcommand{\Dtq}{D_{t,q}}

\newcommand{\istar}{{i^*}}
\newcommand{\jstar}{{j^*}}
\newcommand{\kstar}{{k^*}}
\newcommand{\xistar}{{\xi^*}}

\newcommand{\nstar}{{n^*}}

\newcommand{\lstar}{l^*}
\newcommand{\wstar}{w^*}
\newcommand{\hstar}{h^*}
\newcommand{\LPqn}{\mathbb{P}_{[q,n]}}

\newcommand{\Phiik}{\Phi_{(i,k)}}

\newcommand{\pp}{{\tilde{p}}}

\newcommand{\const}{\mathcal{C}}

\newcommand{\qn}{_{q,n}}
\newcommand{\qnp}{_{q,n,p}}

\newcommand{\qplusnp}{_{q+1,n,p}}

\newcommand{\qnn}{_{q,\nn}}
\newcommand{\qplusnn}{_{q+1,\nn}}

\newcommand{\lessg}{\lesssim}

\newcommand{\LPqnmax}{\mathbb{P}_{\left[q,\nmax+1\right]}}

\newcommand{\Nsharp}{N^{\sharp}}

\def\aaa{{\boldsymbol{\alpha}}}
\def\bbb{{\boldsymbol{\beta}}}

\usepackage[usenames,dvipsnames]{color}
\usepackage[colorlinks=true, pdfstartview=FitV, linkcolor=blue, citecolor=blue, urlcolor=blue,pdfencoding=auto]{hyperref}

\title{An Intermittent Onsager Theorem}

\author{Matthew Novack}
\address{Department of Mathematics, Purdue University, West Lafayette, IN 47907.}
\email{\href{mdnovack@purdue.edu}{mdnovack@purdue.edu}}

\author{Vlad Vicol}
\address{Courant Institute of Mathematical Sciences, New York University, New York, NY 10012.}
\email{\href{vicol@cims.nyu.edu}{vicol@cims.nyu.edu}}

\begin{document}

\maketitle

\begin{abstract}
For any regularity exponent $\beta<\sfrac 12$, we construct non-conservative weak solutions to the 3D incompressible Euler equations in the class $C^0_t (H^{\beta} \cap L^{\sfrac{1}{(1-2\beta)}})$. By interpolation, such solutions belong to $C^0_t B^{s}_{3,\infty}$ for $s$ approaching $\sfrac 13$ as $\beta$ approaches $\sfrac 12$. Hence this result provides {\em a new proof of the flexible side of the $L^3$-based Onsager conjecture}. Of equal importance is that the intermittent nature of our solutions matches that of turbulent flows, which are observed to possess an $L^2$-based regularity index exceeding $\sfrac 13$. Thus our result does not imply, and is not implied by, the work of  Isett~\cite{Isett2018}, who gave a proof of the H\"older-based Onsager conjecture.  Our proof builds on the authors' previous joint work with Buckmaster and Masmoudi~\cite{BMNV21}, in which an intermittent convex integration scheme is developed for the 3D incompressible Euler equations. We employ a scheme with higher-order Reynolds stresses, which are corrected via a combinatorial placement of intermittent pipe flows of optimal relative intermittency. 
\end{abstract}

\setcounter{tocdepth}{1}
\tableofcontents
 
\allowdisplaybreaks

\section{Introduction}
\label{sec:introduction}
We consider  the three-dimensional homogeneous incompressible Euler equations
\begin{subequations}
\label{eq:Euler} 
\begin{align}
\partial_t v + \div(v \otimes v) +\nabla p &=0 \,,\\
\div v &= 0 \,.
\end{align}
\end{subequations}
Here $v(\cdot,t)\colon \T^3 \to \R^3$ is the velocity and $p(\cdot,t) \colon \T^3 \to \R$ is the pressure, and we consider the system \eqref{eq:Euler}  with periodic boundary conditions on $\T^3=[-\pi,\pi]^3$. Without loss of generality, the velocity is taken to have zero mean, 
and the pressure is uniquely determined as the zero mean solution of $- \Delta p = \div \div (v\otimes v)$. Smooth solutions $v$ of the 3D Euler equations  conserve their kinetic energy $\mathcal{E}(t) = \frac 12 \int_{\T^3} |v(x,t)|^2 dx$.

In this paper, we consider weak solutions $v\in C^0_t L^2$ to~\eqref{eq:Euler}. Since the Euler system is in divergence form  and  we consider velocity fields of finite kinetic energy, the definition of  weak  solutions is the usual one. The motivation for considering weak solutions is twofold. First, the Euler equations are expected to dynamically produce singularities, even from smooth initial conditions. Second, matching the mathematical theory with the physical properties of turbulent fluids necessitates the consideration of solutions with singularities. Indeed, the Kolmogorov/Onsager  theories of turbulence postulate that  solutions to the 3D incompressible Navier-Stokes equations, which represent a fully developed turbulent flow, exhibit anomalous dissipation of kinetic energy in the infinite Reynolds number limit. This is an experimental fact~\cite{Frisch95,EyinkSreeniviasan06}. Hence, if the 3D Euler equations are to represent the inertial range of turbulence at very large Reynolds numbers, one is forced to consider non-conservative solutions of \eqref{eq:Euler}, which thus must be weak solutions, not smooth ones.

The conservation of kinetic energy for weak solutions to~\eqref{eq:Euler}  was considered by Onsager~\cite{Onsager49}, who predicted that ``turbulent energy dissipation [...] could take place just as readily without the final assistance of viscosity [...] because the velocity field does not remain differentiable.'' Based on the computation of the energy flux through expanding Fourier domains, Onsager formulated a remarkable 
statement connecting the regularity of a weak solution $v$ to \eqref{eq:Euler} and the validity of the energy conservation law. Onsager's conjecture asserted that any weak solution $v \in C^0_t C^{s}$ with $s>\sfrac 13$ must conserve kinetic energy, whereas for any $s < \sfrac 13$ there exist dissipative weak solutions $v \in C^0_t C^s$ to the 3D Euler equations. The rigidity/flexibility dichotomy expressed by the Onsager conjecture is the mathematical manifestation of an experimental fact in hydrodynamic turbulence:   Kolmogorov's $\sfrac 45$-law regarding third order structure functions~\cite{Frisch95,EyinkSreeniviasan06}. 

Due to the quadratic nature of the nonlinearity in~\eqref{eq:Euler}, the Onsager exponent $\sfrac 13$ is intimately connected to an $L^3$-based regularity scale, such as $C^0_t B^s_{3,\infty}$, where we recall that the Besov norm is given by $\|v\|_{B^{s}_{p,\infty}} = \|v\|_{L^p} + \sup_{|z|>0} |z|^{-s} \| v(\cdot+z) - v(\cdot)\|_{L^p}$, so that $C^s(\T^3) \subset B^{s}_{3,\infty}(\T^3)$. Indeed, the rigidity part of the Onsager conjecture was established by Constantin-E-Titi~\cite{ConstantinETiti94}, who proved that any weak solution $v \in L^3_t B^{s}_{3,\infty} \cap C^0_t L^2_x$ of \eqref{eq:Euler} must conserve kinetic energy if $s>\sfrac 13$; see also the partial result~\cite{Eyink94} and the subsequent refinements in~\cite{DuchonRobert00,CCFS08,DrivasEyink19}. Concerning the flexible part of the Onsager conjecture, after the paradoxical constructions of Scheffer~\cite{Scheffer93} and Shnirelman~\cite{Shnirelman00}, a systematic approach towards the resolution of the conjecture was proposed in the groundbreaking works~\cite{DeLellisSzekelyhidi09,DeLellisSzekelyhidi13} of De~Lellis and Sz\'ekelyhidi~Jr., who introduced $L^\infty$-convex integration and $C^0$-Nash iteration schemes to fluid dynamics. After a series of important partial results~\cite{BDLISZ15,DaneriSzekelyhidi17}, a resolution of the flexible part of the Onsager conjecture was obtained by Isett~\cite{Isett2018} in the setting of weak solutions with {\em compact support in time}. This was further refined by Buckmaster, De Lellis, Sz\'ekelyhidi Jr., and the last author in~\cite{BDLSV17}, by constructing {\em dissipative} weak solutions $v\in C^0_t C^s$ to the 3D Euler equations, for any $s <\sfrac 13$. For a detailed account of the Onsager theory of ideal turbulence, and of the mathematical results which turned the Onsager conjecture into the {\em Onsager theorem}, we refer the reader to~\cite{EyinkSreeniviasan06,Shvydkoy10,DLSZ12,DLSZ17,BV_EMS19,BV20}.

We note that the proofs of rigidity in~\cite{ConstantinETiti94,DuchonRobert00,CCFS08,DrivasEyink19} identify the $L^3$-based spaces $B^{\sfrac 13+}_{3,\infty}$ and $B^{\sfrac 13}_{3,c_0}$, as the borderline regularity spaces for ensuring that weak solutions conserve energy/have vanishing energy flux. These spaces are known to be sharp, for instance in the case of a Burgers shock, which dissipates energy and lies in $B^{\sfrac 13}_{3,\infty}$. See also the incompressible 3D vector fields constructed in~\cite{Eyink94,CCFS08,CFLS16,cheskidovluo,BS22}, which have a nonzero flux at critical regularity. Moreover, the $L^3$-based regularity scale matches the prediction made for third order structure functions in the Kolmogorov theory of turbulence. 

In contrast, the proofs of flexibility in~\cite{Isett2018,BDLSV17,Isett17} are in a certain sense ``too strong,'' since they construct weak solutions in the $L^\infty$-based space $C^{\sfrac 13-}$ (which implies the same result in $B^{\sfrac 13-}_{3,\infty}$).  These solutions thus do not exhibit the observed inertial range intermittency of turbulent flows at large Reynolds number, neither for low order structure functions, nor for high order structure functions. To be more precise, for $p < 3$, the $p^{\rm th}$ order inertial range structure function exponents $\zeta_p$ in fully developed turbulence have consistently been observed to lie above the Kolmogorov predicted value of $\sfrac p 3$. See e.g.~\cite[Figure 8.8]{Frisch95},~\cite[Figures 4\&5]{ChenEtAl05},~\cite[Figure 3]{IshiharaEtAl09},~\cite[Figure 3]{ISY20}. These measurements correspond (see also~\cite{BV_EMS19,BMNV21} for details) to an $L^p$-based regularity exponent of $\sfrac{\zeta_p}{p} > \sfrac 13$. 
Similarly, for $p \gg 3$, experiments and simulations show that the inertial range structure function exponents  $\zeta_p$ saturate (meaning, remain bounded) as $p\to \infty$. See e.g.~\cite[Figure 8.8]{Frisch95},~\cite[Figure 6]{ISY20}, and the discussion in~\cite[Section D]{ISY20}. These measurements correspond to an $L^p$-based regularity exponent of $\sfrac{\zeta_p}{p} \to 0$ as $p\to \infty$, suggesting that the fully developed isotropic turbulent solutions observed in experiments do not retain any positive H\"older exponent, even though weak solutions of Euler may possess H\"older regularity. The culprit is {\em intermittency}.

The main goal of this paper is to give a {\em new proof of the flexible side of the $L^3$-based Onsager conjecture}. We construct weak solutions to the 3D Euler equation in the regularity class $C^0_t ( H^{\sfrac 12 -} \cap  L^{\infty-}) \subset C^0_t B^{\sfrac 13-}_{3,\infty}$, which are non-conservative and exhibit the inertial-range intermittency observed in turbulent flows.

\begin{theorem}[\bf Main result]
\label{thm:main}
Fix $\beta\in (0,\sfrac{1}{2})$. For any divergence-free $v_{\rm start}, v_{\rm end} \in L^2(\T^3)$ which have zero mean, any $T>0$ and any $\epsilon >0$, there exists a weak solution $v\in C([0,T];H^\beta(\mathbb{T}^3) \cap L^{\frac{2 - 2\beta}{1-2\beta}}(\T^3))$ to the 3D Euler equations~\eqref{eq:Euler}  such that $\norm{v(\cdot,0) - v_{\rm start}}_{L^2(\T^3)} \leq \epsilon$ and $\norm{v(\cdot,T)-v_{\rm end}}_{L^2(\T^3)} \leq \epsilon$.
\end{theorem}

Note that as $\beta \to \sfrac 12^-$, the Sobolev regularity index of the weak solutions in Theorem~\ref{thm:main} converges to $\sfrac 12$, while the Lebesgue integrability index converges to $\infty$, explaining the notation $C^0_t ( H^{\sfrac 12 -} \cap  L^{\infty-})$. By interpolation, it follows that for any $s<\sfrac 13$, we may choose $\beta$ sufficiently close to $\sfrac 12$ to ensure that $v\in C^0_t B^s_{3,\infty}$, which is the Onsager regularity threshold (see Remark~\ref{rem:ellthree:regularity}).  

\begin{remark}[\bf \texorpdfstring{$\beta$}{Beta}-model]
We point out that the Sobolev regularity statement in Theorem~\ref{thm:main} corresponds exactly to the predictions of the phenomenological model of turbulence known as the {\em $\beta$-model}, which was introduced by Frisch, Sulem, and Nelkin \cite{FSN06}. Specifically, if one assumes that singularities concentrate on a $2$-dimensional set, then the $\beta$-model predicts that the second order structure function exponent is $1$, which corresponds to $H^{\sfrac 12}$ regularity. Simple heuristic computations indicate that the solutions constructed in this work do indeed concentrate on a two-dimensional set, which is also the prediction of~Iyer, Sreenivasan, and Yeung~\cite{ISY20}.
For a proof of energy conservation within the assumptions of the $\beta$-model, we refer to~\cite{DH21}.
\end{remark}

\begin{remark}[\bf Other flavors of flexibility]
\label{rem:flexibility}
As in~\cite{BMNV21}, we have chosen to state Theorem~\ref{thm:main} in a way that leaves the entire emphasis of the proof on the regularity of the weak solutions. In terms of flexibility,   Theorem~\ref{thm:main} gives the existence of infinitely many non-conservative weak solutions of 3D Euler in the stated regularity class, and moreover shows that the set of wild initial data is dense in the space of $L^2$ periodic functions of given mean. Using well-established techniques, see e.g.~\cite{BDLISZ15,Isett2018,BDLSV17} and~\cite[Remarks~1.2, 3.7, 3.8]{BMNV21}, we may alternatively establish other variants of flexibility for the 3D Euler equations~\eqref{eq:Euler}  in the regularity class $C^0_t ( H^{\sfrac 12 -} \cap  L^{\infty-})$:
 
\begin{enumerate}[(a)]
\item If the functions $v_{\rm start}$ and $v_{\rm end}$ in Theorem~\ref{thm:main} are any two $C^\infty$ smooth stationary solutions of the 3D Euler equations of zero mean, then we may take $\epsilon = 0$. Since the function $0$ and any smooth shear flow are stationary solutions to \eqref{eq:Euler}, this implies the existence of nonzero weak solutions which have compact support in time. Achieving this would require that we introduce a temporal cutoff in the convex integration scheme, which essentially ensures that on temporal regions where a stress is already vanishing identically, no further velocity increments need to be added; see~\cite[Equation (3.14)]{BMNV21}.

\item One may modify the proof of Theorem~\ref{thm:main} to show that any $C^\infty$  function $e \colon [0,T] \to (0,\infty)$ is the kinetic energy of a weak solution to the 3D Euler equations in the regularity class $C^0([0,T]; H^{\sfrac 12 -} \cap  L^{\infty-})$. This implies flexibility within the class of dissipative solutions. Achieving this result would require adding a few inductive assumptions in the convex integration scheme: we need to measure the distance between the energy resolved at every step $q \mapsto q+1$ in the convex integration scheme, and the desired energy profile, see e.g.~\cite{DeLellisSzekelyhidi13,BDLISZ15,DaneriSzekelyhidi17,BDLSV17}. In particular, the energy pumped into the system due to higher order stresses in every sub-step $n\mapsto n+1$ needs to be kept track of, and one also needs to keep track of the amount of energy pumped on the support of each cutoff function, as was done in~\cite{BV19} for stress cutoffs.  

\end{enumerate}
\end{remark}

\subsection{Minimally technical outline of the proof}
We now provide a sketch of the argument used to prove Theorem~\ref{thm:main}, in order to highlight the most important components.  We simultaneously aim to elide certain technical details, while emphasizing the aspects of our argument which are distinct from recent well-known convex integration arguments (see the comparisons in Subsections~\ref{sec:Holder} and~\ref{sec:intermittency}). Finally, while our proof relies fundamentally on the technology developed in~\cite{BMNV21}, it requires several new ingredients in order to ensure that the  solution $v$ belongs to $C^0_t L^{\infty -}$; see Subsection~\ref{sec:BMNV}.

As is customary in Nash-type convex integration schemes for the Euler equations (see e.g.~\cite{DLSZ17,BV_EMS19}), the solution $v$ of Theorem~\ref{thm:main} will be constructed as a limit when $q\to \infty$ of solutions $v_q:\T^3\times\R\rightarrow\R^3$ to the Euler-Reynolds system with a  traceless symmetric stress $\RR_q:\T^3\times\R\rightarrow M^{3\times 3}_{\textnormal{symm}}$
\begin{subequations}
\label{eq:Euler:Reynolds:again}
\begin{align}
\partial_t v_q + \div(v_q \otimes v_q) +\nabla p_q &= \div \RR_q \,,\\
\div v_q &= 0 \,.
\end{align}
\end{subequations}
The pressure $p_q$ is uniquely defined by solving $\Delta p_q = \div \div (\RR_q - v_q\otimes v_q)$, with $\int_{\T^3} p_q dx = 0$. 
The functions $v_q$ and $\RR_q$ are assumed to oscillate at frequencies no larger than $\lambda_q=a^{(b^q)}$, where $a=a(\beta)$ is sufficiently large and the superexponential growth rate $b=b(\beta)$ is slightly larger than $1$. Adhering to the convention that all norms are measured uniformly in time, e.g. $L^p$ refers to $C^0( [0,T]; L^p(\T^3))$, we posit that
\begin{equation}\label{eq:inductive:RRq:practice}
    \bigl\| \RR_q \bigr\|_{L^1} \leq \delta_{q+1} := \lambda_{q+1}^{-2\beta} \, , \qquad  
    \bigl\| \RR_q \bigr\|_{L^{\infty -}} \leq 1 \, .
\end{equation}
Thus $\RR_q \rightarrow 0$ in the $L^1$ topology and is nearly summable in both $W^{1-,1}$ and $L^{\infty-}$.  The quadratic nature of the nonlinearity then leads us to posit furthermore that velocity increments  $w_q=v_q-v_{q-1}$ satisfy
\begin{equation}\label{eq:velocity:inc:basic}
\left\| w_q \right\|_{L^2}\leq \delta_q^{\sfrac 12} \, , \qquad \left\| w_q \right\|_{L^{\infty -}} \leq 1 \, ,    
\end{equation}
so that $w_q\rightarrow 0$ in $L^2$ and is nearly summable in both $H^{\sfrac 12-}$ and $L^{\infty -}$.  The main inductive step on $q$ asserts the existence of a velocity increment $w_{q+1}$ and stress $\RR_{q+1}$ such that \eqref{eq:Euler:Reynolds:again}--\eqref{eq:velocity:inc:basic} hold with $q\mapsto q+1$. 

In order to construct non-conservative solutions with regularity above $\sfrac 13$ on the $L^2$-based Sobolev scale, the results of~\cite{ConstantinETiti94} dictate that the weak solution must be {\em intermittent} -- a term which is used here to mean that the weak solution contains spatial concentrations, not just oscillations, and so it has a different regularity index in an $L^2$-based scale, versus an $L^\infty$-based scale.  A first attempt to define the velocity increment $w_{q+1}$ would then be as a sum of products of the form 
\begin{equation}
    a\bigl(\RR_q, \nabla v_q\bigr) \WW_{q+1,r_q,\xi} \, ,
    \label{eq:amplitude:times:pipe}
\end{equation}
where $a(\RR_q, \nabla  v_q)$ oscillates at spatial frequency $\lambda_q$, and $\WW_{q+1,r_q,\xi}$ is a high-frequency \emph{intermittent pipe flow}. More specifically, $\WW_{q+1,r_q,\xi}$ is a shear flow supported in a thin tube of diameter $\lambda_{q+1}^{-1}$ around a line parallel to a unit vector $\xi$, which has been periodized to scale $(\lambda_{q+1}r_q)^{-1}$, see~Proposition~\ref{prop:pipeconstruction}. The parameter $0<r_q<1$ corresponds both to the measure of the support of the intermittent pipe flow (which is $r_q^2$) and the effective frequency support (which is $[\lambda_{q+1}r_q,\lambda_{q+1}]$). As such, it is clear that $r_q$ quantifies the intermittent nature of the velocity increment $w_{q+1}$. The low-frequency function $a(\RR_q, \nabla  v_q)$ \emph{localizes} the scheme in space and time by zooming down to the scale $\lambda_q^{-1}$, at which $\RR_q$ and $v_q$ may be treated as spatially homogeneous. The ``convex integration step" via which we construct $w_{q+1}$ then consists of essentially independent local iterative steps, which are predicated on the local size of $\RR_q$ and $\nabla v_q$.  The timescale of $a(\RR_q, \nabla  v_q)$ is inversely proportional to $\| \nabla v_q \|_{L^\infty(\supp a)}$. Chebyshev's inequality combined with the global inductive bounds on $\nabla v_q$ and $\RR_q$ then controls the sizes of the space-time sets on which each local iterative step takes place. 

At this stage in the argument, it is not clear how to choose the value of the intermittency parameter $r_q$. It turns out that in order to propagate both $H^{\sfrac 12-}$ and $L^{\infty-}$ bounds, there exists a unique optimal choice of $r_q$! To see this, we inspect the simplest error term in $\RR_{q+1}$, namely the \emph{Nash error} $\RR_{q+1}^{\textnormal{Nash}}$, defined by solving the equation
\begin{equation}\notag
    w_{q+1}\cdot \nabla v_q = \div \RR_{q+1}^{\textnormal{Nash}} \, .
\end{equation}
Using that $\|\WW_{q+1,r_q,\xi}\|_{L^p}\approx r_q^{\sfrac 2p-1}$ and $\| a(\RR_q,\nabla v_q) \|_{L^{2p}}\approx \| \RR_q \|_{L^p}^{\sfrac 12}$, and using the heuristic that the most costly part of $\nabla v_q$ is $\nabla w_q$, we find that
\begin{equation*}
    \left\| \div^{-1} \left( w_{q+1}\cdot\nabla v_q \right) \right\|_{L^1} \lesssim \lambda_{q+1}^{-1} \cdot  \delta_{q+1}^{\sfrac 12}r_q \cdot \delta_q^{\sfrac 12} \lambda_q   \, , \qquad \left\| \div^{-1} \left( w_{q+1}\cdot\nabla v_q \right) \right\|_{L^\infty} \lesssim \lambda_{q+1}^{-1} \cdot r_q^{-1} \cdot r_{q-1}^{-1}\lambda_q \, . 
\end{equation*}
As $b\to 1^+$ and $\beta\to \sfrac{1}{2}^-$, matching the $L^1$ bound for the stress requires $r_q \les \lambda_{q+1}^{-\sfrac 12}\lambda_q^{\sfrac 12}$, while matching the $L^\infty$ bound requires $r_q \gtrsim \lambda_{q+1}^{-\sfrac 12}\lambda_q^{\sfrac 12}$; see \eqref{eq:drq:identity} and \eqref{eq:transport:Loo:ineq} for precise inequalities.  Thus our choice of $r_q$ is completely constrained by the simplest error term in the scheme.  Since we shall always quantify $r_q$ in terms of powers of the quotient of $\lambda_{q+1}^{-1} \lambda_q$, we refer to this constraint on $r_q$ as the {\em one-half rule for intermittency}. Of course, we must then show that the {\em transport and oscillation errors}, defined by solving the equations
\begin{equation}\notag
    \div \RR_{q+1}^{\textnormal{trans}} = \bigl(\partial_t + v_q \cdot \nabla\bigr) w_{q+1} \, , \qquad \div \RR_{q+1}^{\textnormal{osc}} = \div \bigl( \RR_q + w_{q+1}\otimes w_{q+1} \bigr) \, ,
\end{equation}
also respect this one-half rule which is dictated by the Nash error. 

Let us first consider the transport error. Recall cf.~\cite{DLSZ17,BV_EMS19} that $C^\alpha$-based convex integration schemes for the Euler equations  essentially use global Lagrangian coordinate systems, predicated on global $L^\infty$ bounds for $\nabla v_q$. Instead, as in~\cite{BMNV21} we are forced to implement \emph{local} Lagrangian coordinate systems predicated on the \emph{local} $L^\infty$ bounds for $\nabla v_q$ which are available on the support of $a(\RR_q,\nabla v_q)$. Pre-composing the high-frequency pipe flow $\WW_{q+1,r_q,\xi}$ with the local  Lagrangian flow map then gives that the transport error obeys bounds identical to those of the Nash error.
Thus, we may expect the transport error to also respect the one-half intermittency rule. 

Unfortunately, the composition of $\WW_{q+1,r_q,\xi}$ with Lagrangian flow maps introduces an \emph{intersection problem} in the oscillation error: between neighboring cutoffs $a$ and $a'$, it may be the case that
$$  a\bigl(\RR_q,\nabla v_q\bigr) \WW_{q+1,r_q,\xi} \otimes a'\bigl(\RR_q,\nabla v_q\bigr) \WW'_{q+1,r_q,\xi'} \neq 0 \, .  $$
The main innovation in Isett's proof of the Onsager conjecture \cite{Isett2018} was a ``gluing technique," which solved the intersection problem, but which required global $L^\infty$ bounds on $\nabla v_q$.  The localized nature of our scheme, combined with the inherently nonlocal nature of the Euler equations, appears to preclude the usage of a gluing technique, in the spirit of~\cite{Isett2018,BDLSV17}. 

We instead solve the intersection problem directly, using the sparsity of the pipe flows. At an intuitive level, the empty space in between neighboring pipes provides enough space for us to place new sets of intermittent pipes, which do not intersect the already existing ones. We refer to this as {\em pipe dodging}.  However, if one conceptualizes the spatial support of each $a(\RR_q, \nabla v_q)$ as being a spheroid of diameter $\lambda_q^{-1}$, then the one-half rule for intermittency does not provide enough sparsity to solve this intersection problem. Indeed, \cite[Proposition~4.8]{BMNV21} shows that pipe dodging on the support of such an isotropic cutoff requires a {\em three-quarters intermittency  rule}.  We address this issue by \emph{anisotropically shrinking} the diameter of the support of each amplitude function $a$, in a $\xi$-dependent way.  Specifically, if $a(\RR_q,\nabla v_q,\xi)$ is to be multiplied by a pipe flow parallel to $\xi$ as in \eqref{eq:amplitude:times:pipe}, then we extend the  support of $a(\RR_q,\nabla v_q,\xi)$ to length $\lambda_q^{-1}$ in the direction parallel to $\xi$ and $(\lambda_{q+1}r_q)^{-1}$ in the direction perpendicular to $\xi$.  We use the phrase \emph{relative intermittency} to quantify the aspect ratio of the support of $a(\RR_q,\nabla v_q,\xi)$ and implement it technically via a set of {\em checkerboard cutoffs}.  We refer to Subsection~\ref{sec:cutoff:checkerboard:definitions} for a construction of these anisotropic checkerboard cutoffs, Proposition~\ref{prop:disjoint:support:simple:alternate} for a proof that the one-half rule provides sufficient relative intermittency to solve the intersection problem, and Subsection~\ref{ss:stress:oscillation:2} for the implementation of these two ingredients in the context of the oscillation error.

Since the characteristic length scale of $\RR_q$ and $\nabla v_q$ is $\lambda_q^{-1}$, one may expect that introducing the artificially smaller length scale $(\lambda_{q+1} r_q)^{-1} \ll \lambda_q^{-1}$ will produce unnaturally larger error terms.  The first place to look for such a bad error term would be in the oscillation error terms which are given by
\begin{equation}  
\div^{-1} \left( \nabla \left( a(\RR_q,\nabla v_q,\xi)^2\right) \Bigl(\Id - \dashint_{\T^3}\Bigr) (\WW_{q+1,r_q,\xi} \otimes \WW_{q+1,r_q,\xi})  \right)\, . 
\label{eq:osc:1:rough}
\end{equation}
The first key insight is that the differential operator in the above expression is not the full gradient: it is the directional derivative $\xi\cdot\nabla$, as $\WW_{q+1,r_q,\xi}$ is parallel to $\xi$.  Hence, from the perspective of this error term, the anisotropy of $a(\RR_q,\nabla v_q,\xi)$ is essentially free, since in the direction of $\xi$ the amplitude function $a$ only oscillates at frequency $\lambda_q$.

However, the error term in \eqref{eq:osc:1:rough} presents other difficulties.  Since this term inherits its minimum effective frequency of $\lambda_{q+1}r_q$ from the mean-free part of  $\WW_{q+1,r_q,\xi} \otimes \WW_{q+1,r_q,\xi}$, the leftover error terms in \eqref{eq:osc:1:rough}  live at frequencies of absolute value in the range $[\lambda_{q+1}r_q,\lambda_{q+1}]$.  Simple heuristic estimates indicate that the lowest frequency portion of these error terms is too large in $L^1$ to be absorbed into $\RR_{q+1}$, while the highest frequency portion is too large in $L^\infty$ to be absorbed into $\RR_{q+1}$.  Rectifying the first issue requires identifying \emph{higher order stresses} $\RR_{q,n}$ living at intermediate frequencies $\lambda\qn \in [\lambda_{q+1}r_q,\lambda_{q+1}]$, which are corrected by corresponding \emph{higher order perturbations} 
$$w_{q+1,n}=a \bigl(\RR\qn,\nabla v_q,\xi\bigr)\WW_{q+1,r_{q,n},\xi} \,. $$
The minimum frequency of the increment $w_{q+1,n}$, which equals $\lambda_{q+1}r_{q,n}$, is defined to converge to $\lambda_{q+1}$ as $n$ approaches its maximum value of $\nmax$. This allows the $L^1$ stress estimates to just barely close. Rectifying the second issue requires a non-trivial estimate (see Lemma~\ref{lem:tricky:tricky}) on the $L^\infty$ size of the frequency projected squared pipe flow $\mathbb{P}_{[\lambda_{q,n'-1},\lambda_{q,n'}]} (\WW_{q+1,r_{q,n},\xi}\otimes \WW_{q+1,r_{q,n},\xi})$. Somewhat amazingly, this estimate respects the one-half rule in the sense that the $L^{\infty-}$ size of the resulting stress is exactly $1$ if one chooses $r_q=\lambda_q^{\sfrac 12}\lambda_{q+1}^{-\sfrac 12}$.  We then correct the higher order stresses $\RR\qn$ according to a generalization of the one-half rule; in other words, the pipes $\WW_{q+1,r_{q,n},\xi}$ used to correct $\RR\qn$, which lives at frequency $\lambda\qn \in [\lambda_{q+1}r_q,\lambda_{q+1}]$, have minimum frequency $\lambda\qn^{\sfrac 12}\lambda_{q+1}^{\sfrac 12}$. This is again the minimum amount of intermittency needed to ensure higher order pipe dodging, i.e., that pipes from overlapping cutoff functions $a(\RR\qn,\nabla v_q,\xi)$ and $a'(\RR_{q,n'},\nabla v_q,\xi')$ do not intersect.  Thus, $w_{q+1}$ is finally constructed as a sum of terms of the form $a(\RR\qn, \nabla v_q,\xi) \,  \WW_{q+1,r_{q,n},\xi}$, 
which collectively obey the inductive bounds required of velocity increments, i.e.~\eqref{eq:velocity:inc:basic} with $\delta_{q+1}^{\sfrac 12}$ replaced by a suitable $\delta_{q+1,n}^{\sfrac 12}$, and they also produce a stress $\RR_{q+1}$ obeying \eqref{eq:inductive:RRq:practice}.

In summary, in the iteration scheme described above, the one-half rule presents the Goldilocks amount of intermittency needed to obtain both $H^{\sfrac 12 -}$ and $L^{\infty-}$ bounds on the velocity. At a technical level, it appears that the choice of parameters in this scheme is essentially fixed, by scaling: the Nash,  transport, and oscillation errors each impose exactly the same intermittency restrictions.
Implementing the above strategy rigorously is made cumbersome by the need to precisely localize all parts of the argument on suitable regions of space-time. This technically involved part of the proof is encoded in the design of cutoff functions, recursively for the velocities and iteratively for the stresses, which effectively play the role of a joint Eulerian-and-Lagrangian wavelet decomposition (see Section~\ref{sec:cutoff}). This localization machinery was previously developed in our earlier joint work with Buckmaster and Masmoudi~\cite{BMNV21}, and this part of the argument can be used essentially out of the box. In this manuscript, we therefore just focus on the novel aspects of the intermittent convex-integration/Nash iteration scheme.

\subsection{Comparison and contrast with existing works}
\label{sec:comparison}

\subsubsection{H\"older schemes}
\label{sec:Holder}

The techniques in the present work share a number of generic features with the construction of non-conservative solutions in $C^\alpha_{t,x}$ for $\alpha<\sfrac 13$ in~\cite{Isett2018}, and its subsequent optimizations in~\cite{BDLSV17} and~\cite{Isett17}.  Foremost among these features is the usage of some variation of Mikado/pipe flows rather than Beltrami flows, an idea originating in \cite{DaneriSzekelyhidi17} and used additionally in recent works such as \cite{DK2020,GK2021}.  In contrast with Beltrami flows, Mikado/pipe flows enjoy stability on the full Lipschitz timescale, which appears necessary in order to reach sharp thresholds in the Nash and transport errors in both the intermittent and homogeneous settings.  In addition, we require the propagation of material derivative estimates
for the stress, as in the schemes in \cite{BDLISZ15} and \cite{Novack2020}, since in the absence of a gluing step in the iteration, these bounds do not come for free.

Implementation of these basic concepts, however, looks very different in the intermittent setting than in the homogeneous setting. The most glaring difference is in the type of derivative estimates which must be propagated on both the stress $\RR_q$ and the gradient of velocity $\nabla v_q$.  Sharp  material and spatial derivative estimates for homogeneous schemes have typically only been required at very low order, perhaps one or two material derivatives and three spatial derivatives.  Furthermore, such estimates can always be made globally due to the homogeneous character of the stress and velocity.  In our setting, sharp material and spatial derivative estimates have to be made both locally, and to essentially infinite order.  As in \cite{BMNV21}, propagating these estimates requires a careful construction of stress and velocity cutoffs, and a localized inverse divergence operator for which derivative estimates on the input lead directly to corresponding estimates on the output. We expect these tools to be widely applicable in problems which require sharp derivative estimates.

Furthermore, there are significant differences between the present work and \cite{Isett2018,BDLSV17,Isett17} in the estimation of nonlinear error terms.  The most obvious difference is in the approaches used to solve the intersection problem.  The gluing technique in \cite{Isett2018,BDLSV17} relied on a dynamic argument, which used classical stability properties of the Euler equations to localize the stress $\RR_q$ to disjoint regions in time.  Conversely, the pipe dodging technique we use is predicated entirely on an optimal exploitation of the sparsity of intermittent pipe flows.  While we rely on sharp local information about the deformations of various pipes subjected to a background transport velocity, the fact that the transport velocity field solves the Euler-Reynolds system is irrelevant.  

Let us emphasize that our estimates on the error term in \eqref{eq:osc:1:rough}, which includes the nonlinear self-interaction of intermittent pipe flows, are sharp in both $L^1$ and $L^\infty$.  This is in contrast to the estimates on the corresponding nonlinear error term in the homogeneous setting, which are strong enough to allow for $C^{\sfrac 12}$ regularity, and thus offer no relevant regularity restriction. 

Finally, one may draw a connection between our result and the problem of approximating a short embedding of a Riemannian manifold by an isometric embedding, for which there is some evidence that $C^{1,\sfrac 12}$ demarcates the sharp threshold between rigidity and flexibility \cite{Gromov16, DLSZ17}. Our result realizes a version of this ``$\sfrac 12$ threshold", but in the appropriate topology for a different PDE with a quadratic nonlinearity. 

\subsubsection{Intermittent schemes}
\label{sec:intermittency}
The usage of intermittency in Nash-style iterative schemes originated in the work of Buckmaster and the second author \cite{BV19}. The fundamental idea is that an $L^2$-normalized function with significant spatial concentrations has an $L^1$ norm which is much smaller than its $L^2$ norm.  The estimation of linear error terms in $L^1$ then relies crucially on this property. Intermittent building blocks have been used to great effect in a number of works since; we refer for example to \cite{cheskidov2020sharp,cheskidov2020nonuniqueness,BrueColombo2021,BN21,BCDL2021,Luo18}, and to the reviews \cite{BV_EMS19,BV20} and the references cited therein. The intermittent building block utilized in this paper was first used by Modena and Sz\'ekelyhidi in \cite{ModenaSZ17}.  The estimation in $L^1$ of the Nash and transport errors in our scheme relies in part on the intermittency of the pipe flows, and in this limited sense, intermittency serves the same purpose in our context as in other works.  

Sparsity factors into our arguments in several other important ways which however distinguish the present work from other intermittent schemes.  We first point to the oscillation error, in which the sparsity of pipe flows contributes favorably by providing the needed degrees of freedom to solve the intersection problem.  Secondly, and decidedly less favorably, intermittency serves to complicate any local or global $L^\infty$ estimates, especially for the Lagrangian transport maps.  As our previous joint work with Buckmaster and Masmoudi \cite{BMNV21} was the first example of a convex integration scheme which combined intermittency with transport maps, other intermittent convex integration schemes have generally not faced this difficulty; the only other exception to this is joint work of the first author with Beekie for the $\alpha$-Euler equations \cite{BN21}.  Third, the higher order stresses are a feature only shared with \cite{BMNV21}, although it is conceivable that higher-order stresses could sharpen the regularity estimates obtained in other intermittent Nash-style schemes.  Finally, both the sharp $L^{\infty-}$ and $H^{\sfrac 12-}$ require an almost geometric growth of frequencies, which again is a feature only shared with \cite{BMNV21} in the class of intermittent schemes, to the best of the authors' knowledge.

\subsubsection{The \texorpdfstring{$H^{\sfrac 12-}$}{H12} scheme in~\cite{BMNV21}}
\label{sec:BMNV}

More specific comparisons and differences may be identified between the present work and our previous paper joint with Buckmaster, Masmoudi~\cite{BMNV21}.  At a conceptual level, the most significant differences are the new constraints on the amount of intermittency which may be utilized. As described earlier, simultaneously reaching the $H^{\sfrac 12-}$ and $L^{\infty-}$ thresholds in the Nash and transport error terms requires a specific choice of the intermittency parameter $r_q$.  In \cite{BMNV21}, only a lower bound on intermittency was required since the final solution also enjoyed $H^{\sfrac 12-}$ regularity, but Lebesgue integrability only close to $L^4$. Similarly, enacting pipe dodging in the nonlinear error terms in \cite{BMNV21} required only a minimum amount of intermittency, and the self-interaction term in \eqref{eq:osc:1:rough} was essentially impervious to the choice of $r_q$.  In the current argument, the use of anistropy in the pipe dodging scheme improves the approach taken in \cite{BMNV21}, while simultaneously preserving the size of the error term \eqref{eq:osc:1:rough}.  Furthermore, analysis of this error term utilizes the fact that intermittency may not affect the $L^p$ norms of a function itself, but rather the $L^p$ norms of its derivatives. The simplest example of the latter concept is a one-dimensional shock, which is fully intermittent in the sense that it lies in  $B^{\sfrac 1p}_{p,\infty}$ for $1\leq p \leq \infty$, but has $L^p$ norms of order $1$ for all $p$.

At the technical level, there are a few noteworthy similarities and differences between \cite{BMNV21} and the present work.  First, we are able to reuse the framework of the mollification argument, the appendix full of technical lemmas on sums and iterates of operators, and the structure of the inverse divergence operator. The generalizations required for each of these tools are simple, and merely require replacing every instance of $L^1$ or $L^2$ norm in the previous arguments with an $L^\infty$ norm. Furthermore, all estimates related to flow maps (cf. Corollary~\ref{cor:deformation}) and deformations of intermittent pipe flows (cf. Lemma~\ref{lem:axis:control}) have been taken verbatim from \cite{BMNV21}. 
Next, the $L^2$ inductive estimates on velocity increments and the $L^1$ inductive estimates on the stress $\RR_q$ match those from \cite{BMNV21}.  However, we now propagate sharp $L^\infty$ bounds on both velocity increments and stresses, cf. \eqref{eq:inductive:assumption:uniform}, \eqref{eq:inductive:asumption:derivative:q:uniform}, and \eqref{eq:Rq:inductive:uniform}. Small power losses in frequency in these estimates are encoded using the parameter $\badshaq$.   We are   able to reuse the construction of the velocity and stress cutoff functions from \cite{BMNV21}.  However, while the old estimates deferred to the Sobolev inequality to achieve lossy uniform bounds (see the bounds for the parameters $\imax$ in \cite[Lemma~6.14]{BMNV21} and $\jmax$ in \cite[Lemma~6.35]{BMNV21}), the current argument appeals to the new, sharp, $L^\infty$ bounds which have been inductively propagated (see Lemma~\ref{lem:maximal:i} and Lemma~\ref{lem:maximal:j}).  

The identification of the error terms in Subsection~\ref{ss:stress:error:identification} is very similar to that in \cite{BMNV21}, save for two differences.  The first difference is the elimination of the unnecessary parameter $p$ from the scheme, which was used to minimize the accumulation of small power losses in frequency which arise from the repeated cycles of constructing higher order stresses and velocity increments. We instead minimize such losses by ensuring that an error term which arrives at the higher order stress $\RR_{q,n}$ has endured at most $\approx\log_2 n$ previous cycles of higher order stresses and increments.  This requires a choice of $\nmax$ which is large enough to guarantee that $\frac{\log_2 \nmax}{\nmax} \ll 1 $, cf.~\eqref{eq:nmax:cond:all}. Secondly, the identification and estimation of the divergence corrector errors are no longer trivial, due to the anistropy of the checkerboard cutoff functions.  However, we may again use that the anistropy of a cutoff function is fundamentally related to the direction of the axis of the associated pipe to ensure that divergence corrector bounds are satisfactory; see Subsection~\ref{ss:stress:divergence:correctors} for details.

\subsection*{Acknowledgements}
MN thanks Hyunju Kwon and Vikram Giri for many stimulating discussions during the special year on the $h$-principle at the Institute for Advanced Study. MN was supported by the NSF under Grant DMS-1926686 while a member at the IAS. VV is grateful to Tristan Buckmaster for infinitely many (for all practical purposes) discussions about convex integration, and for teaching him everything he knows about this subject. VV was supported in part by the NSF CAREER Grant DMS-1911413.
We thank Theodore Drivas for references and many discussions about structure function exponents in turbulent flows.

\section{Inductive  bounds and the proof of the main theorem}\label{section:inductive:assumptions}
 
\subsection{General notations}
\label{sec:inductive:general:notation}

Throughout the paper, we shall say that {\em the velocity field $v$ solves the Euler-Reynolds system with stress $\RR$, if $(v,\RR)$ solve}
\begin{equation*}
\partial_t v + \div(v\otimes v) + \nabla p = \div \RR,
\qquad \div v = 0\,,
\end{equation*}
for a uniquely defined zero mean pressure $p$.
As already discussed in \eqref{eq:Euler:Reynolds:again}, for $q\geq 0$ we consider a velocity field $v_q$ which solves the Euler-Reynolds system with stress $\RR_q$.

In order to circumvent the derivative-loss problem~\cite{DeLellisSzekelyhidi13}, we use the space-time mollification operator $\Pqxt$ defined in \eqref{mollifier:operators} below, to smoothen $v_q$ and define:
\begin{align}\label{vlq}
    \vlq := \Pqxt v_q 
    \,,
\end{align}
for all $q\geq 0$.
In particular, cf.~\eqref{mollifier:operators} we have that spatial mollification is performed at scale $\tilde \lambda_q^{-1}$ (which is just slightly smaller than $\lambda_q^{-1}$), while temporal mollification is done at scale $\tilde \tau_{q-1}$ (which is much smaller than $\tau_{q-1}$). Next, for all $q \geq 1$, define
\begin{align}\label{eq:cutoffs:wu}
    w_{q}:=v_{q}-\vlqminus, \qquad u_q:= \vlq - \vlqminus.
\end{align}
For consistency of notation, define $w_0 = v_0$ and $u_0 = v_{\ell_0}$.
Note that
\begin{align}
u_q = \Pqxt w_q   + (\Pqxt \vlqminus - \vlqminus)
\label{inductive:velocity:frequency}
\end{align}
so that we may morally think that $u_q = w_q + $ a small error term. We use the following notation for the material derivative corresponding to the vector field $\vlq$:
\begin{align}\label{eq:cutoffs:dtq}
    D_{t,q} := \partial_t + \vlq \cdot \nabla
    \, .
\end{align}
With this notation, we have that
\begin{align}\label{eq:cutoffs:dtqdtq-1}
  D_{t,q} = D_{t,q-1} +  u_q \cdot \nabla =: D_{t,q-1} + D_q \, .
\end{align}

\begin{remark}[\bf Geometric upper bounds with two bases]
For all $n \geq 0$ we define 
\begin{align}
\MM{n,N_*,\lambda,\Lambda} := \lambda^{\min\{n,N_*\}} \Lambda^{\max\{n-N_*,0\}}    \,.
\notag
\end{align}
This notation has the following consequence, which is used throughout the paper:  if $1 \leq \lambda \leq \Lambda$,  then 
\begin{align}
\MM{a,N_*,\lambda,\Lambda} \MM{b,N_*,\lambda,\Lambda} \leq \MM{a+b,N_*,\lambda,\Lambda}.
\notag
\end{align}
When either $a$ or $b$ are larger than $N_*$ the above inequality creates a loss; for $a+b\leq N_*$, it is an equality.
\end{remark}

\begin{remark}[\bf All norms are uniform in time]
Throughout this section, and the remainder of the paper, we shall use the notation $\norm{f}_{L^p}$ to denote $\norm{f}_{L^\infty_t (L^p(\T^3))}$. That is, all $L^p$ norms stand for {\em $L^p$ norms in  space, uniformly in time}. Similarly, when we wish to emphasize a set dependence of an $L^p$ norm, we write $\norm{f}_{L^p(\Omega)}$, for some space-time set $\Omega \subset \R \times \T^3$, to stand for $\norm{{\bf 1}_{\Omega}\; f}_{L^\infty_t (L^p(\T^3))}$.
\end{remark}

\subsection{Inductive estimates} 
\label{sec:inductive:estimates}

The proof is based on propagating estimates for solutions $(v_q,\RR_q)$ of the Euler-Reynolds system~\eqref{eq:Euler:Reynolds:again}, inductively for $q\geq 0$. In order to state these bounds, we first need to fix a number of parameters in terms of which these inductive estimates are stated. We start by picking a regularity exponent $\beta \in [\sfrac 13, \sfrac 12)$, else the theorem is known cf.~\cite{Isett2018,BDLSV17}, and a super-exponential rate parameter $b \in (1,\sfrac 32)$ such that $2\beta b < 1$. In terms of this choice of $\beta$ and $b$, a number of additional parameters ($\nmax, \ldots \Nfin$) are fixed, whose precise definition is summarized for convenience in items~\eqref{item:nmax:pmax:DEF}--\eqref{item:Nfin:DEF} of Section~\ref{sec:parameters:DEF}. Note that at this point the parameter $a_*(\beta,b)$ from item~\eqref{item:astar:DEF} in Section~\ref{sec:parameters:DEF} is not yet fixed. With this choice, we then introduce the fundamental $q$-dependent frequency and amplitude parameters from Section~\ref{ss:q:dependent:parameters}. We state here for convenience the main $q$-dependent parameters defined in  \eqref{def:lambda:q:actual}, \eqref{def:delta:q:actual}, \eqref{eq:Theta:q+1:actual}, \eqref{def:Gamma:q:actual}, and \eqref{def:tau:q:actual}:
\begin{subequations}
\label{eq:Gamma:q+1:def:*}
\begin{align}
\lambda_q 
&= 2^{ \lceil {(b^q) \log_2 a} \rceil} \approx \lambda_{q-1}^b \,, \\
\delta_q 
&=  \lambda_1^{\beta(b+1)} \lambda_q^{-2\beta} \,, \label{eq:delta:early}  \\
\tau_q^{-1} 
&= \delta_q^{\sfrac 12} \lambda_q \Gamma_{q+1}^{\cstar+11} \,, 
\\
\Theta_{q+1}
&= \lambda_{q+1} \lambda_q^{-1} \approx \lambda_q^{(b-1)} \,,\\
\Gamma_{q+1} &=  \Theta_{q+1}^{\eps_\Gamma}  \approx \lambda_q^{(b-1)\eps_\Gamma}\,, \label{eq:theta:def}
\end{align}
\end{subequations}
where the constant $\cstar$ is defined by \eqref{eq:cstar:DEF}, and $\eps_\Gamma$ is chosen as in~\eqref{eq:eps:Gamma:DEF}. 
Next, we define the $n$-dependent frequency, intermittency, and amplitude parameters
\begin{subequations}
\label{eq:n:params:heuristic}
\begin{align}
\lambda_{q,n} 
&\approx \begin{cases}
\lambda_q \Gamma_{q+1}^6 , & n=0 \\
\lambda_{q}^{\frac 12 -\frac{n}{2(\nmax+1)}} \lambda_{q+1}^{\frac 12 + \frac{n}{2(\nmax+1)} }, & 1\leq n \leq \nmax 
\end{cases} \, ,
\label{eq:def:lambda:rq}
\\  
r_{q+1,n} 
&\approx \lambda\qn^{\sfrac 12} \lambda_{q+1}^{-\sfrac 12} \Gamma_{q+1}^{-2} \,,  
\label{eq:def:lambda:r:q+1:n}\\
\qquad
\delta_{q+1,n}
&= \begin{cases}
\delta_{q+1} \Gamma_{q}^\shaq , & n=0 \\
\displaystyle\delta_{q+1,0} \lambda_{q}^{\sfrac 12}  \lambda_{q+1}^{-\sfrac 12} \Gamma_{q+1}^{14+\CLebesgue}, & n=1 \\
\displaystyle \delta_{q+1,0} \lambda_{q}  \lambda_{q,n-1}^{-1} \Gamma_{q+1}^{13+\CLebesgue} \Bigl(\Theta_{q+1}^{\frac{1}{2(\nmax+1)}} \Gamma_{q+1}^{9+\CLebesgue} \Bigr)^{\Upsilon(n)}, & 2\leq n \leq \nmax 
\end{cases} \, . \label{eq:new:delta:def}
\end{align}
\end{subequations}
In the above display, $\delta_{q+1,n}$ is defined to account for small losses (the quantity in parentheses) raised to a power $\Upsilon(n)$ (which is bounded independently of $q$, cf. \eqref{eq:upsawhat} and \eqref{eq:upsa:bound}).  Therefore one may adhere to the heuristic that $\delta_{q+1,n}$ is roughly speaking equal to $ \delta_{q+1}\lambda_q \lambda\qn^{-1}$. We refer also to \eqref{eq:lambda:q:n:def} and \eqref{eq:rqn:perp:definition}, where the precise meaning of $\approx$ in \eqref{eq:def:lambda:rq}--\eqref{eq:def:lambda:r:q+1:n} is given. 

\begin{remark}[\bf Usage of the symbols \texorpdfstring{$\approx$}{approx},  \texorpdfstring{$\lesssim$}{lesssim}, and choice of \texorpdfstring{$a$}{a}]
The $\approx$ symbols in \eqref{eq:Gamma:q+1:def:*} and \eqref{eq:n:params:heuristic} indicate that the left side of the $\approx$ symbol lies between two (universal) constant multiples of the right side, see e.g.~\eqref{eq:lambda:q:to:q+1}. Throughout the paper we make frequent use of the symbol $\lesssim$.  Any implicit constants indicated by $\lesssim$ are only allowed to depend on the parameters defined in Section~\ref{sec:parameters:DEF}, items \eqref{item:beta:DEF}--\eqref{item:Nfin:DEF}. The implicit constants in $\les$ are always independent of the parameters $a$ and $q$, appearing in \eqref{eq:delta:early}. This allows us at the end of the proof, cf.~item~\eqref{item:astar:DEF} in Section~\ref{sec:parameters:DEF} to choose $a_* (\beta,b)$ to be sufficiently large so that for all $a \geq a_*(\beta,b)$ and all $q\geq 0$, the parameter $\Gamma_{q+1}$ appearing in \eqref{eq:theta:def} is larger than all the implicit constants in $\les$ symbols encountered throughout the paper. That is, upon choosing $a_*$ sufficiently large, any inequality of the type $A \les B$ which appears in this manuscript, may be rewritten as $A \leq \Gamma_{q+1} B$, for any $q\geq 0$.
\end{remark}

In order to state the inductive assumptions we use four large integers, defined precisely in Section~\ref{sec:parameters:DEF}. For the moment we simply note that these fixed parameters are independent of $q$ and satisfy the ordering 
\begin{align}
1 \ll \NcutSmall \ll   \Nindvt \ll \Nindv \ll \Nfin 
\,.
\notag
\end{align} 
The precise definitions and the meaning of the $\ll$ symbol in are given in \eqref{eq:Ncut:DEF}, \eqref{eq:Nind:t:DEF}, \eqref{eq:Nind:v:DEF}, and \eqref{eq:Nfin:DEF}.

\subsubsection{Primary inductive assumption for velocity increments}
\label{sec:inductive:primary:velocity}
We make $L^2$ and $L^\infty$ inductive assumptions for $u_{q'}=v_{\ell_{q'}}-v_{\ell_{q'-1}}$ at levels $q'$ strictly below $q$. For all $0 \leq q' \leq q-1$ we assume that
\begin{subequations}
\label{eq:inductive:assumption:uq:all}
\begin{align}
\norm{\psi_{i,q'-1} D^{n} D^m_{t,q'-1} u_{q'}}_{L^2}
&\leq \delta_{q'}^{\sfrac 12} \MM{n,2 \Nindv,\lambda_{q'},\tilde{\lambda}_{q'}} \MM{m, \Nindvt ,\Gamma_{q'}^i \tau_{q'-1}^{-1}, \tilde{\tau}_{q'-1}^{-1}}
\label{eq:inductive:assumption:derivative}
\\
\norm{D^n D_{t,q'-1}^m u_{q'}}_{L^\infty(\supp \psi_{i,q'-1})}
&\leq \Gamma_{q'}^\badshaq \Theta_{q'}^{\sfrac 12}  \MM{n,2 \Nindv,\lambda_{q'},\tilde{\lambda}_{q'}} 
\MM{m,\Nindvt,  \Gamma_{q'}^{i+1} \tau_{q'-1}^{-1}, \tilde{\tau}_{q'-1}^{-1}}   \label{eq:inductive:assumption:uniform}
\end{align}
\end{subequations}
holds for all $0 \leq n+m\leq \Nfin$.

At level $q$, we assume that the velocity increment $w_q$ satisfies corresponding $L^2$ and $L^\infty$ bounds 
\begin{subequations}
\label{eq:inductive:assumption:wq:all}
\begin{align}
\norm{\psi_{i,q-1} D^{n} D^m_{t,q-1} w_{q}}_{L^2}
&\leq 
\Gamma_q^{-1} \delta_{q}^{\sfrac 12} \lambda_{q}^{n} 
\MM{m,\Nindvt, \Gamma_{q}^{i-1} \tau_{q-1}^{-1}, \Gamma_q^{-1} \tilde{\tau}_{q-1}^{-1}} 
\label{eq:inductive:assumption:derivative:q}
\\
\norm{D^n D_{t,q-1}^m w_q}_{L^\infty(\supp \psi_{i,q-1}) }
&\leq \Gamma_q^{\badshaq-1} \Theta_q^{\sfrac 12}  \lambda_q^n 
\MM{m,\Nindvt, \Gamma_{q}^{i} \tau_{q-1}^{-1}, \Gamma_q^{-1} \tilde{\tau}_{q-1}^{-1}} 
\label{eq:inductive:asumption:derivative:q:uniform}
\end{align}
\end{subequations}
for all $0 \leq n,m\leq 7\NindLarge$. 

\subsubsection{Inductive assumptions for the stress}
\label{sec:inductive:primary:stress}
For the Reynolds stress $\RR_q$, we make $L^1$ and $L^\infty$ inductive assumptions
\begin{subequations}
\label{eq:Rq:inductive:assumption:all}
\begin{align}
\bigl\|\psi_{i,q-1} D^n D_{t,q-1}^m \mathring R_{q}\bigr\|_{L^1} 
&\leq  \Gamma_q^\shaq  \delta_{q+1} \lambda_{q}^n \MM{m,\Nindt, \Gamma_{q}^{i+1} \tau_{q-1}^{-1}, \Gamma_q^{-1} \tilde \tau_{q-1}^{-1} }
\label{eq:Rq:inductive:assumption}
\\
\bigl\| D^n D_{t,q-1}^m \mathring{R}_q \bigr\|_{L^\infty(\supp \psi_{i,q-1})}
&\leq \Gamma_q^{\badshaq}  \lambda_q^n
\MM{m,\Nindvt,  \Gamma_{q}^{i+2} \tau_{q-1}^{-1}, \Gamma_q^{-1} \tilde{\tau}_{q-1}^{-1}} \label{eq:Rq:inductive:uniform}
\end{align}
\end{subequations}
for all $0\leq n,m\leq 3 \NindLarge$. 

\subsubsection{Inductive assumptions for the previous generation velocity cutoff functions}
\label{sec:cutoff:inductive}
More assumptions are needed in relation to the previous velocity perturbations and old cutoff functions. First, we assume that the velocity cutoff functions form a partition of unity for $q'\leq q-1$:
\begin{align}\label{eq:inductive:partition}
    \sum_{i\geq 0} \psi_{i,q'}^2 \equiv 1 \, , \qquad \mbox{and} \qquad \psi_{i,q'}\psi_{i',q'}=0 \quad \textnormal{for}\quad|i-i'| \geq 2 \, .
\end{align}
Second, we assume that there exists an $\imax = \imax(q') > 0$, which is bounded uniformly in $q'$ as
\begin{align}
\imax(q') \leq 1 + \badshaq + \frac{\sfrac 12 (b-1) + \beta b}{\eps_\Gamma (b-1) b}
\,,
\label{eq:imax:upper:lower}
\end{align}
such that for all $q'\leq q-1$, 
\begin{align}
\psi_{i,q'} &\equiv 0 \quad \mbox{for all} \quad i > \imax(q')\,,
\qquad \mbox{and} \qquad
\Gamma_{q'+1}^{\imax(q')} \leq \Gamma_{q'+1}^{\badshaq} \Theta_{q'}^{\sfrac 12} \delta_{q'}^{-\sfrac 12}
\label{eq:imax:old}\,.
\end{align}

\begin{remark}[{\bf Products of non-commuting operators}]
\label{eq:products}
The fact that space derivatives $D$ (we do not dinstinguish between $\partial_{x_1}, \partial_{x_2}, \partial_{x_3}$, but rather denote them all with $D$) and time derivatives $\partial_t$ do not commute with the  material derivative $D_{t,q}$ (see~\eqref{eq:cutoffs:dtq}), or with the directional derivative $D_q$ (see~\eqref{eq:cutoffs:dtqdtq-1}), requires that we inductively propagate {\rm mixed derivative} estimates for the velocity cutoff functions. An example of such a {\em mixed derivative} is  $D^{\alpha_1} D_{t,q}^{\beta_1} \ldots D^{\alpha_k} D_{t,q}^{\beta_k}$ for some multi-indices $\aaa = (\alpha_1,\ldots, \alpha_k)$ and $\bbb = (\beta_1, \ldots, \beta_k)$ where $\aaa, \bbb\in \mathbb{N}_0^k$. Throughout the paper, we will accordingly abbreviate these mixed derivative operators as
\begin{align}
  D^{\aaa} D_{t,q}^{\bbb} := \prod_{\ell=1}^k D^{\alpha_\ell} D_{t,q}^{\beta_\ell}
  \,, \qquad \mbox{and} \qquad 
  D^{\aaa} D_{q}^{\bbb} := \prod_{\ell=1}^k D^{\alpha_\ell} D_{q}^{\beta_\ell}
  \,,
  \label{eq:mixed:derivatives}
\end{align}
whenever $\aaa,\bbb \in \mathbb{N}_0^k$, and $q\geq 0$.
\end{remark}

For  all $0 \leq q' \leq q-1$ and $0 \leq i \leq \imax$ we assume the following pointwise derivative bounds for the cutoff functions $\psi_{i,q'}$. For mixed space and material derivatives (recall the notation from \eqref{eq:cutoffs:dtq}, \eqref{eq:mixed:derivatives}) we assume 
\begin{align}
&\frac{|D^{\aaa} D_{t,q'-1}^{\bbb} \psi_{i,q'}|}{\psi_{i,q'}^{1- (|\aaa|+|\bbb|)/\Nfin}}  
\les \MM{|\aaa|,\NindLarge , \Gamma_{q'}  \lambda_{q'}, \Gamma_{q'} \tilde \lambda_{q'} }
\MM{|\bbb|,\NindSmall - \NcutSmall,  \Gamma_{q'+1}^{i+3}  \tau_{q'-1}^{-1}, \Gamma_{q'+1}^{-1}  \tilde \tau_{q'}^{-1}}
\label{eq:sharp:Dt:psi:i:q:old}
\end{align}
for  $k \geq 0$ and $\aaa,\bbb \in {\mathbb N}_0^k$ with $|\aaa|+|\bbb| \leq \Nfin$. 
Lastly, we consider mixtures of space, material, and directional derivatives (recall the notation from \eqref{eq:cutoffs:dtqdtq-1},  \eqref{eq:mixed:derivatives}).
With $M, \aaa, \bbb$ and $k$ as above, and with $N\geq 0$, we assume 
\begin{align}
&\frac{| D^N D_{q'}^{\aaa} D_{t,q'-1}^{\bbb} \psi_{i,q'} |}{\psi_{i,q'}^{1- (N+|\aaa|+|\bbb|)/\Nfin}}  \notag\\
&\qquad  \les  \MM{N,\NindLarge,  \Gamma_{q'}  \lambda_{q'},  \Gamma_{q'} \tilde \lambda_{q'}  } 
(\Gamma_{q'+1}^{i-\cstar} \tau_{q'}^{-1})^{|\aaa|} 
\MM{\bbb,\Nindt-\NcutSmall,  \Gamma_{q'+1}^{i+3}  \tau_{q'-1}^{-1}, \Gamma_{q'+1}^{-1}  \tilde \tau_{q'}^{-1}}
\label{eq:sharp:Dt:psi:i:q:mixed:old}
\end{align}
for all $ N+ |\aaa| + |\bbb| \leq \Nfin$.

In addition to the above pointwise estimates for the cutoff functions $\psi_{i,q'}$, we also assume that we have a good $L^1$ control. More precisely, we postulate  that
\begin{align}
\norm{\psi_{i,q'}}_{L^1}  \lesssim  \Gamma_{q'+1}^{-2i+\CLebesgue} \qquad \mbox{where} \qquad \CLebesgue = \frac{4+b}{b-1}  
\label{eq:psi:i:q:support:old}
\end{align}
holds for $0\leq q' \leq q-1$ and all $0\leq i \leq \imax(q')$.

\subsubsection{Secondary inductive assumptions for velocities}
\label{sec:inductive:secondary:velocity}

Next, for $0\leq q'\leq q-1$, $0 \leq i \leq \imax$, $k\geq 1$,  and $\aaa, \bbb \in \N_0^k$, we assume that the following mixed space-and-material derivative bounds hold
\begin{align}
&\bigl\| D^{\aaa} D_{t,q'-1}^{\bbb} u_{q'} \bigr\|_{L^\infty(\supp \psi_{i,q'})} 
\notag\\
&\qquad \les 
(\Gamma_{q'+1}^{i+1} \delta_{q'}^{\sfrac 12}) \MM{|\aaa|,2\Nindv,\Gamma_{q'} \lambda_{q'},\tilde \lambda_{q'}} 
\MM{|\bbb|,\Nindt, \Gamma_{q'+1}^{i+3}  \tau_{q'-1}^{-1},  \Gamma_{q'+1}^{-1} \tilde \tau_{q'}^{-1}}
\label{eq:nasty:D:wq:old}
\end{align}
for $|\aaa|+|\bbb|  \leq \sfrac{3 \Nfin}{2} +1 $, 
\begin{align}
&\bigl\| D^{\aaa} D_{t,q'}^{\bbb}   D \vlqprime \bigr\|_{L^\infty(\supp \psi_{i,q'})} 
\notag\\
&\qquad \les
(\Gamma_{q'+1}^{i+1} \delta_{q'}^{\sfrac 12}\tilde \lambda_{q'}  ) 
\MM{|\aaa| , 2\Nindv,\Gamma_{q'} \lambda_{q'},\tilde \lambda_{q'}} \MM{|\bbb|,\Nindt,\Gamma_{q'+1}^{i-\cstar} \tau_{q'}^{-1},   \Gamma_{q'+1}^{-1} \tilde \tau_{q'}^{-1}}
\label{eq:nasty:D:vq:old}
\end{align}
for $|\aaa|+|\bbb|  \leq \sfrac{3 \Nfin}{2}$, 
and
\begin{align}
&\bigl\| D^{\aaa} D_{t,q'}^{\bbb} \vlqprime \bigr\|_{L^\infty(\supp \psi_{i,q'})} 
\notag\\
&\qquad \les
(\Gamma_{q'+1}^{i+1} \delta_{q'}^{\sfrac 12}\lambda_{q'}^2 ) 
\MM{|\aaa|,2 \Nindv,\Gamma_{q'} \lambda_{q'},\tilde \lambda_{q'}} 
\MM{|\bbb|,\Nindt,\Gamma_{q'+1}^{i-\cstar} \tau_{q'}^{-1},   \Gamma_{q'+1}^{-1} \tilde \tau_{q'}^{-1}}
\label{eq:bob:Dq':old}
\end{align}
for $|\aaa|+|\bbb|  \leq \sfrac{3 \Nfin}{2}  +1$. Lastly, for $N\geq 0$ and $N+ |\aaa|+|\bbb| \leq  \sfrac{3 \Nfin}{2} +1$, we postulate that mixed space-material-directional derivatives satisfy 
\begin{subequations}
\label{eq:nasty:Dt:wq:WEAK:all}
\begin{align}
&\bigl\| D^N  D_{q'}^{\aaa} D_{t,q'-1}^{\bbb}   u_{q'} \bigr\|_{L^\infty(\supp \psi_{i,q'})} \notag\\
&\qquad \les 
(\Gamma_{q'+1}^{i+1} \delta_{q'}^{\sfrac 12})^{|\aaa|+1} 
\MM{N+|\aaa|,2\Nindv,\Gamma_{q'} \lambda_{q'},\tilde \lambda_{q'}}   
\MM{|\bbb|,\Nindt,  \Gamma_{q'+1}^{i+3}  \tau_{q'-1}^{-1},  \Gamma_{q'+1}^{-1} \tilde \tau_{q'}^{-1}}   \label{eq:nasty:Dt:wq:old} \\
&\qquad \les 
(\Gamma_{q'+1}^{i+1} \delta_{q'}^{\sfrac 12}) \MM{N,2\Nindv,\Gamma_{q'} \lambda_{q'},\tilde \lambda_{q'}} (\Gamma_{q'+1}^{i-\cstar}  \tau_{q'}^{-1})^{|\aaa|}  
\MM{|\bbb|,\Nindt, \Gamma_{q'+1}^{i+3}  \tau_{q'-1}^{-1},  \Gamma_{q'+1}^{-1} \tilde \tau_{q'}^{-1}}
.
\label{eq:nasty:Dt:wq:WEAK:old}
\end{align}
\end{subequations}

\begin{remark}
\label{rem:D:t:q':orangutan}
As shown in \cite[Remark~3.4]{BMNV21}, \eqref{eq:nasty:Dt:wq:WEAK:old} automatically implies the bounds
\begin{align}
&\norm{D^N  D_{t,q'}^{M}  u_{q'} }_{L^\infty(\supp \psi_{i,q'})} 
\les 
(\Gamma_{q'+1}^{i+1} \delta_{q'}^{\sfrac 12}) \MM{N,2\Nindv,\Gamma_{q'} \lambda_{q'},\tilde \lambda_{q'}}
\MM{M,\Nindt, \Gamma_{q'+1}^{i-\cstar}  \tau_{q'}^{-1},  \Gamma_{q'+1}^{-1} \tilde \tau_{q'}^{-1}}
\label{eq:nasty:Dt:uq:orangutan}
\end{align}
for all $N+M \leq \sfrac{3\Nfin}{2}+1$, while in a similar way, \eqref{eq:sharp:Dt:psi:i:q:mixed:old} implies that
\begin{align}
\frac{| D^N  D_{t,q'}^{M}  \psi_{i,q'}|}{\psi_{i,q'}^{1- (N+M)/\Nfin}}  \les  \MM{N,\NindLarge,  \Gamma_{q'}  \lambda_{q'},  \Gamma_{q'} \tilde \lambda_{q'}  } 
\MM{M,\Nindt-\NcutSmall, \Gamma_{q'+1}^{i-\cstar} \tau_{q'}^{-1}, \Gamma_{q'+1}^{-1}  \tilde \tau_{q'}^{-1}}
\label{eq:nasty:Dt:psi:i:q:orangutan}
\end{align}
for all $N+M \leq \Nfin$.
\end{remark}

\subsection{Main inductive proposition}
The main inductive proposition, which propagates the inductive estimates in  Section~\ref{sec:inductive:estimates} from step $q$ to step $q+1$, is as follows.

\begin{proposition}\label{p:main}
Fix $\beta \in [\sfrac 13,\sfrac 12)$ and choose $b\in (1,\sfrac{1}{2\beta})$. Solely in terms of $\beta$ and $b$, define the parameters $\nmax$, $\CLebesgue$, $\mathsf{C_R}$, $\cstar$, $\eps_\Gamma$, $\badshaq$, $\alpha$, $\NcutSmall$, $\NcutLarge$, $\Nindt$, $\Nindv$, $\Ndec$, $\dpot$, and $\Nfin$, by the definitions in Section~\ref{sec:parameters:DEF}, items \eqref{item:beta:DEF}--\eqref{item:Nfin:DEF}.
Then, there exists a sufficiently large $a_* = a_*(\beta,b) \geq 1$, such that for any $a\geq a_*$, the following statement holds for any $q\geq 0$.
Given a velocity field $v_q$ which solves the Euler-Reynolds system with stress $\RR_q$, define $v_{\ell_q}, w_q$, and $u_q$ via \eqref{vlq}--\eqref{eq:cutoffs:wu}. Assume that $\{ u_{q'} \}_{q'=0}^{q-1}$ satisfies \eqref{eq:inductive:assumption:uq:all}, $w_q$ obeys \eqref{eq:inductive:assumption:wq:all}, $\RR_q$ satisfies \eqref{eq:Rq:inductive:assumption:all}, and that for every $q'\leq q-1$ there exists a partition of unity $\{ \psi_{i,q'}\}_{i\geq 0}$ such that properties \eqref{eq:inductive:partition}--\eqref{eq:imax:old} and estimates \eqref{eq:sharp:Dt:psi:i:q:old}--\eqref{eq:nasty:Dt:wq:WEAK:all} hold. Then, there exists a velocity field $v_{q+1}$, a stress $\RR_{q+1}$, and a partition of unity $\{\psi_{i,q}\}_{q\geq 0}$, such that $v_{q+1}$ solves the Euler-Reynolds system with stress $\RR_{q+1}$, $u_q$ satisfies \eqref{eq:inductive:assumption:uq:all} for $q'\mapsto q$, $w_{q+1}$ obeys \eqref{eq:inductive:assumption:wq:all} for $q\mapsto q+1$, $\RR_{q+1}$ satisfies \eqref{eq:Rq:inductive:assumption:all} for $q\mapsto q+1$, and the $\psi_{i,q}$ are such that \eqref{eq:inductive:partition}--\eqref{eq:nasty:Dt:wq:WEAK:all} hold when $q' \mapsto q$.
\end{proposition}

The proof of Proposition~\ref{p:main} takes up the bulk of the remaining part of the paper, cf.~Sections~\ref{sec:building:blocks}--\ref{s:stress:estimates}. Here we just give a road map of which proofs are contained in what sections:
\begin{itemize}
    \item In Section~\ref{sec:building:blocks}, we recall the construction and important properties of intermittent pipe flows from \cite{BMNV21}.  We however prove a new estimate for squared pipe densities in Lemma~\ref{lem:tricky:tricky}, and an updated version of the pipe dodging strategy in Proposition~\ref{prop:disjoint:support:simple:alternate}.
    \item In Section~\ref{sec:mollification:stuff} we mollify the Euler-Reynolds system at level $q$, define $v_{\ell_q}$, and show that $u_q$ satisfies \eqref{eq:inductive:assumption:uq:all} with $q'$ replaced by $q$. This argument requires few changes when compared to~\cite[Section 5]{BMNV21}.
    \item In Section~\ref{sec:cutoff} we construct the velocity cutoffs at level $q$, namely $\{\psi_{i,q}\}_{i\geq 0}$, and show that the inductive assumptions \eqref{eq:inductive:partition}--\eqref{eq:nasty:Dt:wq:WEAK:all} hold for $q'$ replaced by $q$. This part of the argument is technically quite involved, but we take advantage of the fact that it is identical to the proof in~\cite[Section 6]{BMNV21}, except for the new bound for $\imax$. The new bound on $\imax$ is the only place where the propagated $L^\infty$ bounds are required, and we give the full details of this part of the argument in Lemma~\ref{lem:maximal:i}.
    \item In Section~\ref{sec:statements}, we present Proposition~\ref{p:main:inductive:q}, which gives the existence of a pair $(w_{q+1},\RR_{q+1})$ which satisfies the remaining inductive bounds, namely \eqref{eq:inductive:assumption:wq:all} and \eqref{eq:Rq:inductive:assumption:all}, with $q$ replaced by $q+1$.
    \item In Section~\ref{s:stress:estimates} we give the proof of Proposition~\ref{p:main:inductive:q}, thereby concluding the proof of Proposition~\ref{p:main}, once $a$ is taken sufficiently large with respect to $(\beta,b)$, as in Section~\ref{sec:parameters:DEF}, item~\eqref{item:astar:DEF}. This is the main part of the proof, and it is substantially different from the corresponding argument in~\cite[Section 8]{BMNV21}.
\end{itemize}

\subsection{Proof of the main theorem}
We conclude this section by showing how Proposition~\ref{p:main} implies Theorem~\ref{thm:main}, upon potentially choosing $a \geq a_*$ even larger, depending also on the functions $v_{\rm start}, v_{\rm end}$, and on the $T,\epsilon>0$ from the statement of Theorem~\ref{thm:main}. This argument is nearly identical to that in~\cite[Section 3.4]{BMNV21}. 
We also give here the proof that the constructed solutions lie in $C^0_t B^{\sfrac 13-}_{3,\infty}$, cf.~Remark~\ref{rem:ellthree:regularity} below.

First, let $a_*=a_*(\beta,b)$ be as in Proposition~\ref{p:main}, which holds for any $a\geq a_*$. Second, construct the pair $(v_0,\RR_0)$, which solve the Euler-Reynolds system, exactly as in~\cite[Equations (3.30)--(3.31)]{BMNV21}. In essence, $v_0$ is a temporal interpolation between mollified versions of $v_{\rm start}$ and $v_{\rm end}$, and $\RR_0$ is the resulting error made in the Euler equations~\eqref{eq:Euler} .  Third, we define
$v_{-1}= v_{\ell_{-1}}= u_{-1} \equiv 0$, and we let $\psi_{0,-1} \equiv 1$ and $\psi_{i,-1} \equiv 0$ for all $i\geq 1$. Lastly, it is convenient to denote $\tau_{-1}^{-1} = \Gamma_0 := \lambda_0^{\eps_\Gamma}$, $\tilde \tau_{-1}^{-1}= \Gamma_0^3 = \lambda_0^{3\eps_\Gamma}$, and $\Theta_0 = \lambda_0$. 

With these choices, we have already verified in~\cite[Section 3.4]{BMNV21} that if $a\geq a_*$ is taken to be sufficiently large, depending also on $v_{\rm start},v_{\rm end}, T, \epsilon$, then $u_{-1}\equiv 0$ satisfies \eqref{eq:inductive:assumption:derivative} (and trivially also \eqref{eq:inductive:assumption:uniform}), $w_0=v_0$ obeys \eqref{eq:inductive:assumption:derivative:q}, $\RR_0$ satisfies \eqref{eq:Rq:inductive:assumption}, and we have that  \eqref{eq:inductive:partition}--\eqref{eq:nasty:Dt:wq:WEAK:all} hold trivially. Thus it remains to show that $(v_0,\RR_0)$ obey the uniform estimates \eqref{eq:inductive:assumption:uniform} and \eqref{eq:Rq:inductive:uniform}, which were not present in~\cite{BMNV21}. But these estimates are easy to satisfy since both $\Gamma_0^{\badshaq-1} \Theta_0^{\sfrac 12} \geq a^{\sfrac 12}$, and $\Gamma_0^{\badshaq} = a^{\frac{1}{(b-1)(\nmax+1)}}$, may be made arbitrarily large, upon choosing $a$ to be sufficiently large. 

As such, the inductive estimates \eqref{eq:inductive:assumption:uq:all}--\eqref{eq:nasty:Dt:wq:WEAK:all} hold for the base case of the induction $q=0$, and we may inductively apply Proposition~\ref{p:main} for all $q\geq 1$, to produce a sequence of velocity fields $v_q$ which solve the Euler-Reynolds system with stress $\RR_q$, and a sequence of velocity cutoff functions $\psi_{i,q}$, such that the bounds \eqref{eq:inductive:assumption:uq:all}--\eqref{eq:nasty:Dt:wq:WEAK:all} hold for all $q\geq 0$. Then, by construction, we have that for any $\beta'<\beta$, the series $\sum_{q\geq0} (v_{q+1}-v_q) = \sum_{q\geq0} (w_{q+1} + (v_{\ell_q}-v_q))$ is absolutely summable in $C^0_t H^{\beta'}$, justifying the definition of the limiting velocity field $v = v_0 + \sum_{q\geq 0} (v_{q+1}-v_q) \in C^0_t H^{\beta'}$. As $\RR_q \to 0$ in $C^0_t L^1$, the function $v$ is a weak solution of the 3D Euler system \eqref{eq:Euler} . Moreover, as was shown in~\cite[Section 3.4]{BMNV21}, the $L^2$ distance between $v(\cdot,0)$ and $v_{\rm start}$, respectively $v(\cdot,T)$ and $v_{\rm end}$, is less than $\epsilon$.

In order to conclude the proof of the theorem, we only need to show that $v\in C^0_t L^{\frac{2-2\beta}{1-2\beta}}$. For this purpose, note that we have the identity $v= \lim_{q\to \infty} v_q = \sum_{q\geq 0} u_q$. Using the bounds on $u_q$ provided by \eqref{eq:inductive:assumption:uq:all} we may sum over $0\leq i \leq \imax(q)$ using the partition of unity property \eqref{eq:inductive:partition}, and use the definitions \eqref{eq:nmax:cond:3} and \eqref{eq:badshaq:def}, to arrive at
\begin{align*}
    \norm{u_q}_{L^2}  \leq C \delta_q^{\sfrac 12}= \lambda_1^{\frac{\beta(b+1)}{2}} \lambda_q^{-\beta}\,,
    \qquad \mbox{and} \qquad
    \norm{u_q}_{L^\infty}  \leq C \Gamma_q^{\badshaq} \Theta_q^{\frac 12} \approx C \lambda_q^{\frac{b-1}{b} (\frac 12 + \eps_\Gamma \badshaq)} 
    \leq
    C \lambda_q^{\frac{b-1}{b} (\frac 12 + \frac{b-1}{4b})} \,,
    \end{align*}
where the constant $C$ depends only on our upper bound for $\imax(q)$, and so only on $\beta$ and $b$ through \eqref{eq:imax:upper:lower}. Using Lebesgue interpolation, and the above established bounds, for $p\in [2,\infty)$ we obtain
\begin{equation}
    \left\| u_q \right\|_{L^p} 
    \leq 
    \left\| u_q \right\|_{L^2}^{\frac{2}{p}} 
    \left\| u_q \right\|_{L^\infty}^{1-\frac 2p} 
    \leq 
    C 
    \lambda_1^{\frac{\beta(b+1)}{p}}  
    \lambda_q^{-\frac{2\beta}{p} + (1-\frac 2p)\frac{b-1}{b} (\frac 12 + \frac{b-1}{4b})}  \, , \label{bounding:L3}
\end{equation}
where the constant $C\geq 1$ depends only on $\beta$ and $b$. Thus, in order to ensure the absolute summability of $\{u_q\}_{q\geq0}$ in $L^p$, the exponent of $\lambda_q$ appearing on the right side of \eqref{bounding:L3} must be strictly negative. After a short computation, we deduce that we must have 
\begin{align}
    p < p_*(\beta,b) =: 2 + \frac{8 \beta b}{(b-1)(b(b-1)+2)} 
    \,.
    \label{eq:endpoint:p}
\end{align}
At last, we may verify that $\frac{2-2\beta}{1-2\beta} < p_*(\beta,b)$ is equivalent to $\frac{(b-1)(b(b-1)+2)}{4b} < 1- 2\beta$, which in turn is satisfied whenever $2\beta b< 1$  and $\beta \in [\sfrac 13, \sfrac 12)$. 
This concludes the proof of Theorem~\ref{thm:main}.

\begin{remark}[\bf \texorpdfstring{$L^3$}{L3}-based Besov regularity]\label{rem:ellthree:regularity}
From \eqref{bounding:L3} and \eqref{eq:endpoint:p}, we deduce that for $p \in [2,p_*(\beta,b))$, and in particular for $p = \frac{2-2\beta}{1-2\beta}$, we have that $\norm{u_q}_{L^p} \leq C \lambda_1^{\frac{1}{p}} \lambda_q^{-\eta(p,\beta,b)}$, for some  $\eta(p,\beta,b) > 0$. We therefore have that
\begin{align*}
    \left\| u_q \right\|_{B^{0}_{p,\infty}} \leq C \lambda_1^{\frac 1p} \lambda_q^{-\eta(p,\beta,b)} \, ,
    \qquad \mbox{and} \qquad 
    \left\| u_q \right\|_{B^{\beta}_{2,\infty}} \leq 
    \left\| u_q \right\|_{B^{\beta}_{2,2}} 
    \leq C \left\| u_q \right\|_{H^\beta} 
    \leq C \, ,
\end{align*}
where the constant $C>0$ is independent of $q$. 
By interpolation, we have that whenever $s< \beta \theta$, where $\theta = \theta(p) \in (0,1)$ is defined by
solving
\begin{equation}\label{eq:solving:for:theta}
    \frac{1}{3} = \frac{1-\theta}{p} + \frac{\theta}{2} \qquad \implies \qquad \theta = \frac{2p-6}{3p-6} \, ,  
\end{equation}
we have the bound
\begin{align*}
    \left\| u_q \right\|_{B^s_{3,\infty}} 
    \leq C \left\| u_q \right\|_{B^0_{p,\infty}}^{1-\theta} \left\| u_q \right\|_{B^\beta_{2,\infty}}^\theta 
    \leq C \lambda_1^{\frac 1p} \lambda_q^{-(1-\theta) \eta(p,\beta,b)}  \,  ,
\end{align*}
for a constant $C>0$ which is independent of $q\geq 0$.
Taking $p = \frac{2-2\beta}{1-2\beta}$, we obtain from \eqref{eq:solving:for:theta} that $\theta = \frac{4\beta -1}{3\beta}$, and so for any $s < \frac{4\beta -1}{3}$, we have that the series $v = \sum_{q\geq 0} u_q$ is absolutely summable in $C^0_t B^{s}_{3,\infty}$, showing that $v \in C^0_t B^{s}_{3,\infty}$. It is clear that by letting $\beta$ be arbitrarily close to $\sfrac 12$, the value of $s$ may be taken arbitrarily close to $\sfrac 13$, the Onsager threshold.
\end{remark}

\section{Building blocks and pipe dodging}
\label{sec:building:blocks}

\newcommand{\bard}{{\Bar{d}}}

The main results in this section are Proposition~\ref{prop:pipeconstruction} (which describes the intermittent pipe flows and their properties), Lemma~\ref{lem:tricky:tricky} (which gives a sharp bound for the $L^\infty$ norm of frequency truncated square of pipe densities), and Proposition~\ref{prop:disjoint:support:simple:alternate} (which gives the proof of the one-half relative intermittency rule for pipe dodging). 
 First, we recall from~\cite[Lemma~2.4]{DaneriSzekelyhidi17} a version of the following geometric decomposition:
\begin{proposition}[\bf Choosing Vectors for the Axes]\label{p:split}
Let $B_{\sfrac 12}(\Id)$ denote the ball of symmetric $3\times 3$ matrices, centered at $\Id$, of radius $\sfrac 12$. Then, there exists a finite subset $\Xi \subset\mathbb{S}^{2} \cap {\mathbb{Q}}^{3}$, and smooth positive functions $\gamma_{\xi}\colon C^{\infty}\left(B_{\sfrac 12} (\Id)\right) \to \R$ for every $\xi \in \Xi$, such that for each $R\in B_{\sfrac 12} (\Id)$, we have the identity
\begin{equation}\label{e:split}
R = \sum_{\xi\in\Xi} \left(\gamma_{\xi}(R)\right)^2  \xi\otimes \xi.
\end{equation}
Additionally, for every $\xi$ in $\Xi$, there exist vectors $\xi',\xi'' \in \mathbb{S}^2 \cap \mathbb{Q}^3$ such that $\{\xi,\xi',\xi''\}$ is an orthonormal basis of $\R^3$, and there exists a least positive integer $n_\ast$ such that $n_\ast \xi, n_\ast \xi', n_\ast \xi'' \in \mathbb{Z}^3$, for every $\xi\in \Xi$.
\end{proposition}
 
We now recall \cite[Proposition 4.3]{BMNV21} and \cite[Proposition 4.4]{BMNV21} which rigorously construct the intermittent pipe flows and enumerate the necessary properties.
 
\begin{proposition}[\bf Rotating, Shifting, and Periodizing]\label{prop:pipe:shifted}
Fix $\xi\in\Xi$, where $\Xi$ is as in Proposition~\ref{p:split}. Let ${r^{-1},\lambda\in\mathbb{N}}$ be given such that $\lambda r\in\mathbb{N}$. Let $\varkappa:\mathbb{R}^2\rightarrow\mathbb{R}$ be a smooth function with support contained inside a ball of radius $\sfrac{1}{4}$. Then for $k\in\{0,...,r^{-1}-1\}^2$, there exist functions $\varkappa^k_{\lambda,r,\xi}:\mathbb{R}^3\rightarrow\mathbb{R}$ defined in terms of $\varkappa$, satisfying the following additional properties:
\begin{enumerate}[(1)]
    \item \label{item:point:1} 
    We have that  $\varkappa^k_{\lambda,r,\xi}$ is simultaneously $\left(\frac{\mathbb{T}^3}{\lambda r}  \right)$-periodic and $\left(\frac{\Tthreexi}{\lambda r n_\ast}  \right)$-periodic. Here, by  $\T^3_\xi$ we refer to a rotation of the standard torus such that $\T^3_\xi$ has a face perpendicular to $\xi$.
    
 \item \label{item:point:2}  Let $F_\xi$ be one of the two faces of the cube $\frac{\Tthreexi}{\lambda r n_\ast}$ which is perpendicular to $\xi$. Let $\mathbb{G}_{\lambda,r}\subset F_\xi\cap \twopi \mathbb{Q}^3$ be the grid consisting of $r^{-2}$-many points spaced evenly at distance $\twopi  (\lambda n_\ast  )^{-1}$ on $F_\xi$ and containing the origin.  Then each grid point $g_{k}$ for $k\in\{0,...,r^{-1}-1\}^{2}$ satisfies
    \begin{equation}\label{e:shifty:support}
    \left(\supp\varkappa_{\lambda,r,\xi}^k\cap F_\xi \right) \subset \bigl\{x: |x-g_{k}| \leq \twopi\left(4\lambda n_\ast\right)^{-1} \bigr\}.
    \end{equation}
    
    \item \label{item:point:2a} The support of $\varkappa_{\lambda,r,\xi}^k$ is a pipe (cylinder) centered around a $\left(\frac{\mathbb{T}^3}{\lambda r}  \right)$-periodic and $\left(\frac{\Tthreexi}{\lambda r n_\ast}  \right)$-periodic line parallel to $\xi$, which passes through the point $g_k$. The radius of the cylinder's cross-section is as in \eqref{e:shifty:support}.
    \item We have that $\xi \cdot \nabla \varkappa_{\lambda,r,\xi}^k = 0$.
    \item \label{item:point:3} For $k\neq k'$, $\supp \varkappa_{\lambda,r,\xi}^k \cap \supp \varkappa_{\lambda,r,\xi}^{k'}=\emptyset$.
\end{enumerate}
\end{proposition}

\begin{proposition}[\bf Construction and properties of shifted intermittent pipe flows]
\label{prop:pipeconstruction}
Fix a vector $\xi$ belonging to the set of rational vectors $\Xi \subset\mathbb{Q}^{3} \cap \mathbb{S}^2 $ from Proposition~\ref{p:split}, $r^{-1},\lambda \in \mathbb{N}$ with $\lambda r\in \mathbb{N}$, and large integers $3\Nfin$ and $\dpot$. There exist vector fields $\WW^k_{\xi,\lambda,r}:\mathbb{T}^3\rightarrow\mathbb{R}^3$ for $k\in\{0,...,r^{-1}-1\}^2$ and implicit constants depending on $\Nfin$ and $\dpot$ but not on $\lambda$ or $r$ such that:
\begin{enumerate}[(1)]
    \item\label{item:pipe:1} There exists $\varrho:\mathbb{R}^2\rightarrow\mathbb{R}$ given by the iterated Laplacian $\Delta^\dpot  \vartheta =: \varrho$ of a potential $\vartheta:\mathbb{R}^2\rightarrow\mathbb{R}$ with compact support in a ball of radius $\frac{1}{4}$ such that the following holds.  Let $\varrho_{\xi,\lambda,r}^k$ and $\vartheta_{\xi,\lambda,r}^k$ be defined as in Proposition~\ref{prop:pipe:shifted}, in terms of $\varrho$ and $\vartheta$ (instead of $\varkappa$).  Then there exists $\UU^k_{\xi,\lambda,r}:\mathbb{T}^3\rightarrow\mathbb{R}^3$ such that
    if $\{\xi,\xi',\xi''\} \subset \mathbb{Q}^3 \cap \mathbb{S}^2$ form an orthonormal basis of $\R^3$ with $\xi\times\xi'=\xi''$, then we have
    \begin{equation}
    \UU_{\xi,\lambda,r}^k
     =  -  \xi' \underbrace{\lambda^{-2\dpot} \xi''\cdot \nabla \Delta^{\dpot-1} \left(\vartheta_{\xi,\lambda,r}^k \right)}_{=:\varphi_{\xi,\lambda,r}^{\prime \prime k}}
     +  
     \xi'' \underbrace{\lambda^{-2\dpot} \xi'\cdot \nabla \Delta^{\dpot-1} \left( \vartheta_{\xi,\lambda,r}^k \right)
     }_{=:\varphi_{\xi,\lambda,r}^{\prime k}}
     \label{eq:UU:explicit}
        \,, 
    \end{equation}
and thus
\begin{equation}
\notag
\curl \UU^k_{\xi,\lambda,r} = \xi \lambda^{-2\dpot }\Delta^\dpot  \left(\vartheta^k_{\xi,\lambda,r}\right) = \xi \varrho^k_{\xi,\lambda,r} =: \WW^k_{\xi,\lambda,r}
\,,
\end{equation}
and 
\begin{equation}
    \xi \cdot \nabla \WW^k_{\xi,\lambda,r} 
    = \xi \cdot \nabla \UU^k_{\xi,\lambda,r}
    = 0
    \,.
    \label{eq:derivative:along:pipe}
\end{equation}
    \item\label{item:pipe:2} The sets of functions $\{\UU_{\xi,\lambda,r}^k\}_{k}$, $\{\varrho_{\xi,\lambda,r}^k\}_{k}$, $\{\vartheta_{\xi,\lambda,r}^k\}_{k}$, and $\{\WW_{\xi,\lambda,r}^k\}_{k}$ satisfy items~\ref{item:point:1}--\ref{item:point:3} in Proposition~\ref{prop:pipe:shifted}.
    \item\label{item:pipe:3} $\WW^k_{\xi,\lambda,r}$ is a stationary, pressureless solution to the Euler equations.
    \item\label{item:pipe:4} $\displaystyle{\dashint_{\mathbb{T}^3} \WW^k_{\xi,\lambda,r} \otimes \WW^k_{\xi,\lambda,r} = \xi \otimes \xi }$
    \item\label{item:pipe:5} For all $n\leq 3 \Nfin$, 
    \begin{equation}\label{e:pipe:estimates:1}
    {\left\| \nabla^n\vartheta^k_{\xi,\lambda,r} \right\|_{L^p(\mathbb{T}^3)} \lesssim \lambda^{n}r^{\left(\frac{2}{p}-1\right)} }, \qquad {\left\| \nabla^n\varrho^k_{\xi,\lambda,r} \right\|_{L^p(\mathbb{T}^3)} \lesssim \lambda^{n}r^{\left(\frac{2}{p}-1\right)} }
    \end{equation}
    and
    \begin{equation}\label{e:pipe:estimates:2}
    {\left\| \nabla^n\UU^k_{\xi,\lambda,r} \right\|_{L^p(\mathbb{T}^3)} \lesssim \lambda^{n-1}r^{\left(\frac{2}{p}-1\right)} }, \qquad {\left\| \nabla^n\WW^k_{\xi,\lambda,r} \right\|_{L^p(\mathbb{T}^3)} \lesssim \lambda^{n}r^{\left(\frac{2}{p}-1\right)} }.
    \end{equation}
    \item\label{item:pipe:6} Let $\Phi:\mathbb{T}^3\times[0,T]\rightarrow \mathbb{T}^3$ be the periodic solution to the transport equation
\begin{align}
\label{e:phi:transport}
\partial_t \Phi + v\cdot\nabla \Phi =0\,, 
\qquad 
\Phi_{t=t_0} &= x\, ,
\end{align}
with a smooth, divergence-free, periodic velocity field $v$. Then
\begin{equation}\label{eq:pipes:flowed:1}
\nabla \Phi^{-1} \cdot \left( \WW^k_{\xi,\lambda,r} \circ \Phi \right) = \curl \left( \nabla\Phi^T \cdot \left( \mathbb{U}^k_{\xi,\lambda,r} \circ \Phi \right) \right).
\end{equation}

\item\label{item:pipe:7} For $\mathbb{P}_{[\lambda_1,\lambda_2]}$ a Littlewood-Paley projector, $\Phi$ as in \eqref{e:phi:transport}, $A=(\nabla\Phi)^{-1}$, and for $i=1,2,3$,
\begin{align}
\bigg{[} \nabla \cdot \bigg{(} A \, \mathbb{P}_{[\lambda_1,\lambda_2]} \left(  \WW_{\xi,\lambda,r} \otimes  \WW_{\xi,\lambda,r} \right)(\Phi) A^T \bigg{)} \bigg{]}_i 
& = A_{m}^j \mathbb{P}_{[\lambda_1,\lambda_2]} \left( \WW^m_{\xi,\lambda,r}  \WW_{\xi,\lambda,r}^l\right)(\Phi) \partial_j A_{l}^i\nonumber\\
& = A_m^j \xi^m \xi^l \partial_jA_{l}^i \, \mathbb{P}_{[\lambda_1,\lambda_2]}\left( \left( \varrho^k_{\xi,\lambda,r} \right)^2 \right) \, .
\label{eq:pipes:flowed:2}
\end{align}
\end{enumerate}
\end{proposition}

\begin{remark}
\label{sec:mollifiers:Fourier}
In \eqref{eq:pipes:flowed:2} and throughout the rest of the paper, for any interval $I \subset \mathbb{R}_+$ we use the notation 
\begin{equation}\label{eq:PP:def}
 \Proj_{I} 
\end{equation}
to denote the Fourier projection operator onto spatial frequencies $\xi$ such that $ |\xi| \in I$. When $I = [\lambda,\infty)$ we abbreviate this projection as $\Proj_{\geq \lambda}$, while for $I = [0,\lambda]$, we abbreviate this projection as $\Proj_{\leq \lambda}$.
\end{remark}

In order to propagate sharp $L^\infty$ estimates for nonlinear error terms, we will require the following estimates related to the mean-subtracted squared pipe densities.

\begin{lemma}\label{lem:tricky:tricky}
Let $\varrho^k_{\xi,\lambda,r}:\left(\frac{\T^3}{\lambda r}\right)\rightarrow\R$ be defined as in Proposition~\ref{prop:pipeconstruction}. Let $\lambda_1,\lambda_2$ be given with $\lambda r \leq \lambda_1 < \lambda ,\lambda_2 $, and set 
\begin{equation}
\vartheta=\left(\lambda_1^{-2}\Delta\right)^{-\dpot}\mathbb{P}_{[\lambda_1,\lambda_2)}\left((\varrho_{\xi,\lambda,r}^{k})^2-1\right) \, . \notag
\end{equation}
Then, for an arbitrary $\alpha\in(0,1]$ and $N\leq 2\Nfin$, we have the estimates
\begin{subequations}
\begin{align}
    \left\| D^{N} \mathbb{P}_{[\lambda_1,\lambda_2)} \left( \left(\varrho_{\xi,\lambda,r}^{k}\right)^2-1\right)  \right\|_{L^\infty} &\lesssim \left( \frac{\min(\lambda_2,\lambda)}{\lambda r} \right)^2  \min(\lambda_2,\lambda)^N \label{eq:tricky:bounds:1} \\
    \left\| D^N \vartheta \right\|_{L^\infty} &\lesssim \lambda^\alpha \left( \frac{\min(\lambda_2,\lambda)}{\lambda r} \right)^2 \MM{N,2\dpot,\lambda_1,\min(\lambda_2,\lambda)} \label{eq:tricky:bounds:2} 
    \, . 
\end{align}
\end{subequations}
\end{lemma}

\begin{remark}
When $\lambda_2 \ll  \lambda$, we note that \eqref{eq:tricky:bounds:1} contains the nontrivial estimate
$$  \left\| \mathbb{P}_{[\lambda_1,\lambda_2)} \left(\left(\varrho_{e_3,\lambda,r}^{k}\right)^2-1 \right) \right\|_{L^\infty} \lesssim \left( \frac{\lambda_2}{\lambda r} \right)^2 \ll  \frac{1}{r^{2}} \approx \left\| \left(\varrho_{e_3,\lambda,r}^{k}\right)^2-1 \right\|_{L^\infty} \, , $$
which asserts that the $L^\infty$ norms of the Littlewood-Paley projections of the mean-subtracted pipe density increase with respect to frequency from a minimum of $1$ at $\lambda_2=\lambda r$ to $r^{-2}$ at $\lambda_2=\lambda$.
\end{remark}
\begin{proof}[Proof of Lemma~\ref{lem:tricky:tricky}]
For the sake of simplicity, we fix $\xi=e_3$, and abbreviate $   (\varrho_{e_3,\lambda,r}^{k})^2-1 = \Psi = \Psi(x_1,x_2)$.  Then we have from \eqref{eq:PP:def} that
\begin{align}\label{eq:fourier:series}
    \mathbb{P}_{[\lambda_1,\lambda_2)}\Psi(x) = \sum_{\substack{\lambda_1 \leq |k| < \lambda_2, \\ k \in \lambda r \mathbb{Z}^2 }} \hat{\Psi}(k) e^{   i k \cdot x} \, .
\end{align}
From \eqref{e:pipe:estimates:1}, we may bound
\begin{equation}\label{eq:Phi:hat:bound}
   \bigl| \hat{\Psi}(k) \bigr| \lesssim \left\| \Psi \right\|_{L^1(\T^3)} \lesssim \left\| \varrho_{e_3,\lambda,r} \right\|_{L^2}^2 +1 \lesssim 1 \, .
\end{equation}
A simple counting argument further yields that
\begin{equation}\label{eq:counting:frequencies}
    \left| \left\{\lambda_1 \leq |k| < \lambda_2 \, : \,  k \in \lambda r \mathbb{Z}^2\right\} \right| \lesssim \left(\frac{\lambda_2}{\lambda r} \right)^2 \, .
\end{equation}
Then in the case $\lambda_2 \leq \lambda$, the bounds   \eqref{eq:fourier:series}-\eqref{eq:counting:frequencies} give that
\begin{align}
    \left\| D^N \mathbb{P}_{[\lambda_1,\lambda_2)} \Psi \right\|_{L^\infty} \leq \lambda_2^N \sum_{\substack{\lambda_1 \leq |k| < \lambda_2 \\ k \in \lambda r \mathbb{Z}^2}} \bigl|\hat{\Psi}(k)\bigr| \lesssim  \lambda_2^N \left(\frac{\lambda_2}{\lambda r} \right)^2\, ,
\end{align}
which matches the desired bound in \eqref{eq:tricky:bounds:1}. To prove \eqref{eq:tricky:bounds:1} in the case that $\lambda_2>\lambda$, we simply appeal to the boundedness of $\mathbb{P}_{[\lambda_1,\lambda_2)}$ on $L^\infty$ and \eqref{e:pipe:estimates:1}.

In order to prove \eqref{eq:tricky:bounds:2}, standard Littlewood-Paley arguments and the above bound for $\mathbb{P}_{[\lambda_1,\lambda_2)}\Psi$ in $L^\infty$ again give that
\begin{align*}
    \lambda_1^{2\dpot} \left\| D^N
    \Delta^{-\dpot} \mathbb{P}_{[\lambda_1,\lambda_2)} \Psi \right\|_{L^\infty} \lesssim 
    \begin{dcases}
    \left(\frac{\min(\lambda_2,\lambda)}{\lambda r} \right)^2\lambda^\alpha\lambda_1^{2\dpot-(2\dpot-N)} &\mbox{if}\qquad 0 \leq N \leq 2\dpot \\
    \left(\frac{\min(\lambda_2,\lambda)}{\lambda r} \right)^2\lambda^\alpha\lambda_1^{2\dpot}\min(\lambda_2,\lambda)^{N-2\dpot} &\mbox{if}\qquad 2\dpot+1 \leq N \leq 2 \Nfin \, ,
    \end{dcases}
\end{align*}
where the factor of $\lambda^\alpha$ is used to absorb endpoint ($p=\infty$) losses, and $\alpha$ may be taken arbitrarily close to zero at the cost of changing the implicit constants. Translating the above display to incorporate the notation $\MM{N,2\dpot, \lambda_1,\min(\lambda_2,\lambda)}$ concludes the proof.
\end{proof}

We will require \cite[Lemma 4.7]{BMNV21}, which lists the geometric properties of deformed intermittent pipe flows.

\begin{lemma}[\bf Control on Axes, Support, and Spacing]
\label{lem:axis:control}
Consider a convex neighborhood  of space $\Omega\subset \mathbb{T}^3$. Let $v$ be an incompressible velocity field, and define the flow $X(x,t)$ and inverse $\Phi(x,t)=X^{-1}(x,t)$, which solves
\begin{align}\notag
\partial_t \Phi + v\cdot\nabla \Phi =0\,, 
\qquad
\Phi_{t=t_0} &= x\, .
\end{align}
Define $\Omega(t):=\{ x\in\mathbb{T}^3 : \Phi(x,t) \in \Omega \} = X(\Omega,t)$. For an arbitrary $C>0$, let $\tau>0$ be a parameter such that
\begin{equation}\notag
 \tau\leq\bigl(\delta_q^{\sfrac{1}{2}}\lambda_q\Gamma_{q+1}^{C+2}\bigr)^{-1}  \, .
\end{equation}
Furthermore, suppose that  the vector field $v$ satisfies the Lipschitz bound
\begin{equation}\notag
\sup_{t\in [t_0 - \tau,t_0+\tau]} \norm{\nabla v(\cdot,t) }_{L^\infty(\Omega(t))} \lesssim \delta_q^{\sfrac{1}{2}}\lambda_q\Gamma_{q+1}^C \,.
\end{equation}
Let $\WW^k_{\xi,\lambda_{q+1},r}:\mathbb{T}^3\rightarrow\mathbb{R}^3$ be a set of straight pipe flows constructed as in Proposition~\ref{prop:pipe:shifted} and Proposition~\ref{prop:pipeconstruction} which are $\sfrac{\mathbb{T}^3}{\lambda_{q+1}r}$-periodic for $\lambda_q \lambda_{q+1}^{-1}\leq r\leq 1$ and are concentrated around axes $\{A_i\}_{i\in\mathcal{I}}$ oriented in the vector direction $\xi$ for $\xi\in\Xi$, passing through the grid-points in item~\eqref{item:point:2} of Proposition~\ref{prop:pipe:shifted}.  Then $\WW:=\WW^k_{\xi,\lambda_{q+1},r}(\Phi(x,t)):\Omega(t)\times[t_0-\tau,t_0+\tau]$ satisfies the following conditions:
\begin{enumerate}[(1)]
	\item  We have the inequality
	\begin{equation}\label{eq:diameter:inequality}
	\textnormal{diam}(\Omega(t)) \leq \left(1+\Gamma_{q+1}^{-1}\right)\textnormal{diam}(\Omega) \, .
	\end{equation}
    \item If $x$ and $y$ with $x\neq y$ belong to a particular axis $A_i\subset\Omega$, then 
    \begin{equation}\label{e:axis:variation}
    \frac{X(x,t)-X(y,t)}{|X(x,t)-X(y,t)|} = \frac{x-y}{|x-y|} + \delta_i(x,y,t)    
    \end{equation}
    where $|\delta_i(x,y,t)|<\Gamma_{q+1}^{{-1}}$.
    \item Let $x$ and $y$ belong to $A_i \cap\Omega$, for some $i$, where the axes $A_i$ are   defined above.  Denote the length of the axis $A_i(t):=X(A_i\cap\Omega,t)$ in between $X(x,t)$ and $X(y,t)$ by $L(x,y,t)$.  Then
    \begin{equation}\label{e:axis:length}
    L(x,y,t) \leq \left(1+\Gamma_{q+1}^{-1}\right)\left| x-y \right| \, .
    \end{equation}
    \item The support of $\WW$ is contained in a $\displaystyle\left(1+\Gamma_{q+1}^{-1}\right)\twopi (4n_\ast\lambda_{q+1})^{-1}$-neighborhood of the set
    \begin{equation}\label{e:axis:union}
       \bigcup_{i} A_i(t) \, .
    \end{equation}
\item $\WW$ is ``approximately periodic" in the sense that for distinct axes $A_i,A_j$ with $i\neq j$, we have
\begin{equation}\label{e:axis:periodicity:1}
    \left(1-\Gamma_{q+1}^{-1}\right) \dist(A_i\cap\Omega,A_j\cap\Omega)
    \leq \dist\left(A_i(t),A_j(t)\right)
    \leq \left(1+\Gamma_{q+1}^{-1}\right) \dist(A_i\cap\Omega,A_j\cap\Omega) \, .
\end{equation}
\end{enumerate}
\end{lemma}

The following proposition is a variation on the statement and proof of \cite[Proposition 4.8]{BMNV21}. For simplicity, we only consider $\xi = e_3$. The generalization to other vectors $\xi \in \Xi$ follows from incorporating a rotation into the argument; for further details we refer to the final paragraph of the proof of \cite[Proposition 4.8]{BMNV21}. The main difference in the new Proposition is that the set on which placements are made now has dimensions $(\lambda_{q+1}r_2)^{-1}\times(\lambda_{q+1}r_2)^{-1}\times(\lambda_{q+1}r_1)^{-1}$ as opposed to $(\lambda_{q+1}r_1)^{-1}\times(\lambda_{q+1}r_1)^{-1}\times(\lambda_{q+1}r_1)^{-1}$ in~\cite{BMNV21}.

\begin{proposition}[\bf Placing straight pipes which avoid bent pipes]
\label{prop:disjoint:support:simple:alternate} 

Let $\lambda_q\lambda_{q+1}^{-1} \leq r_1 \leq r_2 \leq 1$ be such that $\lambda_{q+1}r_2 \in \mathbb{N}$. Let $\Omega\subset \T^3$ be a rectangular prism with the following properties:
\begin{enumerate}[(1)]
    \item The longest axis of $\Omega$ is parallel to $e_3$ and has length precisely $(\lambda_{q+1}r_1)^{-1}$. 
    \item There exists a constant $\const_\Omega$ (bounded independently of $q$) such that the face of $\Omega$ which is perpendicular to $e_3$ is a square of side length precisely $\const_\Omega(\Gamma_{q+1}^{-1}\lambda_{q+1}r_2)^{-1}$.
    \item\label{item:pipe:placement:three}  There exists a constant $\const_P$ such that for any convex subset $\Omega'\subset\Omega$ with $\textnormal{diam}\left(\Omega'\right) \leq 2\sqrt{3}\pi \left(\lambda_{q+1} r_2 \right)^{-1}$, there exist at most $\const_P\Gamma_{q+1}$ segments of deformed $\sfrac{\T^3}{\lambda_{q+1}r_2}$-periodic pipes of length $4\pi\left(\lambda_{q+1} r_2 \right)^{-1}$. Here, by ``segments of deformed pipes," we mean the objects constructed in Propositions~\ref{prop:pipe:shifted} and \ref{prop:pipeconstruction} which satisfy the conclusions \eqref{eq:diameter:inequality}--\eqref{e:axis:periodicity:1} from Lemma~\ref{lem:axis:control} on $\Omega$. Let $\mathsf{P}$ denote the union of the supports of the deformed pipe segments.
\end{enumerate}
Then, there exists a {\em geometric constant} $C_*\geq 1$ such that if 
\begin{align}\label{eq:r1:r2:condition:alt}
C_* \const_\Omega^2 \const_P \Gamma_{q+1}^3 r_2^2 \leq r_1 \, ,
\end{align}
then there exists a set of pipe flows $\WW^{k_0}_{e_3,\lambda_{q+1},r_2} \colon \T^3 \to \R^3$ which are $\sfrac{\T^3}{\lambda_{q+1}r_2}$-periodic, concentrated to width $2\pi(4\lambda_{q+1} n_*)^{-1}$ around axes with vector direction $e_3$, satisfy the properties listed in Proposition~\ref{prop:pipeconstruction}, and
\begin{align}\label{e:disjoint:conclusion}
\supp  \WW^{k_0}_{e_3,\lambda_{q+1},r_2}  \cap \mathsf{P} \cap \Omega = \emptyset \, .
\end{align}
\end{proposition}

\begin{proof}[Proof of Proposition~\ref{prop:disjoint:support:simple:alternate}]
The proof has been streamlined relative to the original version \cite[Proposition 4.8]{BMNV21}, although the fundamental ideas remain unchanged. We divide the proof into three steps, in which we count the number of segments of deformed pipe of length $\approx(\lambda_{q+1}r_2)^{-1}$, then project each segment onto the smallest face of $\Omega$ and cover it with squares of size $\approx \lambda_{q+1}^{-1}$, and finally use a pigeonhole argument and the bound \eqref{eq:r1:r2:condition:alt} to find a shift $k_0$ satisfying \eqref{e:disjoint:conclusion}.

\textbf{Step 1:}~To count the number of {deformed} segments of pipe which may comprise $\mathsf{P}\cap \Omega$, we appeal to assumption~\eqref{item:pipe:placement:three} and volume considerations.  The dimensions of $\Omega$ imply that $\Omega$ is composed of at most $\const_\Omega^2\Gamma_{q+1}^{2}\cdot r_2 r_1^{-1} $ periodic cells of side length $2\pi(\lambda_{q+1}r_2)^{-1}$. Applying \eqref{item:pipe:placement:three} with each of these cells implies that the number of distinct segments of pipe of length $4\pi(\lambda_{q+1}r_2)^{-1}$ comprising $\mathsf{P}$ is at most
$$\const_P\const_\Omega^2\Gamma_{q+1}^3\cdot r_2 r_1^{-1} \, . $$

\textbf{Step 2:}~We now measure the size of the shadows of the deformed segments of pipe when projected onto the face of $\Omega$ which is perpendicular to $e_3$. First, the length constraint on the segments of deformed pipe implies that the projection of any single segment onto the face of $\Omega$ which is perpendicular to $e_3$ has length at most $4\pi(\lambda_{q+1}r_2)^{-1}$. Now consider the grid $\mathbb{G}_{\lambda_{q+1},r_2}$ from Proposition~\ref{prop:pipe:shifted}, item \eqref{item:point:2}. This grid contains squares of diameter $\approx\lambda_{q+1}^{-1}$, each of which may contain part of the support of an $e_3$-oriented periodic pipe flow, or may be empty, depending on the choice of shift.  Applying a covering argument using the above derived length constraint and \eqref{e:axis:union}, we see that there exists a dimensional constant $C_*$ such that the number of grid squares needed to cover the projection of a single segment is at most $C_* r_2^{-1}$.  Since the number of segments was bounded by $\const_P \const_\Omega^2 \Gamma_{q+1}^3 \cdot r_2 r_1^{-1} $ from Step 1, we see that the \emph{total} number of grid squares needed to cover the projection of $\mathsf{P}$ is at most
$$ \const_P \const_\Omega^2 \Gamma_{q+1}^3 \cdot r_2 r_1^{-1} \cdot C_* r_2^{-1} \leq \const_P \const_\Omega^2 C_* \Gamma_{q+1}^3 r_1^{-1} \, . $$

\textbf{Step 3:}~In order to conclude the proof, we appeal to a pigeonhole argument, made possible by the bound from Step 2. Indeed, we have obtained an upper bound on the number of grid squares which are deemed ``occupied'' by projections of deformed segments of pipe.  Conversely, from Proposition~\ref{prop:pipe:shifted}, the number of possible choices for the shifts $k_0$ is $r_2^{-2}$. Applying assumption~\eqref{eq:r1:r2:condition:alt}, we conclude by the pigeonhole principle that there exists a ``free'' shift  $k_0$ 
such that \emph{none} of the occupied squares intersect the support of $\WW_{\lambda_{q+1},r_2,e_3}^{k_0}$. Thus we have proven \eqref{e:disjoint:conclusion}, concluding the proof of the lemma. 
\end{proof}

\section{Mollification}
\label{sec:mollification:stuff}

Let $\phi(\zeta):\mathbb{R}\rightarrow \mathbb{R}$ be a smooth, $C^\infty$ function compactly supported in the set $\{\zeta: |\zeta|\leq 1 \}$ which in addition satisfies
\begin{equation}\notag
\int_{\R} \phi(\zeta) \,d\zeta = 1, \qquad \int_{\R} \phi(\zeta) \zeta^n =0 \quad \forall n=1,2,...,\Nindv.
\end{equation}
Let $\tilde{\phi}(x):\mathbb{R}^3\rightarrow \mathbb{R}$ be defined by $\tilde{\phi}(x)=\phi(|x|)$. For $\lambda,\mu\in\mathbb{R}$, define
\begin{equation}\notag
    \phi_{\lambda}^{(x)}(x) = {{\lambda}^3} \tilde{\phi}\left( \lambda x \right), \qquad \phi_\mu^{(t)}(t) = \mu \phi(\mu t).
\end{equation}
For $q\in\mathbb{N}$, we will define the spatial and temporal convolution operators
\begin{equation}\label{mollifier:operators}
   \Pqx := \phi_{\tilde{\lambda}_q}^{(x)} \ast, \qquad \Pqt := \phi_{\tilde{\tau}_{q-1}^{-1}}^{(t)} \ast , \qquad \Pqxt := \Pqx \circ \Pqt.
\end{equation}

\begin{lemma}[\textbf{Mollifying the Euler-Reynolds system}]\label{lem:mollifying:ER}
Let $(v_q,\RR_q)$ solve the Euler-Reynolds system \eqref{eq:Euler:Reynolds:again}, and assume that $\psi_{i,q'}, u_{q'}$ for $q'<q$, $w_q$, and $\RR_q$ satisfy \eqref{eq:inductive:assumption:derivative}--\eqref{eq:nasty:Dt:wq:WEAK:old}. Then, we mollify $(v_q,\RR_q)$ at spatial scale $\tilde \lambda_q^{-1}$ and temporal scale $\tilde \tau_{q-1}$ (cf.~the notation in \eqref{mollifier:operators}), and accordingly define 
\begin{align}
\vlq:=\Pqxt v_q
\qquad \mbox{and} \qquad 
\RR_{\ell_q}:=\Pqxt \RR_q \,.
\label{eq:vlq:Rlq:def}
\end{align}
The mollified velocity $v_{\ell_q}$ satisfies the Euler-Reynolds system with stress $\RR_{\ell_q} +  \RR_{q}^{\textnormal{comm}}$, where the commutator stress $\RR_q^{\textnormal{comm}}$ satisfies the estimate (consistent with \eqref{eq:Rq:inductive:assumption} and \eqref{eq:Rq:inductive:uniform} at level $q+1$)
\begin{align}
\label{eq:Rqcomm:bound}
\bigl\| D^N \Dtq^M \RR_q^\textnormal{comm} \bigr\|_{L^\infty} 
\leq  \Gamma_{q+1}^{-1}  \shaqqplusone \delta_{q+2} \lambda_{q+1}^N
\MM{M, \Nindt, \tau_{q}^{-1}, \Gamma_q^{-1}\tilde \tau_q^{-1} }
\end{align}
for all  $N, M \leq 3 \Nindv$, and we have that
\begin{align}
\label{eq:vq:minus:mollified}
\norm{D^N D_{t,q-1}^M (v_{\ell_q} - v_q)}_{L^\infty} 
\leq 
\lambda_q^{-2}   \delta_{q}^{\sfrac 12} \MM{N,2\Nindv,\lambda_q,\tilde \lambda_q} \MM{M, \Nindvt ,\tau_{q-1}^{-1} , \Tilde{\tau}_{q-1}^{-1} \Gamma_q^{-1}}
\end{align}
for all  $N, M \leq 3 \Nindv$. 
Furthermore, $u_q = v_{\ell_q} - v_{\ell_{q-1}}$ satisfies the bound \eqref{eq:inductive:assumption:uq:all} with $q'$ replaced by $q$ 
\begin{subequations}
\begin{align}
    \left\| \psi_{i,q-1} D^N D_{t,q-1}^M u_q \right\|_{L^2} &\leq \delta_{q}^{\sfrac{1}{2}} \MM{N, 2 \Nindv,\lambda_{q},\tilde{\lambda}_q} \MM{M, \Nindvt ,\Gamma_{q}^i \tau_{q-1}^{-1}, \Tilde{\tau}_{q-1}^{-1}} \,,
    \label{eq:mollified:velocity}
    \\
\norm{D^N D_{t,q-1}^M u_{q}}_{L^\infty(\supp \psi_{i,q-1})} 
&\leq \Gamma_{q}^{\badshaq} \Theta_{q}^{\sfrac 12}  \MM{N,2\Nindv,\lambda_q, \tilde\lambda_q} 
\MM{M,\Nindvt,  \Gamma_q^{i+1} \tau_{q-1}^{-1}, \tilde{\tau}_{q-1}^{-1}}  \,,  \label{eq:inductive:assumption:uniform:mol}
\end{align}
\end{subequations}
for all $N+M\leq 2\Nfin$. Finally, $\RR_{\ell_q}$ satisfies bounds which extend \eqref{eq:Rq:inductive:assumption:all} to  the mollified stress
\begin{subequations}
\begin{align}
\bigl\| \psi_{i,q-1} D^N D_{t,q-1}^M \RR_{\ell_q} \bigr\|_{L^1} 
&\lesssim \Gamma_q^\shaq \delta_{q+1} \MM{N, 2 \Nindv ,\lambda_q,\Tilde{\lambda}_q}\MM{M, \NindRt, \Gamma_q^{i+2} \tau_{q-1}^{-1}, \Tilde{\tau}_{q-1}^{-1}}  \,,
\label{eq:mollified:stress:bounds}
\\
\bigl\| D^N D_{t,q-1}^M \mathring{R}_{\ell_q}\bigr\|_{L^\infty(\supp \psi_{i,q-1})}
&\lesssim \Gamma_q^\badshaq \MM{N,2\Nindv,\lambda_q, \tilde\lambda_q} 
\MM{M,\Nindvt,  \Gamma_q^{i+3} \tau_{q-1}^{-1} \,, \tilde{\tau}_{q-1}^{-1}} 
\label{eq:Rq:inductive:uniform:moll}
\end{align}
\end{subequations}
for all $N+M\leq 2 \Nfin$.
\end{lemma}
\begin{proof}[Proof of Lemma~\ref{lem:mollifying:ER}]
The bounds in  \eqref{eq:Rqcomm:bound}--\eqref{eq:mollified:velocity}, and also \eqref{eq:mollified:stress:bounds}, match those of \cite[Lemma~5.1, equations (5.3)--(5.5) and (5.7)]{BMNV21}, and so we omit the proofs.  We note that the analogue of estimate \eqref{eq:vq:minus:mollified} in \cite[equation (5.4)]{BMNV21} contains a typo in the sharp material derivative cost.  Specifically, one may replace the cost of $\tau_{q-1}^{-1}\Gamma_{q}^{i-1}$ simply with $\tau_{q-1}^{-1}$ (which is actually the estimate that can be proved using the argument in \cite{BMNV21}). The only new estimates which would require a proof are \eqref{eq:inductive:assumption:uniform:mol} and \eqref{eq:Rq:inductive:uniform:moll}.

In order to give an idea of how to prove \eqref{eq:inductive:assumption:uniform:mol}, we follow the method of proof from \cite{BMNV21} for \eqref{eq:mollified:velocity}. When either $N\geq 3\NindLarge$ or $M\geq3\NindLarge $, an even stronger bound than \eqref{eq:inductive:assumption:uniform:mol} was previously established in \cite[Lemma~5.1, equation~(5.6)]{BMNV21}.  Thus, we only need to consider \eqref{eq:inductive:assumption:uniform:mol} for $N,M\leq 3\NindLarge$. We appeal to \eqref{inductive:velocity:frequency} and split $u_q = \Pqxt w_q + \left( \Pqxt v_{\ell_{q-1}} - v_{\ell_{q-1}}\right)$. Since the good term $(\Pqxt - \Id) v_{\ell_{q-1}}$ was \emph{already} estimated in $L^\infty$, cf. \cite[equation~(5.43)]{BMNV21} with a stronger bound than that required by \eqref{eq:inductive:assumption:uniform:mol}, we can consider just the main term $\Pqxt w_q$. We split $\Pqxt w_q$ as $\Pqxt w_q = w_q + (\Pqxt-\Id) w_q$. In view of \eqref{eq:inductive:asumption:derivative:q:uniform}, which provides a satisfactory bound on $w_q$, we are only left with $(\Pqxt-\Id) w_q$. However, this term was \emph{already} estimated in $L^\infty$ in \cite[equations~(5.33)--(5.35)]{BMNV21}, and so no new proof is required.  Thus \eqref{eq:inductive:assumption:uniform:mol} is satisfied. 

The proof of \eqref{eq:Rq:inductive:uniform:moll} utilizes the same methodology that produced bounds for $\Pqxt w_q$ from inductive assumptions on $w_q$.  Specifically, the material derivative bounds have been relaxed by a factor of $\Gamma_q$ (the second $\Gamma_q$ loss coming again from the fact that \eqref{eq:Rq:inductive:uniform:moll} is estimated on the support of $\psi_{i,q-1}$), the spatial derivative bounds have been relaxed from $\lambda_q$ to $\tilde\lambda_q$ when $N\geq 2\NindLarge$, and the available number of estimates on the un-mollified stress $\RR_q$ was much more than $2\NindLarge$, specifically $3\NindLarge$. We therefore omit any further discussion and refer the reader to the proof of \cite[Lemma~5.1]{BMNV21}.
\end{proof}

\section{Cutoffs}
\label{sec:cutoff}

\subsection{Velocity cutoff functions}
\label{sec:cutoff:velocity:definitions}
For all $q\geq 1$ and $0\leq m\leq\NcutSmall$, we construct the following cutoff functions.  The specifics of the construction and the proof are contained in \cite[Appendix A.2]{BMNV21}. To avoid abuse of notation, here we denote these smooth cutoffs using the capital letters $\Psi_{m,q}$ and $\tilde \Psi_{m,q}$, instead of the notation in \cite[Appendix A.2]{BMNV21} (which was $\psi_{m,q}$ and $\tilde \psi_{m,q}$).

\begin{lemma}\label{lem:cutoff:construction:first:statement}
For all $q\geq 1$ and $0\leq m \leq \NcutSmall$, there exist smooth cutoff functions $\tilde\Psi_{m,q},\Psi_{m,q}:[0,\infty)\rightarrow[0,1]$ which satisfy the following.
\begin{enumerate}[(1)]
    \item\label{item:cutoff:1} The function $\tilde\Psi_{m,q}$ satisfies ${\bf 1}_{[0,\frac{1}{4}\Gamma_q^{2(m+1)}]} \leq \tilde\Psi_{m,q} \leq {\bf 1}_{[0,\Gamma_q^{2(m+1)}]}$.
    	\item\label{item:cutoff:2}  The function $\Psi_{m,q}$ satisfies ${\bf 1}_{[1,\frac{1}{4}\Gamma_q^{2(m+1)}]} \leq \Psi_{m,q} \leq {\bf 1}_{[\frac{1}{4},\Gamma_q^{2(m+1)}]}$.
    \item For all $y\geq 0$, a partition of unity is formed as
    \begin{align}
    \tilde \Psi_{m,q}^2(y) + \sum_{{i\geq 1}} \Psi_{m,q}^2\bigl(\Gamma_{q}^{-2i(m+1)} y\bigr) = 1 \, .
    \label{eq:tilde:partition}
    \end{align}
    \item $\tilde\Psi_{m,q}$ and $\Psi_{m,q}(\Gamma_q^{-2i(m+1)}\cdot)$ satisfy
    \begin{align}
   \supp \tilde\Psi_{m,q}(\cdot) \cap \supp \Psi_{m,q}\bigl(\Gamma_q^{-2i(m+1)}\cdot\bigr) &= \emptyset \quad \textnormal{if} \quad i \geq 2,\notag\\
   \supp \Psi_{m,q}\bigl(\Gamma_q^{-2i(m+1)}\cdot\bigr) \cap \supp \Psi_{m,q}\bigl(\Gamma_q^{-2i'(m+1)}\cdot\bigr) &= \emptyset \quad \textnormal{if} \quad |i-i'|\geq 2. \label{eq:psi:support:base:case}
    \end{align}
    \item For $0\leq N \leq \Nfin$, when $0\leq y<\Gamma_q^{2(m+1)}$ we have
    \begin{align}
    {|D^N \tilde \Psi_{m,q}(y)|}    &\lesssim  {(\tilde \Psi_{m,q}(y))^{1-N/\Nfin}}
     \Gamma_q^{-2N(m+1)}. \notag
    \end{align}
For $\frac{1}{4}<y<1$ we have
    \begin{align}
     {|D^N  \Psi_{m,q}(y)|} &\lesssim {( \Psi_{m,q}(y))^{1- N / \Nfin}} \notag \, ,
    \end{align}
    while for $\frac{1}{4}\Gamma_q^{2(m+1)}<y<\Gamma_q^{2(m+1)}$ we have
    \begin{align}
     {|D^N  \Psi_{m,q}(y)|} &\lesssim \Gamma_q^{-2N(m+1)} {( \Psi_{m,q}(y))^{1- N / \Nfin}} \notag \, .
    \end{align}
In each of the above inequalities, the implicit constants depend on $N$ but not $m$ or $q$.
\end{enumerate}
\end{lemma}

\begin{definition}\label{def:istar:j}
Given $i,j,q \geq 0$, we define
\begin{align*}
i_* = i_*(j,q) = i_*(j) = \min\{ i \geq 0 \colon \Gamma_{q+1}^{i} \geq \Gamma_q^{j} \}.
\end{align*}
\end{definition}
\noindent Note that for $j=0$, we have that $i_*(j)=0$.

At stage $q\geq 1$ of the iteration (by convention $w_0=u_0=0$) and for $m\leq\NcutSmall$ and $j_m\geq 0$, we define
\begin{align}
h_{m,j_m,q}^2(x,t):= 
\sum_{n=0}^{\NcutLarge} \Gamma_{q+1}^{-2i_*\left(j_m\right)} \delta_{q}^{-1} \bigl(\lambda_{q}\Gamma_q\bigr)^{-2n} \bigl(\tau_{q-1}^{-1}\Gamma_{q+1}^{i_*(j_m)+2}\bigr)^{-2m}    |D^{n} D_{t,q-1}^m u_{q}(x,t)|^2 \, .
\label{eq:h:j:q:def}
\end{align}
One should view $h_{m,j_m,q}$ as a measurement of the extent to which the amplitude of $D_{t,q-1}^m u_q$ (or its spatial derivatives) exceeds $\Gamma_{q+1}^{i_*(j_m)}\approx \Gamma_q^j$, where $\tau_{q-1}^{-1}\Gamma_q^j$ is the material derivative cost on the support of $\psi_{j,q-1}$.  The extra room of $\Gamma_q$ in the spatial derivative cost and $\Gamma_{q+1}^2$ in the material derivative cost accrues extra factors of smallness for high numbers of derivatives.  This allows us to eventually plug in a very lossy bound for $|D^n D_{t,q-1}^m u_q |$ and still show that the resulting contribution to the sum is very small.  To measure the size of $h_{m,j_m,q}$ precisely, we now rescale and plug into a cutoff function.

\begin{definition}[\bf Intermediate Cutoff Functions]\label{def:intermediate:cutoffs}
Given $q\geq 1$, $m\leq\NcutSmall$, and $j_m\geq 0$ we define $\psi_{m,i_m,j_m,q}$ by
\begin{align}
 \psi_{m,i_m,j_m,q}(x,t) 
 &= \Psi_{m,q+1}  \bigl( \Gamma_{q+1}^{-2(i_m-i_*(j_m))(m+1)} h_{m,j_m,q}^2  (x,t) \bigr) 
\label{eq:psi:i:j:def}
\end{align}
for $i_m> i_*(j_m)$, 
while for $i_m=i_*(j_m)$,
\begin{align}
 \psi_{m,i_*(j_m),j_m,q}(x,t) 
 &= \tilde \Psi_{m,q+1} \left( h_{m,j_m,q}^2(x,t) \right).
\label{eq:psi:i:i:def}
\end{align}
The intermediate cutoff functions $\psi_{m,i_m,j_m,q}$ are equal to zero for $ i_m < i_*(j_m)$.  
\end{definition}
The indices $i_m$ and $j_m$ {were shown in see \cite[Lemma~6.14]{BMNV21} and \cite[equation~(6.27)]{BMNV21}} to run up to some maximal values $i_{\rm max}$ and $\tilde{i}_{\textnormal{max}}$, although in the present context, it will be necessary to propagate a much sharper bound on $\imax$; see Lemma~\ref{lem:maximal:i}. With this notation and in view of \eqref{eq:tilde:partition} and \eqref{eq:psi:support:base:case}, it immediately follows that
\begin{align}
\sum_{i_m\geq0} \psi_{m,i_m,j_m,q}^2 = \sum_{i_m\geq i_*(j_m)} \psi_{m,i_m,j_m,q}^2 = \sum_{\{ i_m \colon \Gamma_{q+1}^{i_m} \geq \Gamma_q^{j_m} \}} \psi_{m,i_m,j_m,q}^2 \equiv 1
\notag
\end{align}
for any $m$ and for $|i_m-i'_m|\geq 2$, 
\begin{equation}\notag
  \psi_{m,i_m,j_m,q}\psi_{m,i_m',j_m,q}=0.
\end{equation}

\begin{definition}[\bf $m^{\textnormal{th}}$ Velocity Cutoff Function]\label{def:psi:m:im:q:def}
For $i_m\geq 0$, we inductively define the $m^{\textnormal{th}}$ velocity cutoff function
\begin{equation}\label{eq:psi:m:im:q:def}
\psi_{m,i_m,q}^2 = \sum\limits_{\{j_m\colon i_m\geq i_*(j_m)\}} \psi_{j_m,q-1}^2 \psi_{m,i_m,j_m,q}^2.
\end{equation}
\end{definition}
Informally, one may interpret the definition of the $m^{\textnormal{th}}$ velocity cutoff in \eqref{eq:psi:m:im:q:def} as follows.  To control $D_{t,q}$, one should split into $D_{t,q-1}$ and $u_q\cdot \nabla$.  The inclusion of $\psi_{j_m,q-1}$ ensures that the cost of $D_{t,q-1}$ may be controlled by $\tau_{q-1}^{-1}\Gamma_q^{j}$ and the inclusion of $\psi_{m,i_m,j_m,q}$ ensures that the cost of $u_q \cdot \nabla$ may be controlled as well.  The index $m$ is included for technical reasons, as it is more convenient to control the size of $D_{t,q-1}^m u_q$ for fixed $m$.  Therefore, in reality we then control the size of $u_q\cdot\nabla$ only after incorporating the information provided by the different partitions of unity $\{\psi_{m,i_m,q}\}_{i_m}$ for $0 \leq m \leq \NcutSmall$.  Whichever value of $i_m$ is the largest at any point in spacetime then determines the material derivative cost there.

In order to define the full velocity cutoff function, we use the notation
\begin{equation}\notag
  \Vec{i} =  \{i_m\}_{m=0}^{\NcutSmall} = \left( i_0,...,i_{\NcutSmall} \right) \in \mathbb{N}_0^{\NcutSmall+1}
\end{equation}
to denote a tuple of non-negative integers of length $\NcutSmall+1$, and we shall denote 
$$
{\mathcal I}_i = \left\{\Vec{i} \in \mathbb{N}_0^{\NcutSmall+1}\colon\max\limits_{0\leq m\leq\NcutSmall} i_m =i\right\}.
$$

\begin{definition}[\bf Velocity cutoff function]
\label{def:psi:i:q:def}
For $0 \leq i \leq i_{\rm max}(q)$, we inductively define the velocity cutoff function $\psi_{i,q}$ as follows. When $q= 0$, we let 
\begin{align}
\psi_{i,0}=\begin{cases}1&\mbox{if }i=0\\ 0& \mbox{otherwise}.
\notag
\end{cases}
\end{align}
Then, we inductively on $q$ define
\begin{equation}
    \psi_{i,q}^2 = \sum\limits_{{\mathcal I}_i} \prod\limits_{m=0}^{\NcutSmall} \psi_{m,i_m,q}^2.
  \label{eq:psi:i:q:recursive}
\end{equation}
for all $q \geq 1$. 
\end{definition}

The sum used to define $\psi_{i,q}$ for $q\geq 1$ is over all tuples with a maximum entry of $i$.  The number of such tuples is $q$-independent since it has been demonstrated in \cite[Lemma 6.14]{BMNV21} that $i_m\leq\imax(q)$ (which implies $i\leq \imax(q)$), and $\imax(q)$ is bounded above independently of $q$.

For notational convenience, given an $\Vec{i}$ as in the sum of \eqref{eq:psi:i:q:recursive}, we shall denote 
\begin{align}
\supp \prod\limits_{m=0}^{\NcutSmall} \psi_{m,i_m,q}    = \bigcap_{m=0}^{\NcutSmall} \supp(\psi_{m,i_m,q}) =: \supp (\psi_{\Vec{i},q} )\,.
\notag
\end{align}
In particular, we will frequently use that $(x,t) \in \supp(\psi_{i,q})$ if and only if there exists $\Vec{i}\in \N_0^{\NcutSmall+1}$ such that $\max_{0\leq m\leq\NcutSmall} i_m =i$, and $(x,t) \in \supp(\psi_{\Vec{i},q})$.

\begin{proposition}\label{prop:no:proofs}
With the definitions of the velocity cutoff functions given in the previous subsection, the inductive assumptions from \eqref{eq:inductive:partition} and \eqref{eq:sharp:Dt:psi:i:q:old}--\eqref{eq:nasty:Dt:uq:orangutan} hold.
\end{proposition}
For the proof, see \cite[Section 6]{BMNV21}.
 We however must provide a new estimate for $\imax(q)$ in order to prove \eqref{eq:imax:upper:lower} and \eqref{eq:imax:old}, and we give the details in the following lemma.

\begin{lemma}[\bf Maximal $i$ index in the definition of the cutoff]
\label{lem:maximal:i}
There exists $\imax = \imax(q) \geq 0$, determined by the formula \eqref{eq:imax:def} below, such that 
\begin{align}
\psi_{i,q} \equiv 0 \quad \mbox{for all} \quad i > i_{\rm max}
\label{eq:imax}
\end{align}
and
\begin{align}
\Gamma_{q+1}^{i_{\rm max}} \leq \Gamma_{q+1}^{\badshaq} \Theta_q^{\sfrac 12} \delta_q^{-\sfrac 12}
\label{eq:imax:bound}
\end{align}
for all $q\geq 0$. Moreover, assuming $\lambda_0$ is sufficiently large, $\imax(q)$ is bounded uniformly in $q$ as
\begin{align}
\imax(q)  \leq 1 + \badshaq + \frac{\sfrac{1}{2} (b-1) + \beta b}{\eps_\Gamma (b-1)b}
\,.
\label{eq:imax:upper:bound:uniform}
\end{align}
\end{lemma}

\begin{proof}[Proof of Lemma~\ref{lem:maximal:i}]
Assume $i\geq 0$ is such that $\supp(\psi_{i,q}) \neq  \emptyset$. We will prove that $\Gamma_{q+1}^i \leq \Gamma_{q+1}^\badshaq \Theta_q^{\sfrac 12} \delta_q^{-\sfrac 12}$. From \eqref{eq:psi:i:q:recursive} it follows that for any $(x,t) \in \supp (\psi_{i,q})$, there must exist at least one $\Vec{i} = (i_0,\ldots,i_{\NcutSmall})$ such that $\max\limits_{0\leq m\leq\NcutSmall} i_m =i$, and with $\psi_{m,i_m,q}(x,t) \neq 0$ for all $0\leq m \leq \NcutSmall$. Therefore, in light of \eqref{eq:psi:m:im:q:def}, for each such $m$ there exists a maximal $j_m$ such that $i_*(j_m) \leq i_m$, with $(x,t) \in \supp(\psi_{j_m,q-1}) \cap \supp(\psi_{m,i_m,j_m,q})$. In particular, this holds for any of the indices $m$ such that $i_m = i$. For the remainder of the proof, we fix such an index $0\leq m \leq \NcutSmall$.

If we have $i = i_m = i_{*}(j_m) = i_*(j_m,q)$, since $(x,t) \in \supp(\psi_{j_m,q-1}) $, then by the inductive assumption \eqref{eq:imax:old}, we have that $j_m \leq \imax(q-1)$. Then using $\Gamma_{q+1}^{i-1} < \Gamma_q^{j_m} \leq \Gamma_{q}^{\imax(q-1)}$ and \eqref{eq:imax:old}, we deduce that
\[
\Gamma_{q+1}^{i}
\leq \Gamma_{q+1} \Gamma_{q}^{\imax(q-1)}
\leq \Gamma_{q+1}
\Gamma_q^{\badshaq} \Theta_{q-1}^{\sfrac 12} \delta_{q-1}^{-\sfrac 12}
\leq
\Gamma_{q+1}^{\badshaq}
\Theta_{q}^{\sfrac 12} \delta_{q}^{-\sfrac 12}
\,.
\]
The last inequality above holds  
in light of the parameter inequality
$b \eps_\Gamma + \badshaq \eps_\Gamma + \sfrac{1}{2b}
\leq b \badshaq \eps_\Gamma + \sfrac 12 + \beta$, which in turn follows from $\eps_\Gamma \leq \sfrac{\beta}{b}$. Thus, in this case $\Gamma_{q+1}^i \leq \Gamma_{q+1}^\badshaq \Theta_q^{\sfrac 12} \delta_q^{-\sfrac 12}$ indeed holds.

On the other hand, if $i = i_m \geq i_{*}(j_m) +1$, by the definition of $\Psi_{m,q+1}$ in \eqref{eq:psi:i:j:def}, it follows that $|h_{m,j_m,q}(x,t)| \geq (\sfrac{1}{2}) \Gamma_{q+1}^{(m+1)(i_m - i_*(j_m))}$, and by the pigeonhole principle, there exists $0 \leq n \leq \NcutLarge$  with 
\begin{align*}
|D^n D_{t,q-1}^m u_q(x,t)| &\geq \frac{1}{2 \NcutLarge} \Gamma_{q+1}^{(m+1)(i_m - i_*(j_m))} {\Gamma_{q+1}^{i_*(j_m)}} \delta_q^{\sfrac 12} (\lambda_q \Gamma_q)^{n} (\tau_{q-1}^{-1} \Gamma_{q+1}^{i_*(j_m)+2})^m \notag\\
&\geq   \frac{1}{2 \NcutLarge} \Gamma_{q+1}^{i_m} \delta_q^{\sfrac 12} \lambda_q^{n} (\tau_{q-1}^{-1} \Gamma_{q+1}^{i_m+2})^m ,
\end{align*}
and we also know that $(x,t) \in \supp(\psi_{j_m,q-1})$. By \eqref{eq:inductive:assumption:uniform:mol},  the fact that $\NcutLarge \leq 2 \NindLarge$, and $\NcutSmall\leq \NindSmall$, we know that 
\begin{align*}
|D^n D_{t,q-1}^m u_q(x,t)| 
&\leq \Gamma_q^\badshaq \Theta_q^{\sfrac 12} \lambda_q^{n} (\tau_{q-1}^{-1} \Gamma_q^{j_m+1})^m \notag\\
&\leq \Gamma_q^\badshaq \Theta_q^{\sfrac 12} \lambda_q^{n}  (\tau_{q-1}^{-1} \Gamma_{q+1}^{i_*(j_m)+1})^m  \leq \Gamma_q^\badshaq \Theta_q^{\sfrac 12} \lambda_q^{n} (\tau_{q-1}^{-1} \Gamma_{q+1}^{i_m})^m \, .
\end{align*}
The proof is now completed, since the previous two inequalities and $i_m = i$ imply that 
\begin{equation}\label{eq:imax:contradiction:1}
 \Gamma_{q+1}^i \leq 2\NcutLarge \Gamma_q^\badshaq \Theta_q^{\sfrac 12}\delta_q^{-\sfrac 12} 
 \leq \Gamma_{q+1}^\badshaq \Theta_q^{\sfrac 12}\delta_q^{-\sfrac 12} \, .
\end{equation}

In view of the above inequality, the value of $i_{\rm max}$ is chosen as
\begin{align}
\imax(q) = \sup \{i' \, : \, \Gamma_{q+1}^{i'} \leq \Gamma_{q+1}^\badshaq \Theta_q^{\sfrac 12} \delta_q^{-\sfrac 12} \} \, .
\label{eq:imax:def}
\end{align}
With this definition, if $i > \imax(q)$, then $\Gamma_{q+1}^i > \Gamma_{q+1}^\badshaq \Theta_q^{\sfrac 12}\delta_q^{-\sfrac 12}$, and as such $\supp(\psi_{i,q})= \emptyset$. 
To show that $\imax(q)$ is bounded  independently of $q$, note that 
\begin{align*}
\frac{\log(  \Gamma_{q+1}^\badshaq \Theta_q^{\sfrac 12}\delta_q^{-\sfrac 12})}{\log (\Gamma_{q+1})} = \badshaq + \frac{\left( \sfrac{1}{2}(b-1) + \beta b\right) \log(\lambda_{q-1})}{\eps_\Gamma (b-1) \log (\lambda_q)} \to \badshaq + \frac{\sfrac{1}{2} (b-1) + \beta b}{\eps_\Gamma (b-1)b} \, ,
\end{align*}
as $q\to \infty$. Thus, assuming $\lambda_0$ is sufficiently large,  the bound \eqref{eq:imax:upper:bound:uniform} holds.
\end{proof}

\subsection{Temporal cutoff functions and flow maps}
\label{sec:cutoff:temporal:definitions}
Let $\chi:(-1,1)\rightarrow[0,1]$ be a $C^\infty$ function of compact support which induces a partition of unity according to 
\begin{align}
 \sum_{k \in \Z} \chi^2(\cdot - k) \equiv 1 \, .
\label{eq:chi:cut:partition:unity}
\end{align}
Consider the translated and rescaled function 
\begin{equation*}
    \chi\left(t \tau_{q}^{-1}\Gamma^{i-\cstar+2}_{q+1} - k\right) \, ,
\end{equation*}
which is supported in the set of times $t$ satisfying
\begin{equation}\label{eq:chi:support}
\left| t-\tau_q \Gamma_{q+1}^{-i+\cstar-2} k \right| \leq \tau_q\Gamma_{q+1}^{-i+\cstar-2} \qquad \iff \qquad t\in \left[ (k-1)\tau_q \Gamma_{q+1}^{-i+\cstar-2}, (k+1)\tau_q \Gamma_{q+1}^{-i+\cstar-2} \right]  \, .
\end{equation}
We then define temporal cut-off functions 
\begin{align}
 \chi_{i,k,q}(t)=\chi_{(i)}(t) =  \chi\left(t \tau_{q}^{-1}\Gamma^{i-\cstar+2}_{q+1} - k\right) \, .
 \label{eq:chi:cut:def}
\end{align}
It is then clear that 
\begin{align}
{|\partial_t^m \chi_{i,k,q}| \les (\Gamma_{q+1}^{i-\cstar+2} \tau_{q}^{-1})^m}
\label{eq:chi:cut:dt}
\end{align}
for $m\geq 0$ and
\begin{equation}\label{e:chi:overlap}
    \chi_{i,k_1,q}(t)\chi_{i,k_2,q}(t) = 0
\end{equation}
for all $t\in\mathbb{R}$ unless $|k_1-k_2|\leq 1$. We define
\begin{equation}\notag
    \chi_{(i, k \pm, q)}(t) := \bigl( \chi_{(i,k-1,q)}^2(t) + \chi_{(i,k,q)}^2(t) + \chi_{(i,k+1,q)}^2(t)  \bigr)^{\sfrac{1}{2}},
\end{equation}
which are cutoffs with the property that
\begin{equation}\notag
    \chi_{(i,k\pm,q)} \equiv 1 \qquad \textnormal{ on } \qquad \supp{(\chi_{(i,k,q)})}.
\end{equation}

Next, we define the cutoffs $\tilde\chi_{i,k,q}$ by
\begin{equation}\notag
\tilde\chi_{i,k,q}(t)=\tilde\chi_{(i)}(t) = \chi\left( t \tau_q^{-1}\Gamma_{q+1}^{i-\cstar} - \Gamma_{q+1}^{-\cstar}k \right).
\end{equation}
For comparison with \eqref{eq:chi:support}, we have that $\tilde\chi_{i,k,q}$ is supported in the set of times $t$ satisfying
\begin{equation}\notag
\bigl| t-\tau_q \Gamma_{q+1}^{-i+\cstar} k \bigr| \leq \tau_q\Gamma_{q+1}^{-i+\cstar}.
\end{equation}
As a consequence of these definitions and a sufficiently large choice of $\lambda_0$, if $(i,k)$ and $(\istar,\kstar)$ satisfy $\supp \chi_{i,k,q} \cap \supp \chi_{\istar,\kstar,q}\neq\emptyset$ and $\istar\in\{i-1,i,i+1\}$, then
\begin{equation}\label{eq:tilde:chi:contains}
\supp \chi_{i,k,q} \subset \supp \tilde\chi_{\istar,\kstar,q}.
\end{equation}

We can now make estimates regarding the flows of the vector field $\vlq$ on the support of a cutoff function. The proofs of Lemma~\ref{lem:dornfelder} and Corollary~\ref{cor:deformation} are contained in \cite[Section 6.4]{BMNV21}.

\begin{lemma}[\bf Lagrangian paths don't jump many supports]
\label{lem:dornfelder}
Let $q \geq 0$ and $(x_0,t_0)$ be given. Assume that the index $i$ is such that $\psi_{i,q}^2(x_0,t_0) \geq \kappa^2$, where $\kappa\in\left[\frac{1}{16},1\right]$. Then the forward flow $(X(t),t) := (X(x_0,t_0;t),t)$ of the velocity field $\vlq$ originating at $(x_0,t_0)$ has the property that $\psi_{i,q}^2(X(t),t) \geq\sfrac{\kappa^2}{2}$ for all $t$ be such that $|t - t_0|\leq (\delta_q^{\sfrac 12}\lambda_q \Gamma_{q+1}^{i+3})^{-1}$, which by \eqref{eq:Lambda:q:x:1:NEW} and \eqref{eq:tilde:lambda:q:def} is satisfied for $|t-t_0|\leq \tau_q \Gamma_{q+1}^{-i+5+\cstar}$.
\end{lemma}

\begin{definition}\label{def:transport:maps} We define $\Phi_{i,k,q}(x,t):=\Phi_{(i,k)}(x,t)$ to be the flows induced by $\vlq$ with initial datum at time $k {\tau_{q}}\Gamma_{q+1}^{-i}$ given by the identity, i.e.
\begin{equation}
\notag
(\partial_t + \vlq \cdot\nabla) \Phi_{i,k,q} = 0 \,, 
\qquad
\Phi_{i,k,q}(x,k{\tau_{q}}\Gamma_{q+1}^{-i})=x\, .
\end{equation}
\end{definition}

We will use $D\Phi_{(i,k)}$ to denote the gradient of $\Phi_{(i,k)}$.  The inverse of the matrix $D\Phi_{(i,k)}$ is denoted by $\left(D\Phi_{(i,k)}\right)^{-1}$, in contrast to $D\Phi_{(i,k)}^{-1}$, which is the gradient of the inverse map $\Phi_{(i,k)}^{-1}$.

\begin{corollary}[\bf Deformation bounds]
\label{cor:deformation}
For $k \in \Z$, $0 \leq i \leq  i_{\rm max}$, $q\geq 0$, and $2 \leq N \leq \sfrac{3\Nfin}{2}+1$, we have the following bounds on the support of $\psi_{i,q}(x,t){\tilde\chi_{i,k,q}(t)}$.
\begin{subequations}
\begin{align}
 \norm{D\Phi_{(i,k)} - {\rm Id}}_{L^\infty(\supp(\psi_{i,q} \tilde\chi_{i,k,q} ))} &\lesssim \Gamma_{q+1}^{-1}
 \label{eq:Lagrangian:Jacobian:1}\\
  \norm{D^N\Phi_{(i,k)} }_{L^\infty(\supp(\psi_{i,q} \tilde\chi_{i,k,q} ))} & \lesssim \Gamma_{q+1}^{-1} \MM{N-1, 2\Nindv, \Gamma_q\lambda_q,\tilde\lambda_q}\label{eq:Lagrangian:Jacobian:2}\\
  \norm{(D\Phi_{(i,k)})^{-1} - {\rm Id}}_{L^\infty(\supp(\psi_{i,q} \tilde\chi_{i,k,q} ))} & \lesssim \Gamma_{q+1}^{-1}\label{eq:Lagrangian:Jacobian:3}\\
  \norm{D^{N-1}\left((D\Phi_{(i,k)})^{-1}\right) }_{L^\infty(\supp(\psi_{i,q} \tilde\chi_{i,k,q} ))} & \lesssim \Gamma_{q+1}^{-1} \MM{N-1, 2\Nindv, \Gamma_q\lambda_q,\tilde\lambda_q} \label{eq:Lagrangian:Jacobian:4} \\
   \bigl\|D^N\Phi^{-1}_{(i,k)} \bigr\|_{L^\infty(\supp(\psi_{i,q} \tilde\chi_{i,k,q} ))} & \lesssim \Gamma_{q+1}^{-1} \MM{N-1, 2\Nindv, \Gamma_q\lambda_q,\tilde\lambda_q}\label{eq:Lagrangian:Jacobian:7}
\end{align}
Furthermore, we have the following bounds for $1\leq N+M\leq \sfrac{3\Nfin}{2}$:
\begin{align}
    \bigl\| D^{N-N'} D_{t,q}^M D^{N'+1} \Phi_{(i,k)} \bigr\|_{L^\infty(\supp(\psi_{i,q}\tilde\chi_{i,k,q}))} &\leq  \tilde{\lambda}_q^{N} \MM{M,\NindSmall,\Gamma_{q+1}^{i-\cstar} \tau_q^{-1},\tilde{\tau}_q^{-1}\Gamma_{q+1}^{-1}}\label{eq:Lagrangian:Jacobian:5}\\
     \bigl\| D^{N-N'} D_{t,q}^M D^{N'} (D \Phi_{(i,k)})^{-1} \bigr\|_{L^\infty(\supp(\psi_{i,q}\tilde\chi_{i,k,q}))} &\leq \tilde{\lambda}_q^{N} \MM{M,\NindSmall,\Gamma_{q+1}^{i-\cstar} \tau_q^{-1},\tilde{\tau}_q^{-1}\Gamma_{q+1}^{-1}}\label{eq:Lagrangian:Jacobian:6}
\end{align}
\end{subequations}
for all $0\leq N'\leq N$.
\end{corollary}

\subsection{Stress estimates and stress cutoff functions}
\label{sec:cutoff:stress:bounds:0}
Before giving the definition of the stress cutoffs, we first note that we can upgrade the $L^1$ and $L^\infty$ bounds for $\psi_{i,q-1} D^K D_{t,q-1}^M \RR_{\ell_q}$ available in \eqref{eq:mollified:stress:bounds} and \eqref{eq:Rq:inductive:uniform:moll}, respectively, to $L^1$ and $L^\infty$ bounds for $\psi_{i,q} D^K D_{t,q}^M \RR_{\ell_q}$. We claim that:
\begin{lemma}[\bf $L^1$ and $L^\infty$ estimates for zeroth order stress]
\label{lem:inductive:rq:dtq}
Let $\RR_{\ell_q}$ be as defined in \eqref{eq:vlq:Rlq:def}. For $q\geq 1$ and $0 \leq i \leq \imax(q)$ we have the estimates
\begin{subequations}
\begin{align}
\bigl\|D^K D_{t,q}^M \mathring R_{\ell_q}\bigr\|_{L^1(\supp \psi_{i,q} )}   & \lesssim \Gamma_q^{\shaq} \delta_{q+1} \MM{K,2\Nindv,\lambda_q\Gamma_q,\Tilde{\lambda}_q}\MM{M, \NindRt, \Gamma_{q+1}^{i-\cstar} \tau_q^{-1}, \Gamma_{q+1}^{-1} \Tilde{\tau}_{q}^{-1} }
\label{eq:Rn:inductive:dtq} \\
\bigl\| D^K D_{t,q}^M \mathring R_{\ell_q}\bigr\|_{L^\infty(\supp \psi_{i,q} )}  & \lesssim \Gamma_q^{\badshaq} \MM{K,2\Nindv,\lambda_q\Gamma_q,\Tilde{\lambda}_q}\MM{M, \NindRt, \Gamma_{q+1}^{i-\cstar} \tau_q^{-1}, \Gamma_{q+1}^{-1} \Tilde{\tau}_{q}^{-1} }
\label{eq:Rn:inductive:dtq:uniform}
\end{align}
\end{subequations}
for all $K+M \leq \sfrac{3\Nfin}{2}$. 
\end{lemma}
\begin{proof}[Proof of Lemma~\ref{lem:inductive:rq:dtq}]
The estimate in \eqref{eq:Rn:inductive:dtq} parallels that of \cite[Lemma~6.28]{BMNV21}; the ingredients in the proof were the $L^1$ bounds for the mollified stress, which are available from \eqref{eq:mollified:stress:bounds} (see also \cite[Lemma~5.1]{BMNV21}), and two lemmas regarding sums and iterates of operators. For the sake of clarity, we thus focus on the proof of \eqref{eq:Rn:inductive:dtq:uniform}, which follows the same strategy as the original proof of \eqref{eq:Rn:inductive:dtq}.  The only change is that we simply substitute the $L^\infty$ bound furnished by \eqref{eq:Rq:inductive:uniform:moll} for each instance of an $L^1$ bound in the proof.

The first step is to apply \cite[Lemma~A.14 and Remark~A.15]{BMNV21} to the functions $v = v_{\ell_{q-1}}$, $f = \RR_{\ell_{q}}$, with $p=\infty$, and on the domain $\Omega = \supp (\psi_{i,q-1})$. The bound \cite[equation (A.50)]{BMNV21} holds in view of the inductive assumption \eqref{eq:nasty:D:vq:old} with $q' = q-1$, for the parameters $\const_v = \Gamma_{q}^{i+1} \delta_{q-1}^{\sfrac 12} $, $\lambda_v = \tilde \lambda_v = \tilde \lambda_{q-1}$, $\mu_v = \Gamma_q^{i-\cstar} \tau_{q-1}^{-1}$, $\tilde \mu_v = \Gamma_q^{-1} \tilde \tau_{q-1}^{-1}$, $N_x = 2\Nindv$, $N_t = \Nindt$, and for $N_* = \sfrac{3\Nfin}{2}$. On the other hand, the assumption \cite[equation (A.51)]{BMNV21} holds due to \eqref{eq:Rq:inductive:uniform:moll}, with the parameters $\const_f = \Gamma_q^{\badshaq}$, $\lambda_f = \lambda_q$, $\tilde \lambda_f = \tilde \lambda_q$, $N_x = 2\Nindv$, $\mu_f = \Gamma_q^{i+3} \tau_{q-1}^{-1}$, $\tilde \mu_f = \tilde \tau_{q-1}^{-1}$, $N_t = \Nindt$, and $N_\circ = 2\Nfin$. We thus conclude from \cite[equation (A.54)]{BMNV21} that
\begin{align*}
&\bigl\|D^{\aaa} D_{t,q-1}^{\bbb} \RR_{\ell_q}\bigr\|_{L^{\infty}(\supp(\psi_{i,q-1}))}   \les \Gamma_q^{\badshaq} \MM{|\aaa|,2\Nindv,\lambda_q,\tilde \lambda_q}   \MM{|\bbb|,\Nindt,\Gamma_q^{i+3} \tau_{q-1}^{-1},\tilde \tau_{q-1}^{-1}}
\end{align*}
whenever $|\aaa|+|\bbb| \leq \sfrac{3\Nfin}{2}$. Here we have used that $\tilde \lambda_{q-1} \leq \lambda_q$ and that $\Gamma_{q}^{i+1} \delta_{q-1}^{\sfrac 12} \tilde \lambda_{q-1} \leq \Gamma_q^{i+3} \tau_{q-1}^{-1} \leq \tilde \tau_{q-1}^{-1}$ (in view of \eqref{eq:Lambda:q:x:1:NEW}, \eqref{eq:Lambda:q:t:1}, and \eqref{eq:imax:old}). In particular, the definitions of $\psi_{i,q}$ in \eqref{eq:psi:i:q:recursive} and of $\psi_{m,i_m,q}$ in \eqref{eq:psi:m:im:q:def} imply that for all $|\aaa|+|\bbb| \leq \sfrac{3\Nfin}{2}$,
\begin{align}
&\bigl\|D^{\aaa} D_{t,q-1}^{\bbb}
 \RR_{\ell_q}\bigr\|_{L^{\infty}(\supp(\psi_{i,q}))} \les \Gamma_q^{\badshaq} \MM{|\aaa|,2\Nindv,\lambda_q,\tilde \lambda_q}   \MM{|\bbb|,\Nindt,\Gamma_{q+1}^{i+3} \tau_{q-1}^{-1},\tilde \tau_{q-1}^{-1}} \, .
\label{eq:cooper:2:f:2:temp:1}
\end{align}

The second step is to apply \cite[Lemma~A.10]{BMNV21} with $B = D_{t,q-1}$, $A = u_q \cdot \nabla$, $v = u_q$,  $f = \RR_{\ell_{q}}$, $p=\infty$, and $\Omega = \supp(\psi_{i,q})$. In this case $D^K (A+B)^M f = D^K D_{t,q}^M \RR_{\ell_q}$, which is exactly the object that we need to estimate in \eqref{eq:Rn:inductive:dtq:uniform}. The assumption \cite[equation (A.40)]{BMNV21} holds due to \eqref{eq:nasty:D:wq:old} at level $q$ (which holds due to Proposition~\ref{prop:no:proofs}) with $\const_v = \Gamma_{q+1}^{i+1} \delta_q^{\sfrac 12}$, $\lambda_v = \Gamma_q \lambda_q$, $\tilde \lambda_v = \tilde \lambda_q$, $N_x = 2\Nindv$, $\mu_v = \Gamma_{q+1}^{i+3} \tau_{q-1}^{-1}$, $\tilde \mu_v = \Gamma_{q+1}^{-1} \tilde \tau_q^{-1}$, $N_t = \Nindt$, and $N_* = \sfrac{3\Nfin}{2} + 1$. The assumption \cite[equation (A.41)]{BMNV21} holds due to \eqref{eq:cooper:2:f:2:temp:1} with the parameters $\const_f = \Gamma_q^{\badshaq}$, $\lambda_f = \lambda_q$, $\tilde \lambda_f = \tilde \lambda_q$, $N_x = 2\Nindv$, $\mu_f = \Gamma_{q+1}^{i+3} \tau_{q-1}^{-1}$, $\tilde \mu_f = \tilde \tau_{q-1}^{-1}$, $N_t = \Nindt$, and $N_* = \sfrac{3\Nfin}{2}$. The bound \cite[equation (A.44)]{BMNV21} and the parameter inequalities $\Gamma_{q+1}^{i+1} \delta_q^{\sfrac 12} \tilde \lambda_q \leq \Gamma_{q+1}^{i-\cstar-2} \tau_q^{-1} \leq \Gamma_{q+1}^{-1} \tilde \tau_q^{-1}$ and $\Gamma_{q+1}^{i+3} \tau_{q-1}^{-1} \leq \Gamma_{q+1}^{i-\cstar} \tau_q^{-1}$ (which hold due to \eqref{eq:Tau:q-1:q},  \eqref{eq:Lambda:q:x:1:NEW}, \eqref{eq:Lambda:q:t:1}, and \eqref{eq:imax:old}) then directly imply \eqref{eq:Rn:inductive:dtq:uniform}, concluding the proof.
\end{proof}

\begin{remark}[\bf $L^1$ and $L^\infty$ estimates for higher order stresses]
In order to verify the inductive assumptions in \eqref{eq:Rq:inductive:assumption} and \eqref{eq:Rq:inductive:uniform} for the new stress $\RR_{q+1}$, it will be necessary to consider a sequence of intermediate objects $\RR_{q,n}$ indexed by $n$ for $1\leq n \leq \nmax$. For notational convenience, when $n=0$, we define $\RR_{q,0}:=\RR_{\ell_q}$, and estimates on $\RR_{q,0}$ are already provided by Lemma~\ref{lem:inductive:rq:dtq}. For $1\leq n \leq\nmax$, the higher order stresses $\RR_{q,n}$ are defined in Section~\ref{ss:stress:definition}, specifically in \eqref{e:rqnp:definition}.  Note that the definition of $\RR_{q,n}$ is given as a finite sum of sub-objects $\HH_{q,n}^{n'}$ for $n'\leq n-1$ and thus requires induction on $n$.  The definition of $\HH_{q,n}^{n'}$ is contained in Section~\ref{ss:stress:error:identification}, specifically in \eqref{eq:Hqnpnn:definition}.  Estimates on $\HH_{q,n}^{n'}$ on the support of $\psi_{i,q}$ are stated in \eqref{e:inductive:n:2:Hstress} and \eqref{e:inductive:n:1:Hstress:unif} and proven in Section~\ref{ss:stress:oscillation:1}.  For the time being, we \emph{assume} that $\RR_{q,n}$ is {\em well-defined} and  satisfies  
\begin{align}
\bigl\| D^k D_{t,q}^m \mathring R_{q,n}\bigr\|_{L^1(\supp \psi_{i,q} )} &\les  \delta_{q+1,n}  \lambda\qn^k \MM{m, \NindRt,\Gamma_{q+1}^{i-\cstarn} \tau_{q}^{-1},\Gamma_{q+1}^{-1} \Tilde{\tau}_{q}^{-1}}
\label{eq:Rn:inductive:assumption} \\
\bigl\|   D^k D_{t,q}^m \mathring R_{q,n}\bigr\|_{L^\infty(\supp \psi_{i,q} )} 
&\les  \Gamma_{q}^{\badshaq} \Gamma_{q+1}^{14\Upsilon(n)} \lambda\qn^k \MM{m, \NindRt,\Gamma_{q+1}^{i-\cstarn} \tau_{q}^{-1},\Gamma_{q+1}^{-1} \Tilde{\tau}_{q}^{-1}}
\label{eq:Rn:inductive:assumption:unif}
\end{align}
for $k+m\leq \Nfn$.

For the purpose of defining the stress cutoff functions, the precise definitions of the $n$-dependent parameters $\delta_{q+1,n}, \lambda\qn$, $\Nfn$, and $\cstarn$ present in \eqref{eq:Rn:inductive:assumption} are not relevant. Note however that the definition for $\lambda\qn$ for $0\leq n \leq \nmax$ is given in \eqref{eq:def:lambda:rq}.  Similarly, for $0\leq n \leq \nmax$,  $\delta_{q+1,n}$ is defined in \eqref{eq:new:delta:def}. Finally, note that there are losses in the sharpness and order of the available derivative estimates in \eqref{eq:Rn:inductive:assumption} and \eqref{eq:Rn:inductive:assumption:unif} relative to \eqref{eq:Rn:inductive:dtq} and \eqref{eq:Rn:inductive:dtq:uniform}.  Specifically, the higher order estimates will only be proven up to $\Nfn$, which is a parameter that is decreasing with respect to $n$ and defined in \eqref{def:Nfn:formula}. For the moment it is only important to note  that $\Nfn \gg 14\Nindv$ for all $0\leq n \leq \nmax$, which is necessary in order to establish \eqref{eq:inductive:assumption:derivative:q} and \eqref{eq:Rq:inductive:assumption} at level $q+1$.  Similarly, there is a loss in the cost of sharp material derivatives in \eqref{eq:Rn:inductive:assumption}, as $\cstarn$ will be a parameter which is decreasing with respect to $n$.  When $n=0$, we set $\cstarn=\cstar$ so that \eqref{eq:Rn:inductive:dtq} is consistent with \eqref{eq:Rn:inductive:assumption}. For $1\leq n \leq\nmax$, $\cstarn$ is defined in \eqref{def:cstarn:formula}.
\end{remark}

For $q\geq 1$, $0 \leq i \leq i_{\rm max}$,  and $0 \leq n \leq n_{\rm max}$, we keep in mind the bound \eqref{eq:Rn:inductive:assumption} and define
\begin{align}
g_{i,q,n}^2(x,t) =  1 + \sum_{k=0}^{\NcutLarge}\sum_{m=0}^{\NcutSmall} 
\delta_{q+1,n}^{-2} (\Gamma_{q+1} \lambda\qn)^{-2k}  (\Gamma_{q+1}^{i-\cstarn+2} \tau_{q}^{-1})^{-2m} 
|D^k D_{t,q}^m \RR_{q,n}(x,t)|^2 \, .
\label{eq:g:i:q:n:def}
\end{align}
With this notation, for $j\geq 1$ the stress cut-off functions are defined by
\begin{align}
\omega_{i,j,q,n}(x,t) = \Psi_{0,q+1} \Big( \Gamma_{q+1}^{-2 j} \, g_{i,q,n}(x,t) \Big)
\,,
\label{eq:omega:cut:def}
\end{align}
while for $j=0$ we let 
\begin{align}
\omega_{i,0,q,n}(x,t) = \tilde\Psi_{0,q+1}\Big(g_{i,q,n}(x,t)\Big)
\,,
\label{eq:omega:cut:def:0}
\end{align}
where $\Psi_{0,q+1}$ and $\tilde \Psi_{0,q+1}$ are as in Lemma~\ref{lem:cutoff:construction:first:statement}.
The cutoff functions $\omega_{i,j,q,n}$ defined above will be shown to obey good estimates on the support of the velocity cutoffs $\psi_{i,q}$. An immediate consequence of \eqref{eq:tilde:partition} with $m=0$ is that for every fixed $i,n$, we have 
\begin{align}
\sum_{j\geq 0} \omega_{i,j,q,n}^2 = 1
\label{eq:omega:cut:partition:unity}
\end{align}
on $\T^3 \times \R$. Thus, $\{\omega_{i,j,q,n}^2 \}_{j\geq 0}$ is a partition of unity.

The following Corollary is quite similar to \cite[Corollary~6.34]{BMNV21}.  In fact the method of proof of that Corollary applies \emph{mutatis mutandis} after replacing each instance of $\RR_{q,n,p}$ and $\lambda\qnp$ with $\RR\qn$ and $\lambda\qn$, and so we omit the proof.
 
\begin{corollary}[\bf $L^\infty$ estimates for the higher order stresses]
\label{cor:D:Dt:Rn:sharp:new}
For $q \geq 0$, $0\leq i\leq \imax$, $0 \leq n \leq \nmax$, and $\aaa,\bbb \in \N_0^k$ we have 
\begin{align}
&\bigl\| D^{\aaa} D_{t,q}^{\bbb}  \RR_{q,n}\bigr\|_{L^\infty(\supp \psi_{i,q}  \omega_{i,j,q,n})} 
 \les \Gamma_{q+1}^{2(j+1)} \delta_{q+1,n}  
(\Gamma_{q+1} \lambda\qn)^{|\aaa|}  
\MM{|\bbb|, \NindRt,\Gamma_{q+1}^{i-\cstarn+2} \tau_{q}^{-1},\Gamma_{q+1}^{-1} \Tilde{\tau}_{q}^{-1}}
\label{eq:D:Dt:Rn:sharp:new}
\end{align}
for all $|\aaa| + |\bbb| \leq \Nfn-4$. 
\end{corollary}

The next Lemma provides an estimate on the maximum value of $j$ for which $\psi_{i,q}\omega_{i,j,q,n}$ may be non-zero.  While the proof is similar in spirit to \cite[Lemma~6.35]{BMNV21}, we include the proof since propagating sharp $L^\infty$ estimates of the stress is one of the crucial new ideas in this paper.  

\begin{lemma}[\bf Maximal $j$ index in the stress cutoffs]
\label{lem:maximal:j}
Fix $q\geq 0$ and $0 \leq n \leq n_{\rm max}$. There exists a $\jmax = \jmax(q,n) \geq 1$,  which is bounded as
\begin{align}
\jmax(q,n) \leq \frac{1}{2}\left( 2 + \frac{\badshaq + 3}{b} +  \frac{ 2\beta b^2}{\eps_\Gamma(b-1)}\right)
\,,
\label{eq:jmax:bound}
\end{align}
such that for any $0 \leq i \leq \imax(q)$, we have 
\begin{align*}
\psi_{i,q} \, \omega_{i,j,q,n} \equiv 0 \qquad \mbox{for all} \qquad j > j_{\rm max}.
\end{align*}
Moreover, assuming that $a=\lambda_0$ is sufficiently large, we have the bound
\begin{align}\label{eq:new:jmax:bound}
\Gamma_{q+1}^{2j_{\rm max}(q,n)} 
\leq 
\Gamma_{q}^{\badshaq}   \delta_{q+1,n}^{-1}  \Gamma_{q+1}^{14\Upsilon(n) +3} \, .
\end{align}
\end{lemma}
 \begin{proof}[Proof of Lemma~\ref{lem:maximal:j}]
We define $j_{\rm max}$ by
\begin{align}
\jmax = \jmax(q,n)= \frac{1}{2} \left\lceil \frac{\log(\Gamma_{q}^{\badshaq}    \delta_{q+1,n}^{-1})\Gamma_{q+1}^{14\Upsilon(n)+2}}{\log(\Gamma_{q+1})} \right\rceil
.
\label{eq:j:max:def}
\end{align}
To see that $j_{\rm max}$ may be bounded independently of $q$ and $n$,  we note that $\delta_{q+1,n}^{-1} \leq \delta_{q+2}^{-1}$, and thus 
\begin{align*}
2j_{\rm max} 
\leq 1 + \frac{\badshaq}{b} +  \frac{ \log(\delta_{q+2}^{-1})}{\log(\Gamma_{q+1})} + 14\Upsilon(n)+2 \to 3 + \frac{\badshaq}{b} + \frac{2\beta b^2}{\eps_\Gamma(b-1)} + 14\Upsilon(n) \quad \mbox{as} \quad q\to \infty \, .
\end{align*}
Thus, assuming that $a = \lambda_0$ is sufficiently large, we obtain that 
\begin{align}
2\jmax(q,n) \leq 4 + \frac{\badshaq}{b} +  \frac{ 2\beta b^2}{\eps_\Gamma(b-1)} + 14\Upsilon(\nmax)
\end{align}
for all $q\geq 0$ and $0 \leq n \leq \nmax$.

To conclude the proof of the Lemma, let $j>j_{\rm max}$, as defined in \eqref{eq:j:max:def}, and assume by contradiction that there exists a point $(x,t) \in \supp (\psi_{i,q} \omega_{i,j,q,n}) \neq \emptyset$. In particular, $j\geq 1$. Then, by \eqref{eq:g:i:q:n:def}--\eqref{eq:omega:cut:def} and the pigeonhole principle, we see that there exists $0 \leq k \leq \NcutLarge$ and $0 \leq m \leq \NcutSmall$ such that 
\begin{align*}
|D^k D_{t,q}^m \RR_{q,n}(x,t)| \geq \frac{\Gamma_{q+1}^{2j}}{\sqrt{8 \NcutLarge \NcutSmall}} \delta_{q+1,n} (\Gamma_{q+1} \lambda\qn)^k  (\Gamma_{q+1}^{i-\cstarn+2} \tau_{q}^{-1})^{m}.
\end{align*}
On the other hand, from \eqref{eq:Rn:inductive:dtq:uniform} and \eqref{eq:Rn:inductive:assumption:unif}, we have that 
\begin{align*}
|D^k D_{t,q}^m \RR_{q,n}(x,t)| 
\leq 
\Gamma_{q}^{\badshaq+1} \Gamma_{q+1}^{14\Upsilon(n)} \Gamma_{q+1} \lambda\qn^k (\Gamma_{q+1}^{i-\cstarn} \tau_{q}^{-1})^m
.
\end{align*}
The above two estimates imply that
\begin{align*}
 \Gamma_{q+1}^{2j} 
\leq  \Gamma_{q}^{\badshaq+1} \Gamma_{q+1}^{14\Upsilon(n)}  \sqrt{8 \NcutLarge \NcutSmall} \delta_{q+1,n}^{-1}
\leq \Gamma_{q}^{\badshaq+2} 
\Gamma_{q+1}^{14\Upsilon(n)}
\delta_{q+1,n}^{-1},
\end{align*}
which contradicts the fact that $j> j_{\rm max}$, as defined in \eqref{eq:j:max:def}.
\end{proof}

The following two lemmas correspond to \cite[Lemmas~6.36 and~6.38]{BMNV21}, respectively. As with Corollary~\ref{cor:D:Dt:Rn:sharp:new}, the method of proof applies \emph{mutatis mutandis} after dropping the unnecessary subscript $p$.  We therefore refer the reader to \cite{BMNV21} for further details.

\begin{lemma}[\bf Derivative bounds for the stress cutoffs]
\label{lem:D:Dt:omega:sharp}
For $q\geq 0$, $0 \leq n \leq \nmax$, $0 \leq i \leq \imax$, and $0 \leq j \leq \jmax$, we have that
\begin{align}
\frac{{\bf 1}_{\supp \psi_{i,q}} |D^N D_{t,q}^M \omega_{i,j,q,n}|}{\omega_{i,j,q,n}^{1-(N+M)/\Nfin}} 
\les (\Gamma_{q+1} \lambda\qn)^N \MM{M, \Nindt,\Gamma_{q+1}^{i-\cstarn+3} \tau_{q}^{-1},\Gamma_{q+1}^{-1} \Tilde{\tau}_{q}^{-1}}
\label{eq:D:Dt:omega:sharp}
\end{align}
for all $N + M \leq \Nfn-\NcutLarge-\NcutSmall-4$.
\end{lemma}

\begin{lemma}[\bf $L^r$ norm of the stress cutoffs]
\label{lem:omega:support}
Let $q \geq 0$ and define $\psi_{i\pm,q}=\left( \psi_{i-1,q}^2 + \psi_{i,q}^2 + \psi_{i+1,q}^2 \right)^{\sfrac 12}$.  Then for $r\geq 1$ we have that
\begin{align}
\norm{\omega_{i,j,q,n}}_{L^r\left(\supp \psi_{i\pm,q} \right)}  \lesssim  \Gamma_{q+1}^{-\sfrac{2j}{r}} 
\label{eq:omega:support}
\end{align}
holds for all $0\leq i \leq \imax$, $0 \leq j \leq\jmax$, and $0\leq n \leq \nmax$. The implicit constant is independent of $i,j,q,n$.
\end{lemma}

\subsection{Anisotropic checkerboard cutoff functions}
\label{sec:cutoff:checkerboard:definitions}

We construct anisotropic checkerboard cutoff functions which are well-suited for intermittent pipe flows with axes parallel to $e_3$.  The construction for general $\xi\in\Xi$ follows by rotation.  Consider a partition of $\T^3$ into the rectangular prisms defined using
\begin{equation}\label{eq:prism:zero}
    \left\{ (x_1,x_2,x_3)\in\T^3 \, : \, 0 \leq x_1,x_2 \leq \const_\Gamma \Gamma_{q+1}\left(\lambda_{q+1}r_{q+1,n} \right)^{-1}, \, 0 \leq x_3 \leq 2\pi \lambda\qn^{-1} \right\}
\end{equation}
and its translations by 
$$\bigl(l_1 \const_\Gamma \Gamma_{q+1}(\lambda_{q+1}r_{q+1,n})^{-1},l_2 \const_\Gamma \Gamma_{q+1}(\lambda_{q+1}r_{q+1,n})^{-1},l_3 2\pi \lambda\qn^{-1}\bigr) \, $$
for
$$ l_1,l_2\in\{0,\dots, \sfrac{\const_\Gamma}{2\pi}\Gamma_{q+1}^{-1}\lambda_{q+1}r_{q+1,n} -1\} \, ,  \qquad l_3 \in \{0,\dots,\lambda\qn-1\} \, ,
$$
where $\const_\Gamma \geq 1$ ensures that the prisms evenly partition $[-\pi,\pi]^3$ and is bounded above independently of $q$. Index these prisms by integer triples $\vec{l}=(l_1,l_2,l    _3)$. Let $\mathcal{X}_{q,n,e_3,\vec{l}}$ be a $C^\infty$ partition of unity adapted to this checkerboard of anisotropic rectangular prisms which satisfies
\begin{equation}
\label{eq:checkerboard:useful:1}
\sum_{\vec{l}} \bigl(\mathcal{X}_{q,n,e_3,\vec{l}}\bigr)^2 = 1
\end{equation}
for any $q$ and $n$. Specifically, we impose that spatial derivatives applied to cutoffs belonging to this partition of unity cost $\approx\Gamma_{q+1}(\lambda_{q+1}r_{q+1,n})^{-1}$ in the $x_1$ and $x_2$ directions, and $\approx\lambda\qn^{-1}$ in the $x_3$ direction, so that
$$ \bigl\| \partial_1^{M_1} \partial_2^{M_2} \partial_3^M \mathcal{X}_{q,n,e_3,\vec{l}} \bigr\|_{L^\infty} \lesssim \left(\lambda_{q+1}r_{q+1,n}\Gamma_{q+1}^{-1}\right)^{M_1+M_2} \lambda\qnn^M \,  $$
for $M_1,M_2,M\leq 3\Nfin$. Furthermore, for $\vec{l},\vec{l}^*$ such that 
$$ |l_1-l_1^*| \geq 2, \qquad |l_2-l_2^*| \geq 2, \qquad |l_3-l_3^*| \geq 2,$$
we impose that
\begin{equation}\notag
\mathcal{X}_{q,n,e_3,\vec{l}} \; \mathcal{X}_{q,n,e_3,\vec{l}^*} = 0 \, .
\end{equation}
Incorporating rotations into the above construction, we may similarly produce cutoff functions $\mathcal{X}_{q,n,\xi,\vec{l}}$ satisfying analogous properties for $\xi\in\Xi$. Note that if $\{\xi,\xi',\xi''\}$ forms an orthonormal basis for $\R^3$, then
\begin{equation}\label{eq:nice:derivative:estimate}
\bigl\| \left(\xi'\cdot\nabla\right)^{M_1} \left(\xi''\cdot\nabla\right)^{M_2} (\xi\cdot\nabla)^M \mathcal{X}_{q,n,\xi,\vec{l}} \bigr\|_{L^\infty} \lesssim \left(\lambda_{q+1}r_{q+1,n}\Gamma_{q+1}^{-1}\right)^{M_1+M_2} \lambda\qnn^M \, . 
\end{equation}

\begin{definition}[\bf Anisotropic checkerboard cutoff function]\label{def:checkerboard}
Given $q$, {$\xi\in\Xi$}, $0\leq n\leq \nmax$, $i\leq \imax$, and $k\in\mathbb{Z}$, we define
\begin{equation}\label{eq:checkerboard:definition}
    \zeta_{q,i,k,n,{\xi},\vec{l}}\,(x,t) = \mathcal{X}_{q,n,{\xi},\vec{l}}\left(\Phi_{i,k,q}(x,t)\right).
\end{equation}
\end{definition}
These cutoff functions satisfy properties which we enumerate in the following lemma.

\begin{lemma}\label{lem:checkerboard:estimates}
The cutoff functions $\{\zeta_{q,i,k,n,{\xi},\vec{l}}\}_{\vec{l}}$ satisfy the following properties:
\begin{enumerate}[(1)]
    \item\label{item:check:1} The material derivative $\Dtq (\zeta_{q,i,k,n,{\xi},\vec{l}})$ vanishes.  
    \item\label{item:check:2} For each $t\in\mathbb{R}$ and all $x=(x_1,x_2,x_3)\in\mathbb{T}^3$, 
    \begin{equation}\label{eq:checkerboard:partition}
    \sum_{\vec{l}} \bigl(\zeta_{q,i,k,n,{\xi},\vec{l}}\,(x,t)\bigr)^2 = 1 \, .
    \end{equation}
    \item\label{item:check:3} 
    Let $A=(\nabla\Phi_{i,k,q})^{-1}$.  Then we have the spatial derivative estimate
    \begin{align}\label{eq:checkerboard:derivatives}
        \bigl\| D^{N_1} \Dtq^M  ({\xi^\ell A_\ell^j \partial_j} )^{N_2} \zeta_{q,i,k,n,\xi,\vec{l}} \bigr\|_{L^\infty\left(\supp \psi_{i,q}\tilde\chi_{i,k,q} \right)} &\lesssim \left(\Gamma_{q+1}^{-1} \lambda_{q+1} r_{q+1,n} \right)^{N_1} \lambda\qn^{N_2} \notag\\
        &\qquad \qquad \times \MM{M,\Nindt,\Gamma_{q+1}^{i-\cstarzero}\tau_q^{-1},\tilde\tau_q^{-1}\Gamma_{q+1}^{-1}} \, .
    \end{align}
    for all $N_1+N_2+M\leq\sfrac{3\Nfin}{2}+1$.
    \item\label{item:check:4} There exists an implicit dimensional constant $\const_\chi$ independent of $q$, $n$, $k$, $i$, and $\vec{l}$ such that for all $(x,t)\in\supp\psi_{i,q}\tilde\chi_{i,k,q}$, the support of $\zeta_{q,i,k,n,\xi,\vec{l}}\, (\cdot,t)$ satisfies
    \begin{equation}\label{eq:checkerboard:support}
        \textnormal{diam} ( \supp ( \zeta_{q,i,k,n,\xi,\vec{l}}\,(\cdot,t) ) ) \lesssim \lambda\qn^{-1} \, . 
    \end{equation}
\end{enumerate}
\end{lemma}
\begin{proof}[Proof of Lemma~\ref{lem:checkerboard:estimates}]
The proof of \eqref{item:check:1} is immediate from \eqref{eq:checkerboard:definition}.  \eqref{eq:checkerboard:partition} follows from \eqref{item:check:1} and \eqref{eq:checkerboard:useful:1}. To verify \eqref{item:check:3}, the only nontrivial calculations are those including the differential operator $(\xi^\ell A_{\ell}^j\partial_j)$. Using the Leibniz rule, the contraction
$$ \xi^\ell A_\ell^j \partial_j \zeta_{q,i,k,n,\xi,\vec{l}} = \xi^\ell A_\ell^j (\partial_m \mathcal{X}_{q,n,\xi,\vec{l}})(\Phi_{i,k,q}) \partial_j \Phi_{i,k,q}^m= \xi^m (\partial_m \mathcal{X}_{q,n,\xi,\vec{l}})(\Phi_{i,k,q}) \, , $$
\eqref{eq:nice:derivative:estimate}, and \eqref{eq:Lagrangian:Jacobian:6} gives the desired estimate. The proof of \eqref{eq:checkerboard:support} follows from the construction of $\mathcal{X}_{q,n,\xi,\vec{l}}$ and the Lipschitz bound obeyed by $\nabla \vlq$ on the support of $\psi_{i,q}$; see for example \eqref{eq:diameter:inequality}.
\end{proof}

\subsection{Definition of the cumulative cutoff function}
\label{sec:cutoff:total:definitions}
Finally, combining the  cutoff functions defined in Definition~\ref{def:psi:i:q:def}, \eqref{eq:omega:cut:def}--\eqref{eq:omega:cut:def:0}, and \eqref{eq:chi:cut:def}, we define the cumulative cutoff function by
\begin{align*}
 {\eta_{i,j,k,q,n,\xi,\vec{l}}\,(x,t)=\psi_{i,q}(x,t) \omega_{i,j,q,n}(x,t) \chi_{i,k,q}(t)\zeta_{q,i,k,n,\xi,\vec{l}}\,(x,t) \, .}
\end{align*}
Since the values of $q$ and $n$ are clear from the context and the values of $\xi$ and $\vec{l}$ are irrelevant in many arguments, we may abbreviate the above using any of
\begin{align*}
\eta_{i,j,k,q,n,\xi,\vec{l}}\,(x,t)=\eta_{i,j,k,q,n,\xi}(x,t)=\eta_{(i,j,k)}(x,t)  = \psi_{(i)}(x,t) \omega_{(i,j)}(x,t) \chi_{(i,k)}(t) \zeta_{(i,k)}(x,t) \, .
\end{align*}
It follows from \eqref{eq:inductive:partition} at level $q$, \eqref{eq:omega:cut:partition:unity}, \eqref{eq:chi:cut:partition:unity}, and \eqref{eq:checkerboard:partition} that for every ${(q,n,\xi)}$ fixed, we have
\begin{align}
 \sum_{i,j \geq 0} \sum_{k\in \Z} \sum_{\vec{l}} \eta_{i,j,k,q,n,\xi,\vec{l}}\, ^2(x,t) = 1 \, .
 \label{eq:eta:cut:partition:unity}
\end{align}
The sum in $i$ goes up to $i_{\rm max}$ (defined in \eqref{eq:imax:def}), while the sum in $j$ goes up to $j_{\rm max}$ (defined in \eqref{eq:j:max:def}).

We conclude this section with support estimates on the cumulative cutoff functions $\eta_{i,j,k,q,n,\xi,\vec{l}}$.

\begin{lemma}\label{lemma:cumulative:cutoff:Lp}
For $r_1, r_2 \in [1,\infty]$ with $\frac{1}{r_1}+\frac{1}{r_2}=1$ and any $0\leq i \leq \imax$, $0\leq j \leq \jmax$, and $\xi \in \Xi$, we have that
\begin{equation}
\label{item:lebesgue:1}
\sum_{\vec{l}} \left| \supp ( \eta_{i,j,k,q,n,\xi,\vec{l}} ) \right| \lessg \Gamma_{q+1}^{\frac{-2i+\CLebesgue}{r_1} + \frac{-2j}{r_2} + 2 }
\, .
\end{equation}
\end{lemma}
\begin{proof}[Proof of Lemma~\ref{lemma:cumulative:cutoff:Lp}]
From \eqref{eq:psi:i:q:support:old} at level $q$ and \eqref{eq:omega:support}, we have that for each fixed time $t$,
\begin{align*}
    \left|\supp (\psi_{i,q}) \cap \supp (\omega_{i,j,q,n}) \right| &\leq \norm{\left(\psi_{i-1,q}^2 + \psi_{i,q}^2 + \psi_{i+1,q}^2\right)^{\sfrac 12} \left(\omega_{i,j-1,q,n}^2 + \omega_{i,j,q,n}^2 + \omega_{i,j+1,q,n}^2 \right)^{\sfrac 12}}_{L^1} \notag\\
    &\lesssim \Gamma_{q+1}^{\frac{-2(i-1)+\CLebesgue}{r_1}} \Gamma_{q+1}^{\frac{-2(j-1)}{r_2}} \, .
\end{align*}
Using the fact that $\{\eta_{q,i,k,n,\xi,\vec{l}}\}_{\vec{l}}$ forms a partition of unity from \eqref{eq:checkerboard:partition} and $\frac{1}{r_1}+\frac{1}{r_2}=1$ gives the desired estimate.
\end{proof}

\section{Inductive propositions}
\label{sec:statements}

\subsection{Induction on \texorpdfstring{$q$}{q}} 
The main claim of this section is an induction on $q$. Notice that the estimates in this proposition match the inductive assumptions \eqref{eq:inductive:assumption:wq:all} and \eqref{eq:Rq:inductive:assumption:all} at level $q+1$.

\begin{proposition}[\textbf{Inductive Step on $q$}]\label{p:main:inductive:q}
Given the velocity field $\vlq$ which solves the Euler-Reynolds system with stress $\RR_{\ell_q} + \RR_{q}^{\textnormal{comm}}$, where $\vlq$, $\RR_{\ell_q}$, and $\RR_{q}^{\textnormal{comm}}$ satisfy  the conclusions of Lemma~\ref{lem:mollifying:ER} in addition to \eqref{eq:inductive:assumption:derivative}--\eqref{eq:nasty:Dt:wq:WEAK:old}, there exist $v_{q+1}=\vlq+w_{q+1}$ and $\RR_{q+1}$ which satisfy the following:
\begin{enumerate}[(1)]
\item $v_{q+1}$ solves the Euler-Reynolds system with stress $\RR_{q+1}$.
\item For all $k,m \leq 7\Nindv$, we have
\begin{subequations}
\begin{align}
     \left\| \psi_{i,q} D^k \Dtq^m w_{q+1} \right\|_{L^2} 
 &\leq \Gamma_{q+1}^{-1}\delta_{q+1}^{\sfrac{1}{2}}\lambda_{q+1}^k \MM{m,\Nindt, \Gamma_{q+1}^{i-1} \tau_q^{-1}, \Gamma_{q+1}^{-1} \tilde{\tau}_q^{-1}}
 \label{e:main:inductive:q:velocity} \\
 \norm{D^k D_{t,q}^m w_{q+1}}_{L^\infty(\supp \psi_{i,q})} 
&\leq  \Gamma_{q+1}^{\badshaq - 1} \Theta_{q+1}^{\sfrac 12}  \lambda_{q+1}^k 
\MM{m,\Nindvt, \Gamma_{q+1}^{i}\tau_q^{-1}, \Gamma_{q+1}^{-1} \tilde{\tau}_{q}^{-1}} \,.
\label{e:main:inductive:q:velocity:unif} 
\end{align}
\end{subequations}

\item For all $k,m \leq 3\Nindv$, we have
\begin{subequations}
\begin{align}
\bigl\| \psi_{i,q} D^k \Dtq^m \RR_{q+1} \bigr\|_{L^1} 
&\leq \shaqqplusone  \delta_{q+2} \lambda_{q+1}^k \MM{m,\Nindt,\Gamma_{q+1}^{i+1} \tau_q^{-1},\Gamma_{q+1}^{-1}\tilde\tau_q^{-1}}
\label{e:main:inductive:q:stress}
\\
\bigl\| D^k D_{t,q}^m \mathring{R}_{q+1}\bigr\|_{L^\infty(\supp \psi_{i,q})}
&\leq \Gamma_{q+1}^{\badshaq}  \lambda_{q+1}^k
\MM{m,\Nindvt,  \Gamma_{q+1}^{i+2}\tau_q^{-1}, \Gamma_{q+1}^{-1} \tilde{\tau}_{q}^{-1}}  \, .
\label{e:main:inductive:q:stress:unif}
\end{align}
\end{subequations}
\end{enumerate}
\end{proposition}

\subsection{Notations}\label{ss:notations:prep}
The proof of Proposition~\ref{p:main:inductive:q} will be achieved through an induction with respect to $\tilde{n}$, where $0\leq \nn \leq \nmax$ corresponds to the addition of the perturbation $\displaystyle w_{q+1,\nn}$. We shall employ the notation:
\begin{enumerate}[(1)]
    \item \label{item:notation:1} $\nn$ - An integer taking values $0\leq \nn \leq \nmax$ over which induction is performed, indexing the component $w_{q+1,\nn}$ of the velocity increment $w_{q+1}$.
    We emphasize that the use of $\nn$ at various points in statements and estimates means that we are \emph{currently} working on the inductive step at level $\nn$.
    \item \label{item:notation:2} $n$ - An integer taking values $1\leq n \leq \nmax$ which correspond to the higher order stresses $\RR_{q,n}$.  Occasionally, we shall use the notation $\RR_{q,0}=\RR_{\ell_q}$ to streamline an argument.  We emphasize that $n$ will be used at various points in statements and estimates to reference \textit{higher order} objects in addition to those at level $\nn$, and so will satisfy the inequality $\nn\leq n$.  
    \item \label{item:notation:3} ${\HH}_{q,n}^{n'}$ - The component of $\RR_{q,n}$ originating from an error term produced by the addition of $w_{q+1,n'}$.  The parameter $n'$ will always be a \emph{subsidiary} parameter used to reference objects created at or \emph{below} the level $\nn$ that we are currently working on, and so will satisfy $n'\leq \nn$.
    \item \label{item:notation:4} $\LPqn$ - We use the spatial Littlewood-Paley projectors $\LPqn$ defined by
    \begin{align}
    \LPqn=\begin{cases}
    \mathbb{P}_{\left[\lambda_q^{\sfrac 12}\lambda_{q+1}^{\sfrac
    12}\Gamma_{q+1},\lambda_{q,1}\right)} & \mbox{if }n=1 \, , \\
    \mathbb{P}_{\left[\lambda_{q,n-1},\lambda_{q,n}\right)} &\mbox{if } 2\leq n \leq \nmax \, ,  \\
    \mathbb{P}_{\geq\lambda_{q,\nmax}}& \mbox{if } n=\nmax+1 \, ,
    \label{def:LPqnp}
    \end{cases}
    \end{align}
where $\Proj_{[\lambda_1,\lambda_2)}$ is defined in Remark~\ref{sec:mollifiers:Fourier}  as $\Proj_{\geq \lambda_1} \Proj_{< \lambda_2}$. 
Errors which include the frequency projector $\mathbb{P}_{[q,\nmax+1]}$ will be small enough to be absorbed into $\RR_{q+1}$. We note that if $0\leq \nn \leq \nmax$, then from \eqref{eq:def:lambda:r:q+1:n}, any $\frac{\mathbb{T}^3}{\lambda_{q+1}r_{q+1,\nn}}$-periodic function satisfies
    \begin{align}
    f 
    = \dashint_{\mathbb{T}^3} f + \mathbb{P}_{\geq \lambda_{q+1}r_{q+1,\nn}} f  
    = \dashint_{\mathbb{T}^3} f + \sum_{n=\nn+1}^{\nmax+1} \LPqn f \, .  \label{e:Pqnp:identity}
    \end{align}
    \item \label{eq:notation:new} In order to later deduce a useful refinement of \eqref{e:Pqnp:identity}, we set 
    \begin{align}
        r(\nn) = \begin{cases}
     0 & \mbox{if } \nn=0 \, , \\
    \frac{\nmax+\nn}{2} &\mbox{if } 1\leq \nn \leq \nmax-1 \, .   \label{eq:rofn}
    \end{cases}
    \end{align}
    \item In order to keep track of small losses related to the process of building a stress $\RR\qnn$, corrector $w_{q+1,\nn}$, and new stresses $\RR_{q,n}$ for $n>\nn$, we define
    \begin{align}
        \Upsilon(n) = \begin{cases}
     0 & \mbox{if } n=0 \, , \\
    1 &\mbox{if } 1\leq n \leq \frac{\nmax}{2} \\
    k &\mbox{if } \frac{2^{k-1}-1}{2^{k-1}} \nmax < n \leq \frac{2^k-1}{2^k} \nmax \\
    2 + \lceil \log_2 (\nmax) \rceil  & \mbox{if } n=\nmax
    \, .   \label{eq:upsawhat}
    \end{cases}
    \end{align}
    $\Upsilon(\nn)$ gives an upper bound on the number of steps in the induction on $\nn$ it takes to produce the \emph{entire} error term $\RR\qnn$. A consequence of \eqref{eq:rofn} and \eqref{eq:upsawhat} is that  \begin{align}\label{eq:upsa:ineq}
       n>r(\nn) \qquad \implies \qquad  \Upsilon(n) \geq \Upsilon(\nn) + 1 \, .
    \end{align}
    To prove this, first consider the case $n=\nmax$. Then for all $0\leq \nn \leq \nmax-1$, we have that $r(\nn)<\nmax$, and so \eqref{eq:upsa:ineq} should hold for all $\nn < \nmax$.  Since $\nn<\nmax$, there exists a minimum value of $k$, say $k_\nn$, such that $\nn\leq \nmax- \frac{\nmax}{2^{k_\nn}}$, which implies that $\Upsilon(\nn) \leq k_\nn$.  For $k=\lceil\log_2(\nmax)\rceil+2$, however, we have that $\nmax-\frac{\nmax}{2^{k-1}}\geq\nmax - \frac{1}{2}$, and so it must be the case that $k_\nn \leq \lceil\log_2(\nmax)\rceil+1$, which proves \eqref{eq:upsa:ineq} in the case $n=\nmax$, and shows that
    \begin{equation}\label{eq:upsa:bound}
        \Upsilon(n) \leq 2 + \lceil\log_2(\nmax)\rceil \qquad \forall n \leq \nmax \, .
    \end{equation}
    To prove \eqref{eq:upsa:ineq} in the remaining cases, note that if $\nn=0$, then $n > r(0) \implies n\geq 1$ and so \eqref{eq:upsa:ineq} holds. If $\nn=1$, then $n > \frac{\nmax+1}{2}$, and again \eqref{eq:upsa:ineq} holds. Finally, if $2\leq \nn \leq \nmax-1$ and $\Upsilon(\nn)=k$, then 
    \begin{align*}
        n &> \frac{\nmax+\nn}{2} > \frac{\nmax+\frac{2^{k-1}-1}{2^{k-1}}{\nmax}}{2} = \frac{2^k-1}{2^k} \nmax \qquad \implies \qquad \Upsilon(n) \geq k+1 \, .
    \end{align*}
    \item \label{item:notation:5} $\RR_{q+1}^\nn$ - For any $0\leq \nn \leq\nmax-1$, this is any stress term which satisfies the estimates required of $\RR_{q+1}$ and which has already been estimated at the $\nn^{th}$ stage of the induction; that is, error terms arising from the addition of $w_{q+1,n'}$ for $n'\leq \nn$.  We \emph{exclude} $\RR_{q}^{\textnormal{comm}}$ from $\RR_{q+1}^\nn$, only absorbing it at the very end when we define $\RR_{q+1}$. Thus
    \begin{equation}\label{RR:q+1:n-1:to:n}
    \RR_{q+1}^{\nn+1} = \RR_{q+1}^{\nn} + \left(\textnormal{errors coming from }w_{q+1,\nn+1}\textnormal{ that also go into }\RR_{q+1}\right) \, .
    \end{equation}
   We adopt the convention that $\RR_{q+1}^{-1}=0$.
    \item We adopt the convention that $\sum_{n=0}^{-1} f(n) \equiv \sum_{n=\nmax+1}^{\nmax} f(n)  \equiv 0$ denotes an \emph{empty} summation.
\end{enumerate}

\subsection{Induction on \texorpdfstring{$\tilde{n}$}{tilden}}\label{ss:induction:nn}

We split the verification of Proposition~\ref{p:main:inductive:q} using a sub-inductive procedure on the parameter $\nn$. Note that summing \eqref{e:inductive:n:2:velocity}-\eqref{e:inductive:n:1:Rstress:unif} over $0\leq \nn \leq \nmax$, appealing to \eqref{eq:delta:q:nn:pp:ineq} and \eqref{eq:ineq:badshaq}, and using the extra factor of $\Gamma_{q+1}^{-1}$ to kill implicit constants, we have matched the desired bounds in \eqref{e:main:inductive:q:velocity}-\eqref{e:main:inductive:q:stress:unif}.

\begin{proposition}[\textbf{Induction on $\nn$: From $\nn-1$ to $\nn$ for $0\leq\nn\leq\nmax$}]\label{p:inductive:n:2}
Under the assumptions of Proposition~\ref{p:main:inductive:q} and Lemma~\ref{lem:mollifying:ER}, we let $0\leq \nn \leq\nmax$ be given, and let 
$v_{q,\nn-1} = \vlq+\sum\limits_{n'=0}^{\nn-1}w_{q+1,n'}$,
$\RR_{q+1}^{\nn-1}$, and $\HH\qn^{n'}$ be given for $0\leq n'\leq \nn-1$ and $\nn\leq n\leq\nmax$, such that the following are satisfied:
\begin{enumerate}[(1)]
\item ${v_{q,\nn-1}}$ solves the Euler-Reynolds system with stress
\begin{align}\label{e:inductive:n:2:eulerreynolds}
\mathbf{1}_{\{\nn=0\}} \RR_{\ell_q} + \RR_{q+1}^{\nn-1} +  \sum\limits_{n'=0}^{\nn-1}\sum\limits_{n>r(n')}^{\nmax} \HH\qn^{n'}   +  \RR_{q}^{\textnormal{comm}}  \,.
\end{align}
\item For all $k+m \leq \NN{\textnormal{fin},\textnormal{n}'}-\NcutSmall-\NcutLarge-2\Ndec-9$ and $0\leq n'\leq\nn-1$, 
\begin{subequations}
\begin{align}
\label{e:inductive:n:2:velocity}
 \left\| D^k \Dtq^m w_{q+1,n'} \right\|_{L^2\left(\supp\psi_{i,q}\right)} &\lesssim \delta_{q+1,n'}^{\sfrac{1}{2}} \Gamma_{q+1}^{3}  \lambda_{q+1}^k \MM{m,\Nindt, \tau_q^{-1}\Gamma_{q+1}^{i-\cstarnprime+4}, \tilde{\tau}_q^{-1}\Gamma_{q+1}^{-1}} \\
\label{e:inductive:n:1:velocity:unif}
\norm{D^k D_{t,q}^m w_{q+1,n'}}_{L^\infty(\supp \psi_{i,q})}
&\lesssim \Gamma_q^{\frac{\badshaq}{2}}\Gamma_{q+1}^{7\Upsilon(n')+\frac{7}{2}} r_{q+1,n'}^{-1} \lambda_{q+1}^k 
\MM{m,\Nindvt, \tau_q^{-1}\Gamma_{q+1}^{i-\cstarnprime+4}, \Gamma_{q+1}^{-1} \tilde{\tau}_{q}^{-1}}  \, .
\end{align}
\end{subequations}

\item For all $k,m \leq 3\Nindv$ and $1\leq \nn \leq \nmax$,
\begin{subequations}
\begin{align}
\label{e:inductive:n:2:Rstress}
\bigl\| \psi_{i,q} D^k \Dtq^m \RR_{q+1}^{\nn-1} \bigr\|_{L^1} &\lesssim \Gamma_{q+1}^{\shaq-1} \delta_{q+2} \lambda_{q+1}^k \MM{m,\Nindt,\Gamma_{q+1}^{i+1} \tau_q^{-1},\Gamma_{q+1}^{-1}\tilde\tau_q^{-1}} \\
\label{e:inductive:n:1:Rstress:unif}
\bigl\| D^k D_{t,q}^m \mathring{R}_{q+1}^{\nn-1}\bigr\|_{L^\infty(\supp \psi_{i,q})}
&\lesssim \Gamma_{q+1}^{\badshaq-1}  \lambda_{q+1}^k
\MM{m,\Nindvt,  \Gamma_{q+1}^{i+1}\tau_q^{-1}, \Gamma_{q+1}^{-1} \tilde{\tau}_{q}^{-1}}  \, .
\end{align}
\end{subequations}
 
\item For $0\leq n'\leq\nn-1$, $r(n') < n \leq \nmax$, and all $k+m\leq\Nfn$,
\begin{subequations}
\begin{align}
\label{e:inductive:n:2:Hstress}
\bigl\| D^k \Dtq^m \HH\qn^{n'} \bigr\|_{L^1\left(\supp\psi_{i,q}\right)} &\lesssim \delta_{q+1,n} \lambda\qn^k \MM{m,\Nindt, \tau_q^{-1}\Gamma_{q+1}^{i-\cstarn}, \tilde{\tau}_q^{-1}\Gamma_{q+1}^{-1}} \, , \\
\label{e:inductive:n:1:Hstress:unif}
\bigl\| D^k D_{t,q}^m \HH\qn^{n'}\bigr\|_{L^\infty(\supp \psi_{i,q})}
&\les 
\Gamma_{q}^{\badshaq}\Gamma_{q+1}^{14\Upsilon(n)}  \lambda\qn^k
\MM{m,\Nindvt, \tau_q^{-1}\Gamma_{q+1}^{i-\cstarn}, \tilde{\tau}_{q}^{-1}\Gamma_{q+1}^{-1}} \, .
\end{align}
\end{subequations}
\end{enumerate}
Then if $0 \leq \nn \leq \nmax-1$, there exists $w_{q+1,\nn}$, $\RR_{q+1}^{\nn}$, and $\HH\qn^{n'}$ for $0\leq n' \leq \nn$, such that \eqref{e:inductive:n:2:eulerreynolds}--\eqref{e:inductive:n:1:Hstress:unif} are satisfied with $\nn-1$ replaced with $\nn$.  If $\nn = \nmax$, then there exists $w_{q+1,\nmax}$ and $\RR_{q+1}$ such that $v_{q+1}:=v_{q,\nmax-1}+w_{q+1,\nmax}$ solves the Euler-Reynolds system with stress $\RR_{q+1}$, and $v_{q+1}$, $w_{q+1}$, and $\RR_{q+1}$ satisfy conclusions \eqref{e:main:inductive:q:velocity}--\eqref{e:main:inductive:q:stress:unif} from Proposition~\ref{p:main:inductive:q}.
\end{proposition}

\newcommand{\ijklwhstarzero}{{(\istar,\jstar,\kstar,0,\lstar,\wstar,\hstar)}}
\newcommand{\ijknlwhstarzero}{{(\istar,\jstar,\kstar,\nstar,\lstar,\wstar,\hstar)}}
\newcommand{\ijklwhzero}{{(i,j,k,0,l,w,h)}}
\newcommand{\Phiikstar}{\Phi_{(\istar,\kstar)}}
\newcommand{\vecl}{\vec{l}}
\newcommand{\veclstar}{\vecl^*}

\section{Proving the main inductive estimates}
\label{s:stress:estimates}

\subsection{Definition of \texorpdfstring{$\RR\qnn$}{rqnp} and \texorpdfstring{$w\qplusnn$}{wqnp}}\label{ss:stress:definition}

In this section we define the stresses $\RR\qnn$ and the perturbations $w\qplusnn$ used to correct them.  For $0\leq \nn \leq \nmax$, we define
\begin{equation}\label{e:rqnp:definition}
    \RR\qnn = \mathbf{1}_{\{\nn=0\}} \RR_{\ell_q} + \sum_{0\leq n'\leq \nn-1} \HH\qnn^{n'} \, .
\end{equation}
In Subsection~\ref{ss:stress:error:identification}, we will show that $\HH^{n'}_{q,\nn}$ is zero in certain parameter regimes, although for the moment this is irrelevant. 
Now for any fixed values of $\nn$, $i$, $j$, and $k$, we may define
\begin{equation}\label{eq:rqnpj}
R_{q,\nn,j,i,k}=\nabla\Phi_{(i,k)}\left(\delta_{q+1,\nn}\Gamma^{2j+4}_{q+1}\Id - \mathring{R}\qnn\right)\nabla\Phi_{(i,k)}^T \, .
\end{equation}
Let $\xi\in\Xi$ be a vector from Proposition~\ref{p:split}. For all $\xi\in\Xi$, we define the coefficient function $a_{\xi,i,j,k,q,\nn,\vecl}$ by
\begin{equation}
a_{\xi,i,j,k,q,\nn,\vecl}:=a_{\xi,i,j,k,q,\nn}:=a_{(\xi)}=\delta_{q+1,\nn}^{\sfrac 12}\Gamma^{j+2}_{q+1}\eta_{i,j,k,q,\nn,\xi,\vecl}\, \gamma_{\xi}\left(\frac{R_{q,\nn,j,i,k}}{\delta_{q+1,\nn}\Gamma^{2j+4}_{q+1}}\right) \, .
\label{eq:a:xi:def}
\end{equation}
From Corollary~\ref{cor:D:Dt:Rn:sharp:new}, we see that on the support of $\eta_{(i,j,k)}$ we have $|\RR_{q,\nn}| \lesssim \Gamma_{q+1}^{2j+2} \delta_{q+1,\nn}$, and thus by estimate \eqref{eq:Lagrangian:Jacobian:1} from Corollary~\ref{cor:deformation}, we have that
$$  \left| \frac{R_{q,\nn,j,i,k}}{\delta_{q+1,\nn}\Gamma^{2j+4}_{q+1}} - \Id \right| \leq \Gamma_{q+1}^{-1} < \frac 12  $$
once  $\lambda_0$ is sufficiently large. Thus we may apply Proposition~\ref{p:split}.

The coefficient function $a_{(\xi)}$ is then multiplied by an intermittent pipe flow defined in Proposition~\ref{prop:pipeconstruction} (with $\lambda=\lambda_{q+1}$ and $r=r_{q+1,\nn}$)
$$ \nabla \Phi_{(i,k)}^{-1}  \WW^s_{\xi,\lambda_{q+1},r_{q+1,\nn}} \circ \Phi_{(i,k)},  $$
where the superscript $s=s(i,j,k,\nn,\vecl)$ indicates the placement of the intermittent pipe flow $\WW^{s}_{\xi,\lambda_{q+1},r_{q+1,\nn}}$ (cf. \eqref{item:pipe:2} from Proposition~\ref{prop:pipeconstruction}), which depends on $i$, $j$, $k$, $\nn$, and $\vecl$ and is only relevant in Section~\ref{ss:stress:oscillation:2}.   To ease notation, we will suppress the superscript $s$ (except in Section~\ref{ss:stress:oscillation:2}), and use the shorthand notation
\begin{align} 
\WW_{\xi,q+1,\nn} := \WW_{\xi,\lambda_{q+1},r_{q+1,\nn}}^s \, . 
\label{eq:W:xi:q+1:nn:def}
\end{align}
We will also adopt the same notational conventions for the potentials $\UU_{\xi,q+1,\nn}$. Furthermore, \eqref{eq:pipes:flowed:1} from Proposition~\ref{prop:pipeconstruction} gives that we can now write the principal part of the first term of the perturbation as
\begin{equation}\label{wqplusoneonep}
    w_{q+1,\nn}^{(p)} = \sum_{i,j,k}\sum_{\vecl}\sum_{\xi} a_{(\xi)} \curl \left( \nabla\Phi_{(i,k)}^T \mathbb{U}_{\xi,q+1,\nn} \circ \Phi_{(i,k)} \right): = \sum_{i,j,k}\sum_{\vecl}\sum_{\xi} w_{(\xi)} \, .
\end{equation}
The notation $w_{(\xi)}$ implicitly encodes all indices and thus will be a useful shorthand for the principal part of the perturbation. To make the perturbation divergence free, we add
\begin{equation}\label{wqplusoneonec}
    w_{q+1,\nn}^{(c)} = \sum_{i,j,k}\sum_{\vecl}\sum_{\xi} \nabla a_{(\xi)} \times \left( \nabla\Phi_{(i,k)}^T \mathbb{U}_{\xi,q+1,\nn}\circ \Phi_{(i,k)} \right) = \sum_{i,j,k}\sum_{\vecl}\sum_\xi w_{(\xi)}^{(c)}
\end{equation}
so that
\begin{equation}\label{wqplusoneone}
    w_{q+1,\nn} = w_{q+1,\nn}^{(p)} + w_{q+1,\nn}^{(c)} = \sum_{i,j,k}\sum_{\vecl}\sum_{\xi} \curl \left( a_{(\xi)} \nabla\Phi_{(i,k)}^T \mathbb{U}_{\xi,q+1,\nn} \circ \Phi_{(i,k)} \right) \, .
\end{equation}

\subsection{Estimates for \texorpdfstring{$w\qplusnn$}{wqn}}\label{ss:stress:w:estimates}

In this section, we verify \eqref{e:inductive:n:2:velocity} and \eqref{e:inductive:n:1:velocity:unif}. We first estimate the $L^r$ norms of the coefficient functions $a_{(\xi)}$. We have consolidated the proofs for each value of $\nn$ into the following lemma.
\begin{lemma}
\label{lem:a_master_est_p}
For $N,N',N'',M$ with $N', N'' \in \{0,1\}$ and $N + N' +M \leq \Nfnn-\NcutSmall-\NcutLarge-4$, and $r,r_1,r_2\in[1,\infty]$ with $\frac{1}{r_1}+\frac{1}{r_2}=1$, we have the following estimate.
\begin{align}
\bigl\|D^{N-N''} D_{t,q}^M (\xi^\ell A_\ell^p \partial_p)^{N'} D^{N''} a_{\xi,i,j,k,q,\nn,\vec{l}}\bigr\|_{L^r} &\lessg | \supp ( \eta_{i,j,k,q,\nn,\xi,\vec{l}} ) |^{\sfrac 1r} \delta_{q+1,\nn}^{\sfrac 12} \Gamma_{q+1}^{j+2} \left(\Gamma_{q+1}^{-1}\lambda_{q+1}r_{q+1,\nn}\right)^N \notag\\
&\qquad \qquad \times \left(\Gamma_{q+1}\lambda\qnn\right)^{N'} \MM{M, \NindSmall, \tau_{q}^{-1}\Gamma_{q+1}^{i-\cstarnn+3}, \tilde\tau_{q}^{-1}\Gamma_{q+1}^{-1}}\label{e:a_master_est_p}.
\end{align}
In the case that $r=\infty$, the above estimate gives that
\begin{align}
\bigl\| D^{N-N''} D_{t,q}^M  (\xi^\ell A_\ell^p \partial_p)^{N'} D^{N''} a_{\xi,i,j,k,q,\nn,\vec{l}}\bigr\|_{L^\infty} 
&\lessg \Gamma_{q}^{\frac{\badshaq}{2}} \Gamma_{q+1}^{7\Upsilon(\nn) + \frac{7}{2}} \left(\Gamma_{q+1}^{-1}\lambda_{q+1}r_{q+1,\nn}\right)^N \notag\\
&\qquad \qquad \times \left(\Gamma_{q+1}\lambda\qnn\right)^{N'} \MM{M,\Nindt, \tau_{q}^{-1}\Gamma_{q+1}^{i-\cstarnn+3},\tilde\tau_q^{-1}\Gamma_{q+1}^{-1}} \label{e:a_master_est_p_uniform} \, .
\end{align}
\end{lemma}
\begin{proof}
[Proof of Lemma~\ref{lem:a_master_est_p}]
We first compute \eqref{e:a_master_est_p} for the case $r=\infty$. Recalling estimate \eqref{eq:D:Dt:Rn:sharp:new}, we have that for all $N+M\leq\Nfnn-4$,
\begin{align}
\bigl\| D^N D_{t,q}^M\mathring{R}_{q,\nn} \bigr\|_{L^{\infty}(\supp \eta_{(i,j,k)})}
&\lessg \delta_{q+1,\nn}\Gamma^{2j+2}_{q+1}
 \left(\Gamma_{q+1}\lambda\qnn\right)^N
\MM{M, \NindSmall, \tau_{q}^{-1}\Gamma_{q+1}^{i-\cstarnn+2}, \tilde\tau_{q}^{-1}\Gamma_{q+1}^{-1} }.\nonumber
\end{align}
From Corollary~\ref{cor:deformation}, we have that for all $N+M\leq \sfrac{3\Nfin}{2}$,
\begin{align*}
\left\| D^N D_{t,q}^M D \Phi_{(i,k)} \right\|_{L^\infty(\supp(\psi_{i,q}\chi_{i,k,q}))} &\leq \tilde{\lambda}_q^{N} \MM{M,\NindSmall,\Gamma_{q+1}^{i-\cstar} \tau_q^{-1},\tilde{\tau}_q^{-1}\Gamma_{q+1}^{-1}}.
\end{align*}
Thus from the Leibniz rule and definition \eqref{eq:rqnpj}, for $N+M\leq\Nfnn-4$,
\begin{align}
&\norm{D^N D_{t,q}^M R_{q,\nn,j,i,k} }_{L^{\infty}(\supp \eta_{(i,j,k)})} \lessg \delta_{q+1,\nn}\Gamma^{2j+4}_{q+1}\left(\Gamma_{q+1}\lambda\qnn\right)^N \MM{M, \NindSmall, \tau_{q}^{-1}\Gamma_{q+1}^{i-\cstarnn+2}, \tilde\tau_{q}^{-1}\Gamma_{q+1}^{-1}}\label{eq:davidc:1}
\,.\end{align}
The above estimates allow us to apply \cite[Lemma~A.5]{BMNV21} with $N=N'$, $M=M'$ so that $N+M\leq\Nfnn-4$, $\psi = \gamma_{\xi,}$,
$\Gamma_\psi=1$, $v = \vlq$, $D_t = D_{t,q}$, $h(x,t) = R_{q,\nn,j,i,k}(x,t)$, $C_h = \delta_{q+1,\nn}\Gamma_{q+1}^{2j+4} = \Gamma^2$, $\lambda=\tilde\lambda = \lambda\qnn\Gamma_{q+1}$, $\mu = \tau_{q}^{-1} \Gamma_{q+1}^{i-\cstarnn+2}$, $\tilde \mu = \tilde \tau_{q}^{-1}\Gamma_{q+1}^{-1}$, and $N_t=\Nindt$. We obtain that for all $N+M\leq\Nfnn-4$,
\begin{align*}
\norm{ D^ND_{t,q}^M \gamma_{\xi}\left(\frac{R_{q,\nn,j,i,k}}{\delta_{q+1,\nn,\pp}\Gamma^{2j+4}_{q+1}}\right)}_{L^{\infty}(\supp \eta_{(i,j,k)})} &\lesssim \left(\Gamma_{q+1}\lambda\qnn\right)^N \MM{M, \NindSmall, \tau_{q}^{-1}\Gamma_{q+1}^{i-\cstarnn+2}, \tilde\tau_{q}^{-1}\Gamma_{q+1}^{-1}} \,.
\end{align*}
From the above bound, definition \eqref{eq:a:xi:def}, the Leibniz rule, estimate \eqref{eq:nasty:Dt:psi:i:q:orangutan} at level $q$ in conjunction with \eqref{eq:Nind:cond:2}, \eqref{eq:Lagrangian:Jacobian:6}, \eqref{eq:chi:cut:dt}, \eqref{eq:D:Dt:omega:sharp}, and \eqref{eq:checkerboard:derivatives}, we obtain that for $N+N'+M\leq\Nfnn -\NcutLarge - \NcutSmall-4$,
\begin{align*}
\bigl\| D^ND_{t,q}^M ( \xi^\ell A_\ell^p \partial_p)^{N'} a_{\xi,i,j,k,q,\nn,\vec{l}}\bigr\|_{L^\infty}
& \lessg \delta_{q+1,\nn}^{\sfrac 12} \Gamma_{q+1}^{j+2} (\Gamma_{q+1}^{-1}\lambda_{q+1}r_{q+1,\nn})^N \\
&\qquad \qquad \times (\Gamma_{q+1}\lambda\qnn)^{N'} \MM{M,\Nindt, \tau_{q}^{-1}\Gamma_{q+1}^{i-\cstarnn+3},\tilde\tau_q^{-1}\Gamma_{q+1}^{-1}} \, .
\end{align*}
Then, using \eqref{eq:new:jmax:bound} the above bound becomes \eqref{e:a_master_est_p_uniform} for $N''=0$.  The proof for $N''=1$ is nearly identical, and we omit the details.  When $r\neq \infty$, we use $\left\| f \right\|_{L^r}\leq \left\| f \right\|_{L^\infty} | \{ \supp f \} |^{\sfrac 1r} $ and the demonstrated bound for $r=\infty$ to obtain \eqref{e:a_master_est_p} for the full range of $r$.
\end{proof}
An immediate consequence of Lemma~\ref{lem:a_master_est_p} is that we have estimates for the velocity increments themselves. These are summarized in the following corollary. The proofs for $r\neq \infty$ are analogous to those from \cite[Corollary 8.2]{BMNV21} and therefore use Lemma~\ref{l:slow_fast}. We only note that the gap between the spatial derivative cost of $a_{(\xi)}$ ($\lambda_{q+1}r_{q+1,\nn}\Gamma_{q+1}^{-1}$ from Lemma~\ref{lem:a_master_est_p}) and the minimum frequency of $\WW_{\xi,q+1,\nn}$ ($\lambda_{q+1}r_{q+1,\nn}$ from \eqref{eq:W:xi:q+1:nn:def} and Proposition~\ref{prop:pipeconstruction}) is now only $\Gamma_{q+1}$, and so we need the inequality \eqref{eq:lambdaqn:identity:2} in order to satisfy \eqref{eq:slow_fast_3}. The assumption \eqref{eq:slow_fast_4} follows from \eqref{eq:lambdaqn:identity:3}. The estimates for $r=\infty$ follow directly from \eqref{e:a_master_est_p_uniform} and \eqref{e:pipe:estimates:2}.
\begin{corollary}
\label{cor:corrections:Lp}
For $N+M\leq \Nfnn-\NcutSmall-\NcutLarge-2\Ndec-8$ and $(r,r_1,r_2)\in\left\{(1,2,2),(2,\infty,1)\right\}$, for $w_{(\xi)}$ we have the estimates
\begin{subequations}
\begin{align}
\norm{D^ND_{t,q}^M w_{(\xi)}}_{L^r} 
&\lessg | \supp (\eta_{i,j,k,q,\nn,\xi,\vec{l}}) |^{\sfrac 1r} \delta_{q+1,\nn}^{\sfrac 12}\Gamma^{j+2}_{q+1} {\left(\rqnperptilde\right)}^{\frac2r-1} \lambda_{q+1}^N  \MM{M, \NindSmall, \tau_{q}^{-1}\Gamma_{q+1}^{i-\cstarnn+3}, \tilde\tau_{q}^{-1}\Gamma_{q+1}^{-1}}\label{eq:w:oxi:est}
\\
\norm{D^ND_{t,q}^M w_{(\xi)}}_{L^\infty}
&\lessg 
r_{q+1,\nn}^{-1} \lambda_{q+1}^N \Gamma_{q}^{\frac{\badshaq}{2}} \Gamma_{q+1}^{7\Upsilon(\nn) + \frac{7}{2}} \MM{M,\Nindt, \tau_{q}^{-1}\Gamma_{q+1}^{i-\cstarnn+3},\tilde\tau_q^{-1}\Gamma_{q+1}^{-1}}\, .
\label{eq:w:oxi:unif}
\end{align}
\end{subequations}
For $N+M\leq \Nfnn-\NcutSmall-\NcutLarge-2\Ndec-9$ and $(r,r_1,r_2)\in\left\{(1,2,2),(2,\infty,1)\right\}$, we have that
\begin{subequations}
\begin{align}
\bigl\| D^ND_{t,q}^M w_{(\xi)}^{(c)}\bigr\|_{L^r} &\lessg
\frac{ \lambda_{q+1}r_{q+1,\nn}}{\lambda_{q+1}} |\supp(\eta_{i,j,k,q,\nn,\xi,\vec{l}})|^{\sfrac 1r}
\delta_{q+1,\nn}^{\sfrac 12}\Gamma^{j+2}_{q+1}
{\left(r_{q+1,\nn}\right)}^{\frac2r-1}  \notag\\
&\qquad \qquad \qquad \times\lambda_{q+1}^N \MM{M, \NindSmall, \tau_{q}^{-1}\Gamma_{q+1}^{i-\cstarnn+3}, \tilde\tau_{q}^{-1}\Gamma_{q+1}^{-1}}
\label{eq:w:oxi:c:est} \\
\bigl\| D^ND_{t,q}^M w_{(\xi)}^{(c)}\bigr\|_{L^\infty} & \lesssim r_{q+1,\nn}^{-1} \frac{\lambda_{q+1}r_{q+1,\nn}}{\lambda_{q+1}} \Gamma_{q}^{\frac{\badshaq}{2}} \Gamma_{q+1}^{7\Upsilon(\nn) + \frac{7}{2}} \lambda_{q+1}^{N}
\MM{M,\Nindt,\tau_{q}^{-1}\Gamma_{q+1}^{i-\cstarnn+3},\tilde\tau_q^{-1}\Gamma_{q+1}^{-1}} \label{eq:w:oxi:c:unif}
\end{align}
\end{subequations}
\end{corollary}
\begin{remark}
Note that the above estimates verify the bounds \eqref{e:inductive:n:2:velocity} and \eqref{e:inductive:n:1:velocity:unif} after summing on $(i,j,k,\nn,\xi,\vec{l})$ and using \eqref{item:lebesgue:1} with $r_1=\infty$ and $r_2=2$.  Then from \eqref{wqplusoneonep}--\eqref{wqplusoneone}, \eqref{eq:w:oxi:est}--\eqref{eq:w:oxi:c:unif}, and the parameter inequalities \eqref{ineq:rq:useful}, \eqref{eq:delta:q:nn:pp:ineq}, and \eqref{eq:ineq:badshaq}, the bounds \eqref{e:main:inductive:q:velocity} and \eqref{e:main:inductive:q:velocity:unif} follow after using the extra factor of $\Gamma_{q+1}^{-1}$ to absorb implicit constants.
\end{remark}

\subsection{Identification of error terms}\label{ss:stress:error:identification}

Recall that $v_{q,\nn-1}$ is divergence-free and is a solution to the Euler-Reynolds system with stress given in  \eqref{e:inductive:n:2:eulerreynolds}.  Now using the definition of $\RR_{q,\nn}$ from \eqref{e:rqnp:definition} for $0\leq \nn \leq \nmax$, we add $w_{q+1,\nn}$ as defined in \eqref{wqplusoneone}, we have that $v_{q,\nn}:= v_{q,\nn-1} + w_{q+1,\nn}$ solves
\begin{align}
    \partial_t v_{q,\nn} + \div \left( v_{q,\nn}\otimes v_{q,\nn} \right) + \nabla p_{q,\nn-1} 
    &= \div\left(\RR_{q+1}^{\nn-1}\right) 
    + \div \Biggl( \sum\limits_{n'=0}^{\nn-1} \sum\limits_{n>r(n')}^{\nmax} \HH_{q,n}^{n'} \Biggr) + \div\RR_{q}^{\textnormal{comm}}\notag\\
    &\quad + \Dtq w_{q+1,\nn} + w_{q+1,\nn}\cdot \nabla \vlq + 2 \sum_{n'\leq \nn-1} \div\left( w_{q+1,n'} \otimes_{\rm s} w_{q+1,\nn}\right)\notag\\
    &\quad + \div \left( w_{q+1,\nn} \otimes w_{q+1,\nn}+\RR_{q,\nn} \right).\label{e:expand:nn}
\end{align}
Here we use the notation $
a \otimes_{\rm s} b = \frac 12 (a   \otimes b + b   \otimes a)$.
The first term on the right hand side is $\RR_{q+1}^{\nn-1}$, which for $\nn\geq 1$ satisfies the same estimates as $\RR_{q+1}^\nn$ by \eqref{e:inductive:n:2:Rstress} and will thus be absorbed into $\RR_{q+1}^{\nn}$.  The second term, save for the fact that the sum is over $n'\leq\nn-1$ rather than $n'\leq\nn$ and is therefore missing the terms $\HH_{q,n}^\nn$, matches \eqref{e:inductive:n:2:eulerreynolds} at level $\nn$ (i.e. replacing every instance of $\nn-1$ with $\nn$). We apply the inverse divergence operators from Proposition~\ref{prop:intermittent:inverse:div} to the transport and Nash errors to obtain
$$ \Dtq w_{q+1,\nn} + w_{q+1,\nn}\cdot\nabla\vlq = \div \left( (\divH + \divR) \left( \Dtq w_{q+1,\nn} + w_{q+1,\nn}\cdot\nabla\vlq\right)\right) + \nabla \pi,   $$
and these errors are absorbed into $\RR_{q+1}^\nn$ or the new pressure.  We will show in Section~\ref{ss:stress:oscillation:2} that the interaction of $w_{q+1,\nn}$ with previous terms $w_{q+1,n'}$ is a Type 2 oscillation error so that
\begin{equation}\label{nn:overlap:definition}
2 \sum_{0 \leq n'\leq \nn-1} w_{q+1,n'}\otimes_{\rm s} w_{q+1,\nn} =0  \, .
\end{equation}
So to verify \eqref{e:inductive:n:2:eulerreynolds} at level $\nn$, only the analysis of last line of the right-hand side of \eqref{e:expand:nn} remains. 

For a fixed $\nn$, throughout this section we will consider sums over indices $(\xi,i,j,k,\vecl)$, where the direction vector $\xi$ takes on one of the finitely many values in Proposition~\ref{prop:pipeconstruction}, $0 \leq i \leq \imax(q)$ indexes the velocity cutoffs, $0\leq j \leq \jmax(q,\nn)$ indexes the stress cutoffs, the parameter $k \in \Z$ indexes the time cutoffs defined in \eqref{eq:chi:cut:def}, and lastly, $\vecl \in \N^3_0$ indexes the checkerboard cutoffs from Definition~\ref{def:checkerboard}.  For brevity of notation, we denote sums over such indexes as 
\begin{equation*}
\sum_{\xi,i,j,k,\vecl} \,.
\end{equation*}
Moreover, we shall denote as
\begin{equation}\notag
    \sum_{\neq\{\xi,i,j,k,\vecl\}} 
\end{equation}
the {\em double-summation} over indexes $(\xi,i,j,k, \vecl)$ and $(\xistar, \istar, \jstar, \kstar, \veclstar)$ which belong to the set
\begin{equation}\notag
    \left\{ (\xi,i,j,k,\vecl)  \, , (\xistar, \istar, \jstar, \kstar, \veclstar): \xi\neq\xistar \lor i \neq \istar \lor j \neq \jstar \lor k \neq \kstar  \lor \vecl\neq\veclstar \right \} \, .
\end{equation}
We may now write out the self-interaction of $w_{q+1,\nn}$ as
\begin{align}
    \div\left(w_{q+1,\nn}\otimes w_{q+1,\nn}\right) &= \sum_{\xi,i,j,k,\vecl} \div \left( \curl\bigl(a_{(\xi)} \nabla\Phi_{(i,k)}^T \UU_{\xi,q+1,\nn} \bigr) \otimes \curl\bigl( a_{(\xi)}\nabla\Phi_{i,k}^T \UU_{\xi,q+1,\nn} \bigr) \right)\notag \\
    & + \sum_{\neq\{\xi,i,j,k,\vecl\}} \div \left( \curl\bigl(a_{(\xi)} \nabla\Phi_{(i,k)}^T \UU_{\xi,q+1,\nn} \bigr) \otimes \curl\bigl( a_{(\xistar)}\nabla\Phi_{(\istar,\kstar)}^T \UU_{\xistar,q+1,\nn} \bigr) \right)\notag\\
    &=: \div \mathcal{O}_{\nn,1} + \div \mathcal{O}_{\nn,2} \label{e:split:nn:1}.
\end{align}
We will show that $\mathcal{O}_{\nn,2}$ is a Type 2 oscillation error so that 
\begin{equation}\label{eq:something:is:zero}
\mathcal{O}_{\nn,2} = 0 \, . 
\end{equation}
Splitting $\mathcal{O}_{\nn,1}$ gives
\begin{align}
    \div \mathcal{O}_{\nn,1} &= \sum_{\xi,i,j,k,\vecl} \div \left( \bigl( a_{(\xi)} \nabla\Phi_{(i,k)}^{-1} \WW_{\xi,q+1,\nn}\circ\Phi_{(i,k)} \bigr) \otimes \bigl( a_{(\xi)} \nabla\Phi_{(i,k)}^{-1} \WW_{\xi,q+1,\nn}\circ\Phi_{(i,k)} \bigr) \right) \notag \\
    &\qquad + 2 \sum_{\xi,i,j,k,\vecl} \div \left( \bigl( a_{(\xi)} \nabla\Phi_{(i,k)}^{-1} \WW_{\xi,q+1,\nn}\circ\Phi_{(i,k)} \bigr) \otimes_{\rm s} \bigl( \nabla a_{(\xi)} \times \bigl( \nabla\Phi_{(i,k)}^{T} \UU_{\xi,q+1,\nn}\circ\Phi_{(i,k)}\bigr)  \bigr)\right) \notag \\
    &\qquad + \sum_{\xi,i,j,k,\vecl} \div \left( \bigl( \nabla a_{(\xi)} \times \bigl( \nabla\Phi_{(i,k)}^{T} \UU_{\xi,q+1,\nn}\circ\Phi_{(i,k)}\bigr) \bigr) \otimes \bigl( \nabla a_{(\xi)} \times \bigl( \nabla\Phi_{(i,k)}^{T} \UU_{\xi,q+1,\nn}\circ\Phi_{(i,k)}\bigr) \bigr) \right)\notag\\
    &:= \div\left( \mathcal{O}_{\nn,1,1}+\mathcal{O}_{\nn,1,2}+\mathcal{O}_{\nn,1,3} \right).\label{e:split:nn:2}
\end{align}
The last two of these terms are divergence corrector errors and will therefore be absorbed into $\RR_{q+1}^\nn$ and estimated in Section~\ref{ss:stress:divergence:correctors}.  So the only terms which we have yet to identify from \eqref{e:expand:nn} are $\mathcal{O}_{\nn,1,1}$ and $\RR_{q,\nn}$.

Recall cf.~\eqref{eq:W:xi:q+1:nn:def} that $\WW_{\xi,q+1,\nn}$ is periodized to scale $\left(\lambda_{q+1}r_{q+1,\nn}\right)^{-1}$.  Using \eqref{e:Pqnp:identity}, we have that
\begin{align*}
\WW_{\xi,q+1,\nn}\otimes\WW_{\xi,q+1,\nn} &= \dashint_{\mathbb{T}^3}{\WW_{\xi,q+1,\nn}\otimes\WW_{\xi,q+1,\nn}} + \sum_{n=\nn+1}^{\nmax+1} \mathbb{P}_{[q,n]} \left(\WW_{\xi,q+1,\nn}\otimes\WW_{\xi,q+1,\nn}\right) \, .
\end{align*}
Using \eqref{item:pipe:4} and \eqref{eq:pipes:flowed:2} from Proposition~\ref{prop:pipeconstruction} in combination with the above identity, and the convention that $\bullet$ denotes the unspecified components of a vector field, we then split $\mathcal{O}_{\nn,1,1}$ as
\begin{align}
    \div\left(\mathcal{O}_{\nn,1,1}\right) &= \sum_{\xi,i,j,k,\vecl} \div\left( a_{(\xi)}^2 \nabla\Phi_{(i,k)}^{-1} \left(\xi\otimes \xi\right) \nabla\Phi_{(i,k)}^{-T} \right) \notag \\
    &\quad + \sum_{\xi,i,j,k,\vecl} \div\bigg{(} a_{(\xi)}^2 \nabla\Phi_{(i,k)}^{-1}  \sum_{n=\nn+1}^{\nmax+1} \mathbb{P}_{[q,n]} (\WW\otimes\WW)_{\xi,q+1,\nn}(\Phi_{(i,k)}) \nabla\Phi_{(i,k)}^{-T} \bigg{)}\notag\\
    &= \div \sum_{\xi,i,j,k,\vecl} \delta_{q+1,\nn}\Gamma_{q+1}^{2j+4}\eta_{(i,j,k)}^2 \gamma_\xi^2\left(\frac{R_{q,\nn,j,i,k}}{\delta_{q+1,\nn}\Gamma_{q+1}^{2j+4}}\right)\nabla\Phi_{(i,k)}^{-1}\left(\xi\otimes\xi\right)\nabla\Phi_{(i,k)}^{-T} \notag\\
    &\quad + \sum_{\xi,i,j,k,\vecl} \nabla a_{(\xi)}^2\nabla\Phi_{(i,k)}^{-1} \sum_{n=\nn+1}^{\nmax+1} \mathbb{P}_{[q,n]} (\WW\otimes\WW)_{\xi,q+1,\nn}(\Phi_{(i,k)}) \nabla\Phi_{(i,k)}^{-T}\notag\\
    &\quad + \sum_{\xi,i,j,k,\vecl} a_{(\xi)}^2(\nabla\Phiik^{-1})_{\theta}^\alpha \sum_{n=\nn+1}^{\nmax+1} \mathbb{P}_{[q,n]} (\WW^\theta\WW^\gamma)_{\xi,q+1,\nn}(\Phiik) \partial_\alpha(\nabla\Phiik^{-1})_{\gamma}^\bullet \, .
\label{eq:euler:reynolds:gross:nn}
\end{align}
By \eqref{e:split} from Proposition~\ref{p:split}, identity \eqref{eq:rqnpj}, and \eqref{eq:checkerboard:partition}, we obtain that
\begin{align}
    & \sum_{i,j,k,\xi} \sum_{\vecl} \delta_{q+1,\nn}\Gamma_{q+1}^{2j+4}\eta_{i,j,k,q,\nn,\xi,\vecl}^2 \, \gamma_\xi^2\left(\frac{R_{q,\nn,j,i,k}}{\delta_{q+1,\nn}\Gamma_{q+1}^{2j+4}}\right)\nabla\Phi_{(i,k)}^{-1}\left(\xi\otimes\xi\right)\nabla\Phi_{(i,k)}^{-T} \notag\\
    &\qquad = \sum_{i,j,k,\xi} \delta_{q+1,\nn}\Gamma_{q+1}^{2j+4} \psi_{i,q}^2 \omega_{i,j,q,\nn}^2 \chi_{i,k,q}^2 \gamma_\xi^2\left(\frac{R_{q,\nn,j,i,k}}{\delta_{q+1,\nn}\Gamma_{q+1}^{2j+4}}\right)\nabla\Phi_{(i,k)}^{-1}\left(\xi\otimes\xi\right)\nabla\Phi_{(i,k)}^{-T} \notag\\
    &\qquad =  \sum_{i,j,k} \psi_{i,q}^2 \omega_{i,j,q,\nn}^2 \chi_{i,k,q}^2 \left( \delta_{q+1,\nn}\Gamma_{q+1}^{2j+4}\Id- \RR_{q,\nn}\right)\notag\\
    &\qquad = - \RR_{q,\nn} + \Id \biggl(\sum_{\; i,j,k} \psi_{i,q}^2 \omega_{i,j,q,\nn}^2 \chi_{i,k,q}^2 \delta_{q+1,\nn}\Gamma_{q+1}^{2j+4}\biggr) \,,\label{eq:cancellation:plus:pressure:nn}
\end{align}
where in the last equality we have appealed to the fact that $\eta^2_{i,j,k}$ forms a partition of unity, cf. \eqref{eq:eta:cut:partition:unity}. The second term on the right hand side of \eqref{eq:cancellation:plus:pressure:nn} is a pressure term.

Returning to the second and third lines in \eqref{eq:euler:reynolds:gross:nn}, we first note that when $\nn=0$, \eqref{eq:def:lambda:r:q+1:n} gives that $\lambda_{q+1}r_{q+1,0}=\lambda_{q+1}^{\sfrac 12}\lambda\qn^{\sfrac 12}\Gamma_{q+1}^{-2}=\lambda_{q+1}^{\sfrac 12} \lambda_q^{\sfrac 12}\Gamma_{q+1}$. Then from \eqref{def:LPqnp} and \eqref{eq:W:xi:q+1:nn:def}, for all $1\leq n \leq \nmax+1$, we deduce that $\mathbb{P}_{[q,n]}\left(\WW_{\xi,q+1,0}\otimes\WW_{\xi,q+1,0}\right)\neq 0$.  Conversely, when $1 \leq \nn \leq \nmax$, for all $n \geq \nn + 1$ such that $\lambda_{q,n} < \lambda_{q+1}^{\sfrac 12} \lambda_{q,\nn}^{\sfrac 12}\Gamma_{q+1}^{-2}$,
i.e. such that the maximal frequency of $\Proj_{[q,n]}$ is less than the minimal frequency of $\mathbb{P}_{\neq 0}\left(\WW_{\xi,q+1,\nn} \otimes \WW_{\xi,q+1,\nn}\right)$, we have that $\mathbb{P}_{[q,n]} \left(\WW_{\xi,q+1,\nn}\otimes\WW_{\xi,q+1,\nn}\right) = 0$. Using \eqref{eq:def:lambda:rq}, we write that
\begin{align}
& \qquad \underbrace{\lambda_q^{\frac 12 - \frac{n}{2(\nmax+1)}} \lambda_{q+1}^{\frac 12 + \frac{n}{2(\nmax+1)}}  }_{ = \lambda_{q,n}}
< \underbrace{\lambda_{q+1}^{\frac 12} \lambda_q^{\frac 14 - \frac{\nn}{4(\nmax+1)}} \lambda_{q+1}^{\frac 14 + \frac{\nn}{4(\nmax+1)}} \Gamma_{q+1}^{-2}}_{ = \lambda_{q+1}^{\sfrac 12} \lambda_{q,\nn}^{\sfrac 12} \Gamma_{q+1}^{-2} } \notag\\
\Leftrightarrow 
& \qquad \lambda_q^{\frac 14  + \frac{\nn - 2 n}{4(\nmax+1)} }
< 
\lambda_{q+1}^{\frac 14 + \frac{\nn-2n}{4(\nmax+1)}} \Gamma_{q+1}^{-2}
\notag\\
\Leftrightarrow &\qquad 2\varepsilon_\Gamma < \frac 14  + \frac{\nn - 2 n}{4(\nmax+1)} 
\notag\\
\Leftrightarrow & \qquad 8 \varepsilon_\Gamma (\nmax+1) < \nmax + 1 + \nn - 2 n
\notag\\
\Leftarrow & \qquad n \leq \frac{\nmax + \nn}{2} \, , \qquad \frac{1}{2} - 4\varepsilon_\Gamma (\nmax+1) > 0
\,. \label{eq:nmax:nn:mean}
\end{align}
The second inequality in the last line follows from \eqref{eq:eps:still:doing:it}. Based on \eqref{eq:nmax:nn:mean} and \eqref{eq:rofn}, we apply Proposition~\ref{prop:intermittent:inverse:div} in the parameter regimes $\nn=0,1\leq n \leq \nmax$ and $1\leq \nn \leq \nmax-1, r(\nn) = \frac{\nmax+\nn}{2} < n \leq \nmax$ to define 
\begin{align}
    &\HH_{q,n}^\nn := \divH \bigg{(} \sum_{\xi,i,j,k} \nabla a_{(\xi)}^2\nabla\Phi_{(i,k)}^{-1} \mathbb{P}_{[q,n]} (\WW_{\xi,q+1,\nn}\otimes\WW_{\xi,q+1,\nn})(\Phi_{(i,k)}) \nabla\Phi_{(i,k)}^{-T}\notag\\
    &\qquad \qquad \qquad + \sum_{\xi,i,j,k} a_{(\xi)}^2(\nabla\Phiik^{-1})_{\theta}^\alpha \mathbb{P}_{[q,n]} (\WW_{\xi,q+1,\nn}^\theta\WW_{\xi,q+1,\nn}^\gamma)(\Phiik) \partial_\alpha(\nabla\Phiik^{-1})_{\gamma}^\bullet \bigg{)} \, . \label{eq:Hqnpnn:definition}
\end{align}
The terms from \eqref{eq:euler:reynolds:gross:nn} with $\mathbb{P}_{[q,\nmax+1]}$ will be absorbed into $\RR_{q+1}^\nn$. We will show shortly that the terms $\HH_{q,n}^\nn$ in \eqref{eq:Hqnpnn:definition} are precisely the terms needed to make \eqref{e:expand:nn} match \eqref{e:inductive:n:2:eulerreynolds} at level $\nn$.

Recall from \eqref{RR:q+1:n-1:to:n} that $\RR_{q+1}^\nn$ will include $\RR_{q+1}^{\nn-1}$ in addition to error terms arising from the addition of $w_{q+1,\nn}$ which are small enough to be absorbed in $\RR_{q+1}$.  Then to check \eqref{e:inductive:n:2:eulerreynolds}, we return to \eqref{e:expand:nn} and use \eqref{nn:overlap:definition}, \eqref{e:split:nn:1}, \eqref{e:split:nn:2}, \eqref{eq:euler:reynolds:gross:nn}, \eqref{eq:cancellation:plus:pressure:nn}, \eqref{eq:nmax:nn:mean}, \eqref{eq:rofn}, and \eqref{eq:Hqnpnn:definition} to write 
\begin{align}
    &\partial_t v_{q,\nn} + \div \left( v_{q,\nn}\otimes v_{q,\nn} \right) + \nabla p_{q,\nn-1} \notag\\
    &= \div\RR_{q}^{\textnormal{comm}} + \div \biggl( \sum\limits_{\; n'=0}^{\nn-1} \sum\limits_{n> r(n')}^{\nmax} \HH_{q,n}^{n'} \biggr) + \div\left(\RR_{q+1}^{\nn-1}\right)\notag\\
    &\quad + \Dtq w_{q+1,\nn} + w_{q+1,\nn}\cdot \nabla \vlq + \div \left( \mathcal{O}_{\nn,1,2} + \mathcal{O}_{\nn,1,3} \right) + \div \left( \mathcal{O}_{\nn,1,1} + \RR_{q,\nn} \right) \notag \\
    &= \div\RR_{q}^{\textnormal{comm}} + \div \biggl( \sum\limits_{\; n'=0}^{\nn-1} \sum\limits_{n> r(n')}^{\nmax} \HH_{q,n}^{n'} \biggr) \notag\\
    &\quad + \div \bigg{(} \RR_{q+1}^{\nn-1} + (\divH + \divR) \left(\Dtq w_{q+1,\nn} +  w_{q+1,\nn}\cdot\nabla\vlq \right)  +  \mathcal{O}_{\nn,1,2} + \mathcal{O}_{\nn,1,3} \bigg{)} + \nabla \pi \notag\\
    &\quad +\div\bigg{[}  \left(\divH+\divR\right) \bigg{(} \sum_{\xi,i,j,k,\vecl} \nabla a_{(\xi)}^2\nabla\Phi_{(i,k)}^{-1}  \sum_{n> r(\nn)}^{\nmax+1}\mathbb{P}_{[q,n]} (\WW\otimes\WW)_{\xi,q+1,\nn}(\Phi_{(i,k)}) \nabla\Phi_{(i,k)}^{-T} \bigg{)}  \notag \\
    &\qquad\qquad + (\divH+ \divR) \bigg{(} \sum_{\xi,i,j,k,\vecl} a_{(\xi)}^2(\nabla\Phiik^{-1})_{\theta}^\alpha \sum_{n> r(\nn)}^{\nmax+1} \mathbb{P}_{[q,n]} 
    (\WW^\theta\WW^\gamma)_{\xi,q+1,\nn}(\Phiik) \partial_\alpha(\nabla\Phiik^{-1})_{\gamma}^\bullet \bigg{)} \bigg{]} \label{eq:id:nn:0} \\
    &= \div\RR_{q}^{\textnormal{comm}} + \div \biggl( \sum\limits_{\; n'=0}^{\nn} \sum\limits_{n> r(n')}^{\nmax} \HH_{q,n}^{n'} \biggr) + \div \RR_{q+1}^{\nn} + \nabla \pi \, , \label{eq:id:nn:10}
\end{align}
where
\begin{align}
\RR_{q+1}^{\nn} &= \RR_{q+1}^{\nn-1} + (\divH + \divR) \left(\Dtq w_{q+1,\nn} +  w_{q+1,\nn}\cdot\nabla\vlq \right)  + \mathcal{O}_{\nn, \rm corr} + \pi \Id \notag\\
&\quad +\divH \bigg{(} \sum_{\xi,i,j,k,\vecl} \nabla a_{(\xi)}^2\nabla\Phi_{(i,k)}^{-1}  \mathbb{P}_{[q,\nmax+1]} (\WW\otimes\WW)_{\xi,q+1,\nn}(\Phi_{(i,k)}) \nabla\Phi_{(i,k)}^{-T} \bigg{)}  \notag \\
&\quad + \divH \bigg{(} \sum_{\xi,i,j,k,\vecl} a_{(\xi)}^2(\nabla\Phiik^{-1})_{\theta}^\alpha \mathbb{P}_{[q,\nmax+1]} 
(\WW^\theta\WW^\gamma)_{\xi,q+1,\nn}(\Phiik) \partial_\alpha(\nabla\Phiik^{-1})_{\gamma}^\bullet \bigg{)} \notag\\
&\quad +\divR \bigg{(} \sum_{\xi,i,j,k,\vecl} \nabla a_{(\xi)}^2\nabla\Phi_{(i,k)}^{-1}  \sum_{n>r(\nn)}^{\nmax+1}\mathbb{P}_{[q,n]} (\WW\otimes\WW)_{\xi,q+1,\nn}(\Phi_{(i,k)}) \nabla\Phi_{(i,k)}^{-T} \bigg{)}  \notag \\
&\quad + \divR \bigg{(} \sum_{\xi,i,j,k,\vecl} a_{(\xi)}^2(\nabla\Phiik^{-1})_{\theta}^\alpha \sum_{n>r(\nn)}^{\nmax+1} \mathbb{P}_{[q,n]} 
(\WW^\theta\WW^\gamma)_{\xi,q+1,\nn}(\Phiik) \partial_\alpha(\nabla\Phiik^{-1})_{\gamma}^\bullet \bigg{)} \, . \label{eq:idiing:RRqnn}
\end{align}
We first emphasize that to obtain \eqref{eq:id:nn:10}, we have used that the Type 2 oscillation errors from \eqref{nn:overlap:definition} and \eqref{eq:something:is:zero} will be shown to vanish.  In addition, the symmetric stress $\mathcal{O}_{\nn, \rm corr}$ will be defined in  The equality \eqref{eq:id:nn:10} completes the proof of \eqref{e:inductive:n:2:eulerreynolds} at level $\nn$.

\subsection{Type 1 oscillation errors}\label{ss:stress:oscillation:1}

Recall from \eqref{eq:id:nn:10} that there are two main categories of Type 1 oscillation errors which arise from the addition of $w_{q+1,\nn}$: the higher order stresses $H_{q,n}^{\nn}$, which are defined and non-vanishing in \eqref{eq:Hqnpnn:definition} in the parameter regimes $\nn=0, 1\leq n \leq \nmax$ and $1\leq \nn < \nmax, r(\nn) < n \leq\nmax$, and the portions of $\RR_{q+1}^{\nn}$, which are defined in the last four lines of \eqref{eq:idiing:RRqnn}.  To estimate these error terms, we will first analyze a single term of the form
\begin{align}
 & \left(\divH + \divR \right) \Bigg{(} \sum_{\xi,i,j,k,\vecl} \nabla a_{(\xi)}^2\nabla\Phi_{(i,k)}^{-1}  \LPqn (\WW_{\xi,q+1,\nn}\otimes\WW_{\xi,q+1,\nn})(\Phi_{(i,k)}) \nabla\Phi_{(i,k)}^{-T}\notag\\
    &\qquad\qquad\qquad + \sum_{\xi,i,j,k,\vecl} a_{(\xi)}^2(\nabla\Phiik^{-1})_{\theta}^\alpha \LPqn (\WW_{\xi,q+1,\nn}^\theta\WW_{\xi,q+1,\nn}^\gamma)(\Phiik) \partial_\alpha(\nabla\Phiik^{-1})_{\gamma}^\bullet \Bigg{)} \notag\\
    &\qquad =: \mathcal{O}_{n,\nn} + \mathcal{O}_{n,\nn}^* \, , \label{eq:type:1:general}
\end{align}
where $\bullet$ refers to the unspecified components of a vector field, and superscripts on $\WW_{\xi,q+1,\nn}$ refer to components of vectors over which summation is performed. In the above display, we allow $0\leq \nn \leq \nmax$ and $r(\nn) < n \leq \nmax+1$, thus including both $H_{q,n}^{\nn}$ from \eqref{eq:Hqnpnn:definition} and all Type 1 error terms in \eqref{eq:idiing:RRqnn}. 

\begin{lemma}\label{lem:oscillation:general:estimate}
The terms $\mathcal{O}_{n,\nn}$ and $\mathcal{O}_{n,\nn}^*$ defined in \eqref{eq:type:1:general} satisfy the following estimates.
\begin{enumerate}[(1)]
\item\label{item:Onpnp:1}  For all error terms $\mathcal{O}_{n,\nn}^*$, which are the outputs of $\divR$, we have for all $N,M\leq 3\NindLarge$ that
\begin{equation}\label{eq:Onpnp:estimate:1}
\left\| D^N \Dtq^M \mathcal{O}_{n,\nn}^* \right\|_{L^\infty} \leq {\delta_{q+2}} \lambda_{q+1}^{N-1} \tau_q^{-M} \, .
\end{equation}
\item\label{item:Onpnp:2}  For $0 \leq \nn\leq \nmax$ and $n=\nmax+1$, the high frequency, local part of the Type 1 errors satisfies
\begin{subequations}
\begin{align}
\left\| D^N \Dtq^M \mathcal{O}_{\nmax+1,\nn} \right\|_{L^1\left(\supp\psi_{i,q}\right)} &\lesssim \Gamma_{q+1}^{\shaq-1}  \delta_{q+2} \lambda_{q+1}^N \MM{M, \Nindt, \tau_q^{-1}\Gamma_{q+1}^{i-\cstarnn+4}, \Gamma_{q+1}^{-1}\tilde\tau_q^{-1}}\label{eq:Onpnp:estimate:2} \\
\left\| D^N \Dtq^M \mathcal{O}_{\nmax+1,\nn} \right\|_{L^\infty\left(\supp\psi_{i,q}\right)} 
 & \lesssim \Gamma_{q+1}^{\badshaq-1} \lambda_{q+1}^N \MM{M, \Nindt, \tau_q^{-1}\Gamma_{q+1}^{i-\cstarnn+4}, \Gamma_{q+1}^{-1}\tilde\tau_q^{-1}}\label{eq:Onpnp:estimate:2:new}
\end{align}
\end{subequations}
for all $N,M\leq 3\NindLarge$.
\item\label{item:Onpnp:3} For $0\leq \nn < \nmax$ and $ r(\nn) < n \leq\nmax$, the medium frequency, local part of the Type 1 errors satisfies
\begin{subequations}
\begin{align}
\left\| D^N \Dtq^M \mathcal{O}_{n,\nn} \right\|_{L^1(\supp \psi_{i,q})} &\lesssim \delta_{q+1,n} \lambda\qn^N \MM{M, \Nindt, \tau_q^{-1}\Gamma_{q+1}^{i-\cstarnn+4}, \Gamma_{q+1}^{-1}\tilde\tau_q^{-1}}    \label{eq:Onpnp:estimate:3} \\
\left\| D^N \Dtq^M \mathcal{O}_{n,\nn} \right\|_{L^\infty(\supp \psi_{i,q} )}
&\lesssim \Gamma_q^{\badshaq}\Gamma_{q+1}^{14\Upsilon(n)} \lambda\qn^N \MM{M, \Nindt, \tau_q^{-1}\Gamma_{q+1}^{i-\cstarnn+4}, \Gamma_{q+1}^{-1}\tilde\tau_q^{-1}} \label{eq:Onpnp:estimate:3:new}
\end{align}
\end{subequations}
for all $N+M \leq \Nfn$.
\end{enumerate}
\end{lemma}
\begin{remark}\label{rem:oscillation:general}
In order to verify \eqref{e:inductive:n:2:Hstress} for $n'=\nn$ and $r(\nn)<n\leq \nmax$, we first note that $\mathcal{O}_{n,\nn}=\HH_{q,n}^{\nn}$, and the inequality $  \Gamma_{q+1}^{i-\cstarnn+4}\leq \Gamma_{q+1}^{i-\cstarn} \, ,  $ holds from $\nn \leq n-1$ and \eqref{def:cstarn:formula}.  Then \eqref{eq:Onpnp:estimate:3} provides the desired bound. \eqref{e:inductive:n:1:Hstress:unif} follows similarly from \eqref{eq:Onpnp:estimate:3:new}. The bound in \eqref{e:inductive:n:2:Rstress} follows from \eqref{eq:Onpnp:estimate:1} and \eqref{eq:Onpnp:estimate:2}, since $\cstarnn\geq 4$ from \eqref{eq:cstar:DEF} and \eqref{def:cstarn:formula}.  The bound in \eqref{e:inductive:n:1:Rstress:unif} follows from \eqref{eq:Onpnp:estimate:1} and \eqref{eq:Onpnp:estimate:2:new}. Lastly, when $\nn=\nmax$, and hence $n=\nmax+1$, \eqref{eq:Onpnp:estimate:1}, \eqref{eq:Onpnp:estimate:2}, and \eqref{eq:Onpnp:estimate:2:new} match \eqref{e:main:inductive:q:stress} and \eqref{e:main:inductive:q:stress:unif}.
\end{remark}

\begin{proof}[Proof of Lemma~\ref{lem:oscillation:general:estimate}]
We use \eqref{item:pipe:1} from Proposition~\ref{prop:pipeconstruction} and the notation $A=(\nabla\Phi)^{-1}$ to rewrite \eqref{eq:type:1:general} as
\begin{align}
&\left(\divH + \divR \right) \biggl( \sum_{\; \xi,i,j,k, \vecl} \LPqn \left( \left(\varrho_{\xi,\lambda_{q+1},r_{q+1,\nn}}\right)^2 \right)(\Phi_{(i,k)}) \xi^\theta\xi^\gamma \left(A_{(i,k)}\right)_{\theta}^\alpha \left(  \partial_\alpha a_{(\xi)}^2 \left(A_{(i,k)}\right)_{\gamma}^\bullet + a_{(\xi)}^2 \partial_\alpha\left( A_{(i,k)} \right)_{\gamma}^\bullet \right) \biggr) \, . \notag
\end{align}
Next, we must identify the functions and the values of the parameters which will be used in the application of Proposition~\ref{prop:intermittent:inverse:div}.  We first address the bounds required in \eqref{eq:inverse:div:DN:G}, \eqref{eq:DDpsi}, and \eqref{eq:DDv}, which we can treat simultaneously for items \eqref{item:Onpnp:1}, \eqref{item:Onpnp:2}, and \eqref{item:Onpnp:3}.  Afterwards, we split the proof into two parts. First, we set $n=\nmax+1$ and prove \eqref{eq:Onpnp:estimate:1}, \eqref{eq:Onpnp:estimate:2}, and \eqref{eq:Onpnp:estimate:2:new} for any value of $\nn$.  Next, we consider $0\leq \nn <\nmax$ and $r(\nn)<n\leq\nmax$ and prove \eqref{eq:Onpnp:estimate:1} in the remaining cases, as we simultaneously prove \eqref{eq:Onpnp:estimate:3} and \eqref{eq:Onpnp:estimate:3:new}.

Returning to \eqref{eq:inverse:div:DN:G}, we will verify that this inequality holds with  $v=\vlq$, $D_t=\Dtq=\partial_t + \vlq \cdot \nabla$, and $\displaystyle N_* = M_* = \lfloor \sfrac{\Nsharp}{2} \rfloor$, where $\Nsharp=\Nfnn-\NcutSmall-\NcutLarge-5$.  In order to verify the assumption $N_*- \dpot \geq 2\Ndec + 4$, we use that $\Ndec$ and $\dpot$ satisfy \eqref{eq:lambdaqn:identity:3}. We fix values of $(i,j,k,\nn,\xi,\vec{l})$ and set
\begin{align}
    G^\bullet = \xi^\theta \xi^\gamma \left(A_{(i,k)}\right)_\theta^{\alpha} \left( \partial_\alpha a_{(\xi)}^2  \left(A_{(i,k)}\right)_{\gamma}^\bullet + a_{(\xi)}^2  \partial_\alpha\left(A_{(i,k)}\right)_{\gamma}^\bullet \right) 
    \, . 
    \label{eq:f:zeta:1}
\end{align}
Note crucially that the differential operator falling on $a_{\xi,i,j,k,q,\nn,\vecl}^2$ in the first term is precisely $\xi^\theta \left(A_{(i,k)}\right)_\theta^{\alpha} \partial_\alpha$, which from \eqref{eq:checkerboard:derivatives} and \eqref{e:a_master_est_p} will obey a good bound. We now establish \eqref{eq:inverse:div:DN:G}--\eqref{eq:DDv} with the parameter choices
\begin{align}
\const_{G,1} &= |\supp (\eta_{i,j,k,q,\nn,\xi,\vec{l}})| \delta_{q+1,\nn}\lambda\qnn \Gamma_{q+1}^{2j+5} \, , \qquad \qquad 
 \const_{G,\infty} = \Gamma_q^{\badshaq}\Gamma_{q+1}^{14\Upsilon(\nn)+8}\lambda\qnn
\, , 
\label{eq:crazy:const:G}
\end{align}
$\lambda=\lambda_{q+1}r_{q+1,\nn}\Gamma_{q+1}^{-1}$, $M_t=\Nindt$, $\nu=\tau_q^{-1}\Gamma_{q+1}^{i-\cstarnn+4}$, $\tilde\nu=\tilde\tau_q^{-1}\Gamma_{q+1}^{-1}$, and $\lambda' = \tilde\lambda_q$.

To establish an $L^1$ bound for the first term from \eqref{eq:f:zeta:1}, we appeal to Lemma~\ref{lem:a_master_est_p}, estimate \eqref{e:a_master_est_p} with $N'=1$, and \eqref{eq:Lagrangian:Jacobian:6} to deduce that
\begin{align}
&\Biggl\| D^N \Dtq^M \left( \xi^\theta \left(A_{(i,k)}\right)_{\theta}^{\alpha}  \partial_\alpha a_{(\xi)}^2 \left(A_{(i,k)}\right)_{\gamma}^\bullet \xi^\gamma \right) \Biggr\|_{L^1} \notag\\
&\qquad \lessg | \supp(\eta_{i,j,k,q,\nn,\xi,\vec{l}}) | \delta_{q+1,\nn}\lambda\qnn \Gamma_{q+1}^{2j+5} \left(\Gamma_{q+1}^{-1}\lambda_{q+1}r_{q+1,\nn}\right)^{N } \MM{M,\Nindt,\tau_q^{-1}\Gamma_{q+1}^{i-\cstarnn+4},\tilde\tau_q^{-1}\Gamma_{q+1}^{-1}}
\label{eq:G:estimate:1}
\end{align}
holds for all $N,M\leq \lfloor \sfrac{1}{2}\left( \Nfnn-\NcutSmall-\NcutLarge-5 \right) \rfloor$. It is precisely at this point that we have used that the differential operator $\xi^\theta (A_{(i,k)})_\theta^\alpha\partial_\alpha$ costs only $\lambda\qnn\Gamma_{q+1}$. For the $L^\infty$ bound on the same term, we argue similarly except we apply estimate~\eqref{e:a_master_est_p_uniform} to obtain
\begin{align}
&\Biggl\| D^N \Dtq^M \left( \xi^\theta \left(A_{(i,k)}\right)_{\theta}^{\alpha}  \partial_\alpha a_{(\xi)}^2 \left(A_{(i,k)}\right)_{\gamma}^\bullet \xi^\gamma \right) \Biggr\|_{L^\infty} \notag\\
&\qquad \lessg \Gamma_q^{\badshaq}\Gamma_{q+1}^{14\Upsilon(\nn)+8}\lambda\qnn \left(\Gamma_{q+1}^{-1}\lambda_{q+1}r_{q+1,\nn}\right)^{N} \MM{M,\Nindt,\tau_q^{-1}\Gamma_{q+1}^{i-\cstarnn+4},\tilde\tau_q^{-1}\Gamma_{q+1}^{-1}} \, .
\label{eq:G:estimate:1:1}
\end{align}
For the second term from \eqref{eq:f:zeta:1}, we can appeal to \eqref{eq:Lagrangian:Jacobian:6} and use that $\tilde\lambda_q\leq\lambda\qnn$ for all $\nn$ to deduce that for $N,M\leq \lfloor \sfrac{1}{2}\left( \Nfnn-\NcutSmall-\NcutLarge-5 \right) \rfloor $, we have
\begin{align*}
&\left\|  D^N \Dtq^M  \partial_\alpha \left( A_{(i,k)} \right)_{\gamma}^\bullet \right\|_{L^\infty(\supp\psi_{i,q}\tilde\chi_{i,k,q})}  \lesssim \lambda\qnn^{N+1} \MM{M,\Nindt,\tau_q^{-1}\Gamma_{q+1}^{i-\cstar+1},\tilde\tau_q^{-1}\Gamma_{q+1}^{-1}}.
\end{align*}
Combining this with Lemma~\ref{lem:a_master_est_p}, estimate \eqref{e:a_master_est_p} in the case $p=1$ and \eqref{e:a_master_est_p_uniform} in the case $p=\infty$ produces identical bounds as for the first term and in the range
$N,M\leq \lfloor \sfrac{1}{2}\left( \Nfnn-\NcutSmall-\NcutLarge-5 \right) \rfloor $. Adding both estimates together shows that \eqref{eq:inverse:div:DN:G} has been satisfied for both $p=1,\infty$.

We set the flow in Proposition~\ref{prop:intermittent:inverse:div} as  $\Phi=\Phi_{i,k}$, which by definition satisfies $\Dtq \Phi_{i,k}=0$. Appealing to \eqref{eq:Lagrangian:Jacobian:2} and \eqref{eq:Lagrangian:Jacobian:7}, we have that \eqref{eq:DDpsi} is satisfied. From \eqref{eq:nasty:D:vq:old} at level $q$, which follows from Proposition~\ref{prop:no:proofs}, the choice of $\nu$ from earlier, and \eqref{eq:Lambda:q:x:1:NEW}, we have that $Dv = D \vlq$ satisfies the bound \eqref{eq:DDv}.

{\bf Proof of items \eqref{item:Onpnp:1} and \eqref{item:Onpnp:2} for $0\leq\nn\leq\nmax$ and $n=\nmax+1$.\,}
We first assume that $\nn<\nmax$. With the goal of verifying~\eqref{item:inverse:i}--\eqref{item:inverse:iii} of Proposition~\ref{prop:intermittent:inverse:div}, we choose $\zeta,\mu,\Lambda,\rho$ and $\varrho$ as 
\begin{align}
\zeta =\lambda_{q,\nmax}\,,
\qquad \mu&=\lambda_{q+1}r_{q+1,\nn}\,,
\qquad \Lambda=\lambda_{q+1}\,, \notag \\
\varrho = \LPqnmax \left(\left(\varrho_{\xi,\lambda_{q+1},r_{q+1,\nn}}\right)^2 \right) \, , &\qquad 
\vartheta = \lambda_{q,\nmax}^{2\dpot}\Delta^{-\dpot} \LPqnmax\left( \varrho^2_{\xi,\lambda_{q+1},r_{q+1,\nn}}\right) \, , \label{eq:case:1:1:1}
\end{align}
where we recall that $\varrho_{\xi,\lambda_{q+1},r_{q+1,\nn}}$ is defined in~Propositions~\ref{prop:pipe:shifted} and~\ref{prop:pipeconstruction}. We then have by definition that \eqref{item:inverse:i} from Proposition~\ref{prop:intermittent:inverse:div} is satisfied. By property~\eqref{item:point:1} of Proposition~\ref{prop:pipe:shifted}, we have that the functions $\varrho$ and $\vartheta$ defined in \eqref{eq:case:1:1:1} are both periodic to scale $\left(\lambda_{q+1}r_{q+1,\nn}\right)^{-1}$, and so \eqref{item:inverse:ii} is satisfied. In the case $p=1$, the estimates in \eqref{eq:DN:Mikado:density} follow with $\const_{*,1}=1$ from standard Littlewood-Paley arguments (see also the discussion in part (b) of \cite[Remark~A.21]{BMNV21}) and item \eqref{item:pipe:5} from Proposition~\ref{prop:pipeconstruction}. In the case $p=\infty$, the estimates follow from Lemma~\ref{lem:tricky:tricky}, \eqref{eq:tricky:bounds:2} with the choices $\const_{*,\infty}=r_{q+1,\nn}^{-2}$, $\lambda_1=\lambda_{q,\nmax},\lambda_2=\infty,\lambda=\lambda_{q+1},r=r_{q+1,\nn}$. We recall from \eqref{eq:alpha:equation:1} the choice of $\alpha=\varepsilon_\Gamma \frac{b-1}{b}$, so that the loss $\lambda_{q+1}^{\alpha}$ gives exactly a loss of $\Gamma_{q+1}$. From \eqref{eq:tilde:lambda:q:def}, \eqref{eq:rqn:perp:definition}, and the temporary assumption that $\nn < \nmax$, we have that
$$  \tilde\lambda_q \ll \lambda_{q+1}r_{q+1,\nn}\Gamma_{q+1}^{-1} \ll \lambda_{q+1}r_{q+1,\nn} \leq \lambda_{q,\nmax} \leq \lambda_{q+1}, $$
and so \eqref{eq:inverse:div:parameters:0} is satisfied. From \eqref{eq:lambdaqn:identity:2} we have that
$$  \lambda_{q+1}^4 \leq \left( \frac{\lambda_{q+1}r_{q+1,\nn}}{2\pi \sqrt{3} \Gamma_{q+1}^{-1}\lambda_{q+1}r_{q+1,\nn}} \right)^\Ndec = \left( \frac{\Gamma_{q+1}}{2\pi\sqrt{3}} \right)^\Ndec \, , $$
and so \eqref{eq:inverse:div:parameters:1} is satisfied. Applying the estimate \eqref{eq:inverse:div:stress:1} for $p=1$ with $\alpha$ as in  \eqref{eq:alpha:equation:1}, recalling the value for $\const_{G,1}$ in \eqref{eq:crazy:const:G}, summing over $i$ and using \eqref{eq:inductive:partition} at level $q$, summing over $j,k,\xi$, summing over $\vec{l}$ and using \eqref{item:lebesgue:1} with $r_1=\infty$ and $r_2=2$, and appealing to \eqref{eq:crazy:const:G:ineq} and \eqref{eq:hopeless:mess:new}, we obtain that for $N,M\leq \lfloor \sfrac{1}{2}\left(\Nfnn-\NcutSmall-\NcutLarge-5\right) \rfloor - \dpot$,
\begin{align}
\left\| D^N \Dtq^M  \mathcal{O}_{n,\nn} \right\|_{L^1\left(\supp\psi_{i,q}\right)}  &\lesssim  \lambda_{q+1}^{N} \delta_{q+1,\nn} \lambda\qnn \Gamma_{q+1}^{8} \lambda_{q,\nmax}^{-1} \MM{M,\Nindt,\tau_q^{-1}\Gamma_{q+1}^{i-\cstarnn+4},\tilde\tau_q^{-1}\Gamma_{q+1}^{-1}}\notag\\
&\lessg \Gamma_{q+1}^{\shaq-1} \delta_{q+2} \lambda_{q+1}^N \MM{M,\NindSmall,\tau_q^{-1}\Gamma_{q+1}^{i-\cstarnn+4}, \tilde\tau_q^{-1}\Gamma_{q+1}^{-1}} \,.
 \label{eq:H:estimate:1}
\end{align}
Applying the same steps but in the case $p=\infty$ and using the parameter inequality \eqref{eq:clickity:clackity} yields the bound
\begin{align}
\left\| D^N \Dtq^M  \mathcal{O}_{n,\nn} \right\|_{L^\infty\left(\supp\psi_{i,q}\right)} &\lesssim \Gamma_q^{\badshaq}\Gamma_{q+1}^{14\Upsilon(\nn)+9}\lambda\qnn 
r_{q+1,\nn}^{-2}  \lambda_{q,\nmax}^{-1}
\lambda_{q+1}^{N}   \MM{M,\Nindt,\tau_q^{-1}\Gamma_{q+1}^{i-\cstarnn+4},\tilde\tau_q^{-1}\Gamma_{q+1}^{-1}}\notag\\
& \lessg \Gamma_{q+1}^{\badshaq-2} \lambda_{q+1}^N \MM{M,\NindSmall,\tau_q^{-1}\Gamma_{q+1}^{i-\cstarnn+4}, \tilde\tau_q^{-1}\Gamma_{q+1}^{-1}} \,.
 \label{eq:H:estimate:1:1}
\end{align}
in the same range of $N,M$. The proof is complete after using \eqref{eq:nfnn:mess}, which gives that the range of derivatives allowed in \eqref{eq:H:estimate:1} and \eqref{eq:H:estimate:1:1} is as much as is needed in \eqref{eq:Onpnp:estimate:2}.

Following the parameter choices in \cite[Remark~A.19]{BMNV21}, we set $N_\circ=M_\circ=3\NindLarge$, and $\Nsharp = \Nfnn-\NcutSmall-\NcutLarge-5$.  From \eqref{eq:nfnn:mess:2}, we have that the condition $N_\circ \leq \sfrac{\Nsharp}{4}$ is satisfied. 
The inequalities \eqref{eq:inverse:div:v:global} and \eqref{eq:inverse:div:v:global:parameters} follow from the discussion in \cite[Remark~A.19]{BMNV21}.  The inequality in \eqref{eq:riots:4} follows from the choices $\lambda = \lambda_{q+1}r_{q+1,\nn}\Gamma_{q+1}^{-1}$, $\zeta=\lambda_{q,\nmax}  \geq \lambda_{q+1}r_{q+1,\nn}\Gamma_{q+1}^{-1}$, \eqref{eq:Lambda:q:t:1}, and \eqref{eq:CF:new}.  Having satisfied these assumptions, we may now appeal to estimate \eqref{eq:inverse:div:error:stress:bound} for $p=\infty$ and sum over all parameters $(i,j,k,\xi,\vec{l})$. Since $\vec{l}$ takes at most $\lambda_{q+1}^3$ values, $i$, and $j$ are bounded independently of $q$, and $k$ corresponds to a partition of unity in time, we obtain \eqref{eq:Onpnp:estimate:1} for the case $\nn<\nmax$ and $n=\nmax+1$.

Recall that we began this case with the temporary assumption that $\nn<\nmax$.  In the case $\nn=\nmax$, we have from \eqref{eq:rqn:perp:definition} that $\lambda_{q+1}r_{q+1,\nmax}>\lambda_{q,\nmax}$. Then we can set $\zeta=\mu=\lambda_{q+1}r_{q+1,\nmax}$ and substitute $\mathbb{P}_{\geq \lambda_{q+1}r_{q+1,\nn}}$ for $\mathbb{P}_{[q,\nmax]}$.  The only change is that \eqref{eq:H:estimate:1} and \eqref{eq:H:estimate:1:1} become stronger, since $\lambda_{q,\nmax} < \lambda_{q+1}r_{q+1,\nmax}$, and so the desired estimates follow by arguing as before.  We omit further details.  

{\bf Proof of item~\eqref{item:Onpnp:3} and of item \eqref{item:Onpnp:1} when $0\leq \nn< \nmax$ and $r(\nn)<n\leq\nmax$.\,}
We set 
\begin{align}
\zeta &= 
\begin{dcases} \max\left\{\lambda_{q+1}r_{q+1,\nn},\lambda_{q,n-1}\right\} & \mbox{if} \quad 2 \leq n \leq \nmax \\
\lambda_{q+1}^{\sfrac 12}\lambda_q^{\sfrac 12}\Gamma_{q+1} &\mbox{if} \quad n=1 \, ,
\end{dcases}
\qquad  \mu=\lambda_{q+1}r_{q+1,\nn}\,, \qquad \Lambda=\lambda\qn \, , \label{eq:case:1:1:redux}
\end{align}
and
\begin{align}
\varrho &= \LPqn \left(\left(\varrho_{\xi,\lambda_{q+1},r_{q+1,\nn}}\right)^2 \right) \, , \qquad
\vartheta = \zeta^{2\dpot}\Delta^{-\dpot} \LPqn \left( \varrho^2_{\xi,\lambda_{q+1},r_{q+1,\nn}}\right)
\, . \notag
\end{align}
We then have by definition that \eqref{item:inverse:i} from Proposition~\ref{prop:intermittent:inverse:div} is satisfied.  By property~\eqref{item:point:1} of Proposition~\ref{prop:pipe:shifted}, $\varrho$ and $\vartheta$ are both periodic to scale $\left(\lambda_{q+1}r_{q+1,\nn}\right)^{-1}$, and so \eqref{item:inverse:ii} is satisfied.  The estimates in \eqref{eq:DN:Mikado:density} follow with $\const_{*,1}=1$ in the case $p=1$ as before. In the case $p=\infty$, we appeal to Lemma~\ref{lem:tricky:tricky}, \eqref{eq:tricky:bounds:2} with $\lambda_1=\zeta$, $\lambda_2=\lambda\qn=\Lambda<\lambda_{q+1}$, and $r=r_{q+1,\nn}$ to deduce that \eqref{eq:DN:Mikado:density} holds with $\const_{*,\infty}=\left(\frac{\lambda\qn}{\lambda_{q+1}r_{q+1,\nn}}\right)^2$. We again set $\alpha$ as in \eqref{eq:alpha:equation:1}. From \eqref{eq:rqn:perp:definition} and the condition that $r(\nn)<n$, we have that if $n \neq 1$, then
$$  
\tilde\lambda_q \leq \lambda_{q+1}r_{q+1,\nn}\Gamma_{q+1}^{-1} \ll \lambda_{q+1}r_{q+1,\nn} \leq \max\left\{\lambda_{q+1}r_{q+1,\nn},\lambda_{q,n-1}\right\} \leq \lambda_{q,n} \, ,
$$
and so \eqref{eq:inverse:div:parameters:0} is satisfied if $n\neq 1$. If $n=1$, then it must be the case that $\nn=0$, and so
$$  \tilde\lambda_q \leq \lambda_{q+1}r_{q+1,0}\Gamma_{q+1}^{-1} \ll \lambda_{q+1}r_{q+1,0} \leq \lambda_{q,1} \, . $$
From \eqref{eq:lambdaqn:identity:2}, the inequality $\lambda\qn\leq\lambda_{q+1}$, and the choices of $\mu$ and $\lambda$, we have that \eqref{eq:inverse:div:parameters:1} is satisfied. 

We now use the definition of $\const_{G,p}$ in \eqref{eq:crazy:const:G} and apply the estimate \eqref{eq:inverse:div:stress:1}.  In the case that $p=1$ and $n=1$, then we must have $\nn=0$, and so for all $N,M\leq \lfloor \sfrac{1}{2}\left(\Nfnn-\NcutSmall-\NcutLarge-5\right) \rfloor - \dpot$, we sum over $(i,j,k,\xi,\vec{l})$ as before and obtain that
\begin{align}
\left\| D^N \Dtq^M \mathcal{O}_{0,1} \right\|_{L^1\left(\supp\psi_{i,q}\right)} 
&\lesssim \Gamma_{q}^{\shaq} \delta_{q+1} \tilde\lambda_q \Gamma_{q+1}^{9} \bigl(\lambda_{q}^{\sfrac 12}\lambda_{q+1}^{\sfrac 12}\Gamma_{q+1}\bigr)^{-1} \lambda_{q,1}^{N} \MM{M,\Nindt,\tau_q^{-1}\Gamma_{q+1}^{i-\cstarzero+4},\tilde\tau_q^{-1}\Gamma_{q+1}^{-1}} \notag \\
 &\lessg \delta_{q+1,1} \lambda_{q,1}^N \MM{M,\NindSmall,\tau_q^{-1}\Gamma_{q+1}^{i-\cstarzero+4}, \tilde\tau_q^{-1}\Gamma_{q+1}^{-1}} \, . \label{eq:H:estimate:1:redux}
\end{align}
The inequality in the last line follows immediately from the definitions in \eqref{eq:delta:appendix:def}. Alternatively, if $n>1$, then $\zeta^{-1} \leq \lambda_{q,n-1}^{-1}$ from \eqref{eq:case:1:1:redux}, and so if $N,M\leq \lfloor \sfrac{1}{2}\left(\Nfnn-\NcutSmall-\NcutLarge-5\right) \rfloor - \dpot$,
\begin{align}
\left\| D^N \Dtq^M \mathcal{O}_{\nn, n} \right\|_{L^1\left(\supp\psi_{i,q}\right)} & \lesssim \delta_{q+1,\nn} \lambda\qnn \Gamma_{q+1}^{8} \lambda_{q,n-1}^{-1} \lambda\qn^{N} \MM{M,\Nindt,\tau_q^{-1}\Gamma_{q+1}^{i-\cstarnn+4},\tilde\tau_q^{-1}\Gamma_{q+1}^{-1}} \notag \\
& \lessg \delta_{q+1,n} \lambda\qn^N \MM{M,\NindSmall,\tau_q^{-1}\Gamma_{q+1}^{i-\cstarnn+4}, \tilde\tau_q^{-1}\Gamma_{q+1}^{-1}} \, . \label{eq:H:estimate:1:redux:redux}
\end{align}
In the last inequality, we have used \eqref{ineq:click:clack:1}.  After using \eqref{eq:nfnn:nfn:mess}, which gives   $\lfloor \sfrac{1}{2}\left(\Nfnn-\NcutSmall-\NcutLarge-5\right) \rfloor - \dpot \geq \Nfn$ for all $\nn<n$, we have achieved \eqref{eq:Onpnp:estimate:3}. 

In the case that $p=\infty$ and $n=1$, then we must have that $\nn=0$, and so 
\begin{align}
\left\| D^N \Dtq^M \mathcal{O}_{0,1} \right\|_{L^\infty\left(\supp\psi_{i,q}\right)} & \lesssim \Gamma_q^\badshaq\Gamma_{q+1}^{9} \lambda_{q,0}  \left(\frac{\lambda_{q,1}}{\lambda_{q+1}r_{q+1,0}}\right)^2 \left(\lambda_{q+1}\lambda_q \Gamma_{q+1}^2\right)^{-\sfrac 12} \notag\\
&\qquad \qquad \times \lambda_{q,1}^{N} \MM{M,\Nindt,\tau_q^{-1}\Gamma_{q+1}^{i-\cstarzero+4},\tilde\tau_q^{-1}\Gamma_{q+1}^{-1}} \notag \\
& \lessg \Gamma_{q+1}^{9} \Gamma_q^{\badshaq} \lambda_{q,1}^N \MM{M,\NindSmall,\tau_q^{-1}\Gamma_{q+1}^{i-\cstarzero+4}, \tilde\tau_q^{-1}\Gamma_{q+1}^{-1}} \, . \label{eq:H:estimate:1:redux:uniform}
\end{align}
To achieve the last line, we have appealed to \eqref{eq:rqn:perp:definition} and the inequality $\frac{\lambda_{q,1}^2}{\lambda_{q+1}^{\sfrac 32}\lambda_q^{\sfrac 12}\Gamma_{q+1}}<1$, which is immediate from a large choice of $\nmax$.  In the case that $p=\infty$ and $n\geq 2$, we have that
\begin{align}
\left\| D^N \Dtq^M \mathcal{O}_{n,\nn} \right\|_{L^\infty\left(\supp\psi_{i,q}\right)} & \lesssim \Gamma_q^{\badshaq}\Gamma_{q+1}^{14\Upsilon(\nn)+9}\lambda\qnn \left(\frac{\lambda_{q,n}}{\lambda_{q+1}r_{q+1,\nn}}\right)^2 \lambda_{q,n-1}^{-1} \notag\\
&\qquad \qquad \times \lambda_{q,n}^{N} \MM{M,\Nindt,\tau_q^{-1}\Gamma_{q+1}^{i-\cstarnn+4},\tilde\tau_q^{-1}\Gamma_{q+1}^{-1}} \notag \\
&\lesssim \Gamma_q^{\badshaq} \Gamma_{q+1}^{14\Upsilon(\nn)+13}  \frac{\lambda\qn^2}{\lambda_{q+1}\lambda_{q,n-1}} \lambda\qn^N \MM{M,\Nindt,\tau_q^{-1}\Gamma_{q+1}^{i-\cstarnn+4},\tilde\tau_q^{-1}\Gamma_{q+1}^{-1}} \notag\\ 
& \lessg \Gamma_q^{\badshaq} \Gamma_{q+1}^{14\Upsilon(n)} \lambda_{q,n}^N \MM{M,\NindSmall,\tau_q^{-1}\Gamma_{q+1}^{i-\cstarnn+4}, \tilde\tau_q^{-1}\Gamma_{q+1}^{-1}} \, . 
\end{align}
To achieve the second inequality, we have used \eqref{eq:rqn:perp:definition}. To achieve the third inequality, we have used \eqref{eq:theta:click:clack} and \eqref{eq:upsa:ineq}. The estimates above are again valid in the range $N,M\leq \lfloor \sfrac{1}{2}\left(\Nfnn-\NcutSmall-\NcutLarge-5\right) \rfloor - \dpot$, which from \eqref{eq:nfnn:nfn:mess} completes the proof of \eqref{eq:Onpnp:estimate:3:new}.

Following again the parameter choices in \cite[Remark~A.19]{BMNV21}, we set $N_\circ=M_\circ=3\NindLarge$, and $\Nsharp = \Nfnn-\NcutSmall-\NcutLarge-5$.  From \eqref{eq:nfnn:mess:2}, we have that the condition $N_\circ \leq \sfrac{\Nsharp}{4}$ is satisfied.  The inequalities \eqref{eq:inverse:div:v:global} and \eqref{eq:inverse:div:v:global:parameters} follow from the discussion in \cite[Remark~A.19]{BMNV21}.  The inequality in \eqref{eq:riots:4} follows from the choices $\lambda=\lambda_{q+1}r_{q+1,\nn}\Gamma_{q+1}^{-1}$, $\zeta\geq \lambda_{q+1}r_{q+1,\nn}$, \eqref{eq:Lambda:q:t:1}, and \eqref{eq:CF:new}.  We then achieve the concluded estimate in \eqref{eq:inverse:div:error:stress:bound}, which after summing as before gives \eqref{eq:Onpnp:estimate:1} in the remaining cases $0\leq \nn < \nmax$, $r(\nn)<n\leq \nmax$.
\end{proof}

\subsection{Type 2 oscillation errors}
\label{ss:stress:oscillation:2}

In order to show that the Type 2 errors identified in \eqref{nn:overlap:definition} and \eqref{e:split:nn:1} vanish, we will apply Proposition~\ref{prop:disjoint:support:simple:alternate} on the support of a specific cutoff function
$$  \eta =  \eta_{i,j,k,q,n,\xi,\vecl}=\psi_{i,q}\chi_{i,k,q}\omega_{i,j,q,n}\zeta_{q,i,k,n,\xi,\vecl} \,   $$
in order to place pipes parallel to $\xi$ on $\supp \eta$. We first collect several preliminary estimates in the first subsubsection, mainly with the goal of verifying assumption  \eqref{item:pipe:placement:three} from Proposition~\ref{prop:disjoint:support:simple:alternate}, before applying Proposition~\ref{prop:disjoint:support:simple:alternate} in the second.

\subsubsection{Preliminary estimates}

\begin{lemma}[\bf Keeping Track of Overlap]\label{l:overlap}
For every tuple $(i,j,k,n)$, define the index set $\mathcal{I}$ as
\begin{equation}
    \mathcal{I}
    =
    \mathcal{I}(i,j,k,n)
    = 
    \bigl\{
    (\istar,\jstar,q,\nstar)
    \colon 
    \nstar \leq n, \;
    \psi_{i,q}\omega_{i,j,q,n}\chi_{i,k,q} \psi_{\istar,q}\omega_{\istar,\jstar,q,\nstar}\chi_{\istar,\kstar,q}
    \not\equiv 0
    \bigr\}
    \,.
    \notag
\end{equation}
Then, the cardinality of $\mathcal{I}$ is bounded above by $\const_\eta \Gamma_{q+1}$, where $\const_\eta$ depends only on $\nmax$, $\jmax$, and dimensional constants. In particular, $C_\eta$ is independent of $q$.
\end{lemma}
\begin{proof}[Proof of Lemma~\ref{l:overlap}]
The proof proceeds similarly to the proof of \cite[Lemma~8.8]{BMNV21}. In fact it is somewhat simpler, since the parameter $p$ (see \cite[Definition~2.4]{BMNV21}) is no longer part of the scheme, and we are not considering the checkerboard cutoffs $\zeta_{q,i,k,n,\xi,\vecl}$ yet, but will only incorporate them later. We thus give only an idea of the proof. Once $i$ is fixed, we first note that $\psi_{i,q}$ may only overlap with $\psi_{i+1,q}$ and $\psi_{i-1,q}$ from \eqref{eq:inductive:partition} at level $q$.  The factor of $\Gamma_{q+1}$ in the upper bound for the cardinality of $\mathcal{I}$ comes from the fact that the timescale of the $\chi_{i+1,\kstar,q}$'s on the support of $\psi_{i+1,q}$ is faster by a factor of $\Gamma_{q+1}$ than the timescale of the $\chi_{i,k,q}$'s on the support of $\psi_{i,q}$.  Considering then values of $j$ and $n$ introduces a dependence on $\jmax$ and $\nmax$ which is nevertheless independent of $q$.
\end{proof}

\begin{lemma}\label{lem:overlap:1}
Let $(x,t),(y,t)\in\supp\psi_{i,q}$ be such that  $\psi_{i,q}^2(x,t) \geq \sfrac{1}{4}$ and $\psi_{i,q}^2(y,t) \leq \sfrac{1}{8}$.  Then there exists a geometric constant $\const_\ast > 1$ such that
\begin{equation}\label{eq:xy:distance}
|x-y| \geq \const_* \left(\Gamma_q\lambda_q\right)^{-1}.
\end{equation} 
\end{lemma}
For the proof of Lemma~\ref{lem:overlap:1}, we refer to \cite[Lemma 8.9]{BMNV21}.

\begin{lemma}\label{lem:overlap:2}
Consider cutoff functions
\begin{align*}
\eta&:=\eta_{i,j,k,q,n,\xi,\vecl} = \psi_{i,q}\chi_{i,k,q}\omega_{i,j,q,n}\zeta_{q,i,k,n,\xi,\vecl},\\
\eta^*&:= \eta_{\istar,\jstar,\kstar,q,\nstar,\xistar,\vecl^*} = \psi_{\istar,q}\chi_{\istar,\kstar,q}\omega_{\istar,\jstar,q,\nstar} \zeta_{q,\istar,\kstar,\nstar,\xistar,\vecl^*},
\end{align*}
where $(\istar,\jstar,\kstar,\nstar) \in \mathcal{I}(i,j,k,n)$, as defined in Lemma~\ref{l:overlap}. Let $t^*\in\supp\chi_{\istar,\kstar,q}$ be given. Assume furthermore that $\eta \eta^* \not \equiv 0$, which implies that $\zeta_{q,i,k,n,\xi,\vecl}\,\zeta_{q,\istar,\kstar,\nstar,\xistar,\veclstar}\not\equiv 0$. Then there exists a convex set $\Omega:=\Omega(\eta,\eta^*,t^*)\subset\T^3$ with diameter $\lambda_{q,n}^{-1}\Gamma_{q+1}$ such that
\begin{equation}\notag
\bigl(\supp \zeta_{q,i,k,n,\xi,\vecl} \cap \{t=t^*\}\bigr) \subset \Omega \subset \supp \psi_{i\pm,q} \, .
\end{equation}
\end{lemma}
\begin{proof}[Proof of Lemma~\ref{lem:overlap:2}]
Let $(x,t_0)\in\supp\left(\eta\eta^*\right)$.  Then there exists $i'\in\{i-1,i,i+1\}$ such that $\psi_{i',q}^2(x,t_0)\geq\frac{1}{2}$.  Consider the flow $X(x,t)$ originating from $(x,t_0)$.  Then for any $t$ such that $|t-t_0|\leq\tau_q\Gamma_{q+1}^{-i+5+\cstar}$, we can apply Lemma~\ref{lem:dornfelder} to deduce that $\psi_{i',q}^2(t,X(x,t))\geq\frac{1}{4}$.  By the definition of $\chi_{\istar,\kstar,q}$, the fact that $\istar\in\{i-1,i,i+1\}$, the existence of $(x,t_0)\in\supp(\chi_{i,k,q}\chi_{\istar,\kstar,q})$, and the fact that $t^*\in\supp\chi_{\istar,\kstar,q}$, we in particular deduce that $\psi_{i',q}^2(t^*,X(x,t^*))\geq\frac{1}{4}$.  Now, let $y$ be such that 
$$|X(x,t^*)-y|\leq \lambda_{q,n}^{-1}\Gamma_{q+1} \leq \tilde\lambda_q^{-1} <  \const_* (\Gamma_q \lambda_q)^{-1}$$
for $\const_*$ given in \eqref{eq:xy:distance}, where we have used the definition of $\lambda_{q,n}$ in  \eqref{eq:lambda:q:n:def}. Then from Lemma~\ref{lem:overlap:1}, it cannot be the case that $\psi_{i',q}^2(y,t^*)\leq\frac{1}{8}$, and so 
\begin{align}
y\in\supp\psi_{i',q}\cap\{t=t^*\}\subset \supp\psi_{i\pm,q}\cap\{t=t^*\} \, .
\label{eq:overlap:2:2a}
\end{align}
 Since $y$ is arbitrary, we conclude that the ball of radius $\Gamma_{q+1}\lambda_{q,n}^{-1}$ is contained in $\supp\psi_{i\pm,q}\cap\{t=t^*\}$.  We let $\Omega(\eta,\eta^*,t^*)$ to be precisely this ball.  Since $\Dtq\zeta_{q,i,k,n,\xi,\vecl}=0$ and $(x,t_0)\in\supp\zeta_{q,i,k,n,\xi,\vecl}$, we have that $X(x,t^*)\in\supp\zeta_{q,i,k,n,\xi,\vecl} \cap \{ t = t^*\}$.  Then, recalling that the support of $\zeta_{q,i,k,n,\xi,\vecl}$ must obey the diameter bound in \eqref{eq:checkerboard:support} on the support of $\tilde\chi_{i,k,q}$, which contains the support of $\chi_{\istar,\kstar,q}$ by \eqref{eq:tilde:chi:contains},  we conclude that
\begin{equation}\label{eq:overlap:2:2}
\supp \zeta_{q,i,k,n,\xi,\vecl} \cap\{t=t^*\} \subset \Omega \, .
\end{equation}
Combining \eqref{eq:overlap:2:2a} and \eqref{eq:overlap:2:2} concludes the proof of the lemma.
\end{proof}

\begin{lemma}\label{lem:overlap:3}
As in Lemma~\ref{lem:overlap:2}, consider cutoff functions $\eta$ and $\eta^*$ satisfying the conditions from Lemma~\ref{l:overlap} and the assumption $\eta \eta^* \not \equiv 0$. Let $t^*\in\supp\chi_{\istar,\kstar,q}$ be such that $\Phi^*:=\Phiikstar$ is the identity at time $t^*$. Using Lemma~\ref{lem:overlap:2}, define $\Omega:=\Omega(\eta,\eta^*,t^*)$.  Define $\Omega(t):= \Omega(\eta,\eta^*,t^*,t) := X^*(\Omega,t)$, where $X^*$ is the inverse of $\Phi^*$. Then the following conclusions hold.
\begin{enumerate}[(1)]
\item For $t\in\supp\chi_{i,k,q}$,
\begin{equation}\notag
\supp \eta(\cdot,t) \subset \Omega(t) \subset \supp \psi_{i\pm,q} \, .
\end{equation}
\item Let $\WW^* \circ \Phi^* :=\WW_{\xi*,q+1,n^*}^{\istar,\jstar,\kstar,\nstar,\vecl^*} \circ \Phiikstar$ be the intermittent pipe flow supported on $\eta^*$. Then $\WW^*\circ \Phi^*$ satisfies the conclusion of Lemma~\ref{lem:axis:control} on the set $\Omega(t)$ for $t\in\supp \chi_{i,k,q}$. 
\item\label{item:dodgie:three} For $\mathcal{I} = \mathcal{I}(i,j,k,n)$ defined as in Lemma~\ref{l:overlap}, we denote
\begin{equation}\label{eq:counting:pipelets}
\mathsf{P} := \bigcup_{\mathcal{I}} \left( \supp\left(\psi_{\istar,q}\omega_{\istar,\jstar,q,\nstar} \right) \bigcap \left( \bigcup_{\veclstar,\xistar} \supp\left(\zeta_{q,\istar,\kstar,\nstar,\xistar,\veclstar} \WW_{\xistar,q+1,\nstar}^{\istar,\jstar,\kstar,\nstar,\vecl} \circ \Phi_{(\istar,\kstar)} \right) \right) \right) \, ,
\end{equation}
which is precisely the union of the supports of all pipes living on cutoff functions indexed by tuples belonging to $\mathcal{I}$, which are however \emph{not} restricted to the support of their corresponding time cutoffs $\chi_{\istar,\kstar,q}$. Then there exists $\const_P$ such that for any convex set $\Omega'\subset\T^3$ with $\textnormal{diam}(\Omega')\leq (\lambda_{q+1}r_{q+1,n})^{-1}$ and any $t\in\supp\chi_{i,k,q}$, the set $\mathsf{P} \cap \left(\{t\} \times \Omega'\right)$ consists of at most $\const_P \Gamma_{q+1}$ segments of deformed pipes of length $(\lambda_{q+1}r_{q+1,n})^{-1}$.
\end{enumerate}
\end{lemma}
\begin{remark}
The third item simply asserts that at stage $n$, there exists a geometric constant $\const_P$ such that in any $(\sfrac{\T}{\lambda_{q+1}r_{q+1,n}})^3$-periodic cell of diameter approximately $(\lambda_{q+1}r_{q+1,n})^{-1}$, there exist at most $\const_P\Gamma_{q+1}$ segments of deformed pipes of length $(\lambda_{q+1}r_{q+1,n})^{-1}$.  This will later allow us to apply Proposition~\ref{prop:disjoint:support:simple:alternate}. The factor of $\Gamma_{q+1}$ comes from the fact that overlapping time cutoffs $\chi_{i,k,q}$ and $\chi_{i+1,k',q}$ have timescales which differ by a factor of $\Gamma_{q+1}$, and that we have \emph{not} restricted $\WW_{\xistar,q+1,\nstar}^{\istar,\jstar,\kstar,\nstar,\veclstar}$ to the support of its corresponding time cutoff $\chi_{\istar,\kstar,q}$. Notice also that since choosing a shift moves a segment of pipe inside a $(\sfrac{\T}{\lambda_{q+1}r_{q+1,\nn}})^3$-periodic cell but does not increase the \emph{number} of such segments, the conclusion in \eqref{item:dodgie:three} is independent of the choice of placement.  We may thus appeal to it in the next subsection in order to choose a placement.
\end{remark}
\begin{proof}[Proof of Lemma~\ref{lem:overlap:3}]
The statement and proof are quite similar to the proof of \cite[Lemma~8.11]{BMNV21}, and we refer there for the proof of the first two claims.  The only difference is contained in the third claim above, since we have rephrased the way in which we count the number of deformed segments of pipe comprising $\WW^*\circ\Phi^*$ which may overlap with $\supp \eta$. We remind the reader that a single ``segment of deformed pipe" consists of the support of $\WW^*\circ\Phi^*$ restricted to a single (deformed) $(\sfrac{\T^3}{\lambda_{q+1}r_{q+1,n}})^{-1}$-periodic cell.  Then to prove the third claim, we first fix a tuple $(\istar,\jstar,\kstar,\nstar)\in\mathcal{I}$ and note that in any convex set $\Omega'$ of diameter at most $(\lambda_{q+1}r_{q+1,n})^{-1}$, the conclusions of Lemma~\ref{lem:axis:control} and the construction of the checkerboard cutoff functions implies that there exist at most finitely many $\zeta_{q,\istar,\kstar,\nstar,\xistar,\veclstar}$ such that 
$$\psi_{\istar,q}\omega_{\istar,\jstar,q,\nstar}\chi_{\istar,\kstar,q}\zeta_{q,\istar,\kstar,\nstar,\xistar,\veclstar} \not\equiv 0 \, .$$
From the construction of $\WW^*\circ\Phi^*$ in Proposition~\ref{prop:pipeconstruction} and the fact that $\WW^*\circ\Phi^*$ satisfies the conclusions of Lemma~\ref{lem:axis:control} on $\supp\chi_{i,k,q}$, we then have that taking the union over just $\veclstar$ and $\xistar$ in \eqref{eq:counting:pipelets} allows for the desired conclusion with a $q$-independent constant.  Then applying Lemma~\eqref{lem:overlap:1} and taking the union over the $C_\eta\Gamma_{q+1}$ many tuples in $\mathcal{I}$ then provides the conclusion with a new constant $\const_P$ multiplied by $\Gamma_{q+1}$.
\end{proof}

\subsubsection{Applying Proposition~\ref{prop:disjoint:support:simple:alternate}}

\begin{lemma}\label{lem:osc:2}
The Type 2 oscillation errors identified in \eqref{nn:overlap:definition} and \eqref{e:split:nn:1} vanish.
\end{lemma}

\begin{proof}[Proof of Lemma~\ref{lem:osc:2}]
To show that the errors defined in \eqref{nn:overlap:definition} and \eqref{e:split:nn:1} vanish, it suffices to show the following: for any pairs of cutoff functions $\eta=\eta_{i,j,k,q,\nn,\xi,\vecl}$ and $\eta^*=\eta_{\istar,\jstar,\kstar,q,\nstar,\xistar,\vecl^*}$ where $(\istar,\jstar,\nstar,\kstar) \in \mathcal{I}(i,j,k,n)$, we have that
\begin{align}\label{eq:osc:2:toshow}
\eta_{i,j,k,q,\nn,\xi,\vecl} \;   \eta_{\istar,\jstar,\kstar,q,\nstar,\xistar,\vecl^*} \; \Bigl(\WW_{\xi,q+1,\nn}^{i,j,k,\nn,\vecl}\circ\Phiik
\otimes \WW_{\xistar,q+1,\nstar}^{\istar,\jstar,\kstar,\nstar,\vecl^*}\circ\Phiikstar
\Bigr)
\equiv 0 \, .
\end{align}
The proof of this claim will proceed by fixing $\nn$, using the preliminary estimates, and applying Proposition~\ref{prop:disjoint:support:simple:alternate}.  

Now, consider all cutoff functions $\eta_{i,j,k,q,\nn,\xi,\vecl}$ utilized at stage $\nn$. We may choose an ordering of the tuples $(i,j,k,\xi,\vecl)$ at level $\nn$, which  automatically provides orderings for the cutoff functions $\eta_{i,j,k,q,\nn,\xi,\vecl}$ and associated pipe flows $\WW_{\xi,q+1,\nn}^{i,j,k,\nn,\vecl}\circ\Phiik$.  To lighten the notation, we will abbreviate the newly ordered cutoff functions as $\eta_z$ and the associated intermittent pipe flows as $(\WW\circ\Phi)_z$, where $z\in\mathbb{N}$ corresponds to the ordering.  We will apply Proposition~\ref{prop:disjoint:support:simple:alternate} inductively on $z\in\mathbb{N}$, according to the chosen ordering, so that \eqref{eq:osc:2:toshow} holds. 

Fix $\eta_z$, and fix the associated index set $\mathcal{I}(z) = \mathcal{I}(i,j,k,\nn)$. Since we are proving \eqref{eq:osc:2:toshow} iteratively, we only need to consider the elements $z' \in \mathcal{I}(z)$ such that $\nstar < \nn$, and $\hat{z} \in \mathcal{I}(z)$ such that $\nstar = \nn$ and $\hat z < z$, according to the aforementioned ordering.

We will apply Proposition~\ref{prop:disjoint:support:simple:alternate} with the following choices.  First, we recall that at the time $t_z$ at which $\Phi_z$ is the identity, the cutoff function $\eta_z$ contains a checkerboard cutoff function $\zeta_{z}$ which from \eqref{eq:prism:zero} is adapted to a rectangular prism of dimensions $2\pi\lambda\qnn^{-1}$ in the direction of $\xi_z$, and $\const_\Gamma\Gamma_{q+1}(\lambda_{q+1}r_{q+1,\nn})^{-1}$ in the directions perpendicular to $\xi_z$.  Thus we can bound the dimensions of the support of the anistropic checkerboard cutoff by $4\pi\lambda\qnn^{-1}$ and $2\const_\Gamma\Gamma_{q+1}(\lambda_{q+1}r_{q+1,\nn})^{-1}$, and we thus set 
\begin{equation}\notag
\Omega=\supp \zeta_z \cap \{t=t_z\}\, , \qquad  r_1= \frac{\lambda\qnn}{4\pi\lambda_{q+1}} \, , \qquad r_2 =  r_{q+1,\nn} \, , \qquad  \const_\Omega = 2\const_\Gamma .
\end{equation}
Recalling item~\ref{item:dodgie:three} from Lemma~\ref{lem:overlap:3}, we  choose the support of $(\WW\circ\Phi)_z|_{t=t_z}$ to have empty intersection with
\begin{equation}
    \mathsf{P} \cap \Omega \, , \qquad \mathsf{P}\textnormal{ as defined in }\eqref{eq:counting:pipelets} \, ,
\end{equation}
and so by definition $\mathsf{P}$ satisfies item~\ref{item:pipe:placement:three} from Proposition~\ref{prop:disjoint:support:simple:alternate}. Thus it remains to check \eqref{eq:r1:r2:condition:alt}.  From the definition of $r_{q+1,\nn}$ in \eqref{eq:rqn:perp:definition}, we have that
\begin{align}
    C_* \const_\Omega^2 \const_P \Gamma_{q+1}^3 r_{2}^2 = C_* \const_\Omega^2 \const_P \Gamma_{q+1}^3 r_{q+1,\nn}^2 \lesssim C_* \const_\Omega^2 \const_P \Gamma_{q+1}^3 \frac{\lambda\qnn}{\lambda_{q+1}} \Gamma_{q+1}^{-4} < r_1 
\end{align}
if $a$ is chosen sufficiently large so that $\Gamma_{q+1}^{-1}$ can absorb the constants $C_*$, $\const_\Omega^2$, $\const_P$ and the implicit constant, all of which are bounded independently of $q$. Therefore \eqref{eq:r1:r2:condition:alt} is satisfied, and we may apply Proposition~\ref{prop:disjoint:support:simple:alternate} to choose a placement for $\WW_z$ which has empty intersection with $\mathsf{P}$ at time $t=t_z$. This shows that at time $t=t_z$, $\WW_z$ has empty intersection with \emph{all} previously existing pipes which may be non-zero at \emph{any} time $t\in\supp\chi_{i,k,q}$ but have been flowed to time $t=t_z$. Finally, since $\Dtq(\WW\circ\Phi)_z=\Dtq(\WW\circ\Phi)_{z'}=\Dtq(\WW\circ \Phi)_{\hat{z}}$, and $\mathsf{P}$ has been constructed to contain \emph{all} pipes which are non-zero at \emph{any} time $t\in\supp\chi_{i,k,q}$, \eqref{eq:osc:2:toshow} is satisfied for all $t\in \supp\chi_z$, concluding the proof.
\end{proof}

\subsection{Divergence corrector errors}\label{ss:stress:divergence:correctors}
In this subsection we define and estimate the stress $\mathcal{O}_{\nn, \rm corr}$ written in \eqref{eq:idiing:RRqnn} and arising from the divergence correctors identified in \eqref{e:split:nn:2}, which satisfy
\begin{equation*} 
\div \bigl( w_{q+1,\nn}^{(p)} \otimes w_{q+1,\nn}^{(c)} + w_{q+1,\nn}^{(c)} \otimes w_{q+1,\nn}^{(p)} + w_{q+1,\nn}^{(c)} \otimes w_{q+1,\nn}^{(c)} \bigr) = \div \left( \mathcal{O}_{\nn, \rm corr} \right) \, .
\end{equation*}

\begin{lemma}\label{l:divergence:corrector:error}
For all $0\leq \nn \leq \nmax$, the divergence corrector errors $\mathcal{O}_{\nn,\rm corr}$ satisfy the bounds
\begin{subequations}
\begin{align}
&\left\| \psi_{i,q} D^k \Dtq^m \mathcal{O}_{\nn,\rm corr} \right\|_{L^1} \lesssim \Gamma_{q+1}^{\shaq-1} \delta_{q+2}\lambda_{q+1}^k  \MM{m,\Nindt,\Gamma_{q+1}^{i-\cstarnn + 4} \tau_q^{-1}, \Gamma_{q+1}^{-1}\tilde\tau_q^{-1} }
\label{eq:div:corrector:L1}
\\
&\left\|  D^k \Dtq^m \mathcal{O}_{\nn,\rm corr} \right\|_{L^\infty(\supp \psi_{i,q})} \lesssim 
\Gamma_{q+1}^{\badshaq-1} \lambda_{q+1}^k \MM{m,\Nindt,\Gamma_{q+1}^{i-\cstarnn + 4} \tau_q^{-1}, \Gamma_{q+1}^{-1}\tilde\tau_q^{-1} }
    \label{eq:div:corrector:Linfty}
\end{align}
\end{subequations}
for all $k,m\leq 3\NindLarge$.
\end{lemma}

\begin{proof}[Proof of Lemma~\ref{l:divergence:corrector:error}]
We first present the estimates for the stress $\mathcal{O}_{\nn,1,3} = w_{q+1,\nn}^{(c)} \otimes w_{q+1,\nn}^{(c)}$, which is also given explicitly by the last line in \eqref{e:split:nn:2} and may be absorbed directly into $\mathcal{O}_{\nn,\rm corr}$ and estimated. By the Leibniz rule, the estimate \eqref{eq:w:oxi:c:est} with $(r,r_1,r_2) = (2,\infty,1)$, and the fact that $\supp \psi_{i,q} \cap \supp \eta_{(i',j',k')}\neq \emptyset$ if and only if $|i'-i|\leq 1$, it follows that 
\begin{align*}
\norm{\psi_{i,q} D^k \Dtq^m \mathcal{O}_{\nn,1,3}}_{L^1}
\les 
r_{q+1,\nn}^2 \delta_{q+1,\nn} \Gamma_{q+1}^{6} 
\lambda_{q+1}^k \MM{m, \NindSmall, \tau_{q}^{-1}\Gamma_{q+1}^{i-\cstarnn + 4}, \tilde\tau_{q}^{-1}\Gamma_{q+1}^{-1}}
\,.
\end{align*}
The bound \eqref{eq:div:corrector:L1} for $\mathcal{O}_{\nn,1,3}$ now follows from the   parameter inequality \eqref{eq:div:cor:ineq:1}.
Similarly, from \eqref{eq:w:oxi:c:unif} it follows that 
\begin{align*}
\norm{D^k \Dtq^m \mathcal{O}_{\nn,1,3}}_{L^\infty(\supp \psi_{i,q})} & \lesssim  \Gamma_{q}^{\badshaq} \Gamma_{q+1}^{14 \Upsilon(\nn) + 7} \lambda_{q+1}^{k}\MM{m,\Nindt, \tau_{q}^{-1}\Gamma_{q+1}^{i-\cstarnn + 4},\tilde\tau_q^{-1}\Gamma_{q+1}^{-1}}
\,.
\end{align*}
The bound \eqref{eq:div:corrector:Linfty} for $\mathcal{O}_{\nn,1,3}$ then follows from the   inequality \eqref{eq:div:cor:ineq:2}.

It thus remains to analyze $\div(w_{q+1,\nn}^{(p)} \otimes w_{q+1,\nn}^{(c)} + w_{q+1,\nn}^{(c)} \otimes w_{q+1,\nn}^{(p)})$. Using the second line of \eqref{e:split:nn:2}, we have
\begin{align}
&\div \bigl( w_{q+1,\nn}^{(p)} \otimes w_{q+1,\nn}^{(c)} + w_{q+1,\nn}^{(c)} \otimes w_{q+1,\nn}^{(p)} \bigr)^{\bullet}
\notag \\
&= \sum_{\xi,i,j,k, \vec{l}}
\partial_m \bigl( a_{(\xi)} \varrho_{(\xi)}\circ \Phi_{(i,k)} \xi^\ell \bigl( A_{\ell}^m  \epsilon_{\bullet p r} + A_{\ell}^\bullet  \epsilon_{m p r} \bigr) \partial_p a_{(\xi)}    \partial_r \Phi_{(i,k)}^s \mathbb{U}_{\xi,q+1,\nn}^s \circ \Phi_{(i,k)} 
\bigr)
\label{eq:nets:suck:1}
\end{align} 
where $\epsilon_{i_1i_2i_3}$ is the Levi-Civita alternating tensor, we implicitly contract the repeated indices $\ell,m,p,r,s$, and the $\bullet$ refers to the indices of the vectors on either side of the above display. The subtle point is that if the derivative in $\partial_p a_{(\xi)}$ is not in a good direction, cf. Lemma~\ref{lem:checkerboard:estimates}, one seemingly obtains the wrong bound. As such we use that $\{\xi,\xi',\xi''\}$ is an orthonormal basis associated with the direction vector $\xi$ with $\xi \times \xi' = \xi''$, and so $\xi^n \xi^\ell + (\xi')^n (\xi')^\ell + (\xi'')^n (\xi'')^\ell = \delta^{n\ell}$, and decompose $\partial_p a_{(\xi)}$ into a sum of vector fields $a_{p,(\xi)}^{\rm good}$ and $a_{p,(\xi)}^{\rm bad}$ defined by
\begin{align}
    \partial_p a_{(\xi)} 
    &=
    \underbrace{\partial_p \Phiik^n \xi^n  \xi^{\ell} A_\ell^j \partial_j a_{(\xi)}}_{=: a_{p,(\xi)}^{\rm good}}
      + 
    \underbrace{\partial_p \Phiik^n (\xi^\prime)^n (\xi^\prime)^{\ell} A_\ell^j \partial_j a_{(\xi)}
    +
     \partial_p \Phiik^n (\xi^{\prime \prime})^n (\xi^{\prime \prime})^\ell A_\ell^j \partial_j a_{(\xi)}}_{=: a_{p,(\xi)}^{\rm bad} }
     \,,
     \label{eq:nets:suck:2}
\end{align}
where we have also set $A = A_{(i,k)} =  (\nabla \Phiik)^{-1}$. Using this decomposition, we note that from Lemma~\ref{lem:a_master_est_p}, the derivative of $a_{(\xi)}$ in the ``good'' term costs a factor of $\lambda_{q,\nn}\Gamma_{q+1}$, whereas the derivatives landing on $a_{(\xi)}$ in the ``bad'' terms cost a factor of $\lambda_{q+1} r_{q+1,\nn}\Gamma_{q+1}^{-1} \gg \lambda_{q,\nn}\Gamma_{q+1}$.

In view of \eqref{e:a_master_est_p} and \eqref{e:a_master_est_p_uniform}, we leave the part of \eqref{eq:nets:suck:1} which contains $a_{p,(\xi)}^{\rm good}$ in divergence form and simply move the resulting symmetric stress 
\begin{align}
 (\mathcal{O}_{\nn,1,2}^{\rm good})^{m \bullet}
 := \sum_{\xi,i,j,k, \vec{l}}
  a_{(\xi)} \varrho_{(\xi)}\circ \Phi_{(i,k)} \xi^\ell \bigl( A_{\ell}^m  \epsilon_{\bullet p r} + A_{\ell}^\bullet  \epsilon_{m p r} \bigr) a_{p,(\xi)}^{\rm good}    \partial_r \Phi_{(i,k)}^s \mathbb{U}_{\xi,q+1,\nn}^s \circ \Phi_{(i,k)} 
  \,,
\end{align}
into $\mathcal{O}_{\nn, \rm corr}$ (and thus $\RR_{q+1}^{\nn}$), up to removing a trace term which is thrown into the pressure. This good part of $\mathcal{O}_{\nn,1,2}$ obeys the same $L^1$ and $L^\infty$ bounds as $\mathcal{O}_{\nn,1,3}$ above.   To see this, we apply the $L^1$ de-correlation estimate from Lemma~\ref{l:slow_fast}, for $p=1$, $f= a_{(\xi)} \xi^\ell (A_{\ell}^m  \epsilon_{\bullet p r} + A_{\ell}^\bullet  \epsilon_{m p r} ) a_{p,(\xi)}^{\rm good}    \partial_r \Phi_{(i,k)}^s$, $\Phi = \Phiik$, $v=\vlq$, and $\varphi = \varrho_{(\xi)} \mathbb{U}_{\xi,q+1,\nn}^s$. In light of Proposition~\ref{prop:pipeconstruction}, Corollary~\ref{cor:deformation}, estimate  \eqref{e:a_master_est_p}, and definition \eqref{eq:nets:suck:2}, we have that the assumptions of Lemma~\ref{l:slow_fast} hold with the parameter choices $\mathcal{C}_f = |\supp \eta_{i,j,k,q,\nn,\xi,\vecl}| \delta_{q+1,\nn} \Gamma_{q+1}^{j+7} \lambda_{q,\nn}$, $\lambda = \Gamma_{q+1}^{-1} r_{q+1,\nn} \lambda_{q+1} $, $\nu= \Gamma_{q+1}^{i - \cstarnn+4} \tau_q^{-1}$, $\tilde \nu = \tilde \tau_q^{-1} \Gamma_{q+1}^{-1}$, $N_t = \Nindt$, $\mu= \lambda_{q+1} r_{q+1,\nn} = \Gamma_{q+1} \lambda$, $\mathcal{C}_{\varphi} = \lambda_{q+1}^{-1}$, $\zeta = \tilde \zeta = \lambda_{q+1}$, $N_x = 0$, and $N_\circ = \Nfnn-\NcutSmall-\NcutLarge-5$. By \eqref{eq:lambdaqn:identity:3} we have that $N_\circ \geq 2 \Ndec +4$, and by \eqref{eq:lambdaqn:identity:2} we have that $\lambda_{q+1}^4 \leq (\Gamma_{q+1} (2\pi \sqrt{3})^{-1})^{\Ndec}$, and so condition \eqref{eq:slow_fast_3} is verified. Thus, from \eqref{eq:slow_fast_5} and summing on $\vecl$ using \eqref{item:lebesgue:1}, we deduce the $L^1$ estimate
\begin{align*}
\bigl\| \psi_{i,q} D^k \Dtq^m \mathcal{O}_{\nn,1,2}^{\rm good}\bigr\|_{L^1}
\les \delta_{q+1,\nn} \Gamma_{q+1}^{7}
  \lambda_{q,\nn} \lambda_{q+1}^{-1}
\lambda_{q+1}^k \MM{m, \NindSmall, \tau_{q}^{-1}\Gamma_{q+1}^{i-\cstarnn + 4}, \tilde\tau_{q}^{-1}\Gamma_{q+1}^{-1}}
\,.
\end{align*}
The bound \eqref{eq:div:corrector:L1} for $\mathcal{O}_{\nn,1,2}^{\rm good}$ now follows from the parameter inequality \eqref{eq:div:cor:ineq:3}, and the fact that $\Nfnn-\NcutSmall-\NcutLarge-5 \geq \max\{ 2 \Nindt +4 + 3\Nindv, 6 \Nindv\}$, which is a consequence of \eqref{eq:lambdaqn:identity:3} and \eqref{eq:nfnn:mess}. Similarly, from Proposition~\ref{prop:pipeconstruction}, Corollary~\ref{cor:deformation}, estimate  \eqref{e:a_master_est_p_uniform}, and definition \eqref{eq:nets:suck:2}, we have the $L^\infty$ estimate
\begin{align*}
\bigl\|D^k \Dtq^m \mathcal{O}_{\nn,1,2}^{\rm good}\bigr\|_{L^\infty(\supp \psi_{i,q})}
\les r_{q+1,\nn}^{-2}
\Gamma_q^{\badshaq} \Gamma_{q+1}^{14 \Upsilon(\nn)+8}   \lambda_{q,\nn} \lambda_{q+1}^{-1}
\lambda_{q+1}^k \MM{m, \NindSmall, \tau_{q}^{-1}\Gamma_{q+1}^{i-\cstarnn + 4}, \tilde\tau_{q}^{-1}\Gamma_{q+1}^{-1}}
\,.
\end{align*}
The bound \eqref{eq:div:corrector:Linfty} for $\mathcal{O}_{\nn,1,2}^{\rm good}$ then follows from the parameter  inequality
\eqref{eq:div:cor:ineq:4}.

Returning to \eqref{eq:nets:suck:1}, it remains to  consider the bad part, coming from the second term in \eqref{eq:nets:suck:2}, namely
\begin{align}
&\sum_{\xi,i,j,k, \vec{l}}
\partial_m \left( a_{(\xi)} \varrho_{(\xi)}\circ \Phi_{(i,k)} \xi^\ell \bigl( A_{\ell}^m  \epsilon_{\bullet p r} + A_{\ell}^\bullet  \epsilon_{m p r} \bigr) a_{p,(\xi)}^{\rm bad}   \partial_r \Phi_{(i,k)}^s \mathbb{U}_{\xi,q+1,\nn}^s \circ \Phi_{(i,k)} 
\right) 
= \mathbf{V}_1^\bullet + \mathbf{V}_2^\bullet
\label{eq:nets:suck:3}
\end{align}
where $\mathbf{V}_1$ corresponds to the term containing $A_{\ell}^m  \epsilon_{\bullet p r}$, and $\mathbf{V}_2$ corresponds to the term containing $A_{\ell}^\bullet  \epsilon_{m p r}$. When we distribute the $\partial_m$ derivative in \eqref{eq:nets:suck:3}, we need to be careful that the derivative does not land on the fast (at frequency $\lambda_{q+1}$) object $\varrho_{(\xi)} \UU_{\xi,q+1,\nn}^s$.

Let us first handle $\mathbf{V}_1$. For this purpose, note that 
\begin{align*}
   \xi^\ell A_{\ell}^m \partial_m \left((\varrho_{(\xi)} \UU_{\xi,q+1,\nn}^s) \circ \Phiik \right) 
   &=  \xi^\ell A_{\ell}^m \partial_m \Phiik^r  \left(\partial_r (\varrho_{(\xi)} \UU_{\xi,q+1,\nn}^s)\right) \circ \Phiik  \\
   &=     \left(\xi^\ell  \partial_\ell (\varrho_{(\xi)} \UU_{\xi,q+1,\nn}^s)\right) \circ \Phiik
   \\
   &= 0
\end{align*}
because $\xi \cdot \nabla$ annihilates both $\varrho_{(\xi)}$ and $\UU_{\xi,q+1,\nn}$, from \eqref{eq:derivative:along:pipe}. 
Thus, by \eqref{eq:nets:suck:3}, the term $\mathbf{V}_1^\bullet$ becomes
\begin{align}
\mathbf{V}_1^\bullet
&= \sum_{\xi,i,j,k, \vec{l}}
\partial_m \bigl( a_{(\xi)}  \xi^\ell A_{\ell}^m  \epsilon_{\bullet p r} a_{p,(\xi)}^{\rm bad}    \partial_r \Phi_{(i,k)}^s \bigr) \left(\varrho_{(\xi)} \mathbb{U}_{\xi,q+1,\nn}^s\right)\circ \Phi_{(i,k)} 
\label{eq:nets:suck:4}
\end{align}
Notice that by the Piola identity, we have $\partial_m  ( a_{(\xi)}  \xi^\ell A_{\ell}^m  \epsilon_{\bullet p r}   a_{p,(\xi)}^{\rm bad}    \partial_r \Phi_{(i,k)}^s  ) = \xi^\ell A_{\ell}^m \partial_m  ( a_{(\xi)}    \epsilon_{\bullet p r}  a_{p,(\xi)}^{\rm bad}    \partial_r \Phi_{(i,k)}^s  )$, and so the slow objects contain a derivative that costs the good factor of $\lambda_{q,\nn}\Gamma_{q+1}$, and a derivative that costs the bad factor of $\lambda_{q+1} r_{q+1,\nn}$.
We then apply the inverse divergence operator $\divH + \divR$ from Proposition~\ref{prop:intermittent:inverse:div}, with the following choices: $p=1$, $G=\xi^\ell A_{\ell}^m \partial_m  ( a_{(\xi)}    \epsilon_{\bullet p r} a_{p,(\xi)}^{\rm bad}    \partial_r \Phi_{(i,k)}^s)$, $\varrho = \varrho_{(\xi)} \mathbb{U}_{\xi,q+1,\nn}^s$, $\Phi = \Phi_{(i,k)}$, $v= \vlq$, and $N_* = M_* = \lfloor \frac 12 (\Nfnn-\NcutSmall-\NcutLarge-5) \rfloor$. By \eqref{eq:nasty:D:vq:old}, Corollary~\ref{cor:deformation}, and estimate  \eqref{e:a_master_est_p}, assumption~\eqref{eq:inverse:div:DN:G}  holds for $\mathcal{C}_G = |\supp(\eta_{i,j,k,q,\nn,\xi,\vecl})| \delta_{q+1,\nn} \Gamma_{q+1}^{j+7} (\Gamma_{q+1}^{-1} \lambda_{q+1} r_{q+1,\nn}) \lambda_{q,\nn}$, $\lambda = \lambda_{q+1} r_{q+1,\nn} \Gamma_{q+1}^{-1}$, $N_t = \Nindt$, $\nu = \tau_{q}^{-1}\Gamma_{q+1}^{i-\cstarnn + 4}$, and $\tilde \nu = \tilde\tau_{q}^{-1}\Gamma_{q+1}^{-1}$, while assumptions~\eqref{eq:DDpsi}--\eqref{eq:DDv} hold with $\lambda' = \tilde \lambda_q$. From Proposition~\ref{prop:pipeconstruction} and standard Littlewood-Paley analysis, upon letting $\zeta = \mu = \lambda_{q+1} r_{q+1,\nn}$, $\vartheta =  (\zeta^{-2} \Delta)^{-\dpot} (\varrho_{(\xi)} \mathbb{U}_{\xi,q+1,\nn}^s)$ (we note that $\varrho_{(\xi)} \mathbb{U}_{\xi,q+1,\nn}^s$ has mean zero from a direct computation using the definition of the intermittent pipe flows from Proposition~\ref{prop:pipeconstruction}), $\Lambda = \lambda_{q+1}$, $\mathcal{C}_* = \lambda_{q+1}^{-1}$, and $\alpha$ as in ~\eqref{eq:alpha:equation:1}, we have that condition~\eqref{eq:DN:Mikado:density} is satisfied.
With these chosen parameters, the condition~\eqref{eq:inverse:div:parameters:0} trivially holds, while condition \eqref{eq:inverse:div:parameters:1} is equivalent to $\lambda_{q+1}^4 \leq (\Gamma_{q+1} (2\pi \sqrt{3})^{-1})^{\Ndec}$, which in this case holds due to \eqref{eq:lambdaqn:identity:2}. Conditions \eqref{eq:inverse:div:v:global}--\eqref{eq:inverse:div:v:global:parameters} are verified for $N_\circ = M_\circ = 3 \Nindv$ and $\mathcal{C}_v = \Gamma_{q+1}^{\imax+1} \delta_{q}^{\sfrac 12} \lambda_q^2 \leq \Gamma_{q+1}^{\badshaq} \Theta_{q}^{\sfrac 12} \lambda_q^2 $, in view of \eqref{vlq}, \eqref{eq:imax:old}, and \eqref{eq:bob:Dq':old}, and \eqref{eq:tilde:lambda:q:def}--\eqref{eq:tilde:tau:q:def}. Lastly, the inequality \eqref{eq:riots:4} holds because $\dpot$ is taken to be sufficiently large to ensure \eqref{eq:CF:new:2}. From \eqref{eq:inverse:div:stress:1}, \eqref{eq:inverse:div:error:stress:bound}, and a sum on $\vecl$ as before, we deduce the $L^1$ bound
\begin{align}
&\norm{\psi_{i,q} D^k \Dtq^m (\divH + \divR)  \mathbf{V}_1}_{L^1} 
\notag \\
&\quad \les 
\delta_{q+1,\nn} \Gamma_{q+1}^{8}
\frac{(\Gamma_{q+1}^{-1} \lambda_{q+1} r_{q+1,\nn}) \lambda_{q,\nn}}{(\lambda_{q+1} r_{q+1,\nn}) \lambda_{q+1}}
\lambda_{q+1}^k \MM{m, \NindSmall, \tau_{q}^{-1}\Gamma_{q+1}^{i-\cstarnn + 4}, \tilde\tau_{q}^{-1}\Gamma_{q+1}^{-1}}
\,.
\label{eq:nets:suck:4:L1}
\end{align}
Since $\delta_{q+1,\nn} \Gamma_{q+1}^{7} \lambda_{q,\nn} \lambda_{q+1}^{-1} \leq \Gamma_{q+1}^{\shaq-1} \delta_{q+2}$ -- see~\eqref{eq:div:cor:ineq:3}, and $\lfloor \frac 12 (\Nfnn-\NcutSmall-\NcutLarge-5) \rfloor - \dpot \geq 3 \Nindv$ -- see~\eqref{eq:nfnn:mess}, the above bound is consistent with \eqref{eq:div:corrector:L1}.

The $L^\infty$ estimate is obtained similarly. We again apply Proposition~\ref{prop:intermittent:inverse:div} with the only parameters that change being: $p=\infty$,  $\mathcal{C}_G = \Gamma_q^\badshaq \Gamma_{q+1}^{14 \Upsilon(\nn)+8} \lambda_{q,\nn} (\Gamma_{q+1}^{-1} \lambda_{q+1}r_{q+1,\nn}$) --   see~\eqref{eq:inverse:div:v:global:parameters}, and  $\mathcal{C}_* = r_{q+1,\nn}^{-2} \lambda_{q+1}^{-1}$ -- see~\eqref{e:pipe:estimates:1} and \eqref{e:pipe:estimates:2}.
From \eqref{eq:inverse:div:stress:1} and \eqref{eq:inverse:div:error:stress:bound} we obtain
\begin{align}
&\norm{D^k \Dtq^m (\divH + \divR)  \mathbf{V}_1}_{L^\infty(\supp \psi_{i,q})}
\notag\\
&\quad \les 
\Gamma_q^{\badshaq} \Gamma_{q+1}^{14 \Upsilon(\nn) + 9}
\frac{(\Gamma_{q+1}^{-1} \lambda_{q+1} r_{q+1,\nn})  \lambda_{q,\nn} r_{q+1,\nn}^{-2}}{(\lambda_{q+1} r_{q+1,\nn}) \lambda_{q+1}}
\lambda_{q+1}^k \MM{m, \NindSmall, \tau_{q}^{-1}\Gamma_{q+1}^{i-\cstarnn + 4}, \tilde\tau_{q}^{-1}\Gamma_{q+1}^{-1}}
\,,
\label{eq:nets:suck:4:Linfty}
\end{align}
Since $\Gamma_q^{\badshaq} \Gamma_{q+1}^{14 \Upsilon(\nn) + 8} \lambda_{q,\nn} r_{q+1,\nn}^{-2} \lambda_{q+1}^{-1} \leq \Gamma_{q+1}^{\badshaq-1}$, see~\eqref{eq:div:cor:ineq:4}, the above bound is consistent with \eqref{eq:div:corrector:Linfty}.

It remains to consider the term $\mathbf{V}_2$ in \eqref{eq:nets:suck:3}. We distribute the $\partial_m$ derivative on either the  slow or the fast objects and decompose
\begin{align}
\mathbf{V}_2^\bullet 
&= \sum_{\xi,i,j,k, \vec{l}}
\partial_m \left( a_{(\xi)} \varrho_{(\xi)}\circ \Phi_{(i,k)} \xi^\ell   A_{\ell}^\bullet  \epsilon_{m p r}  a_{p,(\xi)}^{\rm bad}   \partial_r \Phi_{(i,k)}^s \mathbb{U}_{\xi,q+1,\nn}^s \circ \Phi_{(i,k)} 
\right)  \notag\\
&=
\sum_{\xi,i,j,k, \vec{l}}
\Bigl( 
\partial_m \bigl(   \xi^\ell   A_{\ell}^\bullet  \epsilon_{m p r}      \partial_r \Phi_{(i,k)}^s  
\bigr) a_{(\xi)} a_{p,(\xi)}^{\rm bad} 
+
 a_{m,(\xi)}^{\rm good}  \xi^\ell   A_{\ell}^\bullet  \epsilon_{m p r}  a_{p,(\xi)}^{\rm bad}   \partial_r \Phi_{(i,k)}^s \notag\\
&\qquad \qquad \qquad \qquad \qquad \qquad \qquad \qquad \qquad \ \  
-
a_{(\xi)} \xi^\ell   A_{\ell}^\bullet  \epsilon_{m p r}   \partial_m ( a_{p,(\xi)}^{\rm good})   \partial_r \Phi_{(i,k)}^s
\Bigr)
\left(\varrho_{(\xi)}  \mathbb{U}_{\xi,q+1,\nn}^s  \right)\circ \Phi_{(i,k)}
\notag\\
& +
\sum_{\xi,i,j,k, \vec{l}} a_{(\xi)} \xi^\ell A_{\ell}^\bullet  \epsilon_{m p r} a_{p,(\xi)}^{\rm bad}   \partial_r \Phi_{(i,k)}^s
\partial_m \left(  ( \varrho_{(\xi)}     \mathbb{U}_{\xi,q+1,\nn}^s) \circ \Phi_{(i,k)} 
\right)
\,.
\label{eq:nets:suck:5}
\end{align}
In the second equality above we have used the identities $\epsilon_{m p r} \partial_m  (a_{p,(\xi)}^{\rm bad}) = - \epsilon_{mpr} \partial_m (a_{p,(\xi)}^{\rm good})$, and that $\epsilon_{m p r} a_{m,(\xi)}^{\rm bad} a_{p,(\xi)}^{\rm bad} = 0$.
We first consider the terms in which the $\partial_m$ has not landed on functions related to pipe densities.  Similarly to the definition of $\mathbf{V}_1$ in \eqref{eq:nets:suck:4}, the slow functions in each term contain a derivative that costs the good factor of $\lambda_{q,\nn}\Gamma_{q+1}$, and a derivative that costs the bad factor of $\lambda_{q+1} r_{q+1,\nn}$. As such, when applying $\divH+\divR$ to the second to last line of \eqref{eq:nets:suck:5}, the resulting stress obeys exactly the same estimates as \eqref{eq:nets:suck:4:L1} and \eqref{eq:nets:suck:4:Linfty}.

Finally, we are left to consider the term on the last line of \eqref{eq:nets:suck:5}, in which the $\partial_m$ derivative lands on the fast objects, at frequency $\lambda_{q+1}$. The key observation is that this term is in fact equal to $0$! To see this cancellation, we recall the identification of $a_{p,(\xi)}^{\rm bad}$ in \eqref{eq:nets:suck:2}, and we recall from \eqref{eq:UU:explicit} that $\UU_{\xi,q+1,\nn}  
= - \xi' \varphi_{\xi,\lambda_{q+1},r_{q+1,\nn}}^{\prime\prime} 
+  \xi'' \varphi_{\xi,\lambda_{q+1},r_{q+1,\nn}}^{\prime}$.
With these identities, we have
\begin{align*}
&a_{(\xi)} \xi^\ell A_{\ell}^\bullet  \epsilon_{m p r}   a_{p,(\xi)}^{\rm bad}   \partial_r \Phi_{(i,k)}^s
\partial_m \left(  ( \varrho_{(\xi)}     \mathbb{U}_{\xi,q+1,\nn}^s) \circ \Phi_{(i,k)} 
\right)
\notag\\
&\qquad = a_{(\xi)} \xi^\ell A_{\ell}^\bullet  \epsilon_{m p r}  a_{p,(\xi)}^{\rm bad}   \partial_r \Phi_{(i,k)}^s
\partial_m \Phi^n   \partial_n( \varrho_{(\xi)}     \mathbb{U}_{\xi,q+1,\nn}^s) \circ \Phi_{(i,k)} 
\,.
\end{align*}
Note that from \eqref{eq:nets:suck:2}, that $a_{p,(\xi)}^{\rm bad}$ contains either a factor of $\partial_p \Phiik^k \xi'_k$ or a factor of $\partial_p \Phiik^k \xi^{\prime \prime}_k$. From \eqref{eq:derivative:along:pipe}, we also have that 
\begin{align*}
\partial_r \Phi_{(i,k)}^s
\partial_m \Phi^n \partial_n (\varrho_{(\xi)} \UU^s_{\xi,q+1,\nn}) 
&=- \partial_r \Phi_{(i,k)}^s \xi'_s
\partial_m \Phi^n \xi'_n \left((\xi'\cdot\nabla) \left(\varrho_{(\xi)} \varphi_{\xi,\lambda_{q+1},r_{q+1,\nn}}^{\prime\prime} 
\right)\right) \circ \Phi_{(i,k)} \notag\\
&\qquad +
\partial_r \Phi_{(i,k)}^s \xi^{\prime\prime}_s
\partial_m \Phi^n \xi'_n \left((\xi'\cdot\nabla) \left(\varrho_{(\xi)} \varphi_{\xi,\lambda_{q+1},r_{q+1,\nn}}^{\prime} 
\right)\right) \circ \Phi_{(i,k)} \notag\\
&\qquad - 
\partial_r \Phi_{(i,k)}^s \xi^\prime_s
\partial_m \Phi^n \xi_n^{\prime \prime} \left( (\xi^{\prime \prime }\cdot\nabla)\left(\varrho_{(\xi)}   \varphi_{\xi,\lambda_{q+1},r_{q+1,\nn}}^{\prime\prime} 
 \right)\right) \circ \Phi_{(i,k)} \notag\\
&\qquad + 
\partial_r \Phi_{(i,k)}^s \xi^{\prime \prime}_s
\partial_m \Phi^n \xi_n^{\prime \prime} \left( (\xi^{\prime }\cdot\nabla)\left(\varrho_{(\xi)}   \varphi_{\xi,\lambda_{q+1},r_{q+1,\nn}}^{\prime\prime} 
 \right)\right) \circ \Phi_{(i,k)} \, .
\end{align*}
Thus, the expression $\epsilon_{m p r} a_{p,(\xi)}^{\rm bad}   \partial_r \Phi_{(i,k)}^s
\partial_m \Phi^n   \partial_n( \varrho_{(\xi)}     \mathbb{U}_{\xi,q+1,\nn}^s) \circ \Phi_{(i,k)}$ equals the sum of eight terms, each of which is of the type 
\begin{align*}
    \epsilon_{m p r} \partial_p \Phiik^k \xi^{(1)}_k   \partial_r \Phi_{(i,k)}^s \xi^{(2)}_s \partial_m \Phi^n \xi^{(3)}_n
    \times 
    (\mbox{product of fast pipe densities or fast cutoffs})\circ \Phiik
\end{align*}
where  $( \xi^{(1)}, \xi^{(2)}, \xi^{(3)} ) \in \{ \xi^\prime, \xi^{\prime \prime} \}^3$. Since in each of these eight terms, at least two of the vectors in the tuple $( \xi^{(1)}, \xi^{(2)}, \xi^{(3)} )$ are equal to each other, either to $\xi'$ or $\xi''$, by the skew symmetry of the Levi-Civita symbol, we must have
\begin{align*}
    \epsilon_{m p r} \partial_p \Phiik^k \xi^{(1)}_k   \partial_r \Phi_{(i,k)}^s \xi^{(2)}_s \partial_m \Phi^n \xi^{(3)}_n = 0
    \,.
\end{align*}
This proves that the last term on the right side of \eqref{eq:nets:suck:5} is indeed equal to $0$, concluding the proof.
\end{proof}

\subsection{Transport errors}\label{ss:stress:transport}

\begin{lemma}\label{l:transport:error}
For all $0\leq \nn \leq \nmax$, the transport error satisfies the following estimates for $N,M\leq 3\NindLarge$:
\begin{subequations}
\begin{align}
\left\| \psi_{i,q} D^N \Dtq^M \left( \left( \divH + \divR \right)\left(  \Dtq w_{q+1,\nn} \right) \right) \right\|_{L^1} &\lesssim \delta_{q+2}\Gamma_{q+1}^{\shaq-1} \lambda_{q+1}^N \MM{M,\Nindt, \tau_q^{-1}\Gamma_{q+1}^{i-\cstarnn + 5},\Gamma_{q+1}^{-1}\tilde\tau_q^{-1}} \label{eq:trans:L1:est} \\
    \left\| D^k \Dtq^m \left( \left( \divH + \divR \right) (\Dtq w_{q+1,\nn}) \right) \right\|_{L^\infty(\supp \psi_{i,q})}  &\lesssim \Gamma_{q+1}^{\badshaq-1} \lambda_{q+1}^N \MM{M,\Nindt, \tau_q^{-1}\Gamma_{q+1}^{i-\cstarnn + 5},\Gamma_{q+1}^{-1}\tilde\tau_q^{-1}} \, . \label{eq:trans:Loo:est}
\end{align}
\end{subequations}
\end{lemma}
\begin{proof}[Proof of Lemma~\ref{l:transport:error}]
Recall from the first line of \eqref{eq:idiing:RRqnn} that the transport error is given by $\divH+\divR$ applied to $\Dtq w_{q+1,\nn}$, which we further expand as
\begin{align}
    \Dtq w_{q+1,\nn} &=\Dtq \biggl(\sum_{\; i,j,k,\vecl,\xi} \curl \left( a_{\xi,i,j,k,q,\nn,\vecl} \nabla\Phi_{(i,k)}^{T}\UU_{\xi,q+1,\nn} \circ \Phi_{(i,k)} \right) \biggr) \nonumber\\
    &= \sum_{i,j,k,\vecl,\xi}\Dtq \left(a_{(\xi)} \nabla\Phi_{(i,k)}^{-1}\right) \WW_{\xi,q+1,\nn} \circ \Phi_{(i,k)}  + \sum_{i,j,k,\vecl,\xi} \left( \Dtq \nabla a_{(\xi)} \right) \times \left( \nabla\Phiik \UU_{\xi,q+1,\nn}\circ \Phiik \right) \notag\\
    &\qquad + \sum_{i,j,k,\vecl,\xi} \nabla a_{(\xi)} \times \left( \left( \Dtq \nabla\Phiik \right)\UU_{\xi,q+1,\nn}\circ \Phiik \right) \label{eq:transport:estimate:1}
\end{align}
Since the second two terms contain the corrector defined in \eqref{wqplusoneonec}, and the bounds for the corrector in \eqref{eq:w:oxi:c:est} are stronger than that of the principal part of the perturbation, we shall completely estimate only the first term and simply indicate the set-up for the second and third. Before applying Proposition~\ref{prop:intermittent:inverse:div}, recall that the inverse divergence of \eqref{eq:transport:estimate:1} needs to be estimated on the support of a cutoff $\psi_{i,q}$ in order to verify \eqref{eq:trans:L1:est} and \eqref{eq:trans:Loo:est}. Recall that for all $\nn$, $\Dtq w_{q+1,\nn}$ has zero mean.  Thus, although each individual term in the final equality in \eqref{eq:transport:estimate:1} may not have zero mean, we can safely apply $\divH$ and $\divR$ to each term and estimate the outputs while ignoring the last term in \eqref{eq:inverse:div:error:stress}.

We will apply Proposition~\ref{prop:intermittent:inverse:div} to the first term with the following choices.  Let $p\in\{1,\infty\}$.  We set $v=\vlq$, and $D_t=\Dtq=\partial_t+\vlq\cdot\nabla$ as usual.  We set $N_*=M_*=\lfloor \sfrac{1}{2}\left(\Nfnn-\NcutSmall-\NcutLarge-5\right) \rfloor$, with $\Ndec$ and $\dpot$ satisfying \eqref{eq:lambdaqn:identity:3}.  We define
$$  G = \Dtq (a_{(\xi)}\nabla\Phi_{(i,k)}^{-1}) \xi \, , $$
with $\lambda=\Gamma_{q+1}^{-1}\lambda_{q+1}r_{q+1,\nn}$, $\nu=\tau_q^{-1}\Gamma_{q+1}^{i-\cstarnn+5}$, $M_t=\Nindt$, $\tilde\nu=\tilde\tau_{q}^{-1}\Gamma_{q+1}^{-1}$. In order to obtain the value of the amplitude constant $\const_G$, which now depends on $p$, when $p=1$ we use  \eqref{e:a_master_est_p} with $r=1$ and \eqref{eq:Lagrangian:Jacobian:6}, while when $p=\infty$ we use \eqref{e:a_master_est_p_uniform} and \eqref{eq:Lagrangian:Jacobian:6}, obtaining
\begin{subequations}
\label{eq:Cg:transport}
\begin{align}
\const_{G,1} &= |\supp(\eta_{i,j,k,q,\nn,\xi,\vec{l}})| \delta_{q+1,\nn}^{\sfrac{1}{2}} \Gamma_{q+1}^{j+2} \tau_q^{-1} \Gamma_{q+1}^{i-\cstarnn+3} \,,  
\\
\const_{G,\infty} 
&= 
\Gamma_{q}^{\frac{\badshaq}{2}} 
\Gamma_{q+1}^{7 \Upsilon(\nn) + \frac 72}
\tau_q^{-1} \Gamma_{q+1}^{\imax -\cstarnn+4}
\notag\\
&
\leq 
\Gamma_{q}^{\frac{\badshaq}{2}} 
\Gamma_{q+1}^{\badshaq + 7 \Upsilon(\nn) + 9 -\cstarnn }
\tau_q^{-1} \Theta_q^{\sfrac 12} \delta_q^{-\sfrac 12} \leq \Gamma_{q}^{\frac{\badshaq}{2}} 
\Gamma_{q+1}^{\badshaq + 7 \Upsilon(\nmax) + 20 + \cstarzero -\cstarnn } \Theta_q^{\sfrac 12}   \lambda_q \, .  
\end{align}
\end{subequations}
In the above expressions we have used \eqref{eq:cstarn:inequality} to control $\cstarn$, \eqref{eq:imax:bound} to control $\Gamma_{q+1}^{\imax}$, and the definition of $\tau_q$ from \eqref{def:tau:q:actual}. We have that 
\begin{align}
\| D^N \Dtq^M G \|_{L^p} &\lessg \const_{G,p} \left(\lambda_{q+1}r_{q+1,\nn}\Gamma_{q+1}^{-1}\right)^N \MM{M,\Nindt-1,\tau_q^{-1}\Gamma_{q+1}^{i-\cstarnn + 4},\tilde\tau_q^{-1}\Gamma_{q+1}^{-1}} \notag\\
&\lessg \const_{G,p} \left(\lambda_{q+1}r_{q+1,\nn}\Gamma_{q+1}^{-1}\right)^N \MM{M,\Nindt,\tau_q^{-1}\Gamma_{q+1}^{i-\cstarnn + 5},\tilde\tau_q^{-1}\Gamma_{q+1}^{-1}}
,\label{eq:david:transport:0}
\end{align}
for all $N,M \leq \lfloor \sfrac{1}{2}\left(\Nfnn-\NcutSmall-\NcutLarge-5\right) \rfloor $ after using \eqref{eq:Nind:cond:2},
and so \eqref{eq:inverse:div:DN:G} is satisfied.  We set $\Phi=\Phi_{i,k}$ and $\lambda'=\tilde\lambda_q$.  Appealing as usual to Corollary~\ref{cor:deformation} and \eqref{eq:nasty:D:vq:old} with $q'=q$, which is valid from Proposition~\ref{prop:no:proofs}, we have that \eqref{eq:DDpsi} and \eqref{eq:DDv} are satisfied.

Referring to \eqref{item:pipe:1} from Proposition~\ref{prop:pipeconstruction}, we set $\varrho=\varrho_{\xi,\lambda_{q+1},r_{q+1,\nn}}$ and $\vartheta=\vartheta_{\xi,\lambda_{q+1},r_{q+1,\nn}}$.  Setting $\zeta=\lambda_{q+1}$, we have that \eqref{item:inverse:i} is satisfied.  Setting $\mu=\lambda_{q+1}r_{q+1,\nn}$ and referring to \eqref{item:pipe:2} from Proposition~\ref{prop:pipeconstruction}, we have that \eqref{item:inverse:ii} is satisfied.  Setting $\Lambda=\zeta=\lambda_{q+1}$, $C_{*,p}=r_{q+1,\nn}^{\frac{2}{p}-1}$, $\alpha$ as in \eqref{eq:alpha:equation:1}, and referring to \eqref{e:pipe:estimates:1} and \eqref{e:pipe:estimates:2} from Proposition~\ref{prop:pipeconstruction}, we have that \eqref{eq:DN:Mikado:density} is satisfied. \eqref{eq:inverse:div:parameters:0} is immediate from the definitions. Referring to \eqref{eq:lambdaqn:identity:2}, we have that \eqref{eq:inverse:div:parameters:1} is satisfied.

After summing on $(i,j,k,\nn,\xi,\vec{l})$, using \eqref{eq:inductive:partition} at level $q$, and \eqref{item:lebesgue:1} with $r_1=r_2=2$, we conclude from \eqref{eq:inverse:div:stress:1} that for $p=1$ and $N,M \leq\lfloor \sfrac{1}{2}\left(\Nfnn-\NcutSmall-\NcutLarge-5\right) \rfloor -\dpot$,
\begin{align}
&\left\| D^N \Dtq^M \left( \divH \left( \Dtq w_{q+1,\nn} \right) \right) \right\|_{L^1\left(\supp \psi_{i,q}\right)} \notag \\ 
&\quad  \lessg \delta_{q+1,\nn}^{\sfrac{1}{2}} \Gamma_{q+1}^{ \CLebesgue + 9 - \cstarnn} \tau_q^{-1} r_{q+1,\nn} \lambda_{q+1}^{-1} \lambda_{q+1}^N \MM{M,\Nindt,\tau_q^{-1}\Gamma_{q+1}^{i-\cstarnn + 5},\tilde\tau_q^{-1}\Gamma_{q+1}^{-1}} \notag\\
&\quad \lesssim \Gamma_{q+1}^{\shaq-1} \delta_{q+2} \lambda_{q+1}^N  \MM{M,\Nindt,\tau_q^{-1}\Gamma_{q+1}^{i-\cstarnn + 5},\tilde\tau_q^{-1}\Gamma_{q+1}^{-1}}  \label{eq:david:transport:2}
\end{align}
after also using \eqref{eq:drq:identity} and \eqref{eq:cstarn:inequality}.  From \eqref{eq:nfnn:mess}, these bounds are valid for all $N,M\leq 3\NindLarge$.  Similarly, for $p=\infty$, we have 
\begin{align}
&\left\| D^N \Dtq^M \left( \divH \left( \Dtq w_{q+1,\nn} \right) \right) \right\|_{L^\infty\left(\supp \psi_{i,q}\right)} \notag\\
&\quad \lessg 
\Gamma_{q}^{\frac{\badshaq}{2}} 
\Gamma_{q+1}^{\badshaq + 7 \Upsilon(\nmax) + 21 + \cstarzero -\cstarnn } \Theta_q^{\sfrac 12}   \lambda_q  r_{q+1,\nn}^{-1} \lambda_{q+1}^{-1}  \lambda_{q+1}^N \MM{M,\Nindt,\tau_q^{-1}\Gamma_{q+1}^{i-\cstarnn + 5},\tilde\tau_q^{-1}\Gamma_{q+1}^{-1}} \notag\\
&\quad \lesssim \Gamma_{q+1}^{\badshaq-1} \lambda_{q+1}^N  \MM{M,\Nindt,\tau_q^{-1}\Gamma_{q+1}^{i-\cstarnn + 5},\tilde\tau_q^{-1}\Gamma_{q+1}^{-1}}  \label{eq:david:transport:3}
\end{align}
after also using \eqref{def:cstarn:formula} and \eqref{eq:transport:Loo:ineq}.

To conclude the proof, we must still estimate the nonlocal ($\divR$) portion of the inverse divergence, and the error terms coming from the divergence correctors.  These error terms, however, obey stronger estimates than the bounds in \eqref{eq:david:transport:2} and \eqref{eq:david:transport:3}, and so we refer to the proof of \cite[Lemma 8.4]{BMNV21} for further details.
\end{proof}

\subsection{Nash errors}\label{ss:stress:Nash}

\begin{lemma}\label{l:Nash:error}
For all $0\leq \nn \leq \nmax$, the Nash errors satisfy the following estimates for $N,M\leq 3\NindLarge$:
\begin{subequations}
\begin{align}
\left\| \psi_{i,q} D^N \Dtq^M \left( \left( \divH + \divR \right)\left( w_{q+1,\nn}\cdot\nabla\vlq \right) \right) \right\|_{L^1} &\lesssim \delta_{q+2}\Gamma_{q+1}^{\shaq-1} \lambda_{q+1}^N \MM{M,\Nindt, \tau_q^{-1}\Gamma_{q+1}^{i-\cstarnn + 4},\Gamma_{q+1}^{-1}\tilde\tau_q^{-1}} \label{eq:Nash:L1:est} \\
\left\| D^k \Dtq^m \left( \left( \divH + \divR \right) (w_{q+1,\nn}\cdot\nabla\vlq) \right) \right\|_{L^\infty(\supp \psi_{i,q})}  &\lesssim \Gamma_{q+1}^{\badshaq-1} \lambda_{q+1}^N \MM{M,\Nindt, \tau_q^{-1}\Gamma_{q+1}^{i-\cstarnn + 4},\Gamma_{q+1}^{-1}\tilde\tau_q^{-1}} \, . \label{eq:Nash:Loo:est}
\end{align}
\end{subequations}
\end{lemma}
\begin{proof}[Proof of Lemma~\ref{l:Nash:error}]
Recall from the first line of \eqref{eq:idiing:RRqnn} that the Nash error is given by $\divH+\divR$ applied to $w_{q+1,\nn}\cdot\nabla\vlq$, which we further expand as
\begin{align}
    w_{q+1,\nn}\cdot\nabla\vlq &= \sum_{i,j,k,\vecl,\xi} \notag \curl\left( a_{\xi,i,j,k,q,\nn,\vecl} \nabla\Phiik^T \UU_{\xi,q+1,\nn} \circ \Phiik \right) \cdot \nabla\vlq  \\
    &= \Biggl(\sum_{i,j,k,\vecl,\xi} \nabla a_{(\xi)} \times \left( \nabla \Phiik^T \UU_{\xi,q+1,\nn} \circ \Phiik \right)  +   \sum_{i,j,k,\vecl,\xi} a_{(\xi)} \nabla\Phi_{(i,k)}^{-1}\WW_{\xi,q+1,\nn} \circ \Phi_{(i,k)} \Biggr) \cdot \nabla\vlq 
    \,.
    \label{eq:Nash:estimate:1}
\end{align}
Due to the fact that the first term arises from the addition of the corrector defined in \eqref{wqplusoneonec}, and the fact that the bounds for the corrector in \eqref{eq:w:oxi:c:est} are stronger than that of the principal part of the perturbation, we shall only consider the second term. Note that the Nash error can be written as $\div ( w_{q+1,\nn}\otimes \vlq)$ and so has zero mean.  Thus, although each individual term in the final equality in \eqref{eq:Nash:estimate:1} may not have zero mean, we can safely apply $\divH$ and $\divR$ to each term and estimate the outputs while ignoring the last term in \eqref{eq:inverse:div:error:stress}.

We will apply Proposition~\ref{prop:intermittent:inverse:div} to the second term with the following choices.  We set $v=\vlq$, and $D_t=\Dtq=\partial_t + \vlq\cdot\nabla$ as usual.  We set $N_*=M_*= \lfloor \sfrac{1}{2}\left(\Nfnn-\NcutLarge-\NcutSmall-4\right)\rfloor $, with $\Ndec$ and $\dpot$ satisfying \eqref{eq:lambdaqn:identity:3}.  We define
\begin{equation*}
G = a_{(\xi)} \nabla\Phi_{(i',k)}^{-1} \xi \cdot \nabla \vlq
\end{equation*}
and set $\const_{G,1}$, $\const_{G,\infty}$ to be equal to the quantities in \eqref{eq:Cg:transport}, $\lambda=\Gamma_{q+1}^{-1}\lambda_{q+1}r_{q+1,\nn}$, $\nu=\tau_q^{-1}\Gamma_{q+1}^{i-\cstarnn+4}$, $M_t=\Nindt$, and $\tilde\nu=\tilde\tau_q^{-1}\Gamma_{q+1}^{-1}$. Note that these choices match exactly the choices from the estimates on the transport error. From \eqref{e:a_master_est_p} with $r=1$ and $r_1=r_2=2$, \eqref{eq:Lagrangian:Jacobian:6}, and \eqref{eq:nasty:D:vq:old} at level $q$, we have that for $N,M\leq \lfloor \sfrac{1}{2}\left(\Nfnn-\NcutLarge-\NcutSmall-4\right)\rfloor $
\begin{align}
\left\| D^N \Dtq^M G \right\|_{L^1} &\lessg \const_{G,p} \left(\Gamma_{q+1}^{-1}\lambda_{q+1}r_{q+1,\nn}\right)^N \MM{M,\Nindt,\tau_q^{-1}\Gamma_{q+1}^{i+1},\tilde\tau_q^{-1}\Gamma_{q+1}^{-1}}, \label{eq:david:nash:1}
\end{align}
and so \eqref{eq:inverse:div:DN:G} is satisfied.  Note that we have used \eqref{eq:Lambda:q:x:1:NEW} when converting the $\delta_q^{\sfrac{1}{2}} \tilde\lambda_q$ coming from \eqref{eq:nasty:D:vq:old} at level $q$ to a $\tau_q^{-1}$.  Setting $\Phi=\Phi_{(i,k)}$ and $\lambda'=\tilde\lambda_q$, we have that \eqref{eq:DDpsi} and \eqref{eq:DDv} are satisfied as usual.  The choices of $\varrho$, $\vartheta$, $\zeta$, $\mu$, $\Lambda$, and $\const_*$ are identical to those of the transport error (both terms contain $\WW_{\xi,q+1,\nn}\circ\Phiik$), and so we have that \eqref{item:inverse:i}-\eqref{item:inverse:ii}, \eqref{eq:DN:Mikado:density}, \eqref{eq:inverse:div:parameters:0}, and \eqref{eq:inverse:div:parameters:1} are satisfied as well.  Since the bound \eqref{eq:david:nash:1} is identical to that of \eqref{eq:david:transport:0}, we obtain an estimate identical to \eqref{eq:david:transport:2} in the case $p=1$.  The case $p=\infty$ and the estimates for the $\divR$ portion follows analogously to that for the first term from the transport error.  We omit further details.
\end{proof}

\section{Parameters}
\label{sec:parameters}

The purpose of the first subsection is to define the $q$-independent parameters \emph{in order}, beginning with the regularity index $\beta$, and ending with the number $a_*$, which will be used to absorb every implicit constant throughout the paper.  Then in Section~\ref{ss:q:dependent:parameters}, we define the parameters which depend on $q$, as well as the parameters which depend in addition on $n$.  Section~\ref{ss:many:inequalities} contains, in no particular order, consequences of the definitions made in the previous two sections which are necessary to close the estimates in the proof.

\subsection{Definitions and hierarchy of the parameters}
\label{sec:parameters:DEF}
The parameters in our construction are chosen as:
\begin{enumerate}[(i)]
\item \label{item:beta:DEF} Choose an $L^2$ regularity index $\beta \in [\sfrac 13 ,\sfrac 12)$; in light of~\cite{BDLSV17,Isett2018}, there is no reason to take $\beta < \sfrac 13$.
\item Choose $b \in (1, \sfrac 32)$ sufficiently small such that 
\begin{align}
2\beta b &< 1
\label{eq:b:DEF} \,.
\end{align}

\item \label{item:nmax:pmax:DEF}  With $\beta$ and $b$ chosen, we may now designate a number of parameters:
\begin{enumerate}  
\item The parameter $\nmax$, which denotes the total number of higher order stresses $\RR\qn$, is defined as the smallest integer such that
\begin{subequations}
\label{eq:nmax:cond:all}
\begin{align}
\frac{2}{\nmax+1} &< \frac{(b-1)^2}{2b} \label{eq:nmax:cond:3} \\
2\beta b + \frac{3+\lceil \log_2 \nmax \rceil}{2(\nmax+1)} &< \frac{1}{2} + \frac{\nmax}{2(\nmax+1)}
\label{eq:nmax:cond:1}
\, .
\end{align}
\end{subequations}
Notice that the second inequality is possible since $2\beta b < 1$.

\item The parameter $\CLebesgue$ appearing in \eqref{eq:psi:i:q:support:old} to quantify $\left\| \psi_{i,q} \right\|_{L^1}$ is defined as
\begin{align}
\CLebesgue = \frac{b+4}{b-1}
\,.
\label{eq:CLebesgue:DEF}
\end{align}
\item The exponent $\mathsf{C_R}$ is a small parameter used to estimate the Reynolds stress, cf.~\eqref{eq:Rq:inductive:assumption}, and then absorb geometric constants in the construction. It is defined as 
\begin{align}
 \mathsf{C_R} = 10 b + 1 
 \,.
 \label{eq:shaq:DEF}
\end{align}
\end{enumerate}
\item The parameter $\cstar$, which is first introduced in \eqref{eq:sharp:Dt:psi:i:q:mixed:old} and utilized in Sections~\ref{sec:statements} and \ref{s:stress:estimates} to control small losses in the sharp material derivative estimates, is defined in terms of $\nmax$ as 
\begin{align}
 \cstar = 4 \nmax + 5\,.
 \label{eq:cstar:DEF}
\end{align}
\item The parameter $\eps_\Gamma > 0 $, which is used in \eqref{def:Gamma:q:actual} to quantify the \emph{finest} frequency scale between $\lambda_q$ and $\lambda_{q+1}$ utilized throughout the scheme, is defined as any real number such that
\begin{subequations}
\label{eq:eps:Gamma:DEF}
\begin{align}
\varepsilon_\Gamma 300 (\nmax+1)(\lceil \log_2 \nmax\rceil) &< b-1 \label{eq:eps:still:doing:it} \\
7\varepsilon_\Gamma (2 + \lceil \log_2 \nmax \rceil + 22 + 4\nmax ) &< \frac{b-1}{2b} - \frac{3}{2(b-1)(\nmax+1)} \label{eq:eps:gamma:blechhhh} \\
\varepsilon_\Gamma \left( 5 + (2 + \lceil \log_2 \nmax \rceil )(9+\CLebesgue) \right) &< \frac{1}{2} - \frac{2+ \lceil \log_2 \nmax \rceil}{\nmax} \label{eq:eg:new:new} \\
\varepsilon_\Gamma \left( \mathsf{C_R}\left(\frac{b-1}{b}\right) + 15 + 9(3 + \lceil \log_2(\nmax)\rceil) \right) &< \frac{1}{2}\left( 1 + \frac{\nmax}{\nmax+1} \right) - 2\beta b - \frac{3+\lceil \log_2 \nmax \rceil}{2(\nmax+1)} \,  \label{eq:vareps:gamma:new:new} \\
\eps_\Gamma \left(\frac 12\CLebesgue+\cstarzero+ 10 + \frac 12 \mathsf{C_R}\right)  &< 1- 2\beta b
\label{eq:eps:gamma:1:new}
\\
\eps_\Gamma \left( 7 + \mathsf{C_R} + \nmax (8 + \CLebesgue) \right)  
&< \frac{1-2\beta}{10} \label{eq:eps:gamma:1} \\
2b\varepsilon_\Gamma(\cstarzero+7) < 1- \beta \label{eq:eps:gamma:4}
\,. 
\end{align}
\end{subequations}
We note that the right-hand side of \eqref{eq:eps:gamma:blechhhh} is positive from \eqref{eq:nmax:cond:3} and the right-hand sides of \eqref{eq:eg:new:new} and \eqref{eq:vareps:gamma:new:new} are positive from \eqref{eq:nmax:cond:1}.
\item The parameter $\badshaq$ is defined as
\begin{equation}\label{eq:badshaq:def}
    \badshaq = \frac{1}{\varepsilon_\Gamma(b-1)(\nmax+1)} \, .
\end{equation}
\item The parameter $\alpha > 0$ from the $L^1$ loss of the inverse divergence operator  is now defined as 
\begin{align}
\alpha = \frac{\eps_\Gamma (b-1)}{2b}  
\,.
\label{eq:alpha:DEF} 
\end{align}
\item The parameters $\NcutSmall$ and $\NcutLarge$ are used in Section~\ref{sec:cutoff} in order to define the velocity and stress cutoff functions; see~\eqref{eq:h:j:q:def}, \eqref{eq:psi:i:q:recursive}, and \eqref{eq:g:i:q:n:def}. These large integers are chosen solely in terms of $b$ and $\eps_\Gamma$ as 
\begin{align}
\frac 12 \NcutLarge = \NcutSmall = \left\lceil \frac{3b}{ \eps_\Gamma (b-1)} + \frac{15 b}{2} \right\rceil
\,.
\label{eq:Ncut:DEF}
\end{align}
\item The parameter $\Nindt$, which is the number of sharp material derivatives propagated on stresses and velocities in Sections~\ref{section:inductive:assumptions} through \ref{s:stress:estimates}, is chosen as the smallest integer for which we have
\begin{align}
\Nindt  =  \left \lceil\frac{4}{\eps_\Gamma (b-1)} \right \rceil  \NcutSmall
\,.
\label{eq:Nind:t:DEF}
\end{align}
\item The parameter $\Nindv$, whose primary role is to quantify the number of sharp spatial derivatives propagated on the velocity increments and stresses, cf.~\eqref{eq:inductive:assumption:derivative} and~\eqref{eq:Rq:inductive:assumption}, is chosen as the smallest integer for which we have the bound
\begin{align}
4 b \Nindt  + 8 +  b (\mathsf{C_R}+3) \eps_\Gamma (b-1) + 2\beta (b^3-1) 
&< \eps_\Gamma (b-1)  \Nindv 
\,.
\label{eq:Nind:v:DEF}
\end{align}
\item The value of the decoupling parameter $\Ndec$, which is used in the $L^p$ decorrellation conditions \eqref{eq:slow_fast_3} and \eqref{eq:inverse:div:parameters:1}, is chosen as the smallest integer for which
\begin{align}\label{eq:Ndec:DEF}
\Ndec
> \frac{8b}{(b-1)\varepsilon_\Gamma} 
\,.
\end{align}
\item The parameter $\dpot$, which is used in the inverse divergence operator of Proposition~\ref{prop:intermittent:inverse:div} to count the order of a parametrix expansion, is chosen as the smallest integer for which we have
\begin{align}
(\dpot - 1) \varepsilon_\Gamma (b-1) 
&> b(6+ 13\NindLarge) + 2\beta b^2 + (2 + \lceil \log_2 \nmax \rceil ) \left( \frac{b-1}{2(\nmax+1)} + \varepsilon_\Gamma (b-1) (9+\CLebesgue ) \right) 
\label{eq:dpot:DEF}
\,.
\end{align}
\item\label{item:Nfin:DEF} The value of $\Nfin$, which is introduced in Section~\ref{section:inductive:assumptions} and used to quantify the highest order derivative estimates utilized throughout the scheme is chosen as the smallest integer such that 
\begin{align}
\label{eq:Nfin:DEF}
\frac{3}{2} \Nfin >  (2 \NcutSmall + \NcutLarge + 14 \Nindv + 2\dpot + 2\Ndec + 12) 2^{\nmax+1}
\,.
\end{align} 

\item\label{item:astar:DEF} Having chosen all the previous parameters in items \eqref{item:beta:DEF}--\eqref{item:Nfin:DEF}, there exists a {\em sufficiently large} parameter $a_* \geq 1 $, which depends on all the parameters listed above (which recursively means that $a_* = a_*(\beta ,b)$), and which allows us to choose $a$ an {\em arbitrary number} in the interval $[a_*,\infty)$. While we do not give a formula for $a_*$ explicitly, it is chosen so that $a_*^{(b-1)\eps_\Gamma}$ is at least twice larger than {\em all the implicit constants in the $\les$ symbols throughout the paper}; note that these constants only depend on the parameters in items \eqref{item:beta:DEF}--\eqref{item:Nfin:DEF} --- never on $q$ --- which justifies the existence of $a_*$.
\end{enumerate}

\subsection{Definitions of the \texorpdfstring{$q$}{q}-dependent parameters}\label{ss:q:dependent:parameters}

\subsubsection{Parameters which depend only on \texorpdfstring{$q$}{q}}
For $q\geq 0$, we define the fundamental frequency parameter as
\begin{align}
\lambda_q &= 2^{\big{\lceil} (b^q) \log_2 a \big{\rceil}} \,. \label{def:lambda:q:actual} 
\end{align}
Definition~\eqref{def:lambda:q:actual} gives that $\lambda_q$ is an integer power of $2$, and that we have the bounds
\begin{align}
a^{(b^q)} \leq \lambda_q \leq 2 a^{(b^q)}
\qquad\mbox{and}\qquad
\frac 13 \lambda_q^b \leq \lambda_{q+1} \leq 2
\lambda_q^b
\label{eq:lambda:q:to:q+1}
\end{align}
for all $q\geq 0$. Throughout the paper, if there exists a universal constant $C>0$ such that $C^{-1} A \leq B \leq C A$, we say that $A \approx B$. In particular, the above reads $\lambda_q \approx a^{(b^q)}$ and $\lambda_{q+1} \approx \lambda_q^b$. It will be convenient to denote the quotient of two consecutive frequency parameters by
\begin{align}
\Theta_{q+1} =  \lambda_{q+1}\lambda_q^{-1} \approx \lambda_q^{b-1}\,.
\label{eq:Theta:q+1:actual}
\end{align}
The fundamental amplitude parameter is defined in terms of $\lambda_q$ by
\begin{align}
\delta_q &=  \lambda_1^{(b+1)\beta}  \lambda_q^{-2\beta} \label{def:delta:q:actual}
\,.
\end{align}
We now introduce a parameter which is defined in terms of the parameter $\eps_\Gamma$ from \eqref{eq:eps:Gamma:DEF} and used repeatedly to mean ``a tiny power of the frequency parameter'':
\begin{align}
\Gamma_{q+1} &= \Theta_{q+1}^{\varepsilon_\Gamma} 
\,.
\label{def:Gamma:q:actual}
\end{align}
In order to cap off our derivative losses, we need to mollify in space and time using the operators described in Section~\ref{sec:mollification:stuff}. This is done in terms of the following space and time parameters:
\begin{align}
\tilde\lambda_q &= \lambda_q\Gamma_{q+1}^5   \label{eq:tilde:lambda:q:def} \\
\tilde\tau_q^{-1} &= \tau_q^{-1} \tilde\lambda_q^3 \tilde\lambda_{q+1} \label{eq:tilde:tau:q:def}
\,.
\end{align}
While $\tilde\tau_q$ is used for mollification and thus for rough material derivative bounds, the fundamental temporal parameter used in the paper for sharp material derivative bounds is
\begin{align}
\tau_q &= \bigl( \delta_q^{\sfrac 12}\tilde\lambda_q \Gamma_{q+1}^{\cstar+6} \bigr)^{-1} \label{def:tau:q:actual}
\,. 
\end{align}
Note that besides depending on the parameters introduced in \eqref{item:beta:DEF}--\eqref{item:astar:DEF}, the parameters introduced above only depend on $q$, but are independent of $n$. We note that the definitions of the parameters listed so far in this subsection have not been changed from the definitions used in \cite{BMNV21}.

\subsubsection{Parameters which depend on $q$ and $n$}
The rest of the parameters depend on both $q$ and $n$.  We start by defining the frequency parameter $\lambda_{q,n}$ and the intermittency parameter $r_{q+1,n}$ by
\begin{align}
\lambda\qn
&=  \begin{cases}
 2^{\lceil (1 + 6 (b-1) \eps_\Gamma) \log_2 \lambda_q \rceil}, & n=0 \\
 2^{\lceil (\frac 12 -\frac{n}{2(\nmax+1)}) \log_2 \lambda_q +  (\frac 12 + \frac{n}{2(\nmax+1)}) \log_2 \lambda_{q+1}\rceil}, & 1\leq n \leq \nmax 
\end{cases} \, ,
\label{eq:lambda:q:n:def} \\
    r_{q+1,n} &= \lambda_{q+1}^{-1}
    2^{\lceil \frac 12 \log_2\lambda\qn + \frac 12 \log_2 \lambda_{q+1} - 2 \log_2 \Gamma_{q+1} \rceil}
    \label{eq:rqn:perp:definition} 
\end{align}
for $0\leq n \leq \nmax$.
In particular, \eqref{eq:lambda:q:n:def} shows that $\lambda_{q,n}$ is a power of $2$, with
$\lambda_{q,0} \approx \lambda_q \Gamma_{q+1}^6$ and $\lambda_{q,n} \approx \lambda_q^{\frac 12 - \frac{n}{2(\nmax+1)}} \lambda_{q+1}^{\frac 12 + \frac{n}{2(\nmax+1)}}$ for $1\leq n \leq \nmax$. Similarly, \eqref{eq:rqn:perp:definition} shows that $\lambda_{q+1} r_{q+1,n}$ is an integer power of $2$, and we have $\lambda_{q+1}r_{q+1,\nn}\approx \lambda_{q+1}^{\sfrac 12} \lambda\qnn^{\sfrac 12} \Gamma_{q+1}^{-2}$. A consequence of these approximations are the inequalities
\begin{equation}\label{ineq:rq:useful}
    r_{q+1,\nn}^{-2} \leq 2 \frac{\lambda_{q+1}}{\lambda_q} = 2 \Theta_{q+1} \, , \qquad r_{q+1,\nn}^2 \leq 2 \Gamma_{q+1}^{-4} \frac{\lambda\qnn}{\lambda_{q+1}}  \, .
\end{equation}

We recall from \eqref{eq:new:delta:def} that the stresses $\RR_{q,n}$ for $0\leq n \leq \nmax$ will be measured in terms of 
\begin{align}
\delta_{q+1,n} = \begin{cases}
\delta_{q+1} \Gamma_{q}^\shaq , & n=0 \\
\displaystyle\delta_{q+1,0} \frac{\tilde\lambda_{q}}{\lambda_{q}^{\sfrac 12} \lambda_{q+1}^{\sfrac 12}} \Gamma_{q+1}^{9}, & n=1 \\
\displaystyle \delta_{q+1,0} \frac{\tilde \lambda_{q}}{\lambda_{q,n-1}} \Gamma_{q+1}^{8} \left(\Theta_{q+1}^{\frac{1}{2(\nmax+1)}} \Gamma_{q+1}^{9} \right)^{\Upsilon(n)}, & 2\leq n \leq \nmax \, .
\end{cases} \label{eq:delta:appendix:def}
\end{align}
The function $\Upsilon(n)$ is defined in \eqref{eq:upsawhat} to quantify the number of steps required to produce $\RR\qn$. As each step accumulates negligible losses, which correspond to the quantity in parentheses above, one may adhere to the heuristic that $\delta_{q+1,n}$ is roughly speaking equal to $\frac{\delta_{q+1}\lambda_q}{\lambda\qn}$.  We remark that each of the parameters defined so far in this subsubsection has a new definition compared to that of \cite{BMNV21}.

Conversely, the following three parameters remain unchanged when compared to \cite{BMNV21}. For $1\leq n \leq \nmax$, we define $\cstarn$ in terms of $\cstar$ by 
\begin{align}
\label{def:cstarn:formula}
\cstarn &= \cstar - 4n  \,.
\end{align}
For $n=0$, we set
\begin{equation}\label{eq:Nfn0:def}
\NN{\textnormal{fin},0} = \frac{3}{2}\Nfin,
\end{equation}
while for $1\leq n \leq \nmax$, we define $\Nfn$ inductively on $n$ by using \eqref{eq:Nfn0:def} and the formula
\begin{equation}\label{def:Nfn:formula}
\Nfn =  \left\lfloor \frac 12 \left( \NN{\textnormal{fin},\textnormal{n}-1}  - \NcutSmall - \NcutLarge - 6 \right) - \dpot \right\rfloor  \, .
\end{equation}

\subsection{Inequalities and consequences of the parameter definitions}\label{ss:many:inequalities}

Due to \eqref{def:lambda:q:actual} we have that $\Gamma_{q+1} \geq (\sfrac 12)^{b \eps_\Gamma} \lambda_q^{(b-1)\eps_\Gamma} \geq  (\sfrac 12)^{b \eps_\Gamma} \lambda_0^{(b-1)\eps_\Gamma} \geq (\sfrac 12) a_*^{(b-1)\eps_\Gamma}$. As was already mentioned in item~\eqref{item:astar:DEF}, we have chosen $a_*$ to be sufficiently large so that $a_*^{(b-1)\eps_\Gamma}$ is at least twice larger than all the implicit constants appearing in all $\les$ symbols throughout the paper. Therefore, for any $q\geq 0$, we may use a single power of $\Gamma_{q+1}$ to absorb any implicit constant in the paper: an inequality of the type $A \les B$ may be rewritten as $A \leq \Gamma_{q+1} B$.  

From the definition \eqref{def:tau:q:actual} of $\tau_q$ and \eqref{def:cstarn:formula}, which gives that $\cstarn$ is decreasing with respect to $n$, we have that for all $0\leq n \leq \nmax$,
\begin{align}
\Gamma_{q+1}^{\cstarn+6}  \delta_q^{\sfrac 12} {\tilde \lambda_q} &\leq \tau_q^{-1} \, . \label{eq:Lambda:q:x:1:NEW}
\end{align}
Using the definitions \eqref{def:delta:q:actual}, \eqref{def:Gamma:q:actual}, \eqref{eq:tilde:lambda:q:def}, and \eqref{def:tau:q:actual}, and writing out everything in terms of $\lambda_{q-1}$, we have 
\begin{align}
\tau_{q-1}^{ -1 }\Gamma_{q+1}^{3 + \cstar} &\leq \tau_q^{-1} \label{eq:Tau:q-1:q}  \, .
\end{align}
From the definition of $\tilde\tau_q$, it is immediate that
\begin{align}
\tau_{q}^{-1} {\tilde \lambda_{q}^4}  &\leq {\tilde \tau_q^{-1}} \leq  \tau_{q}^{-1} {\tilde \lambda_{q}^3} {\tilde \lambda_{q+1}}  \, .
\label{eq:Lambda:q:t:1}
\end{align}
From the definitions \eqref{eq:cstar:DEF} of $\cstar$ and \eqref{def:cstarn:formula} of $\cstarn$, we have that for all $0\leq n \leq \nmax$,
\begin{equation}\label{eq:cstarn:inequality}
-\cstarn + 4 \leq -1.
\end{equation}
Next, we a list a few consequences of the fact that $\Nindv \gg \Nindt$, as specified in \eqref{eq:Nind:v:DEF}. First, we note from \eqref{eq:Lambda:q:t:1} that 
\begin{align}
 \tilde \tau_{q-1}^{-1} \tau_{q-1} \leq \tilde \lambda_{q-1}^3 \tilde \lambda_q \leq \lambda_q^4  
 \label{eq:trickery:trickery}
\end{align}
where in the second inequality we have used that $\eps_\Gamma \leq \frac{3}{20 b}$. 

The fact that $\Nindt$ is taken to be much larger than $\NcutSmall$, as expressed in \eqref{eq:Nind:t:DEF}, implies when combined with \eqref{eq:trickery:trickery} the following bound, which is also used in Section~\ref{sec:cutoff}:
\begin{align}
 \left(\tau_q  \tilde \tau_q^{-1}\right)^{\Ncut}
 \leq \Gamma_{q+1}^{ \Nindt }
\label{eq:Nind:cond:2}
\end{align}
for all $q\geq 1$.  
The parameter $\alpha$ in \eqref{eq:alpha:DEF} is chosen as such in order to ensure that
\begin{equation}\label{eq:alpha:equation:1}
\lambda_{q+1}^\alpha \approx \Gamma_{q+1}.
\end{equation}
for all $q\geq 0$. 
We note that the previous seven inequalities only involve parameters which have not changed when compared to \cite{BMNV21}.

Next, we list a number of parameter inequalities which are not the same as those in \cite{BMNV21}.  Our choice of $\Ndec$ in \eqref{eq:Ndec:DEF} and the assumption that $a$ is chosen sufficiently large so that $\Gamma_{q+1}^{\sfrac 12}>2\pi\sqrt{3}$ yields
\begin{align}
   \lambda_{q+1}^4 \leq \left( \frac{\Gamma_{q+1}}{2\pi\sqrt{3}} \right)^{\Ndec} \qquad \impliedby \qquad  \frac{8b}{(b-1)\varepsilon_\Gamma} < \Ndec \, . \label{eq:lambdaqn:identity:2}
\end{align}

We need a number of new inequalities to manage the Type 1 oscillation errors. The first of these is 
\begin{align}
\delta_{q+1,\nn}\lambda\qnn \Gamma_{q+1}^{8} \leq
\begin{cases} \Gamma_{q}^{\shaq} \delta_{q+1} \tilde\lambda_q \Gamma_{q+1}^{9} \left( \Theta_{q+1}^{\frac{1}{2(\nmax+1)}} \Gamma_{q+1}^{9} \right)^{\nn)} &\mbox{if} \quad \nn=0,1 \\
\Gamma_{q}^{\shaq} \delta_{q+1} \tilde\lambda_q \Gamma_{q+1}^{9} \left( \Theta_{q+1}^{\frac{1}{2(\nmax+1)}} \Gamma_{q+1}^{9} \right)^{\Upsilon(\nn)+1} &\mbox{if} \quad 2\leq \nn \leq \nmax
\end{cases} \, .
\label{eq:crazy:const:G:ineq}
\end{align}
If $\nn=0$, then the inequality follows from \eqref{eq:delta:appendix:def}, \eqref{eq:tilde:lambda:q:def}, \eqref{eq:lambda:q:n:def}, and \eqref{eq:upsawhat}.  If $\nn=1$, the inequality follows from the aforementioned inequalities and the equality
\begin{equation}\label{eq:theta:click:clack}
\frac{\lambda_{q,1}}{\lambda_q^{\sfrac 12}\lambda_{q+1}^{\sfrac 12}}=\frac{\lambda_{q,\nn}}{\lambda_{q,\nn-1}}=\frac{\lambda_{q+1}}{\lambda_{q,\nmax}}=\Theta_{q+1}^{\frac{1}{2(\nmax+1)}} \, ,
\end{equation}
which holds for $2\leq \nn \leq \nmax$. Finally, if $2\leq \nn \leq \nmax$, we use the aforementioned inequalities in conjunction with \eqref{eq:theta:click:clack}. Next, we claim that for all $0\leq \nn \leq \nmax$,
\begin{equation}
\label{eq:hopeless:mess:new}
\Gamma_q^{\shaq} \delta_{q+1} \tilde\lambda_q \Gamma_{q+1}^{9} \left( \Theta_{q+1}^{\frac{1}{2(\nmax+1)}} \Gamma_{q+1}^{9} \right)^{\Upsilon(\nn)+1} \lambda_{q,\nmax}^{-1}
\leq \Gamma_{q+1}^{\shaq-1}  \delta_{q+2}
\,.
\end{equation}
The above inequality is a consequence of \eqref{eq:upsa:bound} and
\begin{align}
    &2\beta b(b-1) + (b-1) \varepsilon_\Gamma \left( \shaq\left(\frac{1}{b}-1\right) + 15 + 9(3 + \lceil \log_2(\nmax)\rceil) \right) + \left( 3 + \lceil \log_2 \nmax \rceil \right) \frac{b-1}{2(\nmax+1)} \notag\\
    &\qquad \qquad < \frac{b-1}{2} + \frac{\nmax}{2(\nmax+1)}(b-1) \, , \label{eq:stinky:ineq}
\end{align}
which in turn follows from \eqref{eq:nmax:cond:1} and \eqref{eq:vareps:gamma:new:new}.
Finally, we claim that for $n$ such that $n>r(\nn)$, as defined in \eqref{eq:rofn}, and $n>2$,
\begin{align}\label{ineq:click:clack:1}
    \delta_{q+1,\nn} \lambda_{q,\nn} \Gamma_{q+1}^{9} \lambda_{q,n-1}^{-1} \leq \delta_{q+1,n} \, .
\end{align}
If $\nn=0,1$, the inequality follows from the definitions of $\lambda_{q,0}$ and $\lambda_{q,1}$ in \eqref{eq:lambda:q:n:def}, the definition of the $\delta_{q+1,n}$'s in \eqref{eq:delta:appendix:def}, and \eqref{eq:upsawhat}, which guarantees that $\Upsilon(n)\geq 1$ for $n\geq 2$. In the case $2\leq \nn \leq \nmax$, the inequality follows from the aforementioned inequalities combined with \eqref{eq:theta:click:clack} and the fact that for $n>r(\nn)$, \eqref{eq:upsa:ineq} gives that $\Upsilon(n)\geq \Upsilon(\nn)+1$.

The amplitudes of the higher order corrections $w\qplusnp$ must meet the inductive assumptions stated in \eqref{eq:inductive:assumption:derivative:q}.  Towards this end, we claim that for all $0\leq \nn \leq \nmax$,
\begin{equation}\label{eq:delta:q:nn:pp:ineq}
\delta_{q+1,\nn}^{\sfrac{1}{2}} \Gamma_{q+1}^{5} \leq \delta_{q+1}^{\sfrac{1}{2}} \, .
\end{equation}
Indeed, the case $\nn=0$ follows from the definition of $\mathsf{C_R}$ in \eqref{eq:shaq:DEF}, while the case $\nn\geq 1$ is a consequence of the definition \eqref{eq:delta:appendix:def} and the inequality
\begin{equation}\label{eq:vareps:nmax:stuff}
    \varepsilon_\Gamma(b-1) \left( 5 + (2 + \lceil \log_2 \nmax \rceil )(9+\CLebesgue) \right) + (b-1) \frac{2+ \lceil \log_2 \nmax \rceil}{\nmax} < \frac{b-1}{2} \, ,
\end{equation}
which in turn is a consequence of \eqref{eq:nmax:cond:1} and \eqref{eq:eg:new:new}.

We will also need that
\begin{equation}\label{eq:clickity:clackity}
    \Gamma_q^{\badshaq}\Gamma_{q+1}^{14\Upsilon(\nn)+13}\lambda\qnn 
r_{q+1,\nn}^{-2}  
 \lambda_{q,\nmax}^{-1} \leq \Gamma_{q+1}^{\badshaq-2} \, .
\end{equation}
The above inequality is a consequence of \eqref{eq:def:lambda:rq}, \eqref{eq:theta:click:clack}, and
\begin{equation}\label{eq:Cu:ugly}
    \eps_\Gamma \left( \frac{\badshaq}{b} + 14\left(2 + \lceil \log_2 \nmax \rceil \right) + 20 - \badshaq \right) + \frac{1}{2(\nmax+1)} < 0 \, ,
\end{equation}
which holds due to the choice of $\badshaq$ in \eqref{eq:badshaq:def} and \eqref{eq:eps:still:doing:it}. The inequality \eqref{eq:clickity:clackity} then immediately implies that
\begin{align}
\Gamma_q^{\badshaq}  \Gamma_{q+1}^{14 \Upsilon(\nn)+13} \lambda_{q,\nn} r_{q+1,\nn}^{-2} \lambda_{q+1}^{-1}
\leq \Gamma_{q+1}^{\badshaq-2}
\,.
\label{eq:div:cor:ineq:4}
\end{align}
We claim now that $\badshaq$ satisfies
\begin{equation}\label{eq:ineq:badshaq}
\Gamma_q^{\frac{\badshaq}{2}}\Gamma_{q+1}^{7\Upsilon(\nn)+\sfrac 72} r_{q+1,\nn}^{-1} \leq \Gamma_{q+1}^{\badshaq-2} \Theta_{q+1}^{\sfrac 12} \, .
\end{equation}
We may verify this by using \eqref{ineq:rq:useful}, the definition of $\badshaq$ in \eqref{eq:badshaq:def}, and the inequalities
\begin{equation}\label{eq:Cu:ughh}
    \frac{\badshaq}{2b} + 7(2+ \lceil \log_2 \nmax \rceil ) + 4 \leq \badshaq - 2 \, , \quad \impliedby \quad 1- \frac{1}{2b} > \varepsilon_\Gamma(b-1)(\nmax+1)(7\lceil \log_2 \nmax \rceil + 20 ) \, ,
\end{equation}
the second of which follows from \eqref{eq:eps:still:doing:it}.

Next, we claim that due to our choice of $\dpot$, we have
\begin{equation}
\label{eq:CF:new}
\tilde\lambda_q \lambda_{q+1} \left( \Theta_{q+1}^{\frac{1}{2(\nmax+1)}} \Gamma_{q+1}^{9+\CLebesgue} \right)^{\Upsilon(\nn)} \lambda_{q+1}  \left( \frac{\Gamma_{q+1}^{-1}\lambda_{q+1}r_{q+1,\nn}}{\lambda_{q+1}r_{q+1,\nn}} \right)^{\dpot -1} \left( \lambda_{q+1}^{4} \right)^{3\NindLarge} \leq \frac{\delta_{q+2}}{\lambda_{q+1}^{10}} \, ,
\end{equation}
and
\begin{equation}
\label{eq:CF:new:2}
(\delta_{q+1,\nn} \Gamma_{q+1}^{\CLebesgue+5} r_{q+1,\nn} \lambda_{q,\nn} \lambda_{q+1})
\left(\frac{\Gamma_{q+1}^{-1}\lambda_{q+1}r_{q+1,\nn}}{\lambda_{q+1}r_{q+1,\nn}} \right)^{(\dpot -1)}
\left( \lambda_{q+1}^4 \right)^{3 \Nindv}
\leq \frac{\delta_{q+2}}{\lambda_{q+1}^{10}}
\,.    
\end{equation}
The bound \eqref{eq:CF:new} follows from \eqref{eq:dpot:DEF} and \eqref{eq:trickery:trickery}, while \eqref{eq:CF:new:2} follows from \eqref{eq:CF:new} and the parameter inequality $\delta_{q+1,\nn} \Gamma_{q+1}^{\CLebesgue+5} r_{q+1,\nn} \lambda_{q,\nn} \lambda_{q+1} \leq \tilde \lambda_q \lambda_{q+1}^2$ . 

For estimating the stresses emerging from the divergence correctors, we shall need the bound
\begin{align}
r_{q+1,\nn}^2 \delta_{q+1,\nn} \Gamma_{q+1}^{13} \leq \Gamma_{q+1}^{\shaq-1} \delta_{q+2}
\,,
\label{eq:div:cor:ineq:1}
\end{align}
which follows from \eqref{ineq:rq:useful}, \eqref{eq:crazy:const:G:ineq}, and \eqref{eq:hopeless:mess:new} and implies that
\begin{align}
\delta_{q+1,\nn} \Gamma_{q+1}^9 \lambda_{q,\nn} \lambda_{q+1}^{-1} = r_{q+1,\nn}^2 \Gamma_{q+1}^{13} \delta_{q+1,\nn} \leq \Gamma_{q+1}^{\shaq-1} \delta_{q+2}
\,.
\label{eq:div:cor:ineq:3}
\end{align} 
We furthermore need that
\begin{align}
\Gamma_{q}^{\badshaq} \Gamma_{q+1}^{14 \Upsilon(\nn) + 7}
\leq \Gamma_{q}^{\badshaq} \Gamma_{q+1}^{14 \Upsilon(\nmax) + 7}
\leq \Gamma_{q+1}^{\badshaq-1}
\,, 
\label{eq:div:cor:ineq:2}
\end{align} which in turn follows from $\badshaq \geq \frac{b}{b-1}(8 + 14 \Upsilon(\nmax))$, which is a consequence of \eqref{eq:eps:still:doing:it} and \eqref{eq:upsa:bound}. 

In order to estimate the transport and Nash errors in $L^1$ in Sections~\ref{ss:stress:transport} and \ref{ss:stress:Nash}, we claim that
\begin{equation}\label{eq:drq:identity}
\Gamma_{q+1}^{{\CLebesgue}+4} \delta_{q+1,\nn}^{\sfrac{1}{2}}\tau_q^{-1} r_{q+1,\nn} \lambda_{q+1}^{-1} \leq \Gamma_{q+1}^{\shaq-1} \delta_{q+2}
\,.
\end{equation}
In order to verify \eqref{eq:drq:identity}, we note that by \eqref{def:delta:q:actual}, \eqref{def:Gamma:q:actual}, \eqref{def:tau:q:actual}, \eqref{eq:rqn:perp:definition},   \eqref{eq:delta:appendix:def},  the definition of $\CLebesgue$ in \eqref{eq:CLebesgue:DEF}, and the previously established parameter inequalities \eqref{eq:crazy:const:G:ineq} and \eqref{eq:hopeless:mess:new}, the left side of \eqref{eq:drq:identity} is bounded from above by
\begin{align*}
&\Gamma_{q+1}^{\CLebesgue+\cstarzero+ 13} (\delta_{q+1,\nn} \lambda_{q,\nn})^{\sfrac 12} (\delta_q \lambda_q)^{\sfrac 12}
\lambda_q^{\sfrac 12} \lambda_{q+1}^{-\sfrac 32}
\notag\\
&\qquad \leq
\Gamma_{q+1}^{\CLebesgue+\cstarzero+ 13} (\Gamma_{q+1}^{-\CLebesgue-9 \shaq } \lambda_{q+1} \delta_{q+2})^{\sfrac 12} (\delta_{q+2} \lambda_{q+1} \lambda_{q+1}^{2\beta b - 1} \lambda_q^{1-2\beta})^{\sfrac 12}
\lambda_q^{\sfrac 12} \lambda_{q+1}^{-\sfrac 32}
\notag\\
&\qquad \leq
\Gamma_{q+1}^{\frac 12\CLebesgue+\cstarzero+ 10 + \frac 12 \mathsf{C_R}}  
\lambda_q^{(\beta b + \beta - 1)(b-1)}   (\Gamma_{q+1}^{\shaq-1} \delta_{q+2})
\,.
\end{align*}
Thus, \eqref{eq:drq:identity} holds since $\eps_\Gamma (\frac 12\CLebesgue+\cstarzero+ 10 + \frac 12 \mathsf{C_R}) + 2 \beta b < 1$, in view of \eqref{eq:b:DEF} and \eqref{eq:eps:gamma:1:new}. To estimate the transport and Nash errors in $L^\infty$, we finally need that
\begin{equation}\label{eq:transport:Loo:ineq}
     \Gamma_{q}^{\frac{\badshaq}{2}} 
\Gamma_{q+1}^{\badshaq + 7 \Upsilon(\nmax) + 21 + 4\nn } \Theta_q^{\sfrac 12}   \lambda_q  r_{q+1,\nn}^{-1} \lambda_{q+1}^{-1} \leq \Gamma_{q+1}^{\badshaq-1} \, ,
\end{equation}
which follows from the definition of $\badshaq$ in \eqref{eq:badshaq:def}, \eqref{ineq:rq:useful}, and \eqref{eq:eps:gamma:blechhhh}.

In Remark~\ref{rem:ellthree:regularity}, have have used that 
\begin{equation}
    \lim_{(\beta,b) \to (\sfrac 12-, 1+)}
    \frac{2\beta b}{(b-1)(\badshaq\varepsilon_\Gamma+\sfrac 12)} + 2 \rightarrow \infty \label{eq:solving:for:p:back}
\end{equation}
which is a consequence of the choice of $\badshaq$ in \eqref{eq:badshaq:def}.

We conclude this section by verifying a few inequalities concerning the parameter $\Nfn$, which counts the number of available space-plus-material derivative for the residual stress $\RR_{q,n}$. This verification is the same as in~\cite[Section~9.3]{BMNV21}. For all $0 \leq n \leq \nmax$ we require that  
\begin{subequations}
\label{eq:Luigi:is:Mario:s:brother}
\begin{align}
\Nindt, 2\Ndec + 4 
&\leq \lfloor \sfrac{1}{2}\left(\Nfn-\NcutSmall-\NcutLarge-5\right) \rfloor - \dpot \,,
\label{eq:lambdaqn:identity:3} \\
14 \NindLarge
&\leq \NN{\textnormal{fin},\textnormal{n}} - \NcutSmall-\NcutLarge - 2\Ndec - 9 \,,
\label{eq:Nfinn:inequality} \\
6\NindLarge  
&\leq \lfloor \sfrac{1}{2}\left(\Nfn-\NcutSmall-\NcutLarge-6\right) \rfloor - \dpot  \,,
\label{eq:nfnn:mess} \\
6\NindLarge 
&\leq \lfloor \sfrac{1}{4}\left(\Nfn-\NcutSmall-\NcutLarge-7\right) \rfloor\,.
\label{eq:nfnn:mess:2}
\end{align}
\end{subequations}
for all $0\leq n\leq\nmax$.  Additionally for  $0\leq \nn<n \leq \nmax$, we require that 
\begin{align}
\lfloor \sfrac{1}{2}\left(\Nfnn-\NcutSmall-\NcutLarge-6\right) \rfloor - \dpot \geq \Nfn
\label{eq:nfnn:nfn:mess}
\end{align}
holds. The inequality \eqref{eq:nfnn:nfn:mess} is a direct consequence of the recursive formula~\eqref{def:Nfn:formula}  and of the fact that the sequence $\Nfn$ is monotone decreasing with respect to $n$. Using \eqref{eq:Nfn0:def} and \eqref{def:Nfn:formula} one may show that
$$
\Nfn \geq 2^{-n} \NN{\textnormal{fin},0} - (2 \dpot + \NcutSmall + \NcutLarge + 8)
\,.
$$
Noting that the bounds \eqref{eq:Luigi:is:Mario:s:brother} are most restrictive for $n = \nmax$, they now readily follow from  \eqref{eq:Nfin:DEF}.

\appendix
\section{Auxiliary lemmas}

\newtheorem{innercustomlemma}{Lemma}
\newenvironment{customlemma}[1]
  {\renewcommand\theinnercustomlemma{#1}\innercustomlemma}
  {\endinnercustomlemma}
  
\subsection{\texorpdfstring{$L^p$}{Lp} decorrelation}  
\label{sec:appendix}

In order to estimate the perturbation in $L^p$ spaces as well as terms appearing in the Reynolds stress we will need a combination of~\cite[Lemma A.7]{BMNV21} and~\cite[Remark A.9]{BMNV21}, which we recall next.
\begin{lemma}[\bf $L^p$ decorrelation with flows]
\label{l:slow_fast}
Let $p\in\{1,2\}$, and fix integers $N_\circ \geq \Ndec\geq 1$. 
Suppose $f \colon \R^3 \times \R \to \R$ and let $\Phi \colon \R^3\times \R \to \R^3$ be a vector field advected by an incompressible velocity field $v$, i.e. $D_{t}\Phi = (\partial_t + v\cdot \nabla) \Phi=0$. Denote by $\Phi^{-1}$ the inverse of the flow $\Phi$,  which is the identity at a time slice which intersects the support of $f$. Assume that for some $\lambda , \nu, \tilde \nu \geq 1$ and $\const_f>0$  the functions $f$ satisfies 
\begin{align}
\norm{D^N D_{t}^M f}_{L^{p}}&\les \const_f \lambda^N\MM{M,N_{t},\nu,\tilde\nu} \notag
\end{align}
for all $N + M \leq N_\circ$, and that $\Phi$, and $\Phi^{-1}$ are bounded as
\begin{align}
\norm{D^{N+1} \Phi}_{L^{\infty}(\supp f)} +  \norm{D^{N+1} \Phi^{-1}}_{L^{\infty}(\supp f)} 
&\les \lambda^{N}\notag
\end{align}
for all $N\leq N_\circ$.
Lastly, suppose that $\varphi$ is $(\T/ \mu)^3$-periodic, and that there exist parameters  $\tilde\zeta \geq \zeta \geq \mu$  and $\const_\varphi>0$ such that
\begin{align}
\norm{D^N \varphi}_{L^p} \les \const_\varphi \MM{N, N_x, \zeta, \tilde\zeta} \, \label{eq:slow_fast_4}
\end{align}
for all $0\leq N \leq N_\circ$.  If the parameters $$\lambda\leq \mu \leq \zeta \leq \tilde\zeta$$ satisfy
\begin{align}
 \tilde\zeta^4 \bigl(2 \pi \sqrt{3} \lambda \mu^{-1} \bigr)^{\Ndec}
 \leq 1 
 \label{eq:slow_fast_3}
\,,
\end{align}
and we have
\begin{equation}\notag
2 \Ndec + 4 \leq N_\circ \,,
\end{equation}
then the bound
\begin{align}
\norm{D^N  D_{t }^M \left( f\; \varphi\circ \Phi
\right)}_{L^p}
&\les \const_f \const_\varphi \MM{N, N_x, \zeta, \tilde\zeta}  \MM{M,M_{t},\nu,\tilde\nu}
\label{eq:slow_fast_5} 
\end{align}
holds for $N+ M \leq N_\circ$ and $M \leq N_\circ - 2 \Ndec - 4$.  
\end{lemma} 

\subsection{Inversion of the divergence}

\label{sec:inverse:divergence}
Given a vector field $G^i $, a zero mean periodic function $\varrho$ and an incompressible flow $\Phi$,  our goal in this section is to write $G^{i}(x) \varrho(\Phi(x))$  as the divergence of a symmetric tensor. For this purpose, we use~\cite[Proposition A.18]{BMNV21}.

\begin{proposition}[\bf Intermittency-friendly inverse divergence]
\label{prop:intermittent:inverse:div}
Fix an incompressible vector field $v$ and denote its material derivative by $D_t = \partial_t + v\cdot\nabla$.  Fix integers $N_* \geq M_* \geq   1$. Also fix $\Ndec, \dpot \geq 1$ such that  $N_* - \dpot \geq 2\Ndec + 4$, and  let $p \in \{1,\infty\}$.

Let $G$ be a vector field and assume there exists a constant $\const_{G} > 0$ and parameters $\lambda, \nu\geq 1$ such that 
\begin{align}
\norm{D^N D_{t}^M G}_{L^p}\lesssim \const_{G} \lambda^N\MM{M,M_{t},\nu,\tilde\nu}
\label{eq:inverse:div:DN:G}
\end{align}
for all $N \leq N_*$ and $M \leq M_*$.

Let $\Phi$ be a volume preserving transformation of $\T^3$, such that 
\[
D_t \Phi = 0 \,
\qquad \mbox{and} \qquad
\norm{\nabla \Phi - \Id}_{L^\infty(\supp G)} \leq \sfrac 12 \,.
\] 
Denote by $\Phi^{-1}$ the inverse of the flow $\Phi$,  which is the identity at a time slice which intersects the support of $G$.
Assume that  the velocity field $v$ and the flow functions $\Phi$ and $\Phi^{-1}$ satisfy the following bounds 
\begin{align}
\norm{D^{N+1}   \Phi}_{L^{\infty}(\supp G)} + \norm{D^{N+1}   \Phi^{-1}}_{L^{\infty}(\supp G)} 
&\les \lambda'^{N}
\label{eq:DDpsi}\\
\norm{D^ND_t^M D v}_{L^{\infty}(\supp G)}
&\les \nu \lambda'^{N}\MM{M,M_{t},\nu,\tilde\nu}
\label{eq:DDv}
\,,
\end{align}
for all $N \leq N_*$, $M\leq M_*$, and some $\lambda'>0$. 

Lastly, let $\varrho,\vartheta \colon \T^3 \to \R$ be two zero mean functions with  the following properties:
\begin{enumerate}[(i)]
\item \label{item:inverse:i} there exists $\dpot \geq 1$ and a parameter $\zeta\geq 1$ such that $\varrho (x) = \zeta^{-2\dpot } \Delta^\dpot \vartheta(x)$
\item \label{item:inverse:ii} there exists a parameter $\mu\geq 1$ such that $\varrho$ and $\vartheta$ are $(\sfrac{\T}{\mu})^3$-periodic
\item \label{item:inverse:iii} there exists  parameters $\Lambda\geq \zeta$,  $\const_{*} \geq 1$, and $\alpha \in (0,1]$, such that 
\begin{align}
\norm{D^N \vartheta}_{L^p} \les \const_{*} \Lambda^\alpha \MM{N,2\dpot,\zeta,\Lambda} 
\label{eq:DN:Mikado:density}
\end{align} 
for all $0\leq N \leq \Nfin$.
\end{enumerate}

If the above parameters satisfy
\begin{align}
 \lambda' \leq \lambda \ll \mu \leq \zeta \leq \Lambda  \,,
 \label{eq:inverse:div:parameters:0}
\end{align}
where by $\ll$ in \eqref{eq:inverse:div:parameters:0} we mean that 
\begin{align}
 \Lambda^4 \bigl(2\pi \sqrt{3} \lambda \mu^{-1}\bigr)^{\Ndec} \leq 1
 \,,
 \label{eq:inverse:div:parameters:1}
\end{align}
then, we have that 
\begin{align}
G \; \varrho\circ \Phi  
= \div\left( \divH \left( G \varrho \circ \Phi \right) \right) + \nabla P + E. \label{eq:inverse:div}
\end{align}
where the traceless symmetric stress $\divH( G \varrho \circ \Phi)$ and the scalar pressure $P$ are supported in $\supp G$, and for any fixed $\alpha\in (0,1)$ they satisfy 
\begin{align}
\norm{D^N D_{t}^M \divH \left( G \varrho \circ \Phi \right) }_{L^p} + \norm{D^N D_{t}^M P}_{L^p}
 &\les  \const_{G} \const_{*}   \zeta^{-1} \Lambda^{\alpha} \MM{N,1,\zeta,\Lambda} \MM{M,M_{t},\nu,\tilde\nu} 
\label{eq:inverse:div:stress:1}
\end{align}
for all $N \leq N_* - \dpot $ and $M\leq M_*$. The implicit constants  depend on $N,M,\alpha$ but not $G$, $\varrho$, or $\Phi$. Lastly, for $N \leq N_* - \dpot $ and $M\leq M_*$ the error term $E$  in \eqref{eq:inverse:div} satisfies
\begin{align}
\norm{D^N D_{t}^M E}_{L^p}  
\les \const_{G} \const_{*}   \lambda^\dpot  \zeta^{-\dpot } \Lambda^{\alpha+N} \MM{M,M_{t},\nu,\tilde\nu} 
\,.
\label{eq:inverse:div:error:1}
\end{align}
We emphasize that the range of $M$ in \eqref{eq:inverse:div:stress:1} and \eqref{eq:inverse:div:error:1} is exactly the same as the one in \eqref{eq:inverse:div:DN:G}, while the range of permissible values for $N$ shrank from $N_*$ to $N_* - \dpot $.

Lastly, let $N_\circ, M_\circ$ be integers such that $1 \leq M_\circ \leq N_\circ \leq M_*/2$.
Assume that in addition to the bound \eqref{eq:DDv} we have the following global lossy estimates
\begin{align}
\norm{D^N \partial_t^M v}_{L^\infty(\T^3)}\les  \const_v \tilde \lambda_q^N \tilde \tau_q^{-M}  
\label{eq:inverse:div:v:global}
\end{align}
for all  $M \leq M_\circ$ and $N+M \leq N_\circ + M_\circ$, where 
\begin{align}
\const_v \tilde \lambda_q \les \tilde \tau_q^{-1}, \qquad \mbox{and} \qquad  \lambda' \leq \tilde \lambda_q \leq \Lambda \leq \lambda_{q+1}  \,.
\label{eq:inverse:div:v:global:parameters}
\end{align}
If $\dpot $ is chosen {\em large enough} so that  
\begin{align}
\const_G \const_* \Lambda \left(\lambda \zeta^{-1}\right)^{\dpot -1}   \left(1 +\tau_{q} \max\{\tilde \tau_q^{-1}, \tilde \nu, \const_v \Lambda \}\right)^{M_\circ}
\leq  \delta_{q+2} \lambda_{q+1}^{-10}
\,,
\label{eq:riots:4}
\end{align}
then we may write  
\begin{align}
E 
= \div \left(\divR(G \varrho \circ \Phi)\right) + \fint_{\T^3} G \varrho \circ \Phi dx\,,
\label{eq:inverse:div:error:stress}
\end{align}
where $\divR(G \varrho \circ \Phi)$ is a traceless symmetric stress which satisfies
\begin{align}
\norm{D^N D_{t}^M \divR \left( G \varrho \circ \Phi \right)  }_{L^p}  
\leq  \delta_{q+2}  \lambda_{q+1}^{N-10} \tau_{q}^{-M}
\label{eq:inverse:div:error:stress:bound}
\end{align}
for  $N \leq N_\circ$ and $M\leq M_\circ$.
\end{proposition}

The estimates claimed in Proposition~\ref{prop:intermittent:inverse:div} for  $p=1$ are taken as is from \cite[Proposition A.18]{BMNV21}. The definition/construction of the operators $\divH$ and $\divR$ is independent of $p$.  Then the estimates claimed in Proposition~\ref{prop:intermittent:inverse:div} in the case $p=\infty$ follow from the proof of \cite[Proposition A.18]{BMNV21} after replacing each instance of an $L^p$ bound for $p\neq \infty$ in the proof with an $L^\infty$ bound.

\end{document}